\documentclass[12pt]{amsart}
\usepackage{amsmath,amssymb,amsfonts,amsthm,amscd,textcomp,times,booktabs,mathrsfs}
\usepackage[dvips]{graphicx}
\usepackage{braket}
\usepackage{color,float}
\usepackage{bm}
\usepackage[all]{xy}
\usepackage{enumerate}
\usepackage{mathtools}

\newtheorem{theorem}{Theorem}

\newtheorem{proposition}[theorem]{Proposition}

\newtheorem{definition}[theorem]{Definition}
\newtheorem{corollary}[theorem]{Corollary}
\newtheorem{remark}[theorem]{Remark}
\newtheorem{example}{Example}
\newtheorem{problem}[theorem]{Problem}

\numberwithin{equation}{section}
\numberwithin{theorem}{section}
\theoremstyle{definition}
\newtheorem*{acknowledgements}{Acknowledgements}

\newcommand{\bs}{\boldsymbol} 

\newcommand{\din}{\rotatebox{90}{\ensuremath{\in}}}

\begin{document}
\title{Relative Ding and $K$-stability of toric Fano manifolds in low dimensions}

\author{Yasufumi Nitta}
\author{Shunsuke Saito}
\author{Naoto Yotsutani$^*$}

\address{Department of Mathematics, Faculty of Science Division II, Tokyo University of Science, 1-3 Kagurazaka, Shinjuku-ku, Tokyo 162-8601, Japan}
\email{nitta@rs.tus.ac.jp}

\address{Department of Mathematics, Faculty of Science Division I, Tokyo University of Science, 1-3 Kagurazaka, Shinjuku-ku, Tokyo 162-8601, Japan}
\email{saito@rs.tus.ac.jp}

\address{Faculty of Education, Kagawa University,
1-1 Saiwai-cho, Takamatsu 760-8521, Japan}
\email{yotsutani.naoto@kagawa-u.ac.jp}
\thanks{$^*$ Corresponding author}
 \dedicatory{}
\maketitle

\noindent{\bfseries Abstract.}
The purpose of this paper is to clarify all of the uniformly relatively Ding stable toric Fano threefolds and fourfolds as well as unstable ones.  
The key player in our classification result is the Mabuchi constants, which can be calculated by combinatorial data of the associated moment polytopes due to the work of Yao \cite{Y17}. 
In Tables \ref{table:toricFano2RelDSta}--\ref{table:toricFano4RelDSta}, we give the list of uniform relative Ding stability of all toric Fano manifolds in dimension up to four with the values of the Mabuchi constants. 

As an application of our main theorem (Theorem $\ref{thm:main}$), we clarify the difference between relative $K$-stability and relative Ding stability 
by considering some specific toric Fano manifolds (Corollaries $\ref{cor:difference}$ and $\ref{cor:BottFano}$). In the proof of Corollary $\ref{cor:BottFano}$, we used Bott tower structure of relatively Ding unstable toric Fano manifolds.


\section{Introduction}
In \cite{M01}, Mabuchi generalized the notion of K\"ahler-Einstein metrics to Fano manifolds with non-vanishing Futaki invariant as follows: 
let $X$ be an $n$-dimensional Fano manifold defined over $\mathbf{C}$ with complex structure $J$.
Denoting the identity component of the automorphism group of $(X,J)$ by $\mathrm{Aut}^0(X)$, we fix a maximal algebraic torus $T$ of $\mathrm{Aut}^0(X)$.
Let 
$S\coloneqq \operatorname{Hom}(\mathbf C^*, T)\otimes_{\mathbf{Z}} S^{1}$ be the maximal compact torus of $T$ and $\mathcal{H}(X)^{S}$ the space of all $S$-invariant K\"ahler forms representing $c_{1}(X)$. 
Let $F_{\omega}$ be the Ricci potential of $\omega$, that is, 
\[
\operatorname{Ric}(\omega)-\omega=\frac{\sqrt{-1}}{2\pi}\partial\overline{\partial}F_{\omega}, \quad \text{and}\quad
\int_{X}(1-e^{F_{\omega}})\,\frac{\omega^{n}}{n!}= 0. 
\]
A K\"ahler form $\omega\in\mathcal{H}(X)^{S}$ is said to be \emph{Mabuchi soliton}, if the Riemannian gradient
\[
\operatorname{grad}_{\omega}(1-e^{F_{\omega}})\
\]
with respect to $\omega$ is real holomorphic.
For each $\omega\in\mathcal{H}(X)^{S}$, let $\mathfrak{s}_{0}(X, \omega)$ be the subspace of $C^{\infty}(X, \mathbf{R})^{S}$ consisting of all $\phi \in C^{\infty}(X, \mathbf{R})^{S}$ so that the Hamiltonian real vector field $2\pi J \operatorname{grad}_{\omega}\phi$ lies in $\mathrm{Lie}(S)$, where $\mathrm{Lie}(S)$ is the Lie algebra of
the compact torus $S$ and $\int_{X}\phi\, \omega^{n}/n! = 0$. Let $\mathrm{pr}_{\omega} \colon C^{\infty}(X, \mathbf{R})^{S} \to \mathfrak{s}_{0}(X, \omega)$ be the $L^{2}$-orthogonal projection with respect to $\omega$.
Mabuchi also introduced the constant
\begin{align}\label{mabuchi_const}
M_{X}\coloneqq\max_{X}\operatorname{pr}_{\omega}(1-e^{F_{\omega}})
\end{align}
and showed that this is independent of the choice of $\omega$. In particular, 
the existence of  Mabuchi solitons implies $M_{X}<1$. 
Throughout the paper, we call $M_X$ the \emph{Mabuchi constant} of $X$. 

Recently, Mabuchi solitons were rediscovered by Yao \cite{Y17} 
following Donaldson's infinite dimensional moment map picture. He showed that 
Mabuchi soliton is a critical point of the functional
\[
\mathcal{H}(X)^{S}\ni \omega\mapsto \int_{X}(1-e^{F_{\omega}})^{2}\,\frac{{\omega}^{n}}{n!}\in\mathbf{R}, 
\]
which can be interpreted as the norm square of Donaldson's new moment map given in \cite{D17}. 
For a general (not necessarily toric) Fano manifold, the corresponding version of Yau-Tian-Donaldson conjecture regarding Mabuchi soliton has been studied by many researchers.
Han-Li \cite{HL22} 
proved that \emph{uniform relative Ding stability} and $M_X<1$ imply the existence of Mabuchi solitons building upon the work of variational approach \cite{BBJ21}.
In \cite{Y19}, Yao showed that additional assumption $M_X<1$ can be derived from uniform relative Ding stability.
We also mention that Han-Li's work \cite{HL22} deals with more general $g$-solitons which include Mabuchi solitons and K\"ahler-Ricci solitons on (possibly singular) Fano varieties. 
In the toric case, Yao proved equivalences of $M_{X}<1$, the existence of Mabuchi solitons, 
and uniform relative Ding stability for toric Fano manifolds (or more generally toric Fano orbifolds). 

The notion of uniform relative Ding stability for toric Fano manifolds was introduced by Yao \cite{Y17} as a generalization of Ding stability (\cite{B16} and \cite{F18}) and uniform Ding stability 
({\cite{BBJ21}} and \cite{BHJ17}). 
Yao's work is the toric reduction of relative Ding stability of general Fano manifolds. One of the advantage of our toric setting is that everything can be described in terms of convex functions on the associated moment polytopes. In fact, Delcroix \cite{D20} and Li-Li \cite{LL22} showed that there is a one-to-one correspondence between $T$-equivariant test configurations of a polarized toric manifold and rational piecewise affine convex functions on the associated moment polytope. See Sections $\ref{sec:relDing}$ for more details.

Let $v\coloneqq \operatorname{grad}_{\omega}\operatorname{pr}_{\omega}(1-e^{F_{\omega}})$ be the extremal K\"ahler vector 
field of $(X, -K_{X})$ with respect to the maximal algebraic torus $T=(\mathbf{C^*})^{n}$ defined in Futaki-Mabuchi \cite{FM95} , 
and $\langle\cdot, \cdot\rangle$ the bilinear pairing for $\mathbf{C^*}$-actions introduced by Sz\'ekelyhidi \cite{S07}. 
Uniform relative Ding stability requires the existence of a constant $\delta>0$ such that
\[
{D}_v^{NA}(\mathcal{X}, \mathcal{L})\coloneqq D^{NA}(\mathcal{X}, \mathcal{L})+\langle(\mathcal{X}, \mathcal{L}), v\rangle\ge \delta J_{T}^{NA}(\mathcal{X}, \mathcal{L})
\]
holds for any 
$T$-equivariant test configuration $(\mathcal{X}, \mathcal{L})$ of $(X, -K_{X})$, 
where $D^{NA}$ and $J_{T}^{NA}$ denote the \emph{non-Archimedean Ding functional} (introduced in \cite{B16} and \cite{F18}) and the \emph{non-Archimedean reduced J-functional} with respect to $T$ (introduced in \cite{H16}) respectively. 
Later, this condition for (not necessarily toric) Fano manifolds was studied in \cite{Y19}. 

In this paper, we focus on (uniform) relative Ding stability and relative $K$-stability of toric Fano manifolds. 
See Section $\ref{sec:relDing}$ for definitions of relative Ding/$K$-stability.
More precisely, we firstly determine which toric Fano manifolds in dimension three and four are uniformly relatively Ding stable by direct computations of the Mabuchi constants. 
We secondly find explicit examples of toric Fano manifolds which are relatively $K$-stable but which are relatively Ding unstable (see Corollaries $\ref{cor:difference}$ and $\ref{cor:BottFano}$).
For the proof of Theorem $\ref{thm:main}$, we refer the reader to see Tables \ref{table:toricFano3RelDSta}--\ref{table:toricFano4RelDSta}.
\begin{theorem}\label{thm:main}
\begin{enumerate}
\item In dimension three, 12 toric Fano manifolds out of 18 are uniformly relatively Ding stable and the rest are relatively Ding unstable. 
\item In dimension four, 49 toric Fano manifolds out of 124 are uniformly relatively Ding stable  and the rest are relatively Ding unstable. 
\end{enumerate}
\end{theorem}

Note that the uniform relative Ding stability of all Gorenstein toric del Pezzo surfaces was verified  
in \cite[Example $5.14$]{Y17} (see also Table \ref{table:toricFano2RelDSta} in Section \ref{sec:dP}  
for smooth surfaces). We simultaneously complete the classification of toric Fano threefolds and fourfolds admitting Mabuchi solitons 
by Yao's result. 
The key theorem to achieve this is the following. (We refer the reader to Section $\ref{sec:Prelim}$ for notations used in Proposition $\ref{Ding-ps}$.)
\begin{proposition}[Yao \cite{Y17}]\label{Ding-ps}
Let $P$ be a Fano polytope. Then we have the following: 
\begin{enumerate}[\upshape(a)]
\item The following conditions are equivalent. 
\begin{enumerate}[\upshape(1)]
\item $X_{P}$ is uniformly relatively Ding stable. 
\item There exists $\delta>0$ such that 
$D_{v}^{NA}(f)\ge \delta\int_{P}f(x)\,dx$ for any normalized rational PL convex function $f$ on $P$. 
\item $M_{X_{P}}=\max_{P}\theta_P<1$, where $\theta_P$ is the potential function defined in Section $\ref{sec:relDing}$. 
\end{enumerate}
\item The following conditions are equivalent. 
\begin{enumerate}[\upshape(1)]
\item $X_{P}$ is relatively Ding stable. 
\item $X_{P}$ is relatively Ding semistable. 
\item $M_{X_{P}}\le 1$. 
\end{enumerate}
\item The following conditions are equivalent. 
\begin{enumerate}[\upshape(1)]
\item $X_{P}$ is relatively Ding unstable. 
\item $M_{X_{P}}> 1$. 
\end{enumerate}
\end{enumerate}
\end{proposition}
For reader's convenience, we shall give a proof of Proposition $\ref{Ding-ps}$ in Section $\ref{sec:relDing}$.
Theorem $\ref{thm:main}$ and Proposition $\ref{Ding-ps}$ offer concrete examples of Fano manifolds which admit Mabuchi solitons.
\begin{corollary}
\begin{enumerate}
\item In dimension three, 12 toric Fano manifolds admit Mabuchi solitons and the others do not.  
\item In dimension four, 49 toric Fano manifolds admit Mabuchi solitons and the others do not.  
\end{enumerate}
\end{corollary}
With remarkable progress made in research in Ding stability, it is known that uniform Ding stability is equivalent to Ding stability for general Fano manifolds.
Keeping this in our mind, we see that our computation also reveals the following corollary.

\begin{corollary}
Up to dimension four, there are no toric Fano manifolds satisfying $M_{X}=1$.
Hence for each toric Fano $n$-fold with $n\leq 4$, relative Ding stability is equivalent to uniform relative Ding stability. 
\end{corollary}

In \cite[Example 5.14]{Y17}, Yao presented examples of 
Gorenstein toric del Pezzo surfaces with $M_{X}=1$. 
However, any example of Fano \emph{manifolds} with $M_X=1$ has not yet known.
Hence it is natural to expect that there are no Fano manifolds with the Mabuchi constant being precisely 1. 
If this is the case, uniform relative Ding stability is equivalent to relative Ding stability for arbitrary toric Fano manifold.

Finally, we clarify the difference between relative $K$-stability and relative Ding stability with specific examples. Toric reduction of relative $K$-stability
was studied by Zhou and the third author in \cite{YZ19}, where they established instability criterion of relative $K$-stability in terms of the associated moment polytope. 
More precisely, for a given Fano polytope $P$, we consider $P^-=\set{x\in P|1-\theta_P(x)\leq0}$ where $\theta_P(x)=\sum\theta_ix_i+c$ is the potential function of $P$ (see, Section $\ref{sec:relDing}$ for the definition).
If $\mathrm{Int}(P^-)\neq \emptyset$ and the inequality
\begin{equation}\label{ineq:insta}
1-c<\frac{\int_{P^-}(1-\theta_P)^2dx}{\mathrm{vol}(P^-)}
\end{equation}
holds, then the corresponding toric Fano manifold $(X_P,-K_{X_P})$ is relatively $K$-unstable (see \cite[Theorem $1.4$ $(2)$]{YZ19}).
Although the instability criterion $\eqref{ineq:insta}$ is correct, there is a technical mistake in \cite{YZ19} when they apply 
$\eqref{ineq:insta}$ into six toric Fano threefolds 
$\mathcal B_1$, $\mathcal C_2$, $\mathcal D_1$, $\mathcal D_2$, $\mathcal E_1$ and $\mathcal E_2$. 
Here we used same notation in Section \ref{sec:ToricFano3}, Table \ref{table:toricFano3RelDSta}. In fact, the first and the third authors recomputed the values of the right hand side in $\eqref{ineq:insta}$, 
and they found that the above six toric Fano threefolds ($\mathcal B_1$, $\mathcal C_2$, $\mathcal D_1$, $\mathcal D_2$, $\mathcal E_1$ and $\mathcal E_2$) do \emph{not} satisfy the instability condition $\eqref{ineq:insta}$
by using the computational package \cite{Normaliz} (see Section $\ref{sec:ToricFano3}$ for the detail). 

However, we can confirm relative $K$-stability for some of toric Fano manifolds by the existence theorem of extremal K\"ahler metrics due to 
Guan \cite[Theorem 2]{G95}, Hwang \cite[Corollary 1.2]{H94}, Apostolov-Calderbank-Gauduchon-T\o nnesen-Friedman \cite[Proposition 11]{ACGT-F08} as follows.
\begin{proposition}[Guan, Hwang, ACGT-F]\label{prop:extremal}
(1)~~ For the toric Fano threefold $\mathbf P_{\mathbf P^2}(\mathcal O\oplus\mathcal O(2))$ which is labeled $\mathcal B_1$ in Table \ref{table:toricFano3RelDSta}, 
it admits an extremal K\"ahler metric in every K\"ahler classes.

\vskip5pt

\noindent (2) For the toric Fano fourfold $\mathbf P_{\mathbf P^3}(\mathcal O\oplus\mathcal O(k))$ for $k=3$ (resp. $k=2$)
which is labeled $B_1$ (resp. $B_2$) in Table \ref{table:toricFano4RelDSta}, it admits an extremal K\"ahler metric in every K\"ahler classes.
\end{proposition}
In \cite{ZZ08}, Zhou-Zhu proved that relative $K$-stability is a necessary condition for the existence of extremal K\"ahler metrics on polarized toric manifolds.
Combining this, Proposition $\ref{prop:extremal}$ and Theorem $\ref{thm:main}$, we have the following.
\begin{corollary}\label{cor:difference}
\begin{enumerate}
\item The toric Fano threefold $\mathcal B_1$ is relatively $K$-stable, but it is relatively Ding unstable.
\item Moreover, toric Fano fourfolds $B_1$ and $B_2$ are relatively $K$-stable, but they are relatively Ding unstable.
\end{enumerate}
\end{corollary}

After we posted the third version of this paper on arXiv, the third author found that we may adapt the \emph{Bott Manifold structure} for discovering more examples of toric Fano manifolds which clarify the difference between 
relative Ding stability and relative $K$-stability as follows.

Recall that a K\"ahler manifold is called \emph{Calabi dream manifold} if it admits an extremal K\"ahler metric in each K\"ahler class \cite[Definition $1.6$]{CC18}.
Hence the Hirzebruch surfaces, toric Fano manifolds appeared in Proposition $\ref{prop:extremal}$ are examples of Calabi dream manifolds.
The key to make them Calabi dream manifolds is \emph{admissible construction} that has been so successful in producing many explicit examples of extremal K\"ahler metrics \cite{ACGT-F08}.
This construction was applied to \emph{Bott manifolds} and many interesting explicit examples of extremal K\"ahler metrics on certain stage $n$ Bott manifolds were discovered \cite{BCT-F19}.
See, Section \ref{sec:Bott} for the definition of the \emph{stage $k$ Bott manifold of the Bott tower of height $n$}.
We remark that any Bott manifolds is toric since it can be constructed as a quotient of $n$ copies of $\mathbf{C}_*^2:=\mathbf{C}^2\setminus \set{0}$ by an algebraic torus $T=(\mathbf{C}^*)^n$.
The number of holomorphically nontrivial $\mathbf{P}^1$-bundles in the Bott tower is called the \emph{twist}.
In \cite[Proposition $5.4$]{BCT-F19}, 
it has been proved that the following types of Bott manifolds with twist one admit an extremal K\"ahler metric in every K\"ahler class. 
\begin{proposition}[BCT-F]\label{prop:Bott}
For $\boldsymbol k=(k_1,k_2, \ldots , k_{n-1})\in \mathbb Z^{n-1}$ with $\prod_{i=1}^{n-1}k_i\neq 0$, let $X_n(\boldsymbol k)$ be the stage $n$ Bott manifold of the Bott tower of height $n$
\begin{equation}\label{eq:twist1}
\mathbf{P}_{(\mathbf P^1)^{n-1}}(\mathcal O\oplus \mathcal O(k_1,k_2, \ldots ,k_{n-1})) \rightarrow  (\mathbf P^1)^{n-1}
\end{equation}
which is the $\mathbf P^1$-bundle over the product of projective lines. Then $X_{n}(\bs k)$ is a Calabi dream manifold.
\end{proposition}
For our purpose, we shall consider \emph{Fano Bott manifold} which is a special class of toric Fano manifold.
According to the classification results of the Bott manifolds \cite{CLMP20, HKM20, CLMP21}, there are five classes of Fano Bott manifolds at stage three, and $13$ classes of Fano Bott manifolds at stage four.
More precisely, we have the following.
\begin{proposition}
Fano Bott manifolds at stage three are exactly type $\mathcal C$ in Table $\ref{table:toricFano3RelDSta}$, and 
Fano Bott manifolds at stage four are exactly type $L$ in Table $\ref{table:toricFano4RelDSta}$.
\end{proposition}
See \cite[Remark $4.5$]{CLMP20}, \cite[Section $5.3.2$]{HKM20}, and see also \cite[Proposition $3.21$]{BCT-F19} for the proof.
Among these $18$ classes of Fano Bott manifolds at stage $n$ ($n\leq 4$), there is only one Fano Bott manifold with twist one which is 
relatively Ding unstable, that is,
\begin{corollary}\label{cor:BottFano}
The toric Fano fourfold $\mathbf{P}_{(\mathbf P^1)^{3}}(\mathcal O\oplus \mathcal O(1,1,1))$ which is labeled $L_1$ in Table $\ref{table:toricFano4RelDSta}$, 
is a relatively Ding unstable Calabi dream manifold. 
\end{corollary}
We refer the reader to Section $\ref{sec:Bott}$ for further details.

\begin{acknowledgements}
This work was supported by first author's JSPS KAKENHI Grant Number JP$21$K$03234$, second author's JP$20$K$14321$,
and third author's JP$18$K$13406$. The third author would like to thank Professors, Vestislav Apostolov, C.W. T$\o$nnesen-Friedman, Mikiya Masuda, Yi Yao for helpful comments.
Finally, we are grateful to the referees for comments which improved the presentation of our manuscript.
\end{acknowledgements}

\section{Preliminaries}\label{sec:Prelim}
\subsection{Polarized toric varieties}\label{sec:PolarizedToric}
We first recall the construction of polarized toric varieties, as can be found in  \cite{CLS11}. 
Let $P\subset\mathbf{R}^{n}$ be an $n$-dimensional integral convex polytope. 
Define
 \[
 C(P)\coloneqq\{(x, r)\in\mathbf{R}^n\times\mathbf{R}\mid r>0,\, r^{-1}x\in P\}\cup\{0\}
 \]
to be the cone over $P\times\{1\}$ with the vertex at the origin. 
Then the semigroup $S_P\coloneqq C(P)\cap\mathbf{Z}^{n+1}$ is finitely generated by Gordan's lemma.
Let $\mathbf{C}[S_{P}]$ denote its semigroup algebra. The character corresponding to $(m,k)\in S_P$ is $\chi^mt^k$ and $\mathbf{C}[S_P]$
is graded by height, i.e., $\operatorname{deg}(\chi^mt^k)=k$. Consequently, we obtain the graded $\mathbf{C}$-algebra
\[
\mathbf{C}[S_P]=\bigoplus_{k=0}^{+\infty}R_k, \quad R_k\coloneqq\{f\in \mathbf{C}[S_P]\mid\operatorname{deg}f=k\}
\]
from the polytope $P$. We define the polarized toric variety $(X_P, L_P)$ by
\[
(X_P, L_P)\coloneqq(\operatorname{Proj}\mathbf{C}[S_P], \mathcal{O}_{X_P}(1)).
\]
It is well-known that $X_P$ is a smooth projective variety if and only if $P$ is Delzant.

An alternative way to construct toric varieties is to take a GIT quotient of affine $n$-space by a linear torus action.
Then it turns out that every $n$-dimensional projective toric variety can be written as a quotient of $\mathbf C^n$ by some torus action
(cf. Section $\ref{sec:Bott}$).

\subsection{Fano polytopes}
Let $P\subset\mathbf{R}^n$ be an $n$-dimensional integral Delzant polytope containing the origin in its interior. 
The dual polytope $P^{\vee}$ of $P$ is defined by
\[
P^{\vee}=\{y \in \mathbf{R}^n \mid \langle x,y\rangle\ge -1 \text{ for all } x \in P\}, 
\]
which is an $n$-dimensional rational polytope containing the origin in its interior.
$P$ is said to be \emph{reflexive} if $P^{\vee}$ is again an integral polytope.
We call $P$ a \emph{Fano} polytope if it is a reflexive Delzant polytope. 
Remark that we considered the Euclidean space $\mathbf{R}^{n}$ 
including $P$ as $M_\mathbf{R}=\operatorname{Hom}(T, \mathbf{C}^*) \otimes\,\mathbf{R}$, 
and its dual $\mathbf{R}^{n}$ 
containing the extremal K\"ahler vector field as $N_\mathbf{R}=\operatorname{Hom}(\mathbf{C}^*, T)\otimes\mathbf{R}$, the latter identified with the real part of $\operatorname{Lie}(T)$. 
For example, the standard $n$-simplex is a Fano polytope. 
It is well-known that there is a bijective correspondence
between isomorphism classes of Fano polytopes and isomorphism classes of toric Fano manifolds with the anticanonical polarization. 

We will use the notation of the Batyrev classification (3-dimensional case) \cite{B81} 
(see also \cite{WW82} for another geometrical description of toric Fano threefolds), and the Batyrev-Sato classification (4-dimensional case) of Fano polytopes (\cite{B98} and \cite{S00}). 

\subsection{Relative Ding stability}\label{sec:relDing}
We recall the notion of uniform relative Ding stability of toric Fano manifolds. 
Suppose that $P$ is an $n$-dimensional Fano polytope in $\mathbf{R}^{n}$. 
A function $f\colon P\to\mathbf{R}$ is called \emph{rational piecewise affine convex} or \emph{rational PL convex} if $f$ is a convex function of the form
\[
f(x)=\max\{\ell_1(x), \ldots, \ell_m(x)\}
\]
with each $\ell_k$ a rational affine function. 
Moreover, $f$ is said to be \emph{normalized} if it satisfies
\[
\min_{P}f=f(0)=0. 
\]
Note that each rational PL convex function can be assumed to be normalized after adding some rational affine function. 

Here and hereafter, we denote the maximal algebraic torus $(\mathbf{C}^*)^n$ by $T$.
The importance of such functions comes from the fact that each rational PL convex function $f$ corresponds to a \emph{$T$-equivariant} test configuration of $(X_{P}, -K_{X_{P}})$ as follows.
We refer the reader to \cite[Section $4.2$]{D02} for more detailed description to explain why all $T$-equivariant test configurations can be induced by a rational PL convex function.

For a fixed integer $L$ with $L>\max_{P}f$, we obtain an $(n+1)$-dimensional rational convex polytope 
\begin{equation}\label{toric_tc}
\mathcal{P}\coloneqq\{
(x, y)\in\mathbf{R}^{n}\times\mathbf{R}\mid x\in P,\,  0\leq y \leq L-f(x)
\}
\end{equation}
and some multiple of $ \eqref{toric_tc}$ gives an $(n+1)$-dimensional polarized toric variety $(\mathcal{X}_{f}, \mathcal{L}_{f})$ as in Section \ref{sec:PolarizedToric}. 
This can be regarded as a one parameter degeneration of $(X_P, -K_{X_P})$ into a (possibly non-normal) toric variety defined by the ``roof" of $r\mathcal P$.
Specifically, it has the structure of a $T\times\mathbf{C}^*$-equivariant flat projective toric morphism 
$\pi\colon\mathcal{X}_{f}\to{\mathbf P^1}$ and then the test configuration is $\mathcal{X}_{f}\setminus \pi^{-1}(\infty)\to \mathbf P^1 \setminus \set{\infty}$.
Note that the test configuration $(\mathcal{X}_{\ell}, \mathcal{L}_{\ell})$ for affine $\ell$ is \emph{product} in the sense that $\mathcal{X}_{\ell}\cong X_{P}\times\mathbf{P}^{1}$, $T$-equivariantly (see \cite{D02}). 
Conversely, for each $T$-equivariant test configuration $(\mathcal{X}, \mathcal{L})$, there exists a rational piecewise affine convex function $f$ such that $(\mathcal{X}, \mathcal{L}) = (\mathcal{X}_{f}, \mathcal{L}_{f})$
(cf. \cite{D20} and \cite{LL22}). 
Moreover, the non-Archimedean $K$-energy can be written as
\[
M^{NA}(\mathcal{X}_{f}, \mathcal{L}_{f})=\frac{1}{\mathrm{vol}(P)}\left(\int_{\partial P}f(x)\,d\sigma -n \int_P f(x)\,dx \right)
\]
where $\sigma$ denotes the $(n-1)$-dimensional Hausdorff measure on the boundary $\partial P$.

Let $v$ be the extremal K\"ahler vector field of the corresponding toric Fano manifold $(X_P, -K_{X_{P}})$ with respect to the maximal algebraic torus $T$. 
Let $\theta_P$ be the potential function of the extremal K\"ahler vector field $v$ \cite{FM95}, which is affine linear on $P$, and normalized by $\int_P\theta_P\, dx=0$.
For a convex function $f$, we define the functional
\[
\mathscr{L}_P(f)=\int_{\partial P}f(x)\, d\sigma-\int_P(n+\theta_P)f(x)\, dx.
\]
Then the potential function $\theta_P$ is uniquely determined by the $n+1$-equations
\begin{equation}\label{ext_aff_fct}
\mathscr{L}_P(1)=0, \qquad \mathscr{L}_P(x_i)=0 \qquad \text{for} \qquad i=1, \dots , n.
\end{equation}
According to \cite[Definition $1.1$]{YZ19}, $(X_P,-K_{X_P})$ is said to be {\emph{relatively $K$-stable}} if $\mathscr{L}_P(f)\geq 0$ for any rational PL convex function $f$
and the equality holds if and only if $f$ is affine linear.
In \cite{Y17}, Yao observed that
the Mabuchi constant defined in \eqref{mabuchi_const} is given by
\[
M_{X_{P}}
=\max_{P}\theta_P 
\]
by the Atiyah-Guillemin-Sternberg convexity theorem (\cite{A82} and \cite{GS82}).

Let $(\mathcal{X}_{f}, \mathcal{L}_{f})$ be a $T$-equivariant test configuration arising from a rational PL convex function $f$ on $P$. 
The \emph{non-Archimedean relative Ding functional} of $(\mathcal{X}_{f}, \mathcal{L}_{f})$ was computed in \cite[Theorem $4.3$]{Y17} as
\begin{align*}
{D}_v^{NA}(f)\coloneqq {D}_v^{NA} (\mathcal{X}_{f}, \mathcal{L}_{f})
&= -f(0)+\frac{1}{\mathrm{vol}(P)}\int_Pf(x)(1-\theta_P(x))\,dx\\
&=\frac{1}{\mathrm{vol}(P)}\int_P(f(x)-f(0))(1-\theta_P(x))\,dx. 
\end{align*}
Observe that the linear function $\theta_P$ is uniquely determined by the conditions $D^{NA}(\ell)=0$ for all affine $\ell$. 
The \emph{non-Archimedean reduced J-functional} of $(\mathcal{X}_{f}, \mathcal{L}_{f})$, first introduced in \cite{H16}, is expressed by
\[
J_{T}^{NA}(f)\coloneqq J_{T}^{NA}(\mathcal{X}_{f}, \mathcal{L}_{f})=\inf_{\ell}\left(\frac{1}{\mathrm{vol}(P)}\int_P (f(x)+\ell(x))\,dx-\min_{P}(f+\ell)\right), 
\]
where $\ell$ runs over all the affine functions on $\mathbf{R}^{n}$. 
It is clear that the functional is invariant under the (torus) action $f \mapsto f + \ell$ for affine $\ell$. We also know that $J_{T}^{NA}(f)=0$ is equivalent to $f$ being affine. 

Using these quantities, we can define (uniform) relative Ding stability of toric Fano manifolds as follows. 
\begin{definition}\label{relDing}
Let $X$ be a toric Fano manifold and $P$ the corresponding Fano polytope. 
\begin{enumerate}[\upshape(1)]
\item $X$ is relatively Ding semistable if $D_v^{NA}(f)\ge 0$ for any rational PL convex function $f$. 
\item $X$ is relatively Ding stable if $X$ is relatively Ding semistable, and $D_v^{NA}(f)=0$ if and only if $f$ is affine. 
\item $X$ is relatively Ding unstable if $X$ is not relatively Ding semistable. 
\item $X$ is uniformly relatively Ding stable if there exists $\delta>0$ such that 
$D_v^{NA}(f)\ge \delta J_{T}^{NA}(f)$ for any rational PL convex function $f$. 
\end{enumerate}
\end{definition}
Since the triviality of non-Archimedean reduced $J$-functional characterizes affine functions, uniform relative Ding stability is a priori stronger than relative Ding stability. 

It was proved in \cite{Y17} that (uniform) relative Ding stability can be rephrased as the condition on the Mabuchi constant as stated in Proposition $\ref{Ding-ps}$. 
We shall give a proof of it for reader's convenience.

\begin{proof}[Proof of Proposition $\ref{Ding-ps}$]
\begin{enumerate}[\upshape(a)]
\item The equivalence between (1) and (2) is coming from the equivalence of non-Archimedean reduced $J$-functional and $L^{1}$-norm on \emph{normalized} convex functions 
(see \cite[Proposition $5.4.1$ $(3)$]{NS19}). 
First we show (2)$\Rightarrow$(3). Let $\ell\coloneqq (1-\theta_P)/\operatorname{vol}(P)$. Note that $D_v^{NA}(\ell)=0$ implies
\[
\ell(0)=\int_P\ell(x)^2\,dx. 
\]
As a first step, we prove $\ell\ge 0$ on $P$ by contradiction. Assume $\{\ell<0\}$ is non-empty. Then
\[
\int_{\{\ell <0\}}\ell(x)^2\,dx>0. 
\]
 On the other hand, for a rational PL convex function $\ell^+\coloneqq \max\{0, \ell\}$, we get
\begin{align*}
0\le
D_v^{NA}(\ell^+)
=-\ell(0)+\int_{\{\ell\ge 0\}}\ell(x)^2\,dx
=-\int_{\{\ell<0\}}\ell(x)^2\,dx<0, 
\end{align*}
which yields a contradiction. 
Now assume $\min_P\ell=(1-\max_P\theta_P)/\operatorname{vol}(P)=0$. 
Then for each $i\in\mathbf{N}$, $\{\ell<1/i\}$ is non-empty and hence we can define
\[
g_i(x)\coloneqq \max\left\{0, \,\frac{1}{i}-\ell(x)\right\}, \quad f_i(x)\coloneqq \frac{g_i(x)}{\displaystyle \int_Pg_i\,dx}. 
\]
Each $f_i$ is a non-negative PL convex function on $P$. Since $\ell(0)=1/\operatorname{vol}(P)>0$, we have $1/i-\ell(0)<0$ for sufficiently large $i$ and hence $f_i(0)=0$. This implies that $f_i$ is normalized. Furthermore, the rationality of $g_i$ shows that of $f_i$. Since $\{f_i>0\}=\{\ell<1/i\}$, we have
\begin{align*}
\delta=\delta\int_Pf_i(x)\,dx
\le D_v^{NA}(f_i)
=\int_{\{f_i> 0\}}f_i(x)\ell(x)\,dx
\le\frac{1}{i}, 
\end{align*}
which yields a contradiction. 
Now we show (3)$\Rightarrow$(2). Pick a normalized rational PL convex function $f$. Then a simple computation gives
\[
D_v^{NA}(f)
=\frac{1}{\operatorname{vol}(P)}\int_Pf(x)(1-\theta_P(x))\,dx
\ge \frac{1-\max_P\theta_P}{\operatorname{vol}(P)}\int_Pf(x)\,dx. 
\]
\item The proof of (a) shows (1)$\Rightarrow$(2)$\Rightarrow$(3). 
To prove (3)$\Rightarrow$(1), apply Jensen's inequality. \qedhere
\end{enumerate}
\end{proof}

\section{Lists of uniform relative Ding stability of toric Fano manifolds}
\subsection{How to determine relative Ding stability}\label{sec:example}
As it was seen in Proposition \ref{Ding-ps}, we can check uniform relative Ding stability by computing the Mabuchi constants. In this subsection, we illustrate how to compute the potential function 
$\theta_P$ and the Mabuchi constant $M_{X_P}$ for a given Fano polytope $P$.

Let $P$ be an $n$-dimensional Fano polytope in $\mathbf{R}^{n}$ and $v$ the extremal K\"ahler vector field of $(X_P, -K_{X_{P}})$. 
Suppose that $\theta_P$ is of the form
\[ 
\theta_P(x)=\sum_{i=1}^n \theta_{i}\left(x_{i}-\dfrac{b_{i}}{b_{0}}\right), 
\] 
with
\[
b_{0}\coloneqq\operatorname{vol}(P), \quad b_{i}\coloneqq\int_{P}x_{i}\,dx\qquad    \text{for } \quad i=1, \ldots ,n.              
\]
The coefficients $\theta_{1}, \ldots, \theta_{n}$ are uniquely determined by \eqref{ext_aff_fct}, that is, the $n$-equations
\begin{equation}\label{potential_fun}
\sum_{j=1}^{n}(b_{0}c_{ij}-b_{i}b_{j})\theta_{j}+b_{0}b_{i}=0\qquad    \text{for } \quad i=1, \ldots ,n,           
\end{equation}
where
\[
c_{ij}\coloneqq\int_{P}x_{i}x_{j}\,dx\qquad    \text{for } \quad i,j=1, \ldots ,n.          
\]
In \cite[Table $1,2,3$]{N98}, Nakagawa listed the vector $(\theta_{1}, \ldots, \theta_n)$ for all Fano polytopes in dimension up to four. 
Instead of applying $\eqref{potential_fun}$, one can also use his list to compute $\theta_{X_P}$ and $M_{X_P}$ from the polytope data. 
We used the software packages Normaliz \cite{Normaliz} and Maxima\footnote{Maxima is available from http://maxima.sourceforge.net/.} for our computations.
See \cite[Section $5.1$]{YZ19} for more details.

\begin{example} \upshape
Let $P$ be the reflexive polytope in $\mathbf{R}^4$ whose vertices are given by 
\begin{align*}
\begin{Bmatrix}
(-1,-1,-1,-1), (-1,-1,-1,1), (-1,-1,0,-1), (-1,-1,2,1) \\
 (-1,1,-1,-1), (-1,1,0,-1), (-1,3,-1,1), (-1,3,2,1) \\
(1,-1,-1,-1), (1,-1,0,-1), (3,-1,-1,1), (3,-1,2,1) \\
\end{Bmatrix}. 
\end{align*}
The dual polytope $P^{\vee}$ of $P$ is given by
\[
{P}^{\vee}= \operatorname{conv}
\left(
\begin{Bmatrix}
(1,0,0,0), (0,1,0,0), (0,0,1,0), (0,0,0,1) \\
(0,0,0,-1), (-1,-1,0,1), (0,0,-1,1)
\end{Bmatrix}
\right), 
\]
and the associated toric Fano fourfold is $X_{P}=\mathbf{P}_{\mathbf{P}^1 \times \mathbf{P}^2}(\mathcal O \oplus \mathcal{O}(1,1))$. 
Following the computation in $\eqref{potential_fun}$, we find $(\theta_{1}, \theta_{2}, \theta_{3}, \theta_{4})=(0, 0, 0, 2790/1973)$ which is consistent with \cite[Table $3$, $D$-$6$]{N98}.
Since $b_{0}=\operatorname{vol}(P) =62/3$ and $b_{4}=\int_{P} x_4 \,dx=36/5$, we obtain 
\[
\theta_{P}(x) =\frac{1}{1973}(2790x_4-972). 
\]
For the vertex $p=(3,-1,2,1)$ of $P$, one can see that
\[
M_{X_{P}} = \max_P\theta_{P}=\theta_{P}(p)= \frac{1818}{1973} <1, 
\]
which shows that $X_{P}=\mathbf{P}_{\mathbf{P}^1 \times \mathbf{P}^2}(\mathcal O \oplus \mathcal{O}(1,1))$ is uniformly relatively Ding stable. 
\end{example}

\subsection{Toric del Pezzo surfaces}\label{sec:dP}
Although the result of the surface case was already obtained in \cite[Example 16]{Y17}, we list it here for reader's convenience. 
In dimension two, there are five isomorphism classes of Fano polygons and all of them are uniformly relatively Ding stable.
In Table $\ref{table:toricFano2RelDSta}$, $dP_{9-i}$ corresponds to a (smooth) del Pezzo surface with degree $(9-i)$, that is, a smooth projective surface given by the blowing-up of $\mathbf{P}^2$ at $i$ points. 

\begin{table}[H]
\caption{Relative Ding stability of smooth toric del Pezzo surfaces}\label{table:toricFano2RelDSta}
\begin{center} 
\begin{tabular}{cccc} \toprule \vspace{-0.23cm}
    &          &       Relative &       \\  \vspace{-0.23cm}
No. & Notation &                & $M_X$ \\
    &          & Ding stability &       \\   \midrule
  1 &${\mathbf{P}}^2 $                   & uniformly stable & 0 \\ 
  2 &$\mathbf{P}^1\times \mathbf{P}^1 $ & uniformly stable & 0 \\ 
  3 &$ dP_8$                      & uniformly stable & 5/11 \\ 
  4 &$ dP_7$                      & uniformly stable & 304/409 \\ 
  5 &$ dP_6$                      & uniformly stable & 0 \\ 
 \bottomrule
\end{tabular} 
\end{center}
\end{table}


\subsection{Toric Fano threefolds}\label{sec:ToricFano3}
In dimension three, in turn, there are $18$ isomorphism classes of Fano polytopes.
In this subsection, we verify relative Ding stability of all toric Fano threefolds and prove that there is an example of toric Fano threefolds which is relatively $K$-stable, but it is relatively Ding unstable. 

\subsubsection{Relative Ding stability of toric Fano threefolds}
Among all $18$ classes of toric Fano threefolds, if a given toric Fano threefold $X$ is one of $\mathcal{B}_1$, $\mathcal{C}_2$, $\mathcal{D}_1$, $\mathcal{D}_2$, $\mathcal{E}_1$ and 
$\mathcal{E}_2$, then $X$ is relatively Ding unstable. 
On the other hand, if $X$ is one of $\mathbf P^3$, $\mathcal{B}_2$, $\mathcal{B}_3$, $\mathcal{B}_4$, $\mathcal{C}_1$, $\mathcal{C}_3$, $\mathcal{C}_4$, $\mathcal{C}_5$,
$\mathcal{E}_3$, $\mathcal{E}_4$, $\mathcal{F}_1$ and $\mathcal{F}_2$, then $X$ is relatively Ding stable. Remark that these
 $12$ relatively Ding stable ones are actually
{\emph{uniformly}} relatively Ding stable (see \cite[Example 5.15]{Y17}).
We list all results in Table \ref{table:toricFano3RelDSta}.

\subsubsection{Relative $K$-stability of toric Fano threefolds}
This subsection aims to show that the toric Fano threefold $X=\mathbf P_{\mathbf P^2}(\mathcal O\oplus \mathcal O(2))$, which is labelled $\mathcal B_1$ in Table \ref{table:toricFano3RelDSta},
is relatively $K$-stable. Comparing this and the result in Table \ref{table:toricFano3RelDSta}, we see the following (see, Corollary $\ref{cor:difference}$ $(1)$).
\begin{proposition}\label{prop:B1}
The toric Fano threefold $\mathcal B_1$ is relatively $K$-stable, but it is relatively Ding unstable.
\end{proposition}
The rest of this subsection is devoted to prove Proposition $\ref{prop:B1}$ where we use the existence theorem of an extremal K\"ahler metric in each K\"ahler class of $\mathcal B_1$ (Proposition $\ref{prop:extremal}$).
We also clarify which part is a problematic in \cite{YZ19} when the third author apply the instability criterion $\eqref{ineq:insta}$ into $\mathcal B_1$. This effects the ``only if" part of Theorem $1.5$ and Corollary $1.6$
in \cite{YZ19}, and provides inconclusive results for relative $K$-stability of $\mathcal{C}_2$, $\mathcal{D}_1$, $\mathcal{D}_2$, $\mathcal{E}_1$ and $\mathcal{E}_2$.

Let $P$ be the Fano polytope corresponding to $\mathcal B_1$. Let us define $P^-=\set{x\in P| 1-\theta_P \leq 0}$, where $\theta_P=-\frac{620}{349}x_3-\frac{240}{349}$ is the potential function of $P$.
It was computed in \cite[Proposition $5.1$ (b)]{YZ19} that the right-hand side value of $\eqref{ineq:insta}$ is given by
\begin{align} \label{eq:P^-}
\begin{split}
P^-&= \operatorname{conv}
\left(
\begin{Bmatrix}
(4,-1,-1), \left(\frac{39}{10},-1,-\frac{19}{20}\right), \left(-1,-1,-\frac{19}{20}\right), \\
\left(-1,\frac{39}{10},-\frac{19}{20}\right), (-1,-1,-1), (-1,4,-1)
\end{Bmatrix}
\right), \qquad \mathrm{vol}(P^-)=\frac{7351}{12000}   
\end{split}
\end{align}
and
\begin{align}
& \int_{P^-}(1-\theta_P)^2dx=\frac{1475918766336271}{1461612000} \label{eq:integral}
\end{align}
respectively. Although the values of $\eqref{eq:P^-}$ are correct, the integral value of $\eqref{eq:integral}$ is incorrect because the third author forgot to put the normalize condition into the input data of $P^-$
when he run the software \cite{Normaliz} at the time. More precisely, the first and third authors recomputed the value of $\int_{P^-}(1-\theta_P)^2dx$ 
and confirmed that the correct value is
\[
\int_{P^-}(1-\theta_P)^2dx=\frac{23785711}{14616120000}.
\]
Since the left-hand side value of $\eqref{ineq:insta}$ is $1-c=\frac{589}{349}$, we see that the toric Fano threefold $\mathcal B_1$ is inconclusive.

\begin{table}
\caption{Relative  Ding stability of toric Fano threefolds}\label{table:toricFano3RelDSta}
\begin{center} 
\begin{tabular}{cccc} \toprule \vspace{-0.23cm}
    &          &       Relative &       \\  \vspace{-0.23cm}
No. & Notation &                & $M_X$ \\
    &          & Ding stability &       \\   \midrule
  1 &$\mathbf P^3$	       & uniformly stable & 0 \\ 
  2 &$\mathcal{B}_{1}$ &             unstable & 380/349 \\ 
  3 &$\mathcal{B}_{2}$ & uniformly stable & 55/97 \\ 
  4 &$\mathcal{B}_{3}$ & uniformly stable & 35/43 \\ 
  5 &$\mathcal{B}_{4}$ & uniformly stable & 0 \\ 
  6 &$\mathcal{C}_{1}$ & uniformly stable & 60/73 \\ 
  7 &$\mathcal{C}_{2}$ &             unstable & 38/33 \\ 
  8 &$\mathcal{C}_{3}$ & uniformly stable & 0 \\ 
  9 &$\mathcal{C}_{4}$ & uniformly stable & 5/11 \\ 
 10 &$\mathcal{C}_{5}$ & uniformly stable & 0 \\ 
 11 &$\mathcal{D}_{1}$ &             unstable & 711861/467581 \\ 
 12 &$\mathcal{D}_{2}$ &             unstable & 694595/650251 \\ 
 13 &$\mathcal{E}_{1}$ &             unstable & 27195/19651	\\ 
 14 &$\mathcal{E}_{2}$ &             unstable & 2936215/2735927 \\ 
 15 &$\mathcal{E}_{3}$ & uniformly stable & 304/409	\\ 
 16 &$\mathcal{E}_{4}$ & uniformly stable & 185791/394975 \\ 
 17 &$\mathcal{F}_{1}$ & uniformly stable & 0 \\ 
 18 &$\mathcal{F}_{2}$ & uniformly stable & 31/67 \\ 
 \bottomrule
\end{tabular} 
\end{center}
\end{table}

However, if we put emphasis on the geometric description of $\mathcal B_1$ as a projective bundle $\mathbf P_{\mathbf P^2}(\mathcal O\oplus \mathcal O(2))$,
we may conclude that it is relatively $K$-stable as follows (cf. \cite[Section $5.2.4$]{Y17}).

In fact, since the toric Fano manifold $\mathcal B_1$ can be expressed as a $\mathbf P^1$-bundle over the base space $\mathbf P^2$, we may apply ``Calabi ansatz" for constructing extremal K\"ahler metrics which were accomplished
by many researchers  \cite{G95, H94, ACGT-F08}. Consequently, $\mathcal B_1$ admits an extremal K\"ahler metric in every K\"ahler class (see Proposition $\ref{prop:extremal}$ $(1)$).
Hence the assertion is verified by \cite{ZZ08}.

\begin{remark}\rm
The first and third authors also confirmed that other five classes of toric Fano threefolds $\mathcal{C}_2$, $\mathcal{D}_1$, $\mathcal{D}_2$, $\mathcal{E}_1$ and $\mathcal{E}_2$ do \emph{not} satisfy the instability condition 
$\eqref{ineq:insta}$ by direct computations as above. Thus the relative $K$-stability of these five classes are still inconclusive.
\end{remark}

\subsection{Relative Ding stability and relative $K$-stability of toric Fano fourfolds}\label{sec:ToricFano4}
Now we verify relative Ding stability of toric Fano fourfolds. Among all $124$ isomorphism classes of toric Fano fourfolds,
we see that there are $49$ classes of toric Fano fourfolds which are uniformly relatively Ding stable, while the remaining $75$ classes are relatively Ding unstable.

In dimension four, there are at least three examples of toric Fano manifolds which are relatively $K$-stable, but they are relatively Ding unstable.
The geometric description of these manifolds are $X=\mathbf P_{\mathbf P^3}(\mathcal O\oplus \mathcal O(k))$ which are labeled $B_1$ and $B_2$ for $k=3$ and $2$, respectively.
The proof is similar to it of Proposition $\ref{prop:B1}$, and details are left it to the reader.

\subsection{Bott structure of certain toric Fano manifolds}\label{sec:Bott}
In this final section, we prove Corollary $\ref{cor:BottFano}$. We firstly recall toric description of Bott manifolds. For further details, see \cite{BCT-F19,CLMP20, CLMP21, HKM20}.

For a given $n\in \mathbf N$ with $n\geq 2$, we inductively construct complex manifolds $X_k$ for $k=1,2,\ldots , n$ as follows.
Let $X_0$ be a point $\mathrm{Spec}\, \mathbf C$, and let $X_1$ be the projective line.
For $k\geq 2$, we assume that $X_{k-1}$ is already determined and choose a holomorphic line bundle $\mathcal L_k$ on $X_{k-1}$.
Then we define $X_k$ as the total space of the $\mathbf P^1$-bundle over $X_{k-1}$ and denote it by
\[
X_k:=\mathbf P_{X_{k-1}}(\mathcal O_{X_{k-1}}\oplus \mathcal L_k) \stackrel{\pi_k}{\longrightarrow} X_{k-1}.
\]
Through this construction, we call $X_k$ the \emph{stage $k$ Bott manifold of the Bott tower of height $n$:}
\[
X_n \stackrel{\pi_n}{\longrightarrow} X_{n-1} \stackrel{\pi_{n-1}}{\longrightarrow} \cdots \stackrel{\pi_{2}}{\longrightarrow} X_1 \stackrel{\pi_{1}}{\longrightarrow} \mathrm{Spec }\, \mathbf C.
\]

On the other hand, a stage $n$ Bott manifold $X_n$ can be written as a quotient of $(\mathbf C^2_*)^n$ by an algebraic torus $(\mathbf C^*)^n$, where $(\mathbf C^2_*)^n$ is
$n$ copies of $\mathbf C^2_*:=\mathbf C^2\setminus \set{0}$. In order to see this, we denote an algebraic torus $(\mathbf C^*)^n$ by $T$. Let $A$ be a lower triangular unipotent matrix of the form
\begin{equation}\label{eq:matrix}
A=\begin{pmatrix}
1      &   0      &  \cdots   & 0         &  0       \\
  a_2^1         & 1 &  \cdots    &  0     &      0   \\
     \vdots      &  \vdots      & \ddots &  \vdots      &    \vdots   \\
   a_{n-1}^1  &    a_{n-1}^2     &  \cdots      & 1 &   0     \\
      a_n^1      &   a_n^2     &  \cdots     &   a_n^{n-1}    &1
\end{pmatrix}
\end{equation}
with $a_j^i\in \mathbf Z$. Let us consider the induced action $\lambda$ of $\bs t=(t_1,\ldots, t_n)\in T$ on 
$(\bs z; \bs w)=((z_1, \ldots ,z_n);(w_1, \ldots ,w_n))\in (\mathbf C^2_*)^n
$ 
which is given by
\begin{equation}
\begin{split}\label{eq:action}
\xymatrix@R=-1ex{
\lambda:  T\times (\mathbf C^2_*)^n  \ar[r]&(\mathbf C^2_*)^n\\
\din &\din\\
(\bs t, (\bs z; \bs w))\ar@{|->}[r]& \left( (t_1z_1, \ldots , t_nz_n);
\left( \left( \prod_{i=1}^nt_i^{a_1^i}\right) w_1, \ldots, \left( \prod_{i=1}^nt_i^{a_n^i}\right) w_n
 \right)
\right).
}
\end{split}
\end{equation}
For the unit sphere $S^3$ in $\mathbf C^2_*$, there is the induced action of $(\mathbf R_{>0})^n$ on $(\mathbf C^2_*)^n$ by $\eqref{eq:action}$,
which is transverse to $(S^3)^n$. Hence the orbits of this action are bijectively correspond to orbits of the induced free $(S^1)^n$-action on $(S^3)^n$.
Consequently, the geometric quotient $(\mathbf C^2_*)^n/T$ is a complex $n$-dimensional compact manifold which will be denoted by $X_n(A)$.

As a special case, for $\bs k=(k_1, \ldots, k_{n-1})\in \mathbf Z^{n-1}$, we consider the case where $a_n^i=k_i$ and $a_j^i=0$ otherwise.
Then a lower triangular matrix in $\eqref{eq:matrix}$ is written as
\begin{equation*}
A=\begin{pmatrix}
1      &   0      &  \cdots   & 0         &  0       \\
 0        & 1 &  \cdots    &  0     &      0   \\
     \vdots      &  \vdots      & \ddots &  \vdots      &    \vdots   \\
  0  &    0     &  \cdots      & 1 &   0     \\
      k_1      &   k_2     &  \cdots     &   k_{n-1}    &1
\end{pmatrix}.
\end{equation*}
The associated Bott tower $X_n(A)$ is the sequence of $\mathbf P^1$-bundles
\begin{equation}\label{eq:BT}
\mathbf P_{(\mathbf P^1)^{n-1}}(\mathcal O\oplus \mathcal O(k_1, \ldots ,k_{n-1})) \stackrel{\pi_n}{\longrightarrow} (\mathbf P^1)^{n-1} \stackrel{\pi_{n-1}}{\longrightarrow} \cdots \stackrel{\pi_{2}}{\longrightarrow} \mathbf P^1 \stackrel{\pi_{1}}{\longrightarrow} \mathrm{Spec }\, \mathbf C
\end{equation}
which has twist one unless $\bs k=(0, \ldots ,0)$.
Since any stage $n$ Bott manifold in $\eqref{eq:BT}$ with $\prod_{i=1}^{n-1}k_i\neq 0$ admits an extremal K\"ahler metric in each K\"ahler class by Proposition $\ref{prop:Bott}$, our problem
 can be reduced to find examples of relatively Ding unstable Fano Bott manifolds in the form of $\eqref{eq:BT}$. In particular, they must be twist one.
 There are six candidates of such examples in Tables $\ref{table:toricFano3RelDSta}$--$\ref{table:toricFano4RelDSta}$, that is,
 $\mathcal C_2$, $L_1$, $L_2$, $L_3$, $L_4$, and $L_6$. 
Their twists are computable by primitive relations of the vertices of $P^\vee$ (see \cite[Proposition $3.3.6$]{B98}). Alternatively, one can use the \emph{signed rooted forest} associated with a Fano Bott manifold
for computing twist (see \cite[Section $2$]{CLMP21}).
In fact, the height of signed rooted forest is equal to the twist of a Fano Bott manifold, and all signed rooted forests of stage four Fano Bott manifolds are illustrated in \cite[Figure $2$]{CLMP21}.
Consequently, we have the following list of twist for each relatively Ding unstable Fano manifold.
\begin{figure}[H]
\begin{center}
\begin{tabular}{cccccccc} \toprule
Notation &\qquad &$\mathcal C_2$ & $L_1$ & $L_2$ & $L_3$ & $L_4$ & $L_6$\\ \midrule
twist &  &$2$ & $1$ &  $2$  &  $2$  &  $3$ &  $2$  \\
Number in \cite[Figure 2]{CLMP21} & &$-$ & $(12)$ & $(10)$ & $(8)$ & $(7)$ & $(3)$\\
\bottomrule
\end{tabular}
\end{center}
\end{figure}
The assertion is verified.

\begin{remark}\rm
Note that there exists \emph{some} K\"ahler classes of the toric Fano threefold $\mathcal C_2$ that admit extremal K\"ahler metrics by Theorem $1$ $(2)$ in \cite{BCT-F19}.
However, it is difficult to determine \emph{which} K\"ahler class may admit extremal K\"ahler metrics in general.
For the moment, we do not know if every K\"ahler class of $\mathcal C_2$ has an extremal K\"ahler metric, or even if its first Chern class has an extremal K\"ahler metric. 
\end{remark}
Hence the following problem is still open.
\begin{problem}
Does the toric Fano threefold $\mathcal C_2$ admit an extremal K\"ahler metric in its first Chern class? Moreover, is it a Calabi Dream manifold?
\end{problem}
If the answer is affirmative, then any stage three Fano Bott manifold is a Calabi dream manifold.

\newpage

\begin{table}[H]
\caption{Relative Ding stability of toric Fano fourfolds}\label{table:toricFano4RelDSta}
\begin{center} 
\begin{tabular}{cccc} \toprule \vspace{-0.23cm}
    &          &       Relative &       \\  \vspace{-0.23cm}
No. & Notation &                & $M_X$ \\
    &          & Ding stability &       \\   \midrule
  1 & $\mathbf P^4$    & uniformly stable & 0 \\ 
  2 &$B_1$    &       unstable  & $186633/108133$\\ 
  3 &$B_2$    &      unstable & $1827/1577$\\
  4 &$B_3$    &     uniformly stable &   $62937/101357$\\
  5 &$B_4$      & uniformly stable &	0 \\
  6 &$B_5$     & unstable &   $81/71$ \\
  7 &$C_1$    & unstable &   $5168/2803$ \\
  8 &$C_2$    & unstable &   $5056/4961$ \\
  9 &$C_3$    & uniformly stable  &   $1876/2921$ \\
 10 &$C_4$   & uniformly stable  &0\\
 11 &$E_1$  & unstable  &   $3117506664/1528096589$ \\
 12 &$E_2$   & unstable &   $1840195984/1038057499$ \\
 13 &$E_3$ & unstable &   $8725985172/6541476757$ \\
 14 &$D_1$   & unstable  &   $11387/8057$\\
 15 &$D_2$  & unstable &   $2344247/1129107$ \\
 16 &$D_3$   &  unstable &   $23933079/13920179$ \\
 17 &$D_4$   &  unstable &   $10601177/5141177$ \\
 18 &$D_5$  & unstable	&	380/349 \\
 19 &$D_6$    & uniformly stable &   $1818/1973$ \\
 20 &$D_7$   &  unstable &   $748/523$ \\
 21 &$D_8$   &  unstable &   $3165248067/2130145727$ \\
 22 &$D_9$   & unstable  &   $1380/1177$ \\
 23 &$D_{10}$  &  unstable &   $18353/11683$ \\
 24 &$D_{11}$   & unstable &   $1727/1007$ \\
 25 &$D_{12}$   & uniformly stable  & 55/97\\
 26 &$D_{13}$   & uniformly stable & 0 \\
 27 &$D_{14}$   & uniformly stable & 35/43 \\
 28 &$D_{15}$  & uniformly stable & 5/11\\
 29 &$D_{16}$   & uniformly stable  &   $1046933/1061921$ \\
 30 &$D_{17}$  &  uniformly stable &   $59/134$ \\
 31 &$D_{18}$   & uniformly stable  &   $3927/5177$ \\
 32 &$D_{19}$   & uniformly stable  &   $238/1613$ \\
 33 &$G_{1}$   & unstable  &   $29663642268/16641653953$ \\
 34 &$G_{2}$  & unstable &   $372321328/163900203$  \\
  35 &$G_{3}$   & unstable  &   $2130146436/1907713037$ \\
   36 &$G_{4}$   &  unstable &   $91293727706236/48057952407691$ \\
    37 &$G_{5}$   & unstable &   $7423814976/5976434111$  \\
 \bottomrule
\end{tabular} 
\end{center}
\end{table}

\begin{table}[H]
\begin{center}
\begin{tabular}{cccc} \toprule \vspace{-0.23cm}
    &          &       Relative &       \\  \vspace{-0.23cm}
No. & Notation &                & $M_X$ \\
    &          & Ding stability &       \\   \midrule
      38 &$G_{6}$  & unstable  &   $6431616/4388521$ \\
 39 &$H_{1}$  & unstable  &   $391552/172227$ \\
 40 &$H_{2}$   & unstable &   $23712381468/12261327193$ \\
 41 &$H_{3}$    & unstable & $1551503556/999563887$ \\
 42 &$H_{4}$   & unstable  & $173016/111181$ \\
     43 &$H_{5}$   & unstable & $9770225016/8410595791$ \\
 44 &$H_{6}$  & unstable & $7892183028/6776874793$  \\
 45 &$H_{7}$   & unstable & $2510023476/2191715291$  \\
 46 &$H_{8}$   &  uniformly stable & 304/409\\
 47 &$H_{9}$   & uniformly stable & $3339308412/5523943807$ \\
 48 &$H_{10}$   &  uniformly stable & $33344/116609$ \\
 49 &$L_{1}$   & unstable & $1827/1577$ \\
 50 &$L_{2}$ & unstable  & $4491311/2637399$ \\
 51 &$L_{3}$  & unstable  & $29551791/20356871$ \\ 
 52 &$L_{4}$   & unstable & $12523717673/6291985913$ \\
 53 &$L_{5}$   & uniformly stable &60/73  \\
 54 &$L_{6}$   &  unstable & 38/33\\
 55 &$L_{7}$   & uniformly stable &10/11 \\
 56 &$L_{8}$   & uniformly stable &0 \\
 57 &$L_{9}$  &  uniformly stable &5/11\\
 58 &$L_{10}$   & uniformly stable  & $11999/15439$ \\
 59 &$L_{11}$   & uniformly stable &	0 \\ 
 60 &$L_{12}$   & uniformly stable  & $98007/118907$ \\
 61 &$L_{13}$   & uniformly stable & $93/203$  \\
 62 &$I_{1}$   & unstable & $708660642/298795537$  \\
 63 &$I_{2}$   & unstable  & $3024970758/1752843533$ \\
 64 &$I_{3}$   & unstable & $1265748565401/557417025511$  \\
 65 &$I_{4}$   &  unstable & $94549671/42931741$ \\
 66 &$I_{5}$  &  unstable & $11713596999245802/6984760752795427$ \\
 67 &$I_{6}$   & unstable & $448243263/230998253$  \\
 68 &$I_{7}$   &  unstable &711861/467581 \\
 69 &$I_{8}$   & unstable  & $3126600189529/1577014856729$ \\
 70 &$I_{9}$   & unstable & $51542252/31553027$ \\
 71 &$I_{10}$   & unstable & $1376202783/1008086248$ \\
 72 &$I_{11}$  & unstable & $13278385/7524593$ \\
 73 &$I_{12}$   & uniformly stable & $3762595665/10127095471$ \\
 74 &$I_{13}$   & unstable &694595/650251 \\
 75 &$I_{14}$   & uniformly stable & $211431316083/218756440813$ \\
 76 &$I_{15}$  & unstable & $949610187/739116617$ \\
 77 &$M_{1}$   & unstable & $4575996/2853121$ \\
 78 &$M_{2}$   & unstable & $9794436236/4944865931$ \\
 79 &$M_{3}$   & unstable  & $25645181180008/16440690503203$ \\
     80 &$M_{4}$  & unstable & $19189680326248/13060549547443$ \\
\bottomrule
\end{tabular}
\end{center}
\end{table}

\begin{table}[H]
\begin{center} 
\begin{tabular}{cccc} \toprule \vspace{-0.23cm}
    &          &       Relative &       \\  \vspace{-0.23cm}
No. & Notation &                & $M_X$ \\
    &          & Ding stability &       \\   \midrule
 81 &$M_{5}$   & unstable & $12265145105056/5748810067571$ \\
 82 &$J_{1}$   & unstable & $326672481/271746296$ \\
 83 &$J_{2}$  & uniformly stable & $141120/450983$ \\
  84 &$Q_{1}$   & unstable & $510608/266403$ \\
 85 &$Q_{2}$  & unstable & $518137667/237174992$ \\
 86 &$Q_{3}$   & unstable & $90760711/67077371$ \\
 87 &$Q_{4}$  & unstable & $121849781/73313456$ \\
 88 &$Q_{5}$   & unstable &977249482633/573603226348 \\
 89 &$Q_{6}$   & unstable & 27195/19651 \\
 90 &$Q_{7}$   &  unstable & $138419/99094$ \\
 91 &$Q_{8}$   & unstable &2936215/2735927 \\
 92 &$Q_{9}$   & uniformly stable & $173974704/197727059$ \\
 93 &$Q_{10}$   & unstable & 5389/4499 \\
 94 &$Q_{11}$   & uniformly stable &304/409 \\
 95 &$Q_{12}$   & unstable & $203968842/196251577$ \\
 96 &$Q_{13}$   & uniformly stable & $31595933/35875483$ \\
 97 &$Q_{14}$   & unstable & $3046408136973/2450627729348$ \\
 98 &$Q_{15}$   & uniformly stable & 185791/394975 \\
 99 &$Q_{16}$   &  uniformly stable & $80506889/111477439$ \\
100 &$Q_{17}$   & uniformly stable  & $17172/162187$ \\
101 &$K_{1}$   &  unstable & $18300/16571$ \\
102 &$K_{2}$ & unstable & $400/341$ \\
103 &$K_{3}$   & uniformly stable  & $20550/35413$ \\
104 &$K_{4}$   & uniformly stable &0 \\
105 &$R_{1}$  &   unstable & $567397500/324135871$ \\
106 &$R_{2}$   & unstable & $2758057258564/2028544442569$ \\
107 &$R_{3}$   & uniformly stable & $277721596/369479321$ \\
108 &         &         unstable    & $182110284/153124261$   \\
109 &$U_{1}$   & uniformly stable & $2850/3389$ \\
110 &$U_{2}$   & unstable  & $1095285/924382$ \\
111 &$U_{3}$   & uniformly stable & $1275/1424$ \\
112 &$U_{4}$   & uniformly stable  & 5/11 \\
113 &$U_{5}$   & uniformly stable &0 \\
114 &$U_{6}$   & uniformly stable &31/67 \\
115 &$U_{7}$   & uniformly stable & $119/258$ \\
116 &$U_{8}$   & uniformly stable & 0 \\
117 &$\widetilde{V}^4$   & unstable & $737568/395263$ \\
118 &$V^4$  & uniformly stable &0 \\
119 &$dP_7\times dP_7$   & unstable &608/409 \\
120 &$dP_7\times dP_6$    & uniformly stable &304/409\\
121 &$dP_6\times dP_6$   & uniformly stable &0 \\
122 &$Z_1$   & uniformly stable  & $1013328/1289443$  \\
123 &$Z_2$   &  unstable & $3418840588/2592694163$ \\
124 &$W$   & uniformly stable & 0  \\
\bottomrule
\end{tabular}
\end{center}
\end{table}

\newpage

\bibliographystyle{amsalpha}

\appendix
\section{Computational data}
\subsection{2-dimensional Fano polytopes}
\subsubsection{{$\mathbf P^{2}$}}
\begin{align*}
&b_{0} = \mathrm{vol}(P) = \frac{9}{2}, 
&&b_{1} = \int_{P}x_{1}\,dx_{1}dx_{2} = 0, \\
&b_{2} = \int_{P}x_{2}\,dx_{1}dx_{2} = 0, 
&&c_{11} = \int_{P}x_{1}^{2}\,dx_{1}dx_{2} = \frac{9}{4}, \\
&c_{12} = \int_{P}x_{1}x_{2}\,dx_{1}dx_{2} = -\frac{9}{8}, 
&&c_{22} = \int_{P}x_{2}^{2}\,dx_{1}dx_{2} = \frac{9}{4}, 
\end{align*}
\begin{align*}
&\theta_P(x_{1}, x_{2}) = 0, \\
&M_{X} = \max_{P}\theta_P = 0 < 1. 
\end{align*}

\subsubsection{{$\mathbf P^{1} \times \mathbf P^{1}$}}
\begin{align*}
&b_{0} = \mathrm{vol}(P) = 4, 
&&b_{1} = \int_{P}x_{1}\,dx_{1}dx_{2} = 0, \\
&b_{2} = \int_{P}x_{2}\,dx_{1}dx_{2} = 0, 
&&c_{11} = \int_{P}x_{1}^{2}\,dx_{1}dx_{2} = \frac{4}{3}, \\
&c_{12} = \int_{P}x_{1}x_{2}\,dx_{1}dx_{2} = 0, 
&&c_{22} = \int_{P}x_{2}^{2}\,dx_{1}dx_{2} = \frac{4}{3}, 
\end{align*}
\begin{align*}
&\theta_P(x_{1}, x_{2}) = 0, \\
&M_{X} = \max_{P}\theta_P = 0 < 1. 
\end{align*}

\subsubsection{{$dP_{8}$}}
\begin{align*}
&b_{0} = \mathrm{vol}(P) = 4, 
&&b_{1} = \int_{P}x_{1}\,dx_{1}dx_{2} = -\frac{2}{3}, \\
&b_{2} = \int_{P}x_{2}\,dx_{1}dx_{2} = -\frac{1}{3}, 
&&c_{11} = \int_{P}x_{1}^{2}\,dx_{1}dx_{2} = \frac{4}{3}, \\
&c_{12} = \int_{P}x_{1}x_{2}\,dx_{1}dx_{2} = \frac{2}{3}, 
&&c_{22} = \int_{P}x_{2}^{2}\,dx_{1}dx_{2} = 2, 
\end{align*}
\begin{align*}
&\theta_P(x_{1}, x_{2}) = -\frac{6x_{1}+1}{11}, \\
&M_{X} = \max_{P}\theta_P = \frac{5}{11} < 1. 
\end{align*}

\subsubsection{{$dP_{7}$}}
\begin{align*}
&b_{0} = \mathrm{vol}(P) = \frac{7}{2}, 
&&b_{1} = \int_{P}x_{1}\,dx_{1}dx_{2} = -\frac{1}{3}, \\
&b_{2} = \int_{P}x_{2}\,dx_{1}dx_{2} = \frac{1}{3}, 
&&c_{11} = \int_{P}x_{1}^{2}\,dx_{1}dx_{2} = \frac{13}{12}, \\
&c_{12} = \int_{P}x_{1}x_{2}\,dx_{1}dx_{2} = \frac{5}{24}, 
&&c_{22} = \int_{P}x_{2}^{2}\,dx_{1}dx_{2} = \frac{13}{12}, 
\end{align*}
\begin{align*}
&\theta_P(x_{1}, x_{2}) = \frac{-168x_{1}+168x_{2}-32}{409}, \\
&M_{X} = \max_{P}\theta_P = \frac{304}{409} < 1. 
\end{align*}

\subsubsection{{$dP_{6}$}}
\begin{align*}
&b_{0} = \mathrm{vol}(P) = 3, 
&&b_{1} = \int_{P}x_{1}\,dx_{1}dx_{2} = 0, \\
&b_{2} = \int_{P}x_{2}\,dx_{1}dx_{2} = 0, 
&&c_{11} = \int_{P}x_{1}^{2}\,dx_{1}dx_{2} = \frac{5}{6}, \\
&c_{12} = \int_{P}x_{1}x_{2}\,dx_{1}dx_{2} = -\frac{5}{12}, 
&&c_{22} = \int_{P}x_{2}^{2}\,dx_{1}dx_{2} = \frac{5}{6}, 
\end{align*}
\begin{align*}
&\theta_P(x_{1}, x_{2}) = 0, \\
&M_{X} = \max_{P}\theta_P = 0 < 1. 
\end{align*}

\subsection{3-dimensional Fano polytopes}
\subsubsection{{$\mathbf P^{3}$}}
\begin{align*}
&b_{0} = \mathrm{vol}(P) = \frac{32}{3}, 
&&b_{1} = \int_{P}x_{1}dx_{1}dx_{2}dx_{3} = 0, \\
&b_{2} = \int_{P}x_{2}dx_{1}dx_{2}dx_{3} = 0, 
&&b_{3} = \int_{P}x_{3}dx_{1}dx_{2}dx_{3} = 0, \\
&c_{11} = \int_{P}x_{1}^{2}dx_{1}dx_{2}dx_{3} = \frac{32}{5}, 
&&c_{12} = \int_{P}x_{1}x_{2}dx_{1}dx_{2}dx_{3} = -\frac{32}{15}, \\
&c_{13} = \int_{P}x_{1}x_{3}dx_{1}dx_{2}dx_{3} = -\frac{32}{15}, 
&&c_{22} = \int_{P}x_{2}^{2}dx_{1}dx_{2}dx_{3} = \frac{32}{5}, \\
&c_{23} = \int_{P}x_{2}x_{3}dx_{1}dx_{2}dx_{3} = -\frac{32}{15}, 
&&c_{33} = \int_{P}x_{3}^{2}dx_{1}dx_{2}dx_{3} = \frac{32}{5}, 
\end{align*}
\begin{align*}
&\theta_P(x_{1}, x_{2}, x_{3}) = 0, \\
&M_{X} = \max_{P}\theta_P = 0 < 1. 
\end{align*}

\subsubsection{{$\mathcal{B}_{1}$}}
\begin{align*}
&b_{0} = \mathrm{vol}(P) = \frac{31}{3}, 
&&b_{1} = \int_{P}x_{1}dx_{1}dx_{2}dx_{3} = \frac{8}{3}, \\
&b_{2} = \int_{P}x_{2}dx_{1}dx_{2}dx_{3} = \frac{8}{3}, 
&&b_{3} = \int_{P}x_{3}dx_{1}dx_{2}dx_{3} = -4, \\
&c_{11} = \int_{P}x_{1}^{2}dx_{1}dx_{2}dx_{3} = \frac{311}{30}, 
&&c_{12} = \int_{P}x_{1}x_{2}dx_{1}dx_{2}dx_{3} = -\frac{53}{20}, \\
&c_{13} = \int_{P}x_{1}x_{3}dx_{1}dx_{2}dx_{3} = -\frac{38}{15}, 
&&c_{22} = \int_{P}x_{2}^{2}dx_{1}dx_{2}dx_{3} = \frac{311}{30}, \\
&c_{23} = \int_{P}x_{2}x_{3}dx_{1}dx_{2}dx_{3} = -\frac{38}{15}, 
&&c_{33} = \int_{P}x_{3}^{2}dx_{1}dx_{2}dx_{3} = \frac{19}{5}, 
\end{align*}
\begin{align*}
&\theta_P(x_{1}, x_{2}, x_{3}) = -\frac{620x_{3}+240}{349}, \\
&M_{X} = \max_{P}\theta_P = \frac{380}{349} > 1.  
\end{align*}

\subsubsection{{$\mathcal{B}_{2}$}}
\begin{align*}
&b_{0} = \mathrm{vol}(P) = \frac{28}{3}, 
&&b_{1} = \int_{P}x_{1}dx_{1}dx_{2}dx_{3} = \frac{2}{3}, \\
&b_{2} = \int_{P}x_{2}dx_{1}dx_{2}dx_{3} = \frac{2}{3}, 
&&b_{3} = \int_{P}x_{3}dx_{1}dx_{2}dx_{3} = -2, \\
&c_{11} = \int_{P}x_{1}^{2}dx_{1}dx_{2}dx_{3} = \frac{88}{15}, 
&&c_{12} = \int_{P}x_{1}x_{2}dx_{1}dx_{2}dx_{3} = -\frac{12}{5}, \\
&c_{13} = \int_{P}x_{1}x_{3}dx_{1}dx_{2}dx_{3} = -\frac{16}{15}, 
&&c_{22} = \int_{P}x_{2}^{2}dx_{1}dx_{2}dx_{3} = \frac{88}{15}, \\
&c_{23} = \int_{P}x_{2}x_{3}dx_{1}dx_{2}dx_{3} = -\frac{16}{15}, 
&&c_{33} = \int_{P}x_{3}^{2}dx_{1}dx_{2}dx_{3} = \frac{16}{5}, 
\end{align*}
\begin{align*}
&\theta_P(x_{1}, x_{2}, x_{3}) = -\frac{70x_{3}+15}{97}, \\
&M_{X} = \max_{P}\theta_P = \frac{55}{97} < 1. 
\end{align*}

\subsubsection{{$\mathcal{B}_{3}$}}
\begin{align*}
&b_{0} = \mathrm{vol}(P) = 9, 
&&b_{1} = \int_{P}x_{1}dx_{1}dx_{2}dx_{3} = -\frac{9}{8}, \\
&b_{2} = \int_{P}x_{2}dx_{1}dx_{2}dx_{3} = -\frac{9}{8}, 
&&b_{3} = \int_{P}x_{3}dx_{1}dx_{2}dx_{3} = \frac{9}{8}, \\
&c_{11} = \int_{P}x_{1}^{2}dx_{1}dx_{2}dx_{3} = \frac{81}{20}, 
&&c_{12} = \int_{P}x_{1}x_{2}dx_{1}dx_{2}dx_{3} = -\frac{27}{20}, \\
&c_{13} = \int_{P}x_{1}x_{3}dx_{1}dx_{2}dx_{3} = -\frac{27}{20}, 
&&c_{22} = \int_{P}x_{2}^{2}dx_{1}dx_{2}dx_{3} = \frac{81}{20}, \\
&c_{23} = \int_{P}x_{2}x_{3}dx_{1}dx_{2}dx_{3} = -\frac{27}{20}, 
&&c_{33} = \int_{P}x_{3}^{2}dx_{1}dx_{2}dx_{3} = \frac{111}{20}, 
\end{align*}
\begin{align*}
&\theta_P(x_{1}, x_{2}, x_{3}) = -\frac{20x_{1} + 20x_{2}+5}{43}, \\
&M_{X} = \max_{P}\theta_P = \frac{35}{43} < 1. 
\end{align*}

\subsubsection{{$\mathcal{B}_{4}$}}
\begin{align*}
&b_{0} = \mathrm{vol}(P) = 9, 
&&b_{1} = \int_{P}x_{1}dx_{1}dx_{2}dx_{3} = 0, \\
&b_{2} = \int_{P}x_{2}dx_{1}dx_{2}dx_{3} = 0, 
&&b_{3} = \int_{P}x_{3}dx_{1}dx_{2}dx_{3} = 0, \\
&c_{11} = \int_{P}x_{1}^{2}dx_{1}dx_{2}dx_{3} = \frac{9}{2}, 
&&c_{12} = \int_{P}x_{1}x_{2}dx_{1}dx_{2}dx_{3} = -\frac{9}{4}, \\
&c_{13} = \int_{P}x_{1}x_{3}dx_{1}dx_{2}dx_{3} = 0, 
&&c_{22} = \int_{P}x_{2}^{2}dx_{1}dx_{2}dx_{3} = \frac{9}{2}, \\
&c_{23} = \int_{P}x_{2}x_{3}dx_{1}dx_{2}dx_{3} = 0, 
&&c_{33} = \int_{P}x_{3}^{2}dx_{1}dx_{2}dx_{3} = 3, 
\end{align*}
\begin{align*}
&\theta_P(x_{1}, x_{2}, x_{3}) = 0,\\
&M_{X} = \max_{P}\theta_P = 0 < 1. 
\end{align*}

\subsubsection{{$\mathcal{C}_{1}$}}
\begin{align*}
&b_{0} = \mathrm{vol}(P) = \frac{26}{3}, 
&&b_{1} = \int_{P}x_{1}dx_{1}dx_{2}dx_{3} = \frac{4}{3}, \\
&b_{2} = \int_{P}x_{2}dx_{1}dx_{2}dx_{3} = \frac{4}{3}, 
&&b_{3} = \int_{P}x_{3}dx_{1}dx_{2}dx_{3} = -\frac{8}{3}, \\
&c_{11} = \int_{P}x_{1}^{2}dx_{1}dx_{2}dx_{3} = \frac{24}{5}, 
&&c_{12} = \int_{P}x_{1}x_{2}dx_{1}dx_{2}dx_{3} = \frac{23}{30}, \\
&c_{13} = \int_{P}x_{1}x_{3}dx_{1}dx_{2}dx_{3} = -\frac{23}{15}, 
&&c_{22} = \int_{P}x_{2}^{2}dx_{1}dx_{2}dx_{3} = \frac{24}{5}, \\
&c_{23} = \int_{P}x_{2}x_{3}dx_{1}dx_{2}dx_{3} = -\frac{23}{15}, 
&&c_{33} = \int_{P}x_{3}^{2}dx_{1}dx_{2}dx_{3} = \frac{46}{15}, 
\end{align*}
\begin{align*}
&\theta_P(x_{1}, x_{2}, x_{3}) = -\frac{260x_{3}+80}{219}, \\
&M_{X} = \max_{P}\theta_P = \frac{60}{73} < 1. 
\end{align*}

\subsubsection{{$\mathcal{C}_{2}$}}
\begin{align*}
&b_{0} = \mathrm{vol}(P) = \frac{25}{3}, 
&&b_{1} = \int_{P}x_{1}dx_{1}dx_{2}dx_{3} = -\frac{2}{3}, \\
&b_{2} = \int_{P}x_{2}dx_{1}dx_{2}dx_{3} = \frac{4}{3}, 
&&b_{3} = \int_{P}x_{3}dx_{1}dx_{2}dx_{3} = -2, \\
&c_{11} = \int_{P}x_{1}^{2}dx_{1}dx_{2}dx_{3} = \frac{37}{10}, 
&&c_{12} = \int_{P}x_{1}x_{2}dx_{1}dx_{2}dx_{3} = -\frac{89}{60}, \\
&c_{13} = \int_{P}x_{1}x_{3}dx_{1}dx_{2}dx_{3} = -\frac{11}{15}, 
&&c_{22} = \int_{P}x_{2}^{2}dx_{1}dx_{2}dx_{3} = \frac{161}{30}, \\
&c_{23} = \int_{P}x_{2}x_{3}dx_{1}dx_{2}dx_{3} = -\frac{16}{15}, 
&&c_{33} = \int_{P}x_{3}^{2}dx_{1}dx_{2}dx_{3} = \frac{43}{15}, 
\end{align*}
\begin{align*}
&\theta_P(x_{1}, x_{2}, x_{3}) = -\frac{7600x_{1}+17750x_{3}+4868}{17787}, \\
&M_{X} = \max_{P}\theta_P = \frac{38}{33} > 1. 
\end{align*}

\subsubsection{{$\mathcal{C}_{3}$}}
\begin{align*}
&b_{0} = \mathrm{vol}(P) = 8, 
&&b_{1} = \int_{P}x_{1}dx_{1}dx_{2}dx_{3} = 0, \\
&b_{2} = \int_{P}x_{2}dx_{1}dx_{2}dx_{3} = 0, 
&&b_{3} = \int_{P}x_{3}dx_{1}dx_{2}dx_{3} = 0, \\
&c_{11} = \int_{P}x_{1}^{2}dx_{1}dx_{2}dx_{3} = \frac{8}{3}, 
&&c_{12} = \int_{P}x_{1}x_{2}dx_{1}dx_{2}dx_{3} = 0, \\
&c_{13} = \int_{P}x_{1}x_{3}dx_{1}dx_{2}dx_{3} = 0, 
&&c_{22} = \int_{P}x_{2}^{2}dx_{1}dx_{2}dx_{3} = \frac{8}{3}, \\
&c_{23} = \int_{P}x_{2}x_{3}dx_{1}dx_{2}dx_{3} = 0, 
&&c_{33} = \int_{P}x_{3}^{2}dx_{1}dx_{2}dx_{3} = \frac{8}{3}, 
\end{align*}
\begin{align*}
&\theta_P(x_{1}, x_{2}, x_{3}) = 0, \\
&M_{X} = \max_{P}\theta_P = 0 < 1. 
\end{align*}

\subsubsection{{$\mathcal{C}_{4}$}}
\begin{align*}
&b_{0} = \mathrm{vol}(P) = 8, 
&&b_{1} = \int_{P}x_{1}dx_{1}dx_{2}dx_{3} = \frac{2}{3}, \\
&b_{2} = \int_{P}x_{2}dx_{1}dx_{2}dx_{3} = -\frac{4}{3}, 
&&b_{3} = \int_{P}x_{3}dx_{1}dx_{2}dx_{3} = 0, \\
&c_{11} = \int_{P}x_{1}^{2}dx_{1}dx_{2}dx_{3} = 4, 
&&c_{12} = \int_{P}x_{1}x_{2}dx_{1}dx_{2}dx_{3} = -\frac{4}{3}, \\
&c_{13} = \int_{P}x_{1}x_{3}dx_{1}dx_{2}dx_{3} = 0, 
&&c_{22} = \int_{P}x_{2}^{2}dx_{1}dx_{2}dx_{3} = \frac{8}{3}, \\
&c_{23} = \int_{P}x_{2}x_{3}dx_{1}dx_{2}dx_{3} = 0, 
&&c_{33} = \int_{P}x_{3}^{2}dx_{1}dx_{2}dx_{3} = \frac{8}{3}, 
\end{align*}
\begin{align*}
&\theta_P(x_{1}, x_{2}, x_{3}) = -\frac{6x_{2}+1}{11}, \\
&M_{X} = \max_{P}\theta_P = \frac{5}{11} < 1. 
\end{align*}

\subsubsection{{$\mathcal{C}_{5}$}}
\begin{align*}
&b_{0} = \mathrm{vol}(P) = \frac{22}{3}, 
&&b_{1} = \int_{P}x_{1}dx_{1}dx_{2}dx_{3} = 0, \\
&b_{2} = \int_{P}x_{2}dx_{1}dx_{2}dx_{3} = 0,\quad 
&&b_{3} = \int_{P}x_{3}dx_{1}dx_{2}dx_{3} = 0, \\
&c_{11} = \int_{P}x_{1}^{2}dx_{1}dx_{2}dx_{3} = \frac{16}{5}, 
&&c_{12} = \int_{P}x_{1}x_{2}dx_{1}dx_{2}dx_{3} = -\frac{17}{30}, \\
&c_{13} = \int_{P}x_{1}x_{3}dx_{1}dx_{2}dx_{3} = -\frac{17}{15}, 
&&c_{22} = \int_{P}x_{2}^{2}dx_{1}dx_{2}dx_{3} = \frac{16}{5}, \\
&c_{23} = \int_{P}x_{2}x_{3}dx_{1}dx_{2}dx_{3} = \frac{17}{15}, 
&&c_{33} = \int_{P}x_{3}^{2}dx_{1}dx_{2}dx_{3} = \frac{34}{15}, 
\end{align*}
\begin{align*}
&\theta_P(x_{1}, x_{2}, x_{3}) = 0, \\
&M_{X} = \max_{P}\theta_P = 0 < 1. 
\end{align*}

\subsubsection{{$\mathcal{D}_{1}$}}
\begin{align*}
&b_{0} = \mathrm{vol}(P) = \frac{25}{3}, 
&&b_{1} = \int_{P}x_{1}dx_{1}dx_{2}dx_{3} = \frac{11}{8}, \\
&b_{2} = \int_{P}x_{2}dx_{1}dx_{2}dx_{3} = -\frac{21}{8}, 
&&b_{3} = \int_{P}x_{3}dx_{1}dx_{2}dx_{3} = \frac{5}{8}, \\
&c_{11} = \int_{P}x_{1}^{2}dx_{1}dx_{2}dx_{3} = \frac{319}{60}, 
&&c_{12} = \int_{P}x_{1}x_{2}dx_{1}dx_{2}dx_{3} = -\frac{13}{20}, \\
&c_{13} = \int_{P}x_{1}x_{3}dx_{1}dx_{2}dx_{3} = -\frac{7}{3}, 
&&c_{22} = \int_{P}x_{2}^{2}dx_{1}dx_{2}dx_{3} = \frac{11}{4}, \\
&c_{23} = \int_{P}x_{2}x_{3}dx_{1}dx_{2}dx_{3} = -\frac{21}{20}, 
&&c_{33} = \int_{P}x_{3}^{2}dx_{1}dx_{2}dx_{3} = \frac{109}{20}, 
\end{align*}
\begin{align*}
&\theta_P(x_{1}, x_{2}, x_{3}) = \frac{99600x_{1} -627000x_{2}-213939}{467581},\\
&M_{X} = \max_{P}\theta_P = \frac{711861}{467581} > 1. 
\end{align*}

\subsubsection{{$\mathcal{D}_{2}$}}
\begin{align*}
&b_{0} = \mathrm{vol}(P) = \frac{23}{3}, 
&&b_{1} = \int_{P}x_{1}dx_{1}dx_{2}dx_{3} = \frac{11}{12}, \\
&b_{2} = \int_{P}x_{2}dx_{1}dx_{2}dx_{3} = -\frac{7}{8}, 
&&b_{3} = \int_{P}x_{3}dx_{1}dx_{2}dx_{3} = -\frac{11}{24}, \\
&c_{11} = \int_{P}x_{1}^{2}dx_{1}dx_{2}dx_{3} = \frac{19}{5}, 
&&c_{12} = \int_{P}x_{1}x_{2}dx_{1}dx_{2}dx_{3} = \frac{17}{30}, \\
&c_{13} = \int_{P}x_{1}x_{3}dx_{1}dx_{2}dx_{3} = -\frac{19}{10}, 
&&c_{22} = \int_{P}x_{2}^{2}dx_{1}dx_{2}dx_{3} = \frac{47}{20}, \\
&c_{23} = \int_{P}x_{2}x_{3}dx_{1}dx_{2}dx_{3} = -\frac{17}{60}, 
&&c_{33} = \int_{P}x_{3}^{2}dx_{1}dx_{2}dx_{3} = \frac{211}{60}, 
\end{align*}
\begin{align*}
&\theta_P(x_{1}, x_{2}, x_{3}) = \frac{219420x_{1} -318320x_{2} -62565}{650251}, \\
&M_{X} = \max_{P}\theta_P = \frac{694595}{650251} > 1. 
\end{align*}

\subsubsection{{$\mathcal{E}_{1}$}}
\begin{align*}
&b_{0} = \mathrm{vol}(P) = \frac{23}{3}, 
&&b_{1} = \int_{P}x_{1}dx_{1}dx_{2}dx_{3} = -\frac{37}{24}, \\
&b_{2} = \int_{P}x_{2}dx_{1}dx_{2}dx_{3} = -\frac{37}{24}, 
&&b_{3} = \int_{P}x_{3}dx_{1}dx_{2}dx_{3} = \frac{37}{24}, \\
&c_{11} = \int_{P}x_{1}^{2}dx_{1}dx_{2}dx_{3} = \frac{151}{60}, 
&&c_{12} = \int_{P}x_{1}x_{2}dx_{1}dx_{2}dx_{3} = -\frac{7}{60}, \\
&c_{13} = \int_{P}x_{1}x_{3}dx_{1}dx_{2}dx_{3} = -\frac{6}{5}, 
&&c_{22} = \int_{P}x_{2}^{2}dx_{1}dx_{2}dx_{3} = \frac{151}{60}, \\
&c_{23} = \int_{P}x_{2}x_{3}dx_{1}dx_{2}dx_{3} = -\frac{6}{5}, 
&&c_{33} = \int_{P}x_{3}^{2}dx_{1}dx_{2}dx_{3} = \frac{311}{60}, 
\end{align*}
\begin{align*}
&\theta_P(x_{1}, x_{2}, x_{3}) = -\frac{17020x_{1}+17020x_{2}+6845}{19651}, \\
&M_{X} = \max_{P}\theta_P = \frac{27195}{19651} > 1. 
\end{align*}

\subsubsection{{$\mathcal{E}_{2}$}}
\begin{align*}
&b_{0} = \mathrm{vol}(P) = \frac{22}{3}, 
&&b_{1} = \int_{P}x_{1}dx_{1}dx_{2}dx_{3} = -\frac{7}{4}, \\
&b_{2} = \int_{P}x_{2}dx_{1}dx_{2}dx_{3} = -\frac{11}{24}, 
&&b_{3} = \int_{P}x_{3}dx_{1}dx_{2}dx_{3} = \frac{7}{8}, \\
&c_{11} = \int_{P}x_{1}^{2}dx_{1}dx_{2}dx_{3} = \frac{71}{30}, 
&&c_{12} = \int_{P}x_{1}x_{2}dx_{1}dx_{2}dx_{3} = -\frac{4}{15}, \\
&c_{13} = \int_{P}x_{1}x_{3}dx_{1}dx_{2}dx_{3} = -\frac{71}{60}, 
&&c_{22} = \int_{P}x_{2}^{2}dx_{1}dx_{2}dx_{3} = \frac{139}{60}, \\
&c_{23} = \int_{P}x_{2}x_{3}dx_{1}dx_{2}dx_{3} = \frac{2}{15}, 
&&c_{33} = \int_{P}x_{3}^{2}dx_{1}dx_{2}dx_{3} = \frac{229}{60}, 
\end{align*}
\begin{align*}
&\theta_P(x_{1}, x_{2}, x_{3}) = -\frac{2646160x_{1}+982960x_{2}+692905}{2735927}, \\
&M_{X} = \max_{P}\theta_P = \frac{2936215}{2735927} > 1. 
\end{align*}

\subsubsection{{$\mathcal{E}_{3}$}}
\begin{align*}
&b_{0} = \mathrm{vol}(P) = 7, 
&&b_{1} = \int_{P}x_{1}dx_{1}dx_{2}dx_{3} = -\frac{2}{3}, \\
&b_{2} = \int_{P}x_{2}dx_{1}dx_{2}dx_{3} = -\frac{2}{3}, 
&&b_{3} = \int_{P}x_{3}dx_{1}dx_{2}dx_{3} = 0, \\
&c_{11} = \int_{P}x_{1}^{2}dx_{1}dx_{2}dx_{3} = \frac{13}{6}, 
&&c_{12} = \int_{P}x_{1}x_{2}dx_{1}dx_{2}dx_{3} = -\frac{5}{12}, \\
&c_{13} = \int_{P}x_{1}x_{3}dx_{1}dx_{2}dx_{3} = 0, 
&&c_{22} = \int_{P}x_{2}^{2}dx_{1}dx_{2}dx_{3} = \frac{13}{6}, \\
&c_{23} = \int_{P}x_{2}x_{3}dx_{1}dx_{2}dx_{3} = 0, 
&&c_{33} = \int_{P}x_{3}^{2}dx_{1}dx_{2}dx_{3} = 7/3 
\end{align*}
\begin{align*}
&\theta_P(x_{1}, x_{2}, x_{3}) = -\frac{168x_{1}+168x_{2}+32}{409}, \\
&M_{X} = \max_{P}\theta_P = \frac{304}{409} < 1. 
\end{align*}

\subsubsection{{$\mathcal{E}_{4}$}}
\begin{align*}
&b_{0} = \mathrm{vol}(P) = \frac{20}{3}, 
&&b_{1} = \int_{P}x_{1}dx_{1}dx_{2}dx_{3} = -\frac{7}{8}, \\
&b_{2} = \int_{P}x_{2}dx_{1}dx_{2}dx_{3} = \frac{5}{12}, 
&&b_{3} = \int_{P}x_{3}dx_{1}dx_{2}dx_{3} = \frac{5}{24}, \\
&c_{11} = \int_{P}x_{1}^{2}dx_{1}dx_{2}dx_{3} = \frac{121}{60}, 
&&c_{12} = \int_{P}x_{1}x_{2}dx_{1}dx_{2}dx_{3} = -\frac{17}{30}, \\
&c_{13} = \int_{P}x_{1}x_{3}dx_{1}dx_{2}dx_{3} = -\frac{17}{60}, 
&&c_{22} = \int_{P}x_{2}^{2}dx_{1}dx_{2}dx_{3} = \frac{59}{30}, \\
&c_{23} = \int_{P}x_{2}x_{3}dx_{1}dx_{2}dx_{3} = \frac{59}{60}, 
&&c_{33} = \int_{P}x_{3}^{2}dx_{1}dx_{2}dx_{3} = \frac{181}{60}, 
\end{align*}
\begin{align*}
&\theta_P(x_{1}, x_{2}, x_{3}) = -\frac{171040x_{1} -39680x_{2}+24929}{394975}, \\
&M_{X} = \max_{P}\theta_P = \frac{185791}{394975} < 1. 
\end{align*}

\subsubsection{{$\mathcal{F}_{1}$}}
\begin{align*}
&b_{0} = \mathrm{vol}(P) = 6, 
&&b_{1} = \int_{P}x_{1}dx_{1}dx_{2}dx_{3} = 0, \\
&b_{2} = \int_{P}x_{2}dx_{1}dx_{2}dx_{3} = 0, 
&&b_{3} = \int_{P}x_{3}dx_{1}dx_{2}dx_{3} = 0, \\
&c_{11} = \int_{P}x_{1}^{2}dx_{1}dx_{2}dx_{3} = \frac{5}{3}, 
&&c_{12} = \int_{P}x_{1}x_{2}dx_{1}dx_{2}dx_{3} = \frac{5}{6}, \\
&c_{13} = \int_{P}x_{1}x_{3}dx_{1}dx_{2}dx_{3} = 0, 
&&c_{22} = \int_{P}x_{2}^{2}dx_{1}dx_{2}dx_{3} = \frac{5}{3}, \\
&c_{23} = \int_{P}x_{2}x_{3}dx_{1}dx_{2}dx_{3} = 0, 
&&c_{33} = \int_{P}x_{3}^{2}dx_{1}dx_{2}dx_{3} = 2, 
\end{align*}
\begin{align*}
&\theta_P(x_{1}, x_{2}, x_{3}) = 0, \\
&M_{X} = \max_{P}\theta_P = 0 < 1. 
\end{align*}

\subsubsection{{$\mathcal{F}_{2}$}}
\begin{align*}
&b_{0} = \mathrm{vol}(P) = 6, 
&&b_{1} = \int_{P}x_{1}dx_{1}dx_{2}dx_{3} = \frac{5}{12}, \\
&b_{2} = \int_{P}x_{2}dx_{1}dx_{2}dx_{3} = \frac{5}{6}, 
&&b_{3} = \int_{P}x_{3}dx_{1}dx_{2}dx_{3} = \frac{5}{12}, \\
&c_{11} = \int_{P}x_{1}^{2}dx_{1}dx_{2}dx_{3} = \frac{5}{3}, 
&&c_{12} = \int_{P}x_{1}x_{2}dx_{1}dx_{2}dx_{3} = \frac{5}{6}, \\
&c_{13} = \int_{P}x_{1}x_{3}dx_{1}dx_{2}dx_{3} = \frac{5}{12}, 
&&c_{22} = \int_{P}x_{2}^{2}dx_{1}dx_{2}dx_{3} = \frac{5}{3}, \\
&c_{23} = \int_{P}x_{2}x_{3}dx_{1}dx_{2}dx_{3} = \frac{5}{6}, 
&&c_{33} = \int_{P}x_{3}^{2}dx_{1}dx_{2}dx_{3} = \frac{17}{6}, 
\end{align*}
\begin{align*}
&\theta_P(x_{1}, x_{2}, x_{3}) = \frac{36x_{2}-5}{67}, \\
&M_{X} = \max_{P}\theta_P = \frac{31}{67} < 1. 
\end{align*}

\newpage
\subsection{4-dimensional Fano polytopes}
\subsubsection{{$\mathbf P^{4}$}}
\begin{align*}
&b_{0} = \mathrm{vol}(P) = \frac{625}{24}, 
&&b_{1} = \int_{P}x_{1}dx_{1}dx_{2}dx_{3}dx_{4} = 0,\quad \\
&b_{2} = \int_{P}x_{2}dx_{1}dx_{2}dx_{3}dx_{4} = 0, 
&&b_{3} = \int_{P}x_{3}dx_{1}dx_{2}dx_{3}dx_{4} = 0, \\
&b_{4} = \int_{P}x_{4}dx_{1}dx_{2}dx_{3}dx_{4} = 0, 
&&c_{11} = \int_{P}x_{1}^{2}dx_{1}dx_{2}dx_{3}dx_{4} = \frac{625}{36}, \\
&c_{12} = \int_{P}x_{1}x_{2}dx_{1}dx_{2}dx_{3}dx_{4} = -\frac{625}{144}, 
&&c_{13} = \int_{P}x_{1}x_{3}dx_{1}dx_{2}dx_{3}dx_{4} = -\frac{625}{144}, \\
&c_{14} = \int_{P}x_{1}x_{4}dx_{1}dx_{2}dx_{3}dx_{4} = -\frac{625}{144}, 
&&c_{22} = \int_{P}x_{2}^{2}dx_{1}dx_{2}dx_{3}dx_{4} = \frac{625}{36}, \\
&c_{23} = \int_{P}x_{2}x_{3}dx_{1}dx_{2}dx_{3}dx_{4} = -\frac{625}{144}, 
&&c_{24} = \int_{P}x_{2}x_{4}dx_{1}dx_{2}dx_{3}dx_{4} = -\frac{625}{144}, \\
&c_{33} = \int_{P}x_{3}^{2}dx_{1}dx_{2}dx_{3}dx_{4} = \frac{625}{36}, 
&&c_{34} = \int_{P}x_{3}x_{4}dx_{1}dx_{2}dx_{3}dx_{4} = -\frac{625}{144}, \\
&c_{44} = \int_{P}x_{4}^{2}dx_{1}dx_{2}dx_{3}dx_{4} = \frac{625}{36}, 
\end{align*}
\begin{align*}
&\theta_P(x_{1}, x_{2}, x_{3}, x_{4}) = 0, \\
&M_{X} = \max_{P}\theta_P = 0 < 1. 
\end{align*}

\subsubsection{{$B_{1}$}}
\begin{align*}
&b_{0} = \mathrm{vol}(P) = \frac{100}{3}, 
&&b_{1} = \int_{P}x_{1}dx_{1}dx_{2}dx_{3}dx_{4} = \frac{89}{20},\quad \\
&b_{2} = \int_{P}x_{2}dx_{1}dx_{2}dx_{3}dx_{4} = \frac{89}{20}, 
&&b_{3} = \int_{P}x_{3}dx_{1}dx_{2}dx_{3}dx_{4} = \frac{89}{20}, \\
&b_{4} = \int_{P}x_{4}dx_{1}dx_{2}dx_{3}dx_{4} = \frac{89}{20}, 
&&c_{11} = \int_{P}x_{1}^{2}dx_{1}dx_{2}dx_{3}dx_{4} = \frac{3757}{90}, \\
&c_{12} = \int_{P}x_{1}x_{2}dx_{1}dx_{2}dx_{3}dx_{4} = -\frac{229}{18}, 
&&c_{13} = \int_{P}x_{1}x_{3}dx_{1}dx_{2}dx_{3}dx_{4} = -\frac{229}{18}, \\
&c_{14} = \int_{P}x_{1}x_{4}dx_{1}dx_{2}dx_{3}dx_{4} = -\frac{229}{18}, 
&&c_{22} = \int_{P}x_{2}^{2}dx_{1}dx_{2}dx_{3}dx_{4} = \frac{3757}{90}, \\
&c_{23} = \int_{P}x_{2}x_{3}dx_{1}dx_{2}dx_{3}dx_{4} = -\frac{229}{18}, 
&&c_{24} = \int_{P}x_{2}x_{4}dx_{1}dx_{2}dx_{3}dx_{4} = -\frac{229}{18}, \\
&c_{33} = \int_{P}x_{3}^{2}dx_{1}dx_{2}dx_{3}dx_{4} = \frac{3757}{90}, 
&&c_{34} = \int_{P}x_{3}x_{4}dx_{1}dx_{2}dx_{3}dx_{4} = -\frac{229}{18}, \\
&c_{44} = \int_{P}x_{4}^{2}dx_{1}dx_{2}dx_{3}dx_{4} = \frac{3757}{90}, 
\end{align*}
\begin{align*}
&\theta_P(x_{1}, x_{2}, x_{3}, x_{4}) = \frac{400500x_{1}+400500x_{2} + 400500x_{3} + 400500x_{4}-213867}{108133}, \\
&M_{X} = \max_{P}\theta_P = \frac{186633}{108133} > 1. 
\end{align*}

\subsubsection{{$B_{2}$}}
\begin{align*}
&b_{0} = \mathrm{vol}(P) = \frac{80}{3}, 
&&b_{1} = \int_{P}x_{1}dx_{1}dx_{2}dx_{3}dx_{4} = \frac{28}{5},\quad \\
&b_{2} = \int_{P}x_{2}dx_{1}dx_{2}dx_{3}dx_{4} = \frac{28}{5}, 
&&b_{3} = \int_{P}x_{3}dx_{1}dx_{2}dx_{3}dx_{4} = 0, \\
&b_{4} = \int_{P}x_{4}dx_{1}dx_{2}dx_{3}dx_{4} = 0, 
&&c_{11} = \int_{P}x_{1}^{2}dx_{1}dx_{2}dx_{3}dx_{4} = \frac{1208}{45}, \\
&c_{12} = \int_{P}x_{1}x_{2}dx_{1}dx_{2}dx_{3}dx_{4} = \frac{32}{3}, 
&&c_{13} = \int_{P}x_{1}x_{3}dx_{1}dx_{2}dx_{3}dx_{4} = -\frac{728}{45}, \\
&c_{14} = \int_{P}x_{1}x_{4}dx_{1}dx_{2}dx_{3}dx_{4} = -\frac{728}{45}, 
&&c_{22} = \int_{P}x_{2}^{2}dx_{1}dx_{2}dx_{3}dx_{4} = \frac{1208}{45}, \\
&c_{23} = \int_{P}x_{2}x_{3}dx_{1}dx_{2}dx_{3}dx_{4} = -\frac{728}{45}, 
&&c_{24} = \int_{P}x_{2}x_{4}dx_{1}dx_{2}dx_{3}dx_{4} = -\frac{728}{45}, \\
&c_{33} = \int_{P}x_{3}^{2}dx_{1}dx_{2}dx_{3}dx_{4} = \frac{1456}{45}, 
&&c_{34} = \int_{P}x_{3}x_{4}dx_{1}dx_{2}dx_{3}dx_{4} = 0, \\
&c_{44} = \int_{P}x_{4}^{2}dx_{1}dx_{2}dx_{3}dx_{4} = \frac{1456}{45}, 
\end{align*}
\begin{align*}
&\theta_P(x_{1}, x_{2}, x_{3}, x_{4}) = \frac{3150x_{1}+3150x_{2}+3150x_{3}+3150x_{4}-1323}{1577}, \\
&M_{X} = \max_{P}\theta_P = \frac{1827}{1577} > 1. 
\end{align*}

\subsubsection{{$B_{3}$}}
\begin{align*}
&b_{0} = \mathrm{vol}(P) = \frac{68}{3}, 
&&b_{1} = \int_{P}x_{1}dx_{1}dx_{2}dx_{3}dx_{4} = \frac{27}{20},\quad \\
&b_{2} = \int_{P}x_{2}dx_{1}dx_{2}dx_{3}dx_{4} = \frac{27}{20}, 
&&b_{3} = \int_{P}x_{3}dx_{1}dx_{2}dx_{3}dx_{4} = \frac{27}{20}, \\
&b_{4} = \int_{P}x_{4}dx_{1}dx_{2}dx_{3}dx_{4} = \frac{27}{20}, 
&&c_{11} = \int_{P}x_{1}^{2}dx_{1}dx_{2}dx_{3}dx_{4} = \frac{1441}{90}, \\
&c_{12} = \int_{P}x_{1}x_{2}dx_{1}dx_{2}dx_{3}dx_{4} = -\frac{421}{90}, 
&&c_{13} = \int_{P}x_{1}x_{3}dx_{1}dx_{2}dx_{3}dx_{4} = -\frac{421}{90}, \\
&c_{14} = \int_{P}x_{1}x_{4}dx_{1}dx_{2}dx_{3}dx_{4} = -\frac{421}{90}, 
&&c_{22} = \int_{P}x_{2}^{2}dx_{1}dx_{2}dx_{3}dx_{4} = \frac{1441}{90}, \\
&c_{23} = \int_{P}x_{2}x_{3}dx_{1}dx_{2}dx_{3}dx_{4} = -\frac{421}{90}, 
&&c_{24} = \int_{P}x_{2}x_{4}dx_{1}dx_{2}dx_{3}dx_{4} = -\frac{421}{90}, \\
&c_{33} = \int_{P}x_{3}^{2}dx_{1}dx_{2}dx_{3}dx_{4} = \frac{1441}{90}, 
&&c_{34} = \int_{P}x_{3}x_{4}dx_{1}dx_{2}dx_{3}dx_{4} = -\frac{421}{90}, \\
&c_{44} = \int_{P}x_{4}^{2}dx_{1}dx_{2}dx_{3}dx_{4} = \frac{1441}{90}, 
\end{align*}
\begin{align*}
&\theta_P(x_{1}, x_{2}, x_{3}, x_{4}) = \frac{82620x_{1}+82620x_{2}+82620x_{3}+82620x_{4}-19683}{101357}, \\
&M_{X} = \max_{P}\theta_P = \frac{62937}{101357} < 1. 
\end{align*}
\subsubsection{{$B_{4}$}}
\begin{align*}
&b_{0} = \mathrm{vol}(P) = \frac{64}{3}, 
&&b_{1} = \int_{P}x_{1}dx_{1}dx_{2}dx_{3}dx_{4} = 0,\quad \\
&b_{2} = \int_{P}x_{2}dx_{1}dx_{2}dx_{3}dx_{4} = 0, 
&&b_{3} = \int_{P}x_{3}dx_{1}dx_{2}dx_{3}dx_{4} = 0, \\
&b_{4} = \int_{P}x_{4}dx_{1}dx_{2}dx_{3}dx_{4} = 0, 
&&c_{11} = \int_{P}x_{1}^{2}dx_{1}dx_{2}dx_{3}dx_{4} = \frac{64}{5}, \\
&c_{12} = \int_{P}x_{1}x_{2}dx_{1}dx_{2}dx_{3}dx_{4} = -\frac{64}{15}, 
&&c_{13} = \int_{P}x_{1}x_{3}dx_{1}dx_{2}dx_{3}dx_{4} = -\frac{64}{15}, \\
&c_{14} = \int_{P}x_{1}x_{4}dx_{1}dx_{2}dx_{3}dx_{4} = 0, 
&&c_{22} = \int_{P}x_{2}^{2}dx_{1}dx_{2}dx_{3}dx_{4} = \frac{64}{5}, \\
&c_{23} = \int_{P}x_{2}x_{3}dx_{1}dx_{2}dx_{3}dx_{4} = -\frac{64}{15}, 
&&c_{24} = \int_{P}x_{2}x_{4}dx_{1}dx_{2}dx_{3}dx_{4} = 0, \\
&c_{33} = \int_{P}x_{3}^{2}dx_{1}dx_{2}dx_{3}dx_{4} = \frac{64}{5}, 
&&c_{34} = \int_{P}x_{3}x_{4}dx_{1}dx_{2}dx_{3}dx_{4} = 0, \\
&c_{44} = \int_{P}x_{4}^{2}dx_{1}dx_{2}dx_{3}dx_{4} = \frac{64}{9}, 
\end{align*}
\begin{align*}
&\theta_P(x_{1}, x_{2}, x_{3}, x_{4}) = 0, \\
&M_{X} = \max_{P}\theta_P = 0 < 1. 
\end{align*}

\subsubsection{{$B_{5}$}}
\begin{align*}
&b_{0} = \mathrm{vol}(P) = \frac{64}{3}, 
&&b_{1} = \int_{P}x_{1}dx_{1}dx_{2}dx_{3}dx_{4} = -\frac{32}{15},\quad \\
&b_{2} = \int_{P}x_{2}dx_{1}dx_{2}dx_{3}dx_{4} = -\frac{32}{15}, 
&&b_{3} = \int_{P}x_{3}dx_{1}dx_{2}dx_{3}dx_{4} = \frac{16}{5}, \\
&b_{4} = \int_{P}x_{4}dx_{1}dx_{2}dx_{3}dx_{4} = \frac{16}{5}, 
&&c_{11} = \int_{P}x_{1}^{2}dx_{1}dx_{2}dx_{3}dx_{4} = \frac{512}{45}, \\
&c_{12} = \int_{P}x_{1}x_{2}dx_{1}dx_{2}dx_{3}dx_{4} = -\frac{128}{45}, 
&&c_{13} = \int_{P}x_{1}x_{3}dx_{1}dx_{2}dx_{3}dx_{4} = -\frac{128}{45}, \\
&c_{14} = \int_{P}x_{1}x_{4}dx_{1}dx_{2}dx_{3}dx_{4} = -\frac{128}{45}, 
&&c_{22} = \int_{P}x_{2}^{2}dx_{1}dx_{2}dx_{3}dx_{4} = \frac{512}{45}, \\
&c_{23} = \int_{P}x_{2}x_{3}dx_{1}dx_{2}dx_{3}dx_{4} = -\frac{128}{45}, 
&&c_{24} = \int_{P}x_{2}x_{4}dx_{1}dx_{2}dx_{3}dx_{4} = -\frac{128}{45}, \\
&c_{33} = \int_{P}x_{3}^{2}dx_{1}dx_{2}dx_{3}dx_{4} = \frac{224}{15}, 
&&c_{34} = \int_{P}x_{3}x_{4}dx_{1}dx_{2}dx_{3}dx_{4} = -\frac{32}{5}, \\
&c_{44} = \int_{P}x_{4}^{2}dx_{1}dx_{2}dx_{3}dx_{4} = \frac{224}{15}, 
\end{align*}
\begin{align*}
&\theta_P(x_{1}, x_{2}, x_{3}, x_{4}) = \frac{30x_{3}+30x_{4}-9}{71}, \\
&M_{X} = \max_{P}\theta_P = \frac{81}{71} > 1. 
\end{align*}

\subsubsection{{$C_{1}$}}
\begin{align*}
&b_{0} = \mathrm{vol}(P) = \frac{99}{4}, 
&&b_{1} = \int_{P}x_{1}dx_{1}dx_{2}dx_{3}dx_{4} = \frac{153}{10},\quad \\
&b_{2} = \int_{P}x_{2}dx_{1}dx_{2}dx_{3}dx_{4} = -\frac{153}{20}, 
&&b_{3} = \int_{P}x_{3}dx_{1}dx_{2}dx_{3}dx_{4} = \frac{51}{5}, \\
&b_{4} = \int_{P}x_{4}dx_{1}dx_{2}dx_{3}dx_{4} = \frac{51}{5}, 
&&c_{11} = \int_{P}x_{1}^{2}dx_{1}dx_{2}dx_{3}dx_{4} = \frac{837}{40}, \\
&c_{12} = \int_{P}x_{1}x_{2}dx_{1}dx_{2}dx_{3}dx_{4} = -\frac{837}{80}, 
&&c_{13} = \int_{P}x_{1}x_{3}dx_{1}dx_{2}dx_{3}dx_{4} = \frac{279}{20}, \\
&c_{14} = \int_{P}x_{1}x_{4}dx_{1}dx_{2}dx_{3}dx_{4} = \frac{279}{20}, 
&&c_{22} = \int_{P}x_{2}^{2}dx_{1}dx_{2}dx_{3}dx_{4} = \frac{81}{8}, \\
&c_{23} = \int_{P}x_{2}x_{3}dx_{1}dx_{2}dx_{3}dx_{4} = -\frac{279}{40}, 
&&c_{24} = \int_{P}x_{2}x_{4}dx_{1}dx_{2}dx_{3}dx_{4} = -\frac{279}{40}, \\
&c_{33} = \int_{P}x_{3}^{2}dx_{1}dx_{2}dx_{3}dx_{4} = \frac{1461}{40}, 
&&c_{34} = \int_{P}x_{3}x_{4}dx_{1}dx_{2}dx_{3}dx_{4} = -\frac{69}{16}, \\
&c_{44} = \int_{P}x_{4}^{2}dx_{1}dx_{2}dx_{3}dx_{4} = \frac{1461}{40}, 
\end{align*}
\begin{align*}
&\theta_P(x_{1}, x_{2}, x_{3}, x_{4}) = \frac{3740x_{1}-2312}{2803}, \\
&M_{X} = \max_{P}\theta_P = \frac{5168}{2803} > 1. 
\end{align*}

\subsubsection{{$C_{2}$}}
\begin{align*}
&b_{0} = \mathrm{vol}(P) = \frac{171}{8}, 
&&b_{1} = \int_{P}x_{1}dx_{1}dx_{2}dx_{3}dx_{4} = \frac{36}{5},\quad \\
&b_{2} = \int_{P}x_{2}dx_{1}dx_{2}dx_{3}dx_{4} = -\frac{18}{5}, 
&&b_{3} = \int_{P}x_{3}dx_{1}dx_{2}dx_{3}dx_{4} = \frac{12}{5}, \\
&b_{4} = \int_{P}x_{4}dx_{1}dx_{2}dx_{3}dx_{4} = \frac{12}{5}, 
&&c_{11} = \int_{P}x_{1}^{2}dx_{1}dx_{2}dx_{3}dx_{4} = \frac{567}{40}, \\
&c_{12} = \int_{P}x_{1}x_{2}dx_{1}dx_{2}dx_{3}dx_{4} = -\frac{567}{80}, 
&&c_{13} = \int_{P}x_{1}x_{3}dx_{1}dx_{2}dx_{3}dx_{4} = \frac{189}{40}, \\
&c_{14} = \int_{P}x_{1}x_{4}dx_{1}dx_{2}dx_{3}dx_{4} = \frac{189}{40}, 
&&c_{22} = \int_{P}x_{2}^{2}dx_{1}dx_{2}dx_{3}dx_{4} = \frac{189}{20}, \\
&c_{23} = \int_{P}x_{2}x_{3}dx_{1}dx_{2}dx_{3}dx_{4} = -\frac{189}{80}, 
&&c_{24} = \int_{P}x_{2}x_{4}dx_{1}dx_{2}dx_{3}dx_{4} = -\frac{189}{80}, \\
&c_{33} = \int_{P}x_{3}^{2}dx_{1}dx_{2}dx_{3}dx_{4} = \frac{309}{20}, 
&&c_{34} = \int_{P}x_{3}x_{4}dx_{1}dx_{2}dx_{3}dx_{4} = -\frac{429}{80}, \\
&c_{44} = \int_{P}x_{4}^{2}dx_{1}dx_{2}dx_{3}dx_{4} = \frac{309}{20}, 
\end{align*}
\begin{align*}
&\theta_P(x_{1}, x_{2}, x_{3}, x_{4}) = \frac{3040x_{1}-1024}{4961}, \\
&M_{X} = \max_{P}\theta_P = \frac{5056}{4961} > 1. 
\end{align*}

\subsubsection{{$C_{3}$}}
\begin{align*}
&b_{0} = \mathrm{vol}(P) = \frac{171}{8}, 
&&b_{1} = \int_{P}x_{1}dx_{1}dx_{2}dx_{3}dx_{4} = \frac{63}{20},\quad \\
&b_{2} = \int_{P}x_{2}dx_{1}dx_{2}dx_{3}dx_{4} = \frac{63}{20}, 
&&b_{3} = \int_{P}x_{3}dx_{1}dx_{2}dx_{3}dx_{4} = \frac{21}{10}, \\
&b_{4} = \int_{P}x_{4}dx_{1}dx_{2}dx_{3}dx_{4} = \frac{21}{10}, 
&&c_{11} = \int_{P}x_{1}^{2}dx_{1}dx_{2}dx_{3}dx_{4} = \frac{243}{20}, \\
&c_{12} = \int_{P}x_{1}x_{2}dx_{1}dx_{2}dx_{3}dx_{4} = -\frac{621}{80}, 
&&c_{13} = \int_{P}x_{1}x_{3}dx_{1}dx_{2}dx_{3}dx_{4} = \frac{117}{80}, \\
&c_{14} = \int_{P}x_{1}x_{4}dx_{1}dx_{2}dx_{3}dx_{4} = \frac{117}{80}, 
&&c_{22} = \int_{P}x_{2}^{2}dx_{1}dx_{2}dx_{3}dx_{4} = \frac{243}{20}, \\
&c_{23} = \int_{P}x_{2}x_{3}dx_{1}dx_{2}dx_{3}dx_{4} = \frac{117}{80}, 
&&c_{24} = \int_{P}x_{2}x_{4}dx_{1}dx_{2}dx_{3}dx_{4} = \frac{117}{80}, \\
&c_{33} = \int_{P}x_{3}^{2}dx_{1}dx_{2}dx_{3}dx_{4} = \frac{57}{4}, 
&&c_{34} = \int_{P}x_{3}x_{4}dx_{1}dx_{2}dx_{3}dx_{4} =-\frac{453}{80}, \\
&c_{44} = \int_{P}x_{4}^{2}dx_{1}dx_{2}dx_{3}dx_{4} = \frac{57}{4}, 
\end{align*}
\begin{align*}
&\theta_P(x_{1}, x_{2}, x_{3}, x_{4}) = \frac{21280x_{1} + 21280x_{2} - 6272}{23368}, \\
&M_{X} = \max_{P}\theta_P = \frac{1876}{2921} < 1. 
\end{align*}

\subsubsection{{$C_{4}$}}
\begin{align*}
&b_{0} = \mathrm{vol}(P) = \frac{81}{4}, 
&&b_{1} = \int_{P}x_{1}dx_{1}dx_{2}dx_{3}dx_{4} = 0,\quad \\
&b_{2} = \int_{P}x_{2}dx_{1}dx_{2}dx_{3}dx_{4} = 0, 
&&b_{3} = \int_{P}x_{3}dx_{1}dx_{2}dx_{3}dx_{4} = 0, \\
&b_{4} = \int_{P}x_{4}dx_{1}dx_{2}dx_{3}dx_{4} = 0, 
&&c_{11} = \int_{P}x_{1}^{2}dx_{1}dx_{2}dx_{3}dx_{4} = \frac{81}{8}, \\
&c_{12} = \int_{P}x_{1}x_{2}dx_{1}dx_{2}dx_{3}dx_{4} = -\frac{81}{16}, 
&&c_{13} = \int_{P}x_{1}x_{3}dx_{1}dx_{2}dx_{3}dx_{4} = 0, \\
&c_{14} = \int_{P}x_{1}x_{4}dx_{1}dx_{2}dx_{3}dx_{4} = 0, 
&&c_{22} = \int_{P}x_{2}^{2}dx_{1}dx_{2}dx_{3}dx_{4} = \frac{81}{8}, \\
&c_{23} = \int_{P}x_{2}x_{3}dx_{1}dx_{2}dx_{3}dx_{4} = 0, 
&&c_{24} = \int_{P}x_{2}x_{4}dx_{1}dx_{2}dx_{3}dx_{4} = 0, \\
&c_{33} = \int_{P}x_{3}^{2}dx_{1}dx_{2}dx_{3}dx_{4} = \frac{81}{8}, 
&&c_{34} = \int_{P}x_{3}x_{4}dx_{1}dx_{2}dx_{3}dx_{4} = -\frac{81}{16}, \\
&c_{44} = \int_{P}x_{4}^{2}dx_{1}dx_{2}dx_{3}dx_{4} = \frac{81}{8}, 
\end{align*}
\begin{align*}
&\theta_P(x_{1}, x_{2}, x_{3}, x_{4}) = 0, \\
&M_{X} = \max_{P}\theta_P = 0 < 1. 
\end{align*}

\subsubsection{{$E_{1}$}}
\begin{align*}
&b_{0} = \mathrm{vol}(P) = \frac{605}{24}, 
&&b_{1} = \int_{P}x_{1}dx_{1}dx_{2}dx_{3}dx_{4} = 12,\quad \\
&b_{2} = \int_{P}x_{2}dx_{1}dx_{2}dx_{3}dx_{4} = \frac{67}{10}, 
&&b_{3} = \int_{P}x_{3}dx_{1}dx_{2}dx_{3}dx_{4} = \frac{173}{30}, \\
&b_{4} = \int_{P}x_{4}dx_{1}dx_{2}dx_{3}dx_{4} = \frac{173}{30}, 
&&c_{11} = \int_{P}x_{1}^{2}dx_{1}dx_{2}dx_{3}dx_{4} = \frac{587}{60}, \\
&c_{12} = \int_{P}x_{1}x_{2}dx_{1}dx_{2}dx_{3}dx_{4} = \frac{367}{80}, 
&&c_{13} = \int_{P}x_{1}x_{3}dx_{1}dx_{2}dx_{3}dx_{4} = \frac{719}{144}, \\
&c_{14} = \int_{P}x_{1}x_{4}dx_{1}dx_{2}dx_{3}dx_{4} = \frac{719}{144}, 
&&c_{22} = \int_{P}x_{2}^{2}dx_{1}dx_{2}dx_{3}dx_{4} = \frac{623}{24}, \\
&c_{23} = \int_{P}x_{2}x_{3}dx_{1}dx_{2}dx_{3}dx_{4} = -\frac{1007}{180}, 
&&c_{24} = \int_{P}x_{2}x_{4}dx_{1}dx_{2}dx_{3}dx_{4} = -\frac{1007}{180}, \\
&c_{33} = \int_{P}x_{3}^{2}dx_{1}dx_{2}dx_{3}dx_{4} = \frac{4709}{180}, 
&&c_{34} = \int_{P}x_{3}x_{4}dx_{1}dx_{2}dx_{3}dx_{4} = -\frac{3809}{720}, \\
&c_{44} = \int_{P}x_{4}^{2}dx_{1}dx_{2}dx_{3}dx_{4} = \frac{4709}{180}, 
\end{align*}
\begin{align*}
&\theta_P(x_{1}, x_{2}, x_{3}, x_{4}) = \frac{4447297800x_{1}+166293600x_{2}-2161259136}{1528096589}, \\
&M_{X} = \max_{P}\theta_P = \frac{3117506664}{1528096589} > 1. 
\end{align*}

\subsubsection{{$E_{2}$}}
\begin{align*}
&b_{0} = \mathrm{vol}(P) = \frac{163}{8}, 
&&b_{1} = \int_{P}x_{1}dx_{1}dx_{2}dx_{3}dx_{4} = \frac{34}{5},\quad \\
&b_{2} = \int_{P}x_{2}dx_{1}dx_{2}dx_{3}dx_{4} = 3, 
&&b_{3} = \int_{P}x_{3}dx_{1}dx_{2}dx_{3}dx_{4} = \frac{19}{15}, \\
&b_{4} = \int_{P}x_{4}dx_{1}dx_{2}dx_{3}dx_{4} = \frac{19}{15}, 
&&c_{11} = \int_{P}x_{1}^{2}dx_{1}dx_{2}dx_{3}dx_{4} = \frac{83}{12}, \\
&c_{12} = \int_{P}x_{1}x_{2}dx_{1}dx_{2}dx_{3}dx_{4} = \frac{247}{240}, 
&&c_{13} = \int_{P}x_{1}x_{3}dx_{1}dx_{2}dx_{3}dx_{4} = \frac{157}{80}, \\
&c_{14} = \int_{P}x_{1}x_{4}dx_{1}dx_{2}dx_{3}dx_{4} = \frac{157}{80}, 
&&c_{22} = \int_{P}x_{2}^{2}dx_{1}dx_{2}dx_{3}dx_{4} = \frac{883}{60}, \\
&c_{23} = \int_{P}x_{2}x_{3}dx_{1}dx_{2}dx_{3}dx_{4} = -\frac{73}{16}, 
&&c_{24} = \int_{P}x_{2}x_{4}dx_{1}dx_{2}dx_{3}dx_{4} = -\frac{73}{16}, \\
&c_{33} = \int_{P}x_{3}^{2}dx_{1}dx_{2}dx_{3}dx_{4} = \frac{883}{60}, 
&&c_{34} = \int_{P}x_{3}x_{4}dx_{1}dx_{2}dx_{3}dx_{4} = -\frac{983}{240}, \\
&c_{44} = \int_{P}x_{4}^{2}dx_{1}dx_{2}dx_{3}dx_{4} = \frac{883}{60}, 
\end{align*}
\begin{align*}
&\theta_P(x_{1}, x_{2}, x_{3}, x_{4}) = \frac{1517634320x_{1}+215186080x_{2}-538182656}{1038057499}, \\
&M_{X} = \max_{P}\theta_P = \frac{1840195984}{1038057499} > 1. 
\end{align*}

\subsubsection{{$E_{3}$}}
\begin{align*}
&b_{0} = \mathrm{vol}(P) = \frac{431}{24}, 
&&b_{1} = \int_{P}x_{1}dx_{1}dx_{2}dx_{3}dx_{4} = \frac{11}{5},\quad \\
&b_{2} = \int_{P}x_{2}dx_{1}dx_{2}dx_{3}dx_{4} = \frac{47}{20}, 
&&b_{3} = \int_{P}x_{3}dx_{1}dx_{2}dx_{3}dx_{4} = -\frac{47}{60}, \\
&b_{4} = \int_{P}x_{4}dx_{1}dx_{2}dx_{3}dx_{4} = -\frac{47}{60}, 
&&c_{11} = \int_{P}x_{1}^{2}dx_{1}dx_{2}dx_{3}dx_{4} = \frac{329}{60}, \\
&c_{12} = \int_{P}x_{1}x_{2}dx_{1}dx_{2}dx_{3}dx_{4} = -\frac{347}{240}, 
&&c_{13} = \int_{P}x_{1}x_{3}dx_{1}dx_{2}dx_{3}dx_{4} = \frac{347}{720}, \\
&c_{14} = \int_{P}x_{1}x_{4}dx_{1}dx_{2}dx_{3}dx_{4} = \frac{347}{720}, 
&&c_{22} = \int_{P}x_{2}^{2}dx_{1}dx_{2}dx_{3}dx_{4} = \frac{1319}{120}, \\
&c_{23} = \int_{P}x_{2}x_{3}dx_{1}dx_{2}dx_{3}dx_{4} = -\frac{1319}{360}, 
&&c_{24} = \int_{P}x_{2}x_{4}dx_{1}dx_{2}dx_{3}dx_{4} = -\frac{1319}{360}, \\
&c_{33} = \int_{P}x_{3}^{2}dx_{1}dx_{2}dx_{3}dx_{4} = \frac{361}{36}, 
&&c_{34} = \int_{P}x_{3}x_{4}dx_{1}dx_{2}dx_{3}dx_{4} = -\frac{2291}{720}, \\
&c_{44} = \int_{P}x_{4}^{2}dx_{1}dx_{2}dx_{3}dx_{4} = \frac{361}{36}, 
\end{align*}
\begin{align*}
&\theta_P(x_{1}, x_{2}, x_{3}, x_{4}) = \frac{3423372660x_{1}+1994323200x_{2}-680357088}{6541476757}, \\
&M_{X} = \max_{P}\theta_P = \frac{8725985172}{6541476757} > 1. 
\end{align*}

\subsubsection{{$D_{1}$}}
\begin{align*}
&b_{0} = \mathrm{vol}(P) = \frac{74}{3}, 
&&b_{1} = \int_{P}x_{1}dx_{1}dx_{2}dx_{3}dx_{4} = \frac{118}{15},\quad \\
&b_{2} = \int_{P}x_{2}dx_{1}dx_{2}dx_{3}dx_{4} = \frac{118}{15}, 
&&b_{3} = \int_{P}x_{3}dx_{1}dx_{2}dx_{3}dx_{4} = \frac{59}{10}, \\
&b_{4} = \int_{P}x_{4}dx_{1}dx_{2}dx_{3}dx_{4} = \frac{59}{5}, 
&&c_{11} = \int_{P}x_{1}^{2}dx_{1}dx_{2}dx_{3}dx_{4} = \frac{403}{15}, \\
&c_{12} = \int_{P}x_{1}x_{2}dx_{1}dx_{2}dx_{3}dx_{4} = -\frac{203}{30}, 
&&c_{13} = \int_{P}x_{1}x_{3}dx_{1}dx_{2}dx_{3}dx_{4} = \frac{10}{3}, \\
&c_{14} = \int_{P}x_{1}x_{4}dx_{1}dx_{2}dx_{3}dx_{4} = \frac{20}{3}, 
&&c_{22} = \int_{P}x_{2}^{2}dx_{1}dx_{2}dx_{3}dx_{4} = \frac{403}{15}, \\
&c_{23} = \int_{P}x_{2}x_{3}dx_{1}dx_{2}dx_{3}dx_{4} = \frac{10}{3}, 
&&c_{24} = \int_{P}x_{2}x_{4}dx_{1}dx_{2}dx_{3}dx_{4} = \frac{20}{3}, \\
&c_{33} = \int_{P}x_{3}^{2}dx_{1}dx_{2}dx_{3}dx_{4} = \frac{697}{45}, 
&&c_{34} = \int_{P}x_{3}x_{4}dx_{1}dx_{2}dx_{3}dx_{4} = 5, \\
&c_{44} = \int_{P}x_{4}^{2}dx_{1}dx_{2}dx_{3}dx_{4} = 10, 
\end{align*}
\begin{align*}
&\theta_P(x_{1}, x_{2}, x_{3}, x_{4}) = \frac{21830x_{4} - 10443}{8057}, \\
&M_{X} = \max_{P}\theta_P = \frac{11387}{8057} > 1. 
\end{align*}

\subsubsection{{$D_{2}$}}
\begin{align*}
&b_{0} = \mathrm{vol}(P) = 24, 
&&b_{1} = \int_{P}x_{1}dx_{1}dx_{2}dx_{3}dx_{4} = \frac{157}{15},\quad \\
&b_{2} = \int_{P}x_{2}dx_{1}dx_{2}dx_{3}dx_{4} = \frac{157}{15}, 
&&b_{3} = \int_{P}x_{3}dx_{1}dx_{2}dx_{3}dx_{4} = \frac{157}{10}, \\
&b_{4} = \int_{P}x_{4}dx_{1}dx_{2}dx_{3}dx_{4} = \frac{43}{5}, 
&&c_{11} = \int_{P}x_{1}^{2}dx_{1}dx_{2}dx_{3}dx_{4} = \frac{1631}{45}, \\
&c_{12} = \int_{P}x_{1}x_{2}dx_{1}dx_{2}dx_{3}dx_{4} = -\frac{391}{90}, 
&&c_{13} = \int_{P}x_{1}x_{3}dx_{1}dx_{2}dx_{3}dx_{4} = \frac{124}{9}, \\
&c_{14} = \int_{P}x_{1}x_{4}dx_{1}dx_{2}dx_{3}dx_{4} = \frac{298}{45}, 
&&c_{22} = \int_{P}x_{2}^{2}dx_{1}dx_{2}dx_{3}dx_{4} = \frac{1631}{45}, \\
&c_{23} = \int_{P}x_{2}x_{3}dx_{1}dx_{2}dx_{3}dx_{4} = \frac{124}{9}, 
&&c_{24} = \int_{P}x_{2}x_{4}dx_{1}dx_{2}dx_{3}dx_{4} = \frac{298}{45}, \\
&c_{33} = \int_{P}x_{3}^{2}dx_{1}dx_{2}dx_{3}dx_{4} = \frac{62}{3}, 
&&c_{34} = \int_{P}x_{3}x_{4}dx_{1}dx_{2}dx_{3}dx_{4} = \frac{149}{15}, \\
&c_{44} = \int_{P}x_{4}^{2}dx_{1}dx_{2}dx_{3}dx_{4} = \frac{80}{9}, 
\end{align*}
\begin{align*}
&\theta_P(x_{1}, x_{2}, x_{3}, x_{4}) = \frac{1461480x_{3}+588060x_{4}-1166773}{1129107}, \\
&M_{X} = \max_{P}\theta_P = \frac{2344247}{1129107} > 1. 
\end{align*}

\subsubsection{{$D_{3}$}}
\begin{align*}
&b_{0} = \mathrm{vol}(P) = \frac{70}{3}, 
&&b_{1} = \int_{P}x_{1}dx_{1}dx_{2}dx_{3}dx_{4} = \frac{83}{12},\quad \\
&b_{2} = \int_{P}x_{2}dx_{1}dx_{2}dx_{3}dx_{4} = \frac{83}{12}, 
&&b_{3} = \int_{P}x_{3}dx_{1}dx_{2}dx_{3}dx_{4} = \frac{219}{20}, \\
&b_{4} = \int_{P}x_{4}dx_{1}dx_{2}dx_{3}dx_{4} = \frac{49}{5}, 
&&c_{11} = \int_{P}x_{1}^{2}dx_{1}dx_{2}dx_{3}dx_{4} = \frac{2297}{90}, \\
&c_{12} = \int_{P}x_{1}x_{2}dx_{1}dx_{2}dx_{3}dx_{4} = -\frac{262}{45}, 
&&c_{13} = \int_{P}x_{1}x_{3}dx_{1}dx_{2}dx_{3}dx_{4} = \frac{751}{90}, \\
&c_{14} = \int_{P}x_{1}x_{4}dx_{1}dx_{2}dx_{3}dx_{4} = \frac{83}{15}, 
&&c_{22} = \int_{P}x_{2}^{2}dx_{1}dx_{2}dx_{3}dx_{4} = \frac{2297}{90}, \\
&c_{23} = \int_{P}x_{2}x_{3}dx_{1}dx_{2}dx_{3}dx_{4} = \frac{751}{90}, 
&&c_{24} = \int_{P}x_{2}x_{4}dx_{1}dx_{2}dx_{3}dx_{4} = \frac{83}{15}, \\
&c_{33} = \int_{P}x_{3}^{2}dx_{1}dx_{2}dx_{3}dx_{4} = \frac{1571}{90}, 
&&c_{34} = \int_{P}x_{3}x_{4}dx_{1}dx_{2}dx_{3}dx_{4} = \frac{341}{45}, \\
&c_{44} = \int_{P}x_{4}^{2}dx_{1}dx_{2}dx_{3}dx_{4} = \frac{406}{45}, 
\end{align*}
\begin{align*}
&\theta_P(x_{1}, x_{2}, x_{3}, x_{4}) = \frac{6623400x_{3}+23783700x_{4}-13097421}{13920179}, \\
&M_{X} = \max_{P}\theta_P = \frac{23933079}{13920179} > 1. 
\end{align*}

\subsubsection{{$D_{4}$}}
\begin{align*}
&b_{0} = \mathrm{vol}(P) = \frac{70}{3}, 
&&b_{1} = \int_{P}x_{1}dx_{1}dx_{2}dx_{3}dx_{4} = \frac{157}{10},\quad \\
&b_{2} = \int_{P}x_{2}dx_{1}dx_{2}dx_{3}dx_{4} = \frac{161}{60}, 
&&b_{3} = \int_{P}x_{3}dx_{1}dx_{2}dx_{3}dx_{4} = \frac{157}{20}, \\
&b_{4} = \int_{P}x_{4}dx_{1}dx_{2}dx_{3}dx_{4} = \frac{158}{15}, 
&&c_{11} = \int_{P}x_{1}^{2}dx_{1}dx_{2}dx_{3}dx_{4} = \frac{182}{5}, \\
&c_{12} = \int_{P}x_{1}x_{2}dx_{1}dx_{2}dx_{3}dx_{4} = -7, 
&&c_{13} = \int_{P}x_{1}x_{3}dx_{1}dx_{2}dx_{3}dx_{4} = \frac{91}{5}, \\
&c_{14} = \int_{P}x_{1}x_{4}dx_{1}dx_{2}dx_{3}dx_{4} = \frac{56}{5}, 
&&c_{22} = \int_{P}x_{2}^{2}dx_{1}dx_{2}dx_{3}dx_{4} = \frac{571}{30}, \\
&c_{23} = \int_{P}x_{2}x_{3}dx_{1}dx_{2}dx_{3}dx_{4} = -\frac{7}{2}, 
&&c_{24} = \int_{P}x_{2}x_{4}dx_{1}dx_{2}dx_{3}dx_{4} = \frac{18}{5}, \\
&c_{33} = \int_{P}x_{3}^{2}dx_{1}dx_{2}dx_{3}dx_{4} = \frac{2263}{90}, 
&&c_{34} = \int_{P}x_{3}x_{4}dx_{1}dx_{2}dx_{3}dx_{4} = \frac{28}{5}, \\
&c_{44} = \int_{P}x_{4}^{2}dx_{1}dx_{2}dx_{3}dx_{4} = \frac{46}{5}, 
\end{align*}
\begin{align*}
&\theta_P(x_{1}, x_{2}, x_{3}, x_{4}) = \frac{1389500x_{1}+10897600x_{4}-5854423}{5141177}, \\
&M_{X} = \max_{P}\theta_P = \frac{10601177}{5141177} > 1. 
\end{align*}

\subsubsection{{$D_{5}$}}
\begin{align*}
&b_{0} = \mathrm{vol}(P) = \frac{62}{3}, 
&&b_{1} = \int_{P}x_{1}dx_{1}dx_{2}dx_{3}dx_{4} = \frac{16}{3},\quad \\
&b_{2} = \int_{P}x_{2}dx_{1}dx_{2}dx_{3}dx_{4} = \frac{16}{3}, 
&&b_{3} = \int_{P}x_{3}dx_{1}dx_{2}dx_{3}dx_{4} = 8, \\
&b_{4} = \int_{P}x_{4}dx_{1}dx_{2}dx_{3}dx_{4} = 0, 
&&c_{11} = \int_{P}x_{1}^{2}dx_{1}dx_{2}dx_{3}dx_{4} = 0, \\
&c_{12} = \int_{P}x_{1}x_{2}dx_{1}dx_{2}dx_{3}dx_{4} = -\frac{53}{10}, 
&&c_{13} = \int_{P}x_{1}x_{3}dx_{1}dx_{2}dx_{3}dx_{4} = \frac{76}{15}, \\
&c_{14} = \int_{P}x_{1}x_{4}dx_{1}dx_{2}dx_{3}dx_{4} = 0, 
&&c_{22} = \int_{P}x_{2}^{2}dx_{1}dx_{2}dx_{3}dx_{4} = \frac{311}{15}, \\
&c_{23} = \int_{P}x_{2}x_{3}dx_{1}dx_{2}dx_{3}dx_{4} = \frac{76}{15}, 
&&c_{24} = \int_{P}x_{2}x_{4}dx_{1}dx_{2}dx_{3}dx_{4} = 0, \\
&c_{33} = \int_{P}x_{3}^{2}dx_{1}dx_{2}dx_{3}dx_{4} = \frac{38}{5}, 
&&c_{34} = \int_{P}x_{3}x_{4}dx_{1}dx_{2}dx_{3}dx_{4} = 0, \\
&c_{44} = \int_{P}x_{4}^{2}dx_{1}dx_{2}dx_{3}dx_{4} = \frac{62}{9}, 
\end{align*}
\begin{align*}
&\theta_P(x_{1}, x_{2}, x_{3}, x_{4}) = \frac{620{{x}_{3}}-240}{349}, \\
&M_{X} = \max_{P}\theta_P = \frac{380}{349} > 1. 
\end{align*}

\subsubsection{{$D_{6}$}}
\begin{align*}
&b_{0} = \mathrm{vol}(P) = \frac{62}{3}, 
&&b_{1} = \int_{P}x_{1}dx_{1}dx_{2}dx_{3}dx_{4} = \frac{12}{5},\quad \\
&b_{2} = \int_{P}x_{2}dx_{1}dx_{2}dx_{3}dx_{4} = \frac{12}{5}, 
&&b_{3} = \int_{P}x_{3}dx_{1}dx_{2}dx_{3}dx_{4} = \frac{18}{5}, \\
&b_{4} = \int_{P}x_{4}dx_{1}dx_{2}dx_{3}dx_{4} = \frac{36}{5}, 
&&c_{11} = \int_{P}x_{1}^{2}dx_{1}dx_{2}dx_{3}dx_{4} = 14, \\
&c_{12} = \int_{P}x_{1}x_{2}dx_{1}dx_{2}dx_{3}dx_{4} = -\frac{86}{15}, 
&&c_{13} = \int_{P}x_{1}x_{3}dx_{1}dx_{2}dx_{3}dx_{4} = \frac{19}{15}, \\
&c_{14} = \int_{P}x_{1}x_{4}dx_{1}dx_{2}dx_{3}dx_{4} = \frac{38}{15}, 
&&c_{22} = \int_{P}x_{2}^{2}dx_{1}dx_{2}dx_{3}dx_{4} = 14, \\
&c_{23} = \int_{P}x_{2}x_{3}dx_{1}dx_{2}dx_{3}dx_{4} = \frac{19}{15}, 
&&c_{24} = \int_{P}x_{2}x_{4}dx_{1}dx_{2}dx_{3}dx_{4} = \frac{38}{15}, \\
&c_{33} = \int_{P}x_{3}^{2}dx_{1}dx_{2}dx_{3}dx_{4} = \frac{532}{45}, 
&&c_{34} = \int_{P}x_{3}x_{4}dx_{1}dx_{2}dx_{3}dx_{4} = \frac{19}{5}, \\
&c_{44} = \int_{P}x_{4}^{2}dx_{1}dx_{2}dx_{3}dx_{4} = \frac{38}{5}, 
\end{align*}
\begin{align*}
&\theta_P(x_{1}, x_{2}, x_{3}, x_{4}) = \frac{2790{{x}_{4}}-972}{1973}, \\
&M_{X} = \max_{P}\theta_P = \frac{1818}{1973} < 1. 
\end{align*}

\subsubsection{{$D_{7}$}}
\begin{align*}
&b_{0} = \mathrm{vol}(P) = \frac{81}{4}, 
&&b_{1} = \int_{P}x_{1}dx_{1}dx_{2}dx_{3}dx_{4} = \frac{99}{10},\quad \\
&b_{2} = \int_{P}x_{2}dx_{1}dx_{2}dx_{3}dx_{4} = -\frac{99}{20}, 
&&b_{3} = \int_{P}x_{3}dx_{1}dx_{2}dx_{3}dx_{4} = \frac{99}{20}, \\
&b_{4} = \int_{P}x_{4}dx_{1}dx_{2}dx_{3}dx_{4} = \frac{99}{20}, 
&&c_{11} = \int_{P}x_{1}^{2}dx_{1}dx_{2}dx_{3}dx_{4} = \frac{153}{10}, \\
&c_{12} = \int_{P}x_{1}x_{2}dx_{1}dx_{2}dx_{3}dx_{4} = -\frac{153}{20}, 
&&c_{13} = \int_{P}x_{1}x_{3}dx_{1}dx_{2}dx_{3}dx_{4} = \frac{153}{20}, \\
&c_{14} = \int_{P}x_{1}x_{4}dx_{1}dx_{2}dx_{3}dx_{4} = \frac{153}{20}, 
&&c_{22} = \int_{P}x_{2}^{2}dx_{1}dx_{2}dx_{3}dx_{4} = \frac{171}{20}, \\
&c_{23} = \int_{P}x_{2}x_{3}dx_{1}dx_{2}dx_{3}dx_{4} = -\frac{153}{40}, 
&&c_{24} = \int_{P}x_{2}x_{4}dx_{1}dx_{2}dx_{3}dx_{4} = -\frac{153}{40}, \\
&c_{33} = \int_{P}x_{3}^{2}dx_{1}dx_{2}dx_{3}dx_{4} = \frac{303}{20}, 
&&c_{34} = \int_{P}x_{3}x_{4}dx_{1}dx_{2}dx_{3}dx_{4} = \frac{153}{40}, \\
&c_{44} = \int_{P}x_{4}^{2}dx_{1}dx_{2}dx_{3}dx_{4} = \frac{303}{20}, 
\end{align*}
\begin{align*}
&\theta_P(x_{1}, x_{2}, x_{3}, x_{4}) = \frac{495{{x}_{1}}-242}{523}, \\
&M_{X} = \max_{P}\theta_P = \frac{748}{523} > 1. 
\end{align*}

\subsubsection{$D_{8}$}
\begin{align*}
&b_{0} = \mathrm{vol}(P) = 20, 
&&b_{1} = \int_{P}x_{1}dx_{1}dx_{2}dx_{3}dx_{4} = \frac{161}{60},\quad \\
&b_{2} = \int_{P}x_{2}dx_{1}dx_{2}dx_{3}dx_{4} = \frac{161}{60}, 
&&b_{3} = \int_{P}x_{3}dx_{1}dx_{2}dx_{3}dx_{4} = \frac{153}{40}, \\
&b_{4} = \int_{P}x_{4}dx_{1}dx_{2}dx_{3}dx_{4} = \frac{27}{5}, 
&&c_{11} = \int_{P}x_{1}^{2}dx_{1}dx_{2}dx_{3}dx_{4} = \frac{269}{18}, \\
&c_{12} = \int_{P}x_{1}x_{2}dx_{1}dx_{2}dx_{3}dx_{4} = -\frac{469}{90}, 
&&c_{13} = \int_{P}x_{1}x_{3}dx_{1}dx_{2}dx_{3}dx_{4} = \frac{407}{90}, \\
&c_{14} = \int_{P}x_{1}x_{4}dx_{1}dx_{2}dx_{3}dx_{4} = \frac{89}{45}, 
&&c_{22} = \int_{P}x_{2}^{2}dx_{1}dx_{2}dx_{3}dx_{4} = \frac{269}{18}, \\
&c_{23} = \int_{P}x_{2}x_{3}dx_{1}dx_{2}dx_{3}dx_{4} = \frac{407}{90}, 
&&c_{24} = \int_{P}x_{2}x_{4}dx_{1}dx_{2}dx_{3}dx_{4} = \frac{89}{45}, \\
&c_{33} = \int_{P}x_{3}^{2}dx_{1}dx_{2}dx_{3}dx_{4} = \frac{407}{30}, 
&&c_{34} = \int_{P}x_{3}x_{4}dx_{1}dx_{2}dx_{3}dx_{4} = \frac{89}{15}, \\
&c_{44} = \int_{P}x_{4}^{2}dx_{1}dx_{2}dx_{3}dx_{4} = \frac{316}{45}, 
\end{align*}
\begin{align*}
&\theta_P(x_{1}, x_{2}, x_{3}, x_{4}) = \frac{783290760{{x}_{1}}+783290760{{x}_{2}}-1168558320{{x}_{3}}}{2130145727}+ \frac{2808344100{{x}_{4}} -811654353}{2130145727}, \\
&M_{X} = \max_{P}\theta_P = \frac{3165248067}{2130145727} > 1. 
\end{align*}

\subsubsection{$D_{9}$}
\begin{align*}
&b_{0} = \mathrm{vol}(P) = \frac{58}{3}, 
&&b_{1} = \int_{P}x_{1}dx_{1}dx_{2}dx_{3}dx_{4} = \frac{8}{3},\quad \\
&b_{2} = \int_{P}x_{2}dx_{1}dx_{2}dx_{3}dx_{4} = \frac{8}{3}, 
&&b_{3} = \int_{P}x_{3}dx_{1}dx_{2}dx_{3}dx_{4} = 4, \\
&b_{4} = \int_{P}x_{4}dx_{1}dx_{2}dx_{3}dx_{4} = 4, 
&&c_{11} = \int_{P}x_{1}^{2}dx_{1}dx_{2}dx_{3}dx_{4} = \frac{661}{45}, \\
&c_{12} = \int_{P}x_{1}x_{2}dx_{1}dx_{2}dx_{3}dx_{4} = -\frac{449}{90}, 
&&c_{13} = \int_{P}x_{1}x_{3}dx_{1}dx_{2}dx_{3}dx_{4} = \frac{106}{45}, \\
&c_{14} = \int_{P}x_{1}x_{4}dx_{1}dx_{2}dx_{3}dx_{4} = \frac{106}{45}, 
&&c_{22} = \int_{P}x_{2}^{2}dx_{1}dx_{2}dx_{3}dx_{4} = \frac{661}{45}, \\
&c_{23} = \int_{P}x_{2}x_{3}dx_{1}dx_{2}dx_{3}dx_{4} = \frac{106}{45}, 
&&c_{24} = \int_{P}x_{2}x_{4}dx_{1}dx_{2}dx_{3}dx_{4} = \frac{106}{45}, \\
&c_{33} = \int_{P}x_{3}^{2}dx_{1}dx_{2}dx_{3}dx_{4} = \frac{298}{45}, 
&&c_{34} = \int_{P}x_{3}x_{4}dx_{1}dx_{2}dx_{3}dx_{4} = \frac{4}{9}, \\
&c_{44} = \int_{P}x_{4}^{2}dx_{1}dx_{2}dx_{3}dx_{4} = \frac{298}{45}, 
\end{align*}
\begin{align*}
&\theta_P(x_{1}, x_{2}, x_{3}, x_{4}) = \frac{870{{x}_{3}}+870{{x}_{4}}-360}{1177}, \\
&M_{X} = \max_{P}\theta_P = \frac{1380}{1177} > 1. 
\end{align*}

\subsubsection{$D_{10}$}
\begin{align*}
&b_{0} = \mathrm{vol}(P) = \frac{58}{3}, 
&&b_{1} = \int_{P}x_{1}dx_{1}dx_{2}dx_{3}dx_{4} = \frac{36}{5},\quad \\
&b_{2} = \int_{P}x_{2}dx_{1}dx_{2}dx_{3}dx_{4} = -\frac{16}{15}, 
&&b_{3} = \int_{P}x_{3}dx_{1}dx_{2}dx_{3}dx_{4} = \frac{18}{5}, \\
&b_{4} = \int_{P}x_{4}dx_{1}dx_{2}dx_{3}dx_{4} = \frac{76}{15}, 
&&c_{11} = \int_{P}x_{1}^{2}dx_{1}dx_{2}dx_{3}dx_{4} = \frac{82}{5}, \\
&c_{12} = \int_{P}x_{1}x_{2}dx_{1}dx_{2}dx_{3}dx_{4} = -6, 
&&c_{13} = \int_{P}x_{1}x_{3}dx_{1}dx_{2}dx_{3}dx_{4} = \frac{41}{5}, \\
&c_{14} = \int_{P}x_{1}x_{4}dx_{1}dx_{2}dx_{3}dx_{4} = \frac{22}{5}, 
&&c_{22} = \int_{P}x_{2}^{2}dx_{1}dx_{2}dx_{3}dx_{4} = \frac{158}{15}, \\
&c_{23} = \int_{P}x_{2}x_{3}dx_{1}dx_{2}dx_{3}dx_{4} = -3, 
&&c_{24} = \int_{P}x_{2}x_{4}dx_{1}dx_{2}dx_{3}dx_{4} = \frac{6}{5}, \\
&c_{33} = \int_{P}x_{3}^{2}dx_{1}dx_{2}dx_{3}dx_{4} = \frac{644}{45}, 
&&c_{34} = \int_{P}x_{3}x_{4}dx_{1}dx_{2}dx_{3}dx_{4} = \frac{11}{5}, \\
&c_{44} = \int_{P}x_{4}^{2}dx_{1}dx_{2}dx_{3}dx_{4} = \frac{34}{5}, 
\end{align*}
\begin{align*}
&\theta_P(x_{1}, x_{2}, x_{3}, x_{4}) = \frac{18125{{x}_{1}}+34945{{x}_{4}}-15908}{46732}, \\
&M_{X} = \max_{P}\theta_P = \frac{18353}{11683} > 1. 
\end{align*}

\subsubsection{{$D_{11}$}}
\begin{align*}
&b_{0} = \mathrm{vol}(P) = \frac{153}{8}, 
&&b_{1} = \int_{P}x_{1}dx_{1}dx_{2}dx_{3}dx_{4} = \frac{36}{5},\quad \\
&b_{2} = \int_{P}x_{2}dx_{1}dx_{2}dx_{3}dx_{4} = -\frac{18}{5}, 
&&b_{3} = \int_{P}x_{3}dx_{1}dx_{2}dx_{3}dx_{4} = \frac{39}{5}, \\
&b_{4} = \int_{P}x_{4}dx_{1}dx_{2}dx_{3}dx_{4} = \frac{39}{10}, 
&&c_{11} = \int_{P}x_{1}^{2}dx_{1}dx_{2}dx_{3}dx_{4} = \frac{261}{20}, \\
&c_{12} = \int_{P}x_{1}x_{2}dx_{1}dx_{2}dx_{3}dx_{4} = -\frac{261}{40}, 
&&c_{13} = \int_{P}x_{1}x_{3}dx_{1}dx_{2}dx_{3}dx_{4} = \frac{189}{20}, \\
&c_{14} = \int_{P}x_{1}x_{4}dx_{1}dx_{2}dx_{3}dx_{4} = \frac{189}{40}, 
&&c_{22} = \int_{P}x_{2}^{2}dx_{1}dx_{2}dx_{3}dx_{4} = \frac{333}{40}, \\
&c_{23} = \int_{P}x_{2}x_{3}dx_{1}dx_{2}dx_{3}dx_{4} = -\frac{189}{40}, 
&&c_{24} = \int_{P}x_{2}x_{4}dx_{1}dx_{2}dx_{3}dx_{4} = -\frac{189}{80}, \\
&c_{33} = \int_{P}x_{3}^{2}dx_{1}dx_{2}dx_{3}dx_{4} = \frac{321}{20}, 
&&c_{34} = \int_{P}x_{3}x_{4}dx_{1}dx_{2}dx_{3}dx_{4} = \frac{321}{40}, \\
&c_{44} = \int_{P}x_{4}^{2}dx_{1}dx_{2}dx_{3}dx_{4} = \frac{573}{40}, 
\end{align*}
\begin{align*}
&\theta_P(x_{1}, x_{2}, x_{3}, x_{4}) = \frac{465{{x}_{1}}+375{{x}_{3}}-328}{1007}, \\
&M_{X} = \max_{P}\theta_P = \frac{1727}{1007} > 1. 
\end{align*}

\subsubsection{$D_{12}$}
\begin{align*}
&b_{0} = \mathrm{vol}(P) = \frac{56}{3}, 
&&b_{1} = \int_{P}x_{1}dx_{1}dx_{2}dx_{3}dx_{4} = \frac{4}{3},\quad \\
&b_{2} = \int_{P}x_{2}dx_{1}dx_{2}dx_{3}dx_{4} = \frac{4}{3}, 
&&b_{3} = \int_{P}x_{3}dx_{1}dx_{2}dx_{3}dx_{4} = 4, \\
&b_{4} = \int_{P}x_{4}dx_{1}dx_{2}dx_{3}dx_{4} = 0, 
&&c_{11} = \int_{P}x_{1}^{2}dx_{1}dx_{2}dx_{3}dx_{4} = \frac{176}{15}, \\
&c_{12} = \int_{P}x_{1}x_{2}dx_{1}dx_{2}dx_{3}dx_{4} = -\frac{24}{5}, 
&&c_{13} = \int_{P}x_{1}x_{3}dx_{1}dx_{2}dx_{3}dx_{4} = \frac{32}{15}, \\
&c_{14} = \int_{P}x_{1}x_{4}dx_{1}dx_{2}dx_{3}dx_{4} = 0, 
&&c_{22} = \int_{P}x_{2}^{2}dx_{1}dx_{2}dx_{3}dx_{4} = \frac{176}{15}, \\
&c_{23} = \int_{P}x_{2}x_{3}dx_{1}dx_{2}dx_{3}dx_{4} = \frac{32}{15}, 
&&c_{24} = \int_{P}x_{2}x_{4}dx_{1}dx_{2}dx_{3}dx_{4} = 0, \\
&c_{33} = \int_{P}x_{3}^{2}dx_{1}dx_{2}dx_{3}dx_{4} = \frac{32}{5}, 
&&c_{34} = \int_{P}x_{3}x_{4}dx_{1}dx_{2}dx_{3}dx_{4} = 0, \\
&c_{44} = \int_{P}x_{4}^{2}dx_{1}dx_{2}dx_{3}dx_{4} = -\frac{56}{9}, 
\end{align*}
\begin{align*}
&\theta_P(x_{1}, x_{2}, x_{3}, x_{4}) = \frac{70{{x}_{3}}-15}{97}, \\
&M_{X} = \max_{P}\theta_P = \frac{55}{97} < 1. 
\end{align*}

\subsubsection{{$D_{13}$}}
\begin{align*}
&b_{0} = \mathrm{vol}(P) = 18, 
&&b_{1} = \int_{P}x_{1}dx_{1}dx_{2}dx_{3}dx_{4} = 0, \\
&b_{2} = \int_{P}x_{2}dx_{1}dx_{2}dx_{3}dx_{4} = 0, 
&&b_{3} = \int_{P}x_{3}dx_{1}dx_{2}dx_{3}dx_{4} = 0, \\
&b_{4} = \int_{P}x_{4}dx_{1}dx_{2}dx_{3}dx_{4} = 0, 
&&c_{11} = \int_{P}x_{1}^{2}dx_{1}dx_{2}dx_{3}dx_{4} = 9, \\
&c_{12} = \int_{P}x_{1}x_{2}dx_{1}dx_{2}dx_{3}dx_{4} = -\frac{9}{2}, 
&&c_{13} = \int_{P}x_{1}x_{3}dx_{1}dx_{2}dx_{3}dx_{4} = 0, \\
&c_{14} = \int_{P}x_{1}x_{4}dx_{1}dx_{2}dx_{3}dx_{4} = 0, 
&&c_{22} = \int_{P}x_{2}^{2}dx_{1}dx_{2}dx_{3}dx_{4} = 3, \\
&c_{23} = \int_{P}x_{2}x_{3}dx_{1}dx_{2}dx_{3}dx_{4} = 0, 
&&c_{24} = \int_{P}x_{2}x_{4}dx_{1}dx_{2}dx_{3}dx_{4} = 0, \\
&c_{33} = \int_{P}x_{3}^{2}dx_{1}dx_{2}dx_{3}dx_{4} = 6, 
&&c_{34} = \int_{P}x_{3}x_{4}dx_{1}dx_{2}dx_{3}dx_{4} = 0, \\
&c_{44} = \int_{P}x_{4}^{2}dx_{1}dx_{2}dx_{3}dx_{4} = 6, 
\end{align*}
\begin{align*}
&\theta_P(x_{1}, x_{2}, x_{3}, x_{4}) = 0, \\
&M_{X} = \max_{P}\theta_P = 0 < 1. 
\end{align*}

\subsubsection{{$D_{14}$}}
\begin{align*}
&b_{0} = \mathrm{vol}(P) = 18, 
&&b_{1} = \int_{P}x_{1}dx_{1}dx_{2}dx_{3}dx_{4} = \frac{9}{2},\quad \\
&b_{2} = \int_{P}x_{2}dx_{1}dx_{2}dx_{3}dx_{4} = -\frac{9}{4}, 
&&b_{3} = \int_{P}x_{3}dx_{1}dx_{2}dx_{3}dx_{4} = 0, \\
&b_{4} = \int_{P}x_{4}dx_{1}dx_{2}dx_{3}dx_{4} = \frac{9}{4}, 
&&c_{11} = \int_{P}x_{1}^{2}dx_{1}dx_{2}dx_{3}dx_{4} = \frac{54}{5}, \\
&c_{12} = \int_{P}x_{1}x_{2}dx_{1}dx_{2}dx_{3}dx_{4} = -\frac{27}{5}, 
&&c_{13} = \int_{P}x_{1}x_{3}dx_{1}dx_{2}dx_{3}dx_{4} = 0, \\
&c_{14} = \int_{P}x_{1}x_{4}dx_{1}dx_{2}dx_{3}dx_{4} = \frac{27}{5}, 
&&c_{22} = \int_{P}x_{2}^{2}dx_{1}dx_{2}dx_{3}dx_{4} = \frac{81}{10}, \\
&c_{23} = \int_{P}x_{2}x_{3}dx_{1}dx_{2}dx_{3}dx_{4} = 0, 
&&c_{24} = \int_{P}x_{2}x_{4}dx_{1}dx_{2}dx_{3}dx_{4} = -\frac{27}{10}, \\
&c_{33} = \int_{P}x_{3}^{2}dx_{1}dx_{2}dx_{3}dx_{4} = 6, 
&&c_{34} = \int_{P}x_{3}x_{4}dx_{1}dx_{2}dx_{3}dx_{4} = 0, \\
&c_{44} = \int_{P}x_{4}^{2}dx_{1}dx_{2}dx_{3}dx_{4} = \frac{111}{10}, 
\end{align*}
\begin{align*}
&\theta_P(x_{1}, x_{2}, x_{3}, x_{4}) = \frac{20{{x}_{1}}-5}{43}, \\
&M_{X} = \max_{P}\theta_P = \frac{35}{43} < 1. 
\end{align*}

\subsubsection{{$D_{15}$}}
\begin{align*}
&b_{0} = \mathrm{vol}(P) = 18, 
&&b_{1} = \int_{P}x_{1}dx_{1}dx_{2}dx_{3}dx_{4} = 0, \\
&b_{2} = \int_{P}x_{2}dx_{1}dx_{2}dx_{3}dx_{4} = 0, 
&&b_{3} = \int_{P}x_{3}dx_{1}dx_{2}dx_{3}dx_{4} = 3, \\
&b_{4} = \int_{P}x_{4}dx_{1}dx_{2}dx_{3}dx_{4} = \frac{3}{2}, 
&&c_{11} = \int_{P}x_{1}^{2}dx_{1}dx_{2}dx_{3}dx_{4} = 9, \\
&c_{12} = \int_{P}x_{1}x_{2}dx_{1}dx_{2}dx_{3}dx_{4} = -\frac{9}{2}, 
&&c_{13} = \int_{P}x_{1}x_{3}dx_{1}dx_{2}dx_{3}dx_{4} = 0, \\
&c_{14} = \int_{P}x_{1}x_{4}dx_{1}dx_{2}dx_{3}dx_{4} = 0, 
&&c_{22} = \int_{P}x_{2}^{2}dx_{1}dx_{2}dx_{3}dx_{4} = 9, \\
&c_{23} = \int_{P}x_{2}x_{3}dx_{1}dx_{2}dx_{3}dx_{4} = 0, 
&&c_{24} = \int_{P}x_{2}x_{4}dx_{1}dx_{2}dx_{3}dx_{4} = 0, \\
&c_{33} = \int_{P}x_{3}^{2}dx_{1}dx_{2}dx_{3}dx_{4} = 6, 
&&c_{34} = \int_{P}x_{3}x_{4}dx_{1}dx_{2}dx_{3}dx_{4} = 3, \\
&c_{44} = \int_{P}x_{4}^{2}dx_{1}dx_{2}dx_{3}dx_{4} = 9, 
\end{align*}
\begin{align*}
&\theta_P(x_{1}, x_{2}, x_{3}, x_{4}) = \frac{6{{x}_{3}}-1}{11}, \\
&M_{X} = \max_{P}\theta_P = \frac{5}{11} < 1. 
\end{align*}

\subsubsection{{$D_{16}$}}
\begin{align*}
&b_{0} = \mathrm{vol}(P) = 18, 
&&b_{1} = \int_{P}x_{1}dx_{1}dx_{2}dx_{3}dx_{4} = \frac{79}{60}, \\
&b_{2} = \int_{P}x_{2}dx_{1}dx_{2}dx_{3}dx_{4} = \frac{79}{60}, 
&&b_{3} = \int_{P}x_{3}dx_{1}dx_{2}dx_{3}dx_{4} = \frac{107}{20}, \\
&b_{4} = \int_{P}x_{4}dx_{1}dx_{2}dx_{3}dx_{4} = \frac{7}{5}, 
&&c_{11} = \int_{P}x_{1}^{2}dx_{1}dx_{2}dx_{3}dx_{4} = \frac{1033}{90}, \\
&c_{12} = \int_{P}x_{1}x_{2}dx_{1}dx_{2}dx_{3}dx_{4} = -\frac{206}{45}, 
&&c_{13} = \int_{P}x_{1}x_{3}dx_{1}dx_{2}dx_{3}dx_{4} = \frac{35}{18}, \\
&c_{14} = \int_{P}x_{1}x_{4}dx_{1}dx_{2}dx_{3}dx_{4} = -\frac{17}{45}, 
&&c_{22} = \int_{P}x_{2}^{2}dx_{1}dx_{2}dx_{3}dx_{4} = \frac{1033}{90}, \\
&c_{23} = \int_{P}x_{2}x_{3}dx_{1}dx_{2}dx_{3}dx_{4} = \frac{35}{18}, 
&&c_{24} = \int_{P}x_{2}x_{4}dx_{1}dx_{2}dx_{3}dx_{4} = -\frac{17}{45}, \\
&c_{33} = \int_{P}x_{3}^{2}dx_{1}dx_{2}dx_{3}dx_{4} = \frac{947}{90}, 
&&c_{34} = \int_{P}x_{3}x_{4}dx_{1}dx_{2}dx_{3}dx_{4} = \frac{211}{45}, \\
&c_{44} = \int_{P}x_{4}^{2}dx_{1}dx_{2}dx_{3}dx_{4} = \frac{262}{45}, 
\end{align*}
\begin{align*}
&\theta_P(x_{1}, x_{2}, x_{3}, x_{4}) = \frac{796536{{x}_{3}}-335484{{x}_{4}}-210655}{1061921}, \\
&M_{X} = \max_{P}\theta_P = \frac{1046933}{1061921} < 1. 
\end{align*}

\subsubsection{{$D_{17}$}}
\begin{align*}
&b_{0} = \mathrm{vol}(P) = \frac{135}{8}, 
&&b_{1} = \int_{P}x_{1}dx_{1}dx_{2}dx_{3}dx_{4} = \frac{9}{5},\quad \\
&b_{2} = \int_{P}x_{2}dx_{1}dx_{2}dx_{3}dx_{4} = \frac{9}{5}, 
&&b_{3} = \int_{P}x_{3}dx_{1}dx_{2}dx_{3}dx_{4} = \frac{9}{10}, \\
&b_{4} = \int_{P}x_{4}dx_{1}dx_{2}dx_{3}dx_{4} = \frac{9}{10}, 
&&c_{11} = \int_{P}x_{1}^{2}dx_{1}dx_{2}dx_{3}dx_{4} = \frac{171}{20}, \\
&c_{12} = \int_{P}x_{1}x_{2}dx_{1}dx_{2}dx_{3}dx_{4} = -\frac{99}{20}, 
&&c_{13} = \int_{P}x_{1}x_{3}dx_{1}dx_{2}dx_{3}dx_{4} = \frac{171}{40}, \\
&c_{14} = \int_{P}x_{1}x_{4}dx_{1}dx_{2}dx_{3}dx_{4} = -\frac{99}{40}, 
&&c_{22} = \int_{P}x_{2}^{2}dx_{1}dx_{2}dx_{3}dx_{4} = \frac{171}{20}, \\
&c_{23} = \int_{P}x_{2}x_{3}dx_{1}dx_{2}dx_{3}dx_{4} = -\frac{99}{40}, 
&&c_{24} = \int_{P}x_{2}x_{4}dx_{1}dx_{2}dx_{3}dx_{4} = \frac{171}{40}, \\
&c_{33} = \int_{P}x_{3}^{2}dx_{1}dx_{2}dx_{3}dx_{4} = \frac{363}{40}, 
&&c_{34} = \int_{P}x_{3}x_{4}dx_{1}dx_{2}dx_{3}dx_{4} = -\frac{99}{80}, \\
&c_{44} = \int_{P}x_{4}^{2}dx_{1}dx_{2}dx_{3}dx_{4} = \frac{363}{40}, 
\end{align*}
\begin{align*}
&\theta_P(x_{1}, x_{2}, x_{3}, x_{4}) = \frac{75{{x}_{1}}+75{{x}_{2}}-16}{134}, \\
&M_{X} = \max_{P}\theta_P = \frac{59}{134} < 1. 
\end{align*}

\subsubsection{$D_{18}$}
\begin{align*}
&b_{0} = \mathrm{vol}(P) = -\frac{99}{80}, 
&&b_{1} = \int_{P}x_{1}dx_{1}dx_{2}dx_{3}dx_{4} = \frac{14}{5}, \\
&b_{2} = \int_{P}x_{2}dx_{1}dx_{2}dx_{3}dx_{4} = \frac{14}{5}, 
&&b_{3} = \int_{P}x_{3}dx_{1}dx_{2}dx_{3}dx_{4} = -\frac{21}{10}, \\
&b_{4} = \int_{P}x_{4}dx_{1}dx_{2}dx_{3}dx_{4} = -\frac{21}{5}, 
&&c_{11} = \int_{P}x_{1}^{2}dx_{1}dx_{2}dx_{3}dx_{4} = \frac{73}{5}, \\
&c_{12} = \int_{P}x_{1}x_{2}dx_{1}dx_{2}dx_{3}dx_{4} = -\frac{23}{6}, 
&&c_{13} = \int_{P}x_{1}x_{3}dx_{1}dx_{2}dx_{3}dx_{4} = -\frac{26}{15}, \\
&c_{14} = \int_{P}x_{1}x_{4}dx_{1}dx_{2}dx_{3}dx_{4} = -\frac{52}{15}, 
&&c_{22} = \int_{P}x_{2}^{2}dx_{1}dx_{2}dx_{3}dx_{4} = \frac{73}{5}, \\
&c_{23} = \int_{P}x_{2}x_{3}dx_{1}dx_{2}dx_{3}dx_{4} = -\frac{26}{15}, 
&&c_{24} = \int_{P}x_{2}x_{4}dx_{1}dx_{2}dx_{3}dx_{4} = -\frac{52}{15}, \\
&c_{33} = \int_{P}x_{3}^{2}dx_{1}dx_{2}dx_{3}dx_{4} = \frac{53}{9}, 
&&c_{34} = \int_{P}x_{3}x_{4}dx_{1}dx_{2}dx_{3}dx_{4} = \frac{13}{5}, \\
&c_{44} = \int_{P}x_{4}^{2}dx_{1}dx_{2}dx_{3}dx_{4} = \frac{26}{5}, 
\end{align*}
\begin{align*}
&\theta_P(x_{1}, x_{2}, x_{3}, x_{4}) = \frac{-5250{{x}_{4}}-1323}{5177}, \\
&M_{X} = \max_{P}\theta_P = \frac{3927}{5177} < 1. 
\end{align*}

\subsubsection{$D_{19}$}
\begin{align*}
&b_{0} = \mathrm{vol}(P) = -\frac{99}{80}, 
&&b_{1} = \int_{P}x_{1}dx_{1}dx_{2}dx_{3}dx_{4} = \frac{4}{15},\quad \\
&b_{2} = \int_{P}x_{2}dx_{1}dx_{2}dx_{3}dx_{4} = \frac{4}{15}, 
&&b_{3} = \int_{P}x_{3}dx_{1}dx_{2}dx_{3}dx_{4} = -\frac{2}{5}, \\
&b_{4} = \int_{P}x_{4}dx_{1}dx_{2}dx_{3}dx_{4} = -\frac{4}{5}, 
&&c_{11} = \int_{P}x_{1}^{2}dx_{1}dx_{2}dx_{3}dx_{4} = \frac{142}{15}, \\
&c_{12} = \int_{P}x_{1}x_{2}dx_{1}dx_{2}dx_{3}dx_{4} = -\frac{58}{15}, 
&&c_{13} = \int_{P}x_{1}x_{3}dx_{1}dx_{2}dx_{3}dx_{4} = -\frac{13}{15}, \\
&c_{14} = \int_{P}x_{1}x_{4}dx_{1}dx_{2}dx_{3}dx_{4} = -\frac{26}{15}, 
&&c_{22} = \int_{P}x_{2}^{2}dx_{1}dx_{2}dx_{3}dx_{4} = \frac{142}{15}, \\
&c_{23} = \int_{P}x_{2}x_{3}dx_{1}dx_{2}dx_{3}dx_{4} = -\frac{13}{15}, 
&&c_{24} = \int_{P}x_{2}x_{4}dx_{1}dx_{2}dx_{3}dx_{4} = -\frac{26}{15}, \\
&c_{33} = \int_{P}x_{3}^{2}dx_{1}dx_{2}dx_{3}dx_{4} = \frac{316}{45}, 
&&c_{34} = \int_{P}x_{3}x_{4}dx_{1}dx_{2}dx_{3}dx_{4} = \frac{13}{5}, \\
&c_{44} = \int_{P}x_{4}^{2}dx_{1}dx_{2}dx_{3}dx_{4} = \frac{26}{5}, 
\end{align*}
\begin{align*}
&\theta_P(x_{1}, x_{2}, x_{3}, x_{4}) = -\frac{250{{x}_{4}}+12}{1613}, \\
&M_{X} = \max_{P}\theta_P = \frac{238}{1613} < 1. 
\end{align*}

\subsubsection{$G_{1}$}
\begin{align*}
&b_{0} = \mathrm{vol}(P) = \frac{529}{24}, 
&&b_{1} = \int_{P}x_{1}dx_{1}dx_{2}dx_{3}dx_{4} = \frac{93}{10}, \\
&b_{2} = \int_{P}x_{2}dx_{1}dx_{2}dx_{3}dx_{4} = \frac{35}{6}, 
&&b_{3} = \int_{P}x_{3}dx_{1}dx_{2}dx_{3}dx_{4} = -\frac{71}{30}, \\
&b_{4} = \int_{P}x_{4}dx_{1}dx_{2}dx_{3}dx_{4} = \frac{629}{60}, 
&&c_{11} = \int_{P}x_{1}^{2}dx_{1}dx_{2}dx_{3}dx_{4} = \frac{1517}{180}, \\
&c_{12} = \int_{P}x_{1}x_{2}dx_{1}dx_{2}dx_{3}dx_{4} = \frac{2803}{720}, 
&&c_{13} = \int_{P}x_{1}x_{3}dx_{1}dx_{2}dx_{3}dx_{4} = \frac{77}{120}, \\
&c_{14} = \int_{P}x_{1}x_{4}dx_{1}dx_{2}dx_{3}dx_{4} = \frac{5837}{720}, 
&&c_{22} = \int_{P}x_{2}^{2}dx_{1}dx_{2}dx_{3}dx_{4} = \frac{3149}{180}, \\
&c_{23} = \int_{P}x_{2}x_{3}dx_{1}dx_{2}dx_{3}dx_{4} = -\frac{989}{180}, 
&&c_{24} = \int_{P}x_{2}x_{4}dx_{1}dx_{2}dx_{3}dx_{4} = \frac{4781}{720}, \\
&c_{33} = \int_{P}x_{3}^{2}dx_{1}dx_{2}dx_{3}dx_{4} = \frac{4187}{360}, 
&&c_{34} = \int_{P}x_{3}x_{4}dx_{1}dx_{2}dx_{3}dx_{4} = -\frac{745}{144}, \\
&c_{44} = \int_{P}x_{4}^{2}dx_{1}dx_{2}dx_{3}dx_{4} = \frac{5147}{180}, 
\end{align*}
\begin{align*}
&\theta_P(x_{1}, x_{2}, x_{3}, x_{4}) = \frac{37598378760{{x}_{1}}-8882851620{{x}_{3}}-16817588112}{16641653953}, \\
&M_{X} = \max_{P}\theta_P = \frac{29663642268}{16641653953} > 1. 
\end{align*}

\subsubsection{$G_{2}$}
\begin{align*}
&b_{0} = \mathrm{vol}(P) = \frac{75}{4}, 
&&b_{1} = \int_{P}x_{1}dx_{1}dx_{2}dx_{3}dx_{4} = \frac{89}{10},\quad \\
&b_{2} = \int_{P}x_{2}dx_{1}dx_{2}dx_{3}dx_{4} = -\frac{89}{20}, 
&&b_{3} = \int_{P}x_{3}dx_{1}dx_{2}dx_{3}dx_{4} = \frac{197}{15}, \\
&b_{4} = \int_{P}x_{4}dx_{1}dx_{2}dx_{3}dx_{4} = \frac{7}{3}, 
&&c_{11} = \int_{P}x_{1}^{2}dx_{1}dx_{2}dx_{3}dx_{4} = \frac{517}{40}, \\
&c_{12} = \int_{P}x_{1}x_{2}dx_{1}dx_{2}dx_{3}dx_{4} = -\frac{517}{80}, 
&&c_{13} = \int_{P}x_{1}x_{3}dx_{1}dx_{2}dx_{3}dx_{4} = \frac{595}{36}, \\
&c_{14} = \int_{P}x_{1}x_{4}dx_{1}dx_{2}dx_{3}dx_{4} = \frac{839}{180}, 
&&c_{22} = \int_{P}x_{2}^{2}dx_{1}dx_{2}dx_{3}dx_{4} = \frac{911}{120}, \\
&c_{23} = \int_{P}x_{2}x_{3}dx_{1}dx_{2}dx_{3}dx_{4} = -\frac{595}{72}, 
&&c_{24} = \int_{P}x_{2}x_{4}dx_{1}dx_{2}dx_{3}dx_{4} = -\frac{839}{360}, \\
&c_{33} = \int_{P}x_{3}^{2}dx_{1}dx_{2}dx_{3}dx_{4} = \frac{12301}{360}, 
&&c_{34} = \int_{P}x_{3}x_{4}dx_{1}dx_{2}dx_{3}dx_{4} = -\frac{401}{720}, \\
&c_{44} = \int_{P}x_{4}^{2}dx_{1}dx_{2}dx_{3}dx_{4} = \frac{1727}{120}, 
\end{align*}
\begin{align*}
&\theta_P(x_{1}, x_{2}, x_{3}, x_{4}) = \frac{1410100500{{x}_{1}}+366948000{{x}_{3}}-926354392}{1802902233}, \\
&M_{X} = \max_{P}\theta_P = \frac{372321328}{163900203} > 1. 
\end{align*}

\subsubsection{$G_{3}$}
\begin{align*}
&b_{0} = \mathrm{vol}(P) = \frac{433}{24}, 
&&b_{1} = \int_{P}x_{1}dx_{1}dx_{2}dx_{3}dx_{4} = \frac{89}{30},\quad \\
&b_{2} = \int_{P}x_{2}dx_{1}dx_{2}dx_{3}dx_{4} = \frac{133}{60}, 
&&b_{3} = \int_{P}x_{3}dx_{1}dx_{2}dx_{3}dx_{4} = -\frac{133}{30}, \\
&b_{4} = \int_{P}x_{4}dx_{1}dx_{2}dx_{3}dx_{4} = \frac{37}{10}, 
&&c_{11} = \int_{P}x_{1}^{2}dx_{1}dx_{2}dx_{3}dx_{4} = \frac{217}{36}, \\
&c_{12} = \int_{P}x_{1}x_{2}dx_{1}dx_{2}dx_{3}dx_{4} = \frac{19}{720}, 
&&c_{13} = \int_{P}x_{1}x_{3}dx_{1}dx_{2}dx_{3}dx_{4} = -\frac{19}{360}, \\
&c_{14} = \int_{P}x_{1}x_{4}dx_{1}dx_{2}dx_{3}dx_{4} = \frac{2189}{720}, 
&&c_{22} = \int_{P}x_{2}^{2}dx_{1}dx_{2}dx_{3}dx_{4} = \frac{1787}{180}, \\
&c_{23} = \int_{P}x_{2}x_{3}dx_{1}dx_{2}dx_{3}dx_{4} = -\frac{877}{240}, 
&&c_{24} = \int_{P}x_{2}x_{4}dx_{1}dx_{2}dx_{3}dx_{4} = \frac{265}{144}, \\
&c_{33} = \int_{P}x_{3}^{2}dx_{1}dx_{2}dx_{3}dx_{4} = \frac{877}{120}, 
&&c_{34} = \int_{P}x_{3}x_{4}dx_{1}dx_{2}dx_{3}dx_{4} = -\frac{265}{72}, \\
&c_{44} = \int_{P}x_{4}^{2}dx_{1}dx_{2}dx_{3}dx_{4} = \frac{2333}{180}, 
\end{align*}
\begin{align*}
&\theta_P(x_{1}, x_{2}, x_{3}, x_{4}) = \frac{28885436928x_{1} - 35779946976x_{3} -13541869440}{45785112888}, \\
&M_{X} = \max_{P}\theta_P = \frac{2130146436}{1907713037} > 1. 
\end{align*}

\subsubsection{$G_{4}$}
\begin{align*}
&b_{0} = \mathrm{vol}(P) = \frac{139}{8}, 
&&b_{1} = \int_{P}x_{1}dx_{1}dx_{2}dx_{3}dx_{4} = -\frac{17}{20},\quad \\
&b_{2} = \int_{P}x_{2}dx_{1}dx_{2}dx_{3}dx_{4} = \frac{301}{60}, 
&&b_{3} = \int_{P}x_{3}dx_{1}dx_{2}dx_{3}dx_{4} = \frac{127}{30}, \\
&b_{4} = \int_{P}x_{4}dx_{1}dx_{2}dx_{3}dx_{4} = -\frac{1}{30}, 
&&c_{11} = \int_{P}x_{1}^{2}dx_{1}dx_{2}dx_{3}dx_{4} = \frac{147}{20}, \\
&c_{12} = \int_{P}x_{1}x_{2}dx_{1}dx_{2}dx_{3}dx_{4} = -\frac{785}{144}, 
&&c_{13} = \int_{P}x_{1}x_{3}dx_{1}dx_{2}dx_{3}dx_{4} = \frac{2333}{720}, \\
&c_{14} = \int_{P}x_{1}x_{4}dx_{1}dx_{2}dx_{3}dx_{4} = -\frac{161}{240}, 
&&c_{22} = \int_{P}x_{2}^{2}dx_{1}dx_{2}dx_{3}dx_{4} = \frac{383}{36}, \\
&c_{23} = \int_{P}x_{2}x_{3}dx_{1}dx_{2}dx_{3}dx_{4} = \frac{53}{80}, 
&&c_{24} = \int_{P}x_{2}x_{4}dx_{1}dx_{2}dx_{3}dx_{4} = \frac{181}{80}, \\
&c_{33} = \int_{P}x_{3}x_{3}dx_{1}dx_{2}dx_{3}dx_{4} = \frac{2261}{180}, 
&&c_{34} = \int_{P}x_{3}x_{4}dx_{1}dx_{2}dx_{3}dx_{4} = -\frac{1039}{240}, \\
&c_{44} = \int_{P}x_{4}^{2}dx_{1}dx_{2}dx_{3}dx_{4} = \frac{583}{60}, 
\end{align*}
\begin{align*}
&\theta_P(x_{1}, x_{2}, x_{3}, x_{4}) = \frac{10356297206700x_{1}+33083176976940x_{2}+16155069457920x_{3}}{48057952407691} \\
&\quad \quad \quad \quad \quad \quad \quad \quad - \frac{12981537414704}{48057952407691}, \\
&M_{X} = \max_{P}\theta_P = \frac{91293727706236}{48057952407691} > 1. 
\end{align*}

\bigskip 

\subsubsection{$G_{5}$}
\begin{align*}
&b_{0} = \mathrm{vol}(P) = \frac{203}{12}, 
&&b_{1} = \int_{P}x_{1}dx_{1}dx_{2}dx_{3}dx_{4} = -\frac{16}{5},\quad \\
&b_{2} = \int_{P}x_{2}dx_{1}dx_{2}dx_{3}dx_{4} = \frac{8}{5}, 
&&b_{3} = \int_{P}x_{3}dx_{1}dx_{2}dx_{3}dx_{4} = \frac{28}{15}, \\
&b_{4} = \int_{P}x_{4}dx_{1}dx_{2}dx_{3}dx_{4} = -\frac{14}{15}, 
&&c_{11} = \int_{P}x_{1}^{2}dx_{1}dx_{2}dx_{3}dx_{4} = \frac{767}{120}, \\
&c_{12} = \int_{P}x_{1}x_{2}dx_{1}dx_{2}dx_{3}dx_{4} = -\frac{767}{240}, 
&&c_{13} = \int_{P}x_{1}x_{3}dx_{1}dx_{2}dx_{3}dx_{4} = \frac{68}{45}, \\
&c_{14} = \int_{P}x_{1}x_{4}dx_{1}dx_{2}dx_{3}dx_{4} = -\frac{34}{45}, 
&&c_{22} = \int_{P}x_{2}^{2}dx_{1}dx_{2}dx_{3}dx_{4} = \frac{3181}{360}, \\
&c_{23} = \int_{P}x_{2}x_{3}dx_{1}dx_{2}dx_{3}dx_{4} = -\frac{34}{45}, 
&&c_{24} = \int_{P}x_{2}x_{4}dx_{1}dx_{2}dx_{3}dx_{4} = \frac{17}{45}, \\
&c_{33} = \int_{P}x_{3}^{2}dx_{1}dx_{2}dx_{3}dx_{4} = \frac{1039}{120}, 
&&c_{34} = \int_{P}x_{3}x_{4}dx_{1}dx_{2}dx_{3}dx_{4} = -\frac{1039}{240}, \\
&c_{44} = \int_{P}x_{4}^{2}dx_{1}dx_{2}dx_{3}dx_{4} = \frac{569}{72}, 
\end{align*}
\begin{align*}
&\theta_P(x_{1}, x_{2}, x_{3}, x_{4}) = \frac{-4015697280{{x}_{1}}+2205554400{{x}_{3}}-1002991104}{5976434111}, \\
&M_{X} = \max_{P}\theta_P = \frac{7423814976}{5976434111} > 1. 
\end{align*}

\newpage
\subsubsection{$G_{6}$}
\begin{align*}
&b_{0} = \mathrm{vol}(P) = \frac{401}{24}, 
&&b_{1} = \int_{P}x_{1}dx_{1}dx_{2}dx_{3}dx_{4} = \frac{12}{5}, \\
&b_{2} = \int_{P}x_{2}dx_{1}dx_{2}dx_{3}dx_{4} = -\frac{6}{5}, 
&&b_{3} = \int_{P}x_{3}dx_{1}dx_{2}dx_{3}dx_{4} = \frac{24}{5}, \\
&b_{4} = \int_{P}x_{4}dx_{1}dx_{2}dx_{3}dx_{4} = -\frac{6}{5}, 
&&c_{11} = \int_{P}x_{1}^{2}dx_{1}dx_{2}dx_{3}dx_{4} = \frac{997}{120}, \\
&c_{12} = \int_{P}x_{1}x_{2}dx_{1}dx_{2}dx_{3}dx_{4} = -\frac{997}{240}, 
&&c_{13} = \int_{P}x_{1}x_{3}dx_{1}dx_{2}dx_{3}dx_{4} = \frac{2437}{360}, \\
&c_{14} = \int_{P}x_{1}x_{4}dx_{1}dx_{2}dx_{3}dx_{4} = \frac{277}{360}, 
&&c_{22} = \int_{P}x_{2}^{2}dx_{1}dx_{2}dx_{3}dx_{4} = \frac{1357}{180}, \\
&c_{23} = \int_{P}x_{2}x_{3}dx_{1}dx_{2}dx_{3}dx_{4} = -\frac{2437}{720}, 
&&c_{24} = \int_{P}x_{2}x_{4}dx_{1}dx_{2}dx_{3}dx_{4} = -\frac{277}{720}, \\
&c_{33} = \int_{P}x_{3}^{2}dx_{1}dx_{2}dx_{3}dx_{4} = \frac{2437}{180}, 
&&c_{34} = \int_{P}x_{3}x_{4}dx_{1}dx_{2}dx_{3}dx_{4} = -\frac{2437}{720}, \\
&c_{44} = \int_{P}x_{4}^{2}dx_{1}dx_{2}dx_{3}dx_{4} = \frac{1357}{180}, 
\end{align*}
\begin{align*}
&\theta_P(x_{1}, x_{2}, x_{3}, x_{4}) = \frac{1732320{{x}_{3}}-497664}{4388521}, \\
&M_{X} = \max_{P}\theta_P = \frac{6431616}{4388521} > 1. 
\end{align*}

\subsubsection{{$H_{1}$}}
\begin{align*}
&b_{0} = \mathrm{vol}(P) = \frac{93}{4}, 
&&b_{1} = \int_{P}x_{1}dx_{1}dx_{2}dx_{3}dx_{4} = \frac{153}{40}, \\
&b_{2} = \int_{P}x_{2}dx_{1}dx_{2}dx_{3}dx_{4} = \frac{161}{10}, 
&&b_{3} = \int_{P}x_{3}dx_{1}dx_{2}dx_{3}dx_{4} = \frac{161}{15}, \\
&b_{4} = \int_{P}x_{4}dx_{1}dx_{2}dx_{3}dx_{4} = \frac{161}{15}, 
&&c_{11} = \int_{P}x_{1}^{2}dx_{1}dx_{2}dx_{3}dx_{4} = \frac{3043}{360}, \\
&c_{12} = \int_{P}x_{1}x_{2}dx_{1}dx_{2}dx_{3}dx_{4} = \frac{2449}{240}, 
&&c_{13} = \int_{P}x_{1}x_{3}dx_{1}dx_{2}dx_{3}dx_{4} = \frac{2449}{360}, \\
&c_{14} = \int_{P}x_{1}x_{4}dx_{1}dx_{2}dx_{3}dx_{4} = \frac{2449}{360}, 
&&c_{22} = \int_{P}x_{2}^{2}dx_{1}dx_{2}dx_{3}dx_{4} = \frac{2449}{120}, \\
&c_{23} = \int_{P}x_{2}x_{3}dx_{1}dx_{2}dx_{3}dx_{4} = \frac{2449}{180}, 
&&c_{24} = \int_{P}x_{2}x_{4}dx_{1}dx_{2}dx_{3}dx_{4} = \frac{2449}{180}, \\
&c_{33} = \int_{P}x_{3}^{2}dx_{1}dx_{2}dx_{3}dx_{4} = \frac{12947}{360}, 
&&c_{34} = \int_{P}x_{3}x_{4}dx_{1}dx_{2}dx_{3}dx_{4} = -\frac{3151}{720}, \\
&c_{44} = \int_{P}x_{4}^{2}dx_{1}dx_{2}dx_{3}dx_{4} = \frac{12947}{360}, 
\end{align*}
\begin{align*}
&\theta_P(x_{1}, x_{2}, x_{3}, x_{4}) = \frac{299460{{x}_{2}}-207368}{172227}, \\
&M_{X} = \max_{P}\theta_P = \frac{391552}{172227} > 1. 
\end{align*}

\subsubsection{{$H_{2}$}}
\begin{align*}
&b_{0} = \mathrm{vol}(P) = \frac{505}{24}, 
&&b_{1} = \int_{P}x_{1}dx_{1}dx_{2}dx_{3}dx_{4} = \frac{163}{20}, \\
&b_{2} = \int_{P}x_{2}dx_{1}dx_{2}dx_{3}dx_{4} = \frac{247}{20}, 
&&b_{3} = \int_{P}x_{3}dx_{1}dx_{2}dx_{3}dx_{4} = \frac{41}{6}, \\
&b_{4} = \int_{P}x_{4}dx_{1}dx_{2}dx_{3}dx_{4} = \frac{41}{6}, 
&&c_{11} = \int_{P}x_{1}^{2}dx_{1}dx_{2}dx_{3}dx_{4} = \frac{1391}{180}, \\
&c_{12} = \int_{P}x_{1}x_{2}dx_{1}dx_{2}dx_{3}dx_{4} = \frac{6139}{720}, 
&&c_{13} = \int_{P}x_{1}x_{3}dx_{1}dx_{2}dx_{3}dx_{4} = \frac{3901}{720}, \\
&c_{14} = \int_{P}x_{1}x_{4}dx_{1}dx_{2}dx_{3}dx_{4} = \frac{3901}{720}, 
&&c_{22} = \int_{P}x_{2}^{2}dx_{1}dx_{2}dx_{3}dx_{4} = \frac{2963}{180}, \\
&c_{23} = \int_{P}x_{2}x_{3}dx_{1}dx_{2}dx_{3}dx_{4} = \frac{1999}{240}, 
&&c_{24} = \int_{P}x_{2}x_{4}dx_{1}dx_{2}dx_{3}dx_{4} = \frac{1999}{240}, \\
&c_{33} = \int_{P}x_{3}^{2}dx_{1}dx_{2}dx_{3}dx_{4} = \frac{4361}{180}, 
&&c_{34} = \int_{P}x_{3}x_{4}dx_{1}dx_{2}dx_{3}dx_{4} = -\frac{3773}{720}, \\
&c_{44} = \int_{P}x_{4}^{2}dx_{1}dx_{2}dx_{3}dx_{4} = \frac{4361}{180}, 
\end{align*}
\begin{align*}
&\theta_P(x_{1}, x_{2}, x_{3}, x_{4}) = \frac{12590952900{{x}_{1}}+11321625300{{x}_{2}}-11521822032}{12261327193}, \\
&M_{X} = \max_{P}\theta_P = \frac{23712381468}{12261327193} > 1. 
\end{align*}

\subsubsection{{$H_{3}$}}
\begin{align*}
&b_{0} = \mathrm{vol}(P) = \frac{239}{12}, 
&&b_{1} = \int_{P}x_{1}dx_{1}dx_{2}dx_{3}dx_{4} = \frac{42}{5}, \\
&b_{2} = \int_{P}x_{2}dx_{1}dx_{2}dx_{3}dx_{4} = \frac{179}{20}, 
&&b_{3} = \int_{P}x_{3}dx_{1}dx_{2}dx_{3}dx_{4} = \frac{28}{5}, \\
&b_{4} = \int_{P}x_{4}dx_{1}dx_{2}dx_{3}dx_{4} = \frac{28}{5}, 
&&c_{11} = \int_{P}x_{1}^{2}dx_{1}dx_{2}dx_{3}dx_{4} = \frac{881}{120}, \\
&c_{12} = \int_{P}x_{1}x_{2}dx_{1}dx_{2}dx_{3}dx_{4} = \frac{1697}{240}, 
&&c_{13} = \int_{P}x_{1}x_{3}dx_{1}dx_{2}dx_{3}dx_{4} = \frac{881}{180}, \\
&c_{14} = \int_{P}x_{1}x_{4}dx_{1}dx_{2}dx_{3}dx_{4} = \frac{881}{180}, 
&&c_{22} = \int_{P}x_{2}^{2}dx_{1}dx_{2}dx_{3}dx_{4} = \frac{4771}{360}, \\
&c_{23} = \int_{P}x_{2}x_{3}dx_{1}dx_{2}dx_{3}dx_{4} = \frac{1697}{360}, 
&&c_{24} = \int_{P}x_{2}x_{4}dx_{1}dx_{2}dx_{3}dx_{4} = \frac{1697}{360}, \\
&c_{33} = \int_{P}x_{3}^{2}dx_{1}dx_{2}dx_{3}dx_{4} = \frac{7363}{360}, 
&&c_{34} = \int_{P}x_{3}x_{4}dx_{1}dx_{2}dx_{3}dx_{4} = -\frac{3839}{720}, \\
&c_{44} = \int_{P}x_{4}^{2}dx_{1}dx_{2}dx_{3}dx_{4} = \frac{7363}{360}, 
\end{align*}
\begin{align*}
&\theta_P(x_{1}, x_{2}, x_{3}, x_{4}) = \frac{1983987756{{x}_{1}}+260718408{{{{x}_{}}}_{2}}-953921016}{999563887}, \\
&M_{X} = \max_{P}\theta_P = \frac{1551503556}{999563887} > 1. 
\end{align*}

\subsubsection{{$H_{4}$}}
\begin{align*}
&b_{0} = \mathrm{vol}(P) = \frac{149}{8}, 
&&b_{1} = \int_{P}x_{1}dx_{1}dx_{2}dx_{3}dx_{4} = \frac{89}{20}, \\
&b_{2} = \int_{P}x_{2}dx_{1}dx_{2}dx_{3}dx_{4} = \frac{89}{10}, 
&&b_{3} = \int_{P}x_{3}dx_{1}dx_{2}dx_{3}dx_{4} = \frac{89}{30}, \\
&b_{4} = \int_{P}x_{4}dx_{1}dx_{2}dx_{3}dx_{4} = \frac{89}{30}, 
&&c_{11} = \int_{P}x_{1}^{2}dx_{1}dx_{2}dx_{3}dx_{4} = \frac{1133}{180}, \\
&c_{12} = \int_{P}x_{1}x_{2}dx_{1}dx_{2}dx_{3}dx_{4} = \frac{311}{48}, 
&&c_{13} = \int_{P}x_{1}x_{3}dx_{1}dx_{2}dx_{3}dx_{4} = \frac{311}{144}, \\
&c_{14} = \int_{P}x_{1}x_{4}dx_{1}dx_{2}dx_{3}dx_{4} = \frac{311}{144}, 
&&c_{22} = \int_{P}x_{2}^{2}dx_{1}dx_{2}dx_{3}dx_{4} = \frac{311}{24}, \\
&c_{23} = \int_{P}x_{2}x_{3}dx_{1}dx_{2}dx_{3}dx_{4} = \frac{311}{72}, 
&&c_{24} = \int_{P}x_{2}x_{4}dx_{1}dx_{2}dx_{3}dx_{4} = \frac{311}{72}, \\
&c_{33} = \int_{P}x_{3}^{2}dx_{1}dx_{2}dx_{3}dx_{4} = \frac{2599}{180}, 
&&c_{34} = \int_{P}x_{3}x_{4}dx_{1}dx_{2}dx_{3}dx_{4} = -\frac{3643}{720}, \\
&c_{44} = \int_{P}x_{4}^{2}dx_{1}dx_{2}dx_{3}dx_{4} = \frac{2599}{180}, 
\end{align*}
\begin{align*}
&\theta_P(x_{1}, x_{2}, x_{3}, x_{4}) = \frac{795660{{x}_{2}}-380208}{778267}, \\
&M_{X} = \max_{P}\theta_P = \frac{173016}{111181} > 1. 
\end{align*}

\subsubsection{{$H_{5}$}}
\begin{align*}
&b_{0} = \mathrm{vol}(P) = \frac{415}{24}, 
&&b_{1} = \int_{P}x_{1}dx_{1}dx_{2}dx_{3}dx_{4} = \frac{97}{20}, \\
&b_{2} = \int_{P}x_{2}dx_{1}dx_{2}dx_{3}dx_{4} = \frac{29}{5}, 
&&b_{3} = \int_{P}x_{3}dx_{1}dx_{2}dx_{3}dx_{4} = \frac{97}{60}, \\
&b_{4} = \int_{P}x_{4}dx_{1}dx_{2}dx_{3}dx_{4} = \frac{97}{60}, 
&&c_{11} = \int_{P}x_{1}^{2}dx_{1}dx_{2}dx_{3}dx_{4} = \frac{139}{24}, \\
&c_{12} = \int_{P}x_{1}x_{2}dx_{1}dx_{2}dx_{3}dx_{4} = \frac{1259}{240}, 
&&c_{13} = \int_{P}x_{1}x_{3}dx_{1}dx_{2}dx_{3}dx_{4} = \frac{139}{72}, \\
&c_{14} = \int_{P}x_{1}x_{4}dx_{1}dx_{2}dx_{3}dx_{4} = \frac{139}{72}, 
&&c_{22} = \int_{P}x_{2}^{2}dx_{1}dx_{2}dx_{3}dx_{4} = \frac{367}{36}, \\
&c_{23} = \int_{P}x_{2}x_{3}dx_{1}dx_{2}dx_{3}dx_{4} = \frac{1259}{720}, 
&&c_{24} = \int_{P}x_{2}x_{4}dx_{1}dx_{2}dx_{3}dx_{4} = \frac{1259}{720}, \\
&c_{33} = \int_{P}x_{3}^{2}dx_{1}dx_{2}dx_{3}dx_{4} = \frac{2021}{180}, 
&&c_{34} = \int_{P}x_{3}x_{4}dx_{1}dx_{2}dx_{3}dx_{4} = -\frac{3347}{720}, \\
&c_{44} = \int_{P}x_{4}^{2}dx_{1}dx_{2}dx_{3}dx_{4} = \frac{2021}{180}, 
\end{align*}
\begin{align*}
&\theta_P(x_{1}, x_{2}, x_{3}, x_{4}) = \frac{6818815200{{x}_{1}}+2922039900{{x}_{2}}-2892669984}{8410595791}, \\
&M_{X} = \max_{P}\theta_P = \frac{9770225016}{8410595791} > 1. 
\end{align*}

\subsubsection{{$H_{6}$}}
\begin{align*}
&b_{0} = \mathrm{vol}(P) = \frac{409}{24}, 
&&b_{1} = \int_{P}x_{1}dx_{1}dx_{2}dx_{3}dx_{4} = \frac{27}{5}, \\
&b_{2} = \int_{P}x_{2}dx_{1}dx_{2}dx_{3}dx_{4} = \frac{61}{20}, 
&&b_{3} = \int_{P}x_{3}dx_{1}dx_{2}dx_{3}dx_{4} = \frac{31}{12}, \\
&b_{4} = \int_{P}x_{4}dx_{1}dx_{2}dx_{3}dx_{4} = \frac{31}{12}, 
&&c_{11} = \int_{P}x_{1}^{2}dx_{1}dx_{2}dx_{3}dx_{4} = \frac{1013}{180}, \\
&c_{12} = \int_{P}x_{1}x_{2}dx_{1}dx_{2}dx_{3}dx_{4} = \frac{3049}{720}, 
&&c_{13} = \int_{P}x_{1}x_{3}dx_{1}dx_{2}dx_{3}dx_{4} = \frac{337}{144}, \\
&c_{14} = \int_{P}x_{1}x_{4}dx_{1}dx_{2}dx_{3}dx_{4} = \frac{337}{144}, 
&&c_{22} = \int_{P}x_{2}^{2}dx_{1}dx_{2}dx_{3}dx_{4} = \frac{2941}{260}, \\
&c_{23} = \int_{P}x_{2}x_{3}dx_{1}dx_{2}dx_{3}dx_{4} = \frac{1}{10}, 
&&c_{24} = \int_{P}x_{2}x_{4}dx_{1}dx_{2}dx_{3}dx_{4} = \frac{1}{10}, \\
&c_{33} = \int_{P}x_{3}^{2}dx_{1}dx_{2}dx_{3}dx_{4} = \frac{2411}{180}, 
&&c_{34} = \int_{P}x_{3}x_{4}dx_{1}dx_{2}dx_{3}dx_{4} = -\frac{3173}{720}, \\
&c_{44} = \int_{P}x_{4}^{2}dx_{1}dx_{2}dx_{3}dx_{4} = \frac{2411}{180}, 
\end{align*}
\begin{align*}
&\theta_P(x_{1}, x_{2}, x_{3}, x_{4}) = \frac{11025012180{{x}_{1}}-2015224800{{x}_{2}} -3132829152}{6776874793}, \\
&M_{X} = \max_{P}\theta_P = \frac{7892183028}{6776874793} > 1. 
\end{align*}

\subsubsection{{$H_{7}$}}
\begin{align*}
&b_{0} = \mathrm{vol}(P) = \frac{191}{12}, 
&&b_{1} = \int_{P}x_{1}dx_{1}dx_{2}dx_{3}dx_{4} = \frac{61}{20}, \\
&b_{2} = \int_{P}x_{2}dx_{1}dx_{2}dx_{3}dx_{4} = -\frac{31}{20}, 
&&b_{3} = \int_{P}x_{3}dx_{1}dx_{2}dx_{3}dx_{4} = \frac{46}{15}, \\
&b_{4} = \int_{P}x_{4}dx_{1}dx_{2}dx_{3}dx_{4} = \frac{46}{15}, 
&&c_{11} = \int_{P}x_{1}^{2}dx_{1}dx_{2}dx_{3}dx_{4} = \frac{335}{72}, \\
&c_{12} = \int_{P}x_{1}x_{2}dx_{1}dx_{2}dx_{3}dx_{4} = \frac{1859}{720}, 
&&c_{13} = \int_{P}x_{1}x_{3}dx_{1}dx_{2}dx_{3}dx_{4} = \frac{497}{360}, \\
&c_{14} = \int_{P}x_{1}x_{4}dx_{1}dx_{2}dx_{3}dx_{4} = \frac{497}{360}, 
&&c_{22} = \int_{P}x_{2}^{2}dx_{1}dx_{2}dx_{3}dx_{4} = \frac{1963}{360}, \\
&c_{23} = \int_{P}x_{2}x_{3}dx_{1}dx_{2}dx_{3}dx_{4} = -\frac{689}{360}, 
&&c_{24} = \int_{P}x_{2}x_{4}dx_{1}dx_{2}dx_{3}dx_{4} = -\frac{689}{360}, \\
&c_{33} = \int_{P}x_{3}^{2}dx_{1}dx_{2}dx_{3}dx_{4} = \frac{1031}{72}, 
&&c_{34} = \int_{P}x_{3}x_{4}dx_{1}dx_{2}dx_{3}dx_{4} = -\frac{2783}{720}, \\
&c_{44} = \int_{P}x_{4}^{2}dx_{1}dx_{2}dx_{3}dx_{4} = \frac{1031}{72}, 
\end{align*}
\begin{align*}
&\theta_P(x_{1}, x_{2}, x_{3}, x_{4}) = \frac{3404937900{{x}_{1}}-2489673540{{x}_{2}}-894914424}{2191715291}, \\
&M_{X} = \max_{P}\theta_P = \frac{2510023476}{2191715291} > 1. 
\end{align*}

\subsubsection{{$H_{8}$}}
\begin{align*}
&b_{0} = \mathrm{vol}(P) = \frac{63}{4}, 
&&b_{1} = \int_{P}x_{1}dx_{1}dx_{2}dx_{3}dx_{4} = \frac{3}{2}, \\
&b_{2} = \int_{P}x_{2}dx_{1}dx_{2}dx_{3}dx_{4} = 3, 
&&b_{3} = \int_{P}x_{3}dx_{1}dx_{2}dx_{3}dx_{4} = 0, \\
&b_{4} = \int_{P}x_{4}dx_{1}dx_{2}dx_{3}dx_{4} = 0, 
&&c_{11} = \int_{P}x_{1}^{2}dx_{1}dx_{2}dx_{3}dx_{4} = \frac{33}{8}, \\
&c_{12} = \int_{P}x_{1}x_{2}dx_{1}dx_{2}dx_{3}dx_{4} = \frac{63}{16}, 
&&c_{13} = \int_{P}x_{1}x_{3}dx_{1}dx_{2}dx_{3}dx_{4} = 0, \\
&c_{14} = \int_{P}x_{1}x_{4}dx_{1}dx_{2}dx_{3}dx_{4} = 0, 
&&c_{22} = \int_{P}x_{2}^{2}dx_{1}dx_{2}dx_{3}dx_{4} = \frac{63}{8}, \\
&c_{23} = \int_{P}x_{2}x_{3}dx_{1}dx_{2}dx_{3}dx_{4} = 0, 
&&c_{24} = \int_{P}x_{2}x_{4}dx_{1}dx_{2}dx_{3}dx_{4} = 0, \\
&c_{33} = \int_{P}x_{3}^{2}dx_{1}dx_{2}dx_{3}dx_{4} = \frac{63}{8}, 
&&c_{34} = \int_{P}x_{3}x_{4}dx_{1}dx_{2}dx_{3}dx_{4} = -\frac{63}{16}, \\
&c_{44} = \int_{P}x_{4}^{2}dx_{1}dx_{2}dx_{3}dx_{4} = \frac{63}{8}, 
\end{align*}
\begin{align*}
&\theta_P(x_{1}, x_{2}, x_{3}, x_{4}) = \frac{168{{x}_{2}}-32}{409}, \\
&M_{X} = \max_{P}\theta_P = \frac{304}{409} < 1. 
\end{align*}

\subsubsection{$H_{9}$}
\begin{align*}
&b_{0} = \mathrm{vol}(P) = \frac{367}{24}, 
&&b_{1} = \int_{P}x_{1}dx_{1}dx_{2}dx_{3}dx_{4} = \frac{11}{5},\quad \\
&b_{2} = \int_{P}x_{2}dx_{1}dx_{2}dx_{3}dx_{4} = \frac{11}{20}, 
&&b_{3} = \int_{P}x_{3}dx_{1}dx_{2}dx_{3}dx_{4} = \frac{11}{20}, \\
&b_{4} = \int_{P}x_{4}dx_{1}dx_{2}dx_{3}dx_{4} = \frac{11}{20}, 
&&c_{11} = \int_{P}x_{1}^{2}dx_{1}dx_{2}dx_{3}dx_{4} = \frac{827}{180}, \\
&c_{12} = \int_{P}x_{1}x_{2}dx_{1}dx_{2}dx_{3}dx_{4} = \frac{2267}{720}, 
&&c_{13} = \int_{P}x_{1}x_{3}dx_{1}dx_{2}dx_{3}dx_{4} = \frac{347}{720}, \\
&c_{14} = \int_{P}x_{1}x_{4}dx_{1}dx_{2}dx_{3}dx_{4} = \frac{347}{720}, 
&&c_{22} = \int_{P}x_{2}^{2}dx_{1}dx_{2}dx_{3}dx_{4} = \frac{1133}{180}, \\
&c_{23} = \int_{P}x_{2}x_{3}dx_{1}dx_{2}dx_{3}dx_{4} = -\frac{151}{144}, 
&&c_{24} = \int_{P}x_{2}x_{4}dx_{1}dx_{2}dx_{3}dx_{4} = -\frac{151}{144}, \\
&c_{33} = \int_{P}x_{3}^{2}dx_{1}dx_{2}dx_{3}dx_{4} = \frac{1613}{180}, 
&&c_{34} = \int_{P}x_{3}x_{4}dx_{1}dx_{2}dx_{3}dx_{4} = -\frac{535}{144}, \\
&c_{44} = \int_{P}x_{4}^{2}dx_{1}dx_{2}dx_{3}dx_{4} = \frac{1613}{180}, 
\end{align*}
\begin{align*}
&\theta_P(x_{1}, x_{2}, x_{3}, x_{4}) = -\frac{-3841851420{{x}_{1}}+1395187200{{x}_{2}}+502543008}{5523943807}, \\
&M_{X} = \max_{P}\theta_P = \frac{3339308412}{5523943807} < 1. 
\end{align*}

\subsubsection{{$H_{10}$}}
\begin{align*}
&b_{0} = \mathrm{vol}(P) = \frac{117}{8}, 
&&b_{1} = \int_{P}x_{1}dx_{1}dx_{2}dx_{3}dx_{4} = -\frac{4}{5}, \\
&b_{2} = \int_{P}x_{2}dx_{1}dx_{2}dx_{3}dx_{4} = \frac{117}{8}, 
&&b_{3} = \int_{P}x_{3}dx_{1}dx_{2}dx_{3}dx_{4} = \frac{117}{8}, \\
&b_{4} = \int_{P}x_{4}dx_{1}dx_{2}dx_{3}dx_{4} = \frac{117}{8}, 
&&c_{11} = \int_{P}x_{1}^{2}dx_{1}dx_{2}dx_{3}dx_{4} = \frac{151}{36}, \\
&c_{12} = \int_{P}x_{1}x_{2}dx_{1}dx_{2}dx_{3}dx_{4} = \frac{619}{240}, 
&&c_{13} = \int_{P}x_{1}x_{3}dx_{1}dx_{2}dx_{3}dx_{4} = -\frac{619}{720}, \\
&c_{14} = \int_{P}x_{1}x_{4}dx_{1}dx_{2}dx_{3}dx_{4} = -\frac{619}{720}, 
&&c_{22} = \int_{P}x_{2}^{2}dx_{1}dx_{2}dx_{3}dx_{4} = \frac{619}{120}, \\
&c_{23} = \int_{P}x_{2}x_{3}dx_{1}dx_{2}dx_{3}dx_{4} = -\frac{619}{360}, 
&&c_{24} = \int_{P}x_{2}x_{4}dx_{1}dx_{2}dx_{3}dx_{4} = -\frac{619}{360}, \\
&c_{33} = \int_{P}x_{3}^{2}dx_{1}dx_{2}dx_{3}dx_{4} = \frac{1567}{180}, 
&&c_{34} = \int_{P}x_{3}x_{4}dx_{1}dx_{2}dx_{3}dx_{4} = -\frac{503}{144}, \\
&c_{44} = \int_{P}x_{4}^{2}dx_{1}dx_{2}dx_{3}dx_{4} = \frac{1567}{180}, 
\end{align*}
\begin{align*}
&\theta_P(x_{1}, x_{2}, x_{3}, x_{4}) = \frac{-37440{{x}_{2}}-4096}{116609}, \\
&M_{X} = \max_{P}\theta_P = \frac{33344}{116609} < 1. 
\end{align*}

\subsubsection{{$L_{1}$}}
\begin{align*}
&b_{0} = \mathrm{vol}(P) = 20, 
&&b_{1} = \int_{P}x_{1}dx_{1}dx_{2}dx_{3}dx_{4} = \frac{42}{5}, \\
&b_{2} = \int_{P}x_{2}dx_{1}dx_{2}dx_{3}dx_{4} = \frac{21}{5}, 
&&b_{3} = \int_{P}x_{3}dx_{1}dx_{2}dx_{3}dx_{4} = \frac{21}{5}, \\
&b_{4} = \int_{P}x_{4}dx_{1}dx_{2}dx_{3}dx_{4} = \frac{21}{5}, 
&&c_{11} = \int_{P}x_{1}^{2}dx_{1}dx_{2}dx_{3}dx_{4} = \frac{116}{15}, \\
&c_{12} = \int_{P}x_{1}x_{2}dx_{1}dx_{2}dx_{3}dx_{4} = \frac{58}{15}, 
&&c_{13} = \int_{P}x_{1}x_{3}dx_{1}dx_{2}dx_{3}dx_{4} = \frac{58}{15}, \\
&c_{14} = \int_{P}x_{1}x_{4}dx_{1}dx_{2}dx_{3}dx_{4} = \frac{58}{15}, 
&&c_{22} = \int_{P}x_{2}^{2}dx_{1}dx_{2}dx_{3}dx_{4} = \frac{542}{45}, \\
&c_{23} = \int_{P}x_{2}x_{3}dx_{1}dx_{2}dx_{3}dx_{4} = \frac{29}{15}, 
&&c_{24} = \int_{P}x_{2}x_{4}dx_{1}dx_{2}dx_{3}dx_{4} = \frac{29}{15}, \\
&c_{33} = \int_{P}x_{3}^{2}dx_{1}dx_{2}dx_{3}dx_{4} = \frac{542}{45}, 
&&c_{34} = \int_{P}x_{3}x_{4}dx_{1}dx_{2}dx_{3}dx_{4} = \frac{29}{15}, \\
&c_{44} = \int_{P}x_{4}^{2}dx_{1}dx_{2}dx_{3}dx_{4} = \frac{542}{45}, 
\end{align*}
\begin{align*}
&\theta_P(x_{1}, x_{2}, x_{3}, x_{4}) = \frac{3150{{x}_{1}}-1323}{1577}, \\
&M_{X} = \max_{P}\theta_P = \frac{1827}{1577} > 1. 
\end{align*}

\subsubsection{{$L_{2}$}}
\begin{align*}
&b_{0} = \mathrm{vol}(P) = \frac{58}{3}, 
&&b_{1} = \int_{P}x_{1}dx_{1}dx_{2}dx_{3}dx_{4} = \frac{92}{15}, \\
&b_{2} = \int_{P}x_{2}dx_{1}dx_{2}dx_{3}dx_{4} = -\frac{43}{10}, 
&&b_{3} = \int_{P}x_{3}dx_{1}dx_{2}dx_{3}dx_{4} = \frac{313}{60}, \\
&b_{4} = \int_{P}x_{4}dx_{1}dx_{2}dx_{3}dx_{4} = \frac{313}{60}, 
&&c_{11} = \int_{P}x_{1}^{2}dx_{1}dx_{2}dx_{3}dx_{4} = \frac{314}{45}, \\
&c_{12} = \int_{P}x_{1}x_{2}dx_{1}dx_{2}dx_{3}dx_{4} = \frac{2}{45}, 
&&c_{13} = \int_{P}x_{1}x_{3}dx_{1}dx_{2}dx_{3}dx_{4} = \frac{52}{15}, \\
&c_{14} = \int_{P}x_{1}x_{4}dx_{1}dx_{2}dx_{3}dx_{4} = \frac{52}{15}, 
&&c_{22} = \int_{P}x_{2}^{2}dx_{1}dx_{2}dx_{3}dx_{4} = \frac{362}{45}, \\
&c_{23} = \int_{P}x_{2}x_{3}dx_{1}dx_{2}dx_{3}dx_{4} = -4, 
&&c_{24} = \int_{P}x_{2}x_{4}dx_{1}dx_{2}dx_{3}dx_{4} = -4, \\
&c_{33} = \int_{P}x_{3}^{2}dx_{1}dx_{2}dx_{3}dx_{4} = \frac{149}{10}, 
&&c_{34} = \int_{P}x_{3}x_{4}dx_{1}dx_{2}dx_{3}dx_{4} = \frac{56}{15}, \\
&c_{44} = \int_{P}x_{4}^{2}dx_{1}dx_{2}dx_{3}dx_{4} = \frac{149}{10}, 
\end{align*}
\begin{align*}
&\theta_P(x_{1}, x_{2}, x_{3}, x_{4}) = \frac{3878228{{x}_{1}}-2370692{{x}_{2}}-1757609}{2637399}, \\
&M_{X} = \max_{P}\theta_P = \frac{4491311}{2637399} > 1. 
\end{align*}

\subsubsection{{$L_{3}$}}
\begin{align*}
&b_{0} = \mathrm{vol}(P) = \frac{56}{3}, 
&&b_{1} = \int_{P}x_{1}dx_{1}dx_{2}dx_{3}dx_{4} = \frac{103}{15}, \\
&b_{2} = \int_{P}x_{2}dx_{1}dx_{2}dx_{3}dx_{4} = \frac{103}{30}, 
&&b_{3} = \int_{P}x_{3}dx_{1}dx_{2}dx_{3}dx_{4} = \frac{112}{15}, \\
&b_{4} = \int_{P}x_{4}dx_{1}dx_{2}dx_{3}dx_{4} = \frac{56}{15}, 
&&c_{11} = \int_{P}x_{1}^{2}dx_{1}dx_{2}dx_{3}dx_{4} = \frac{104}{15}, \\
&c_{12} = \int_{P}x_{1}x_{2}dx_{1}dx_{2}dx_{3}dx_{4} = \frac{52}{15}, 
&&c_{13} = \int_{P}x_{1}x_{3}dx_{1}dx_{2}dx_{3}dx_{4} = \frac{248}{45}, \\
&c_{14} = \int_{P}x_{1}x_{4}dx_{1}dx_{2}dx_{3}dx_{4} = \frac{124}{45}, 
&&c_{22} = \int_{P}x_{2}^{2}dx_{1}dx_{2}dx_{3}dx_{4} = \frac{487}{45}, \\
&c_{23} = \int_{P}x_{2}x_{3}dx_{1}dx_{2}dx_{3}dx_{4} = \frac{124}{45}, 
&&c_{24} = \int_{P}x_{2}x_{4}dx_{1}dx_{2}dx_{3}dx_{4} = \frac{62}{45}, \\
&c_{33} = \int_{P}x_{3}^{2}dx_{1}dx_{2}dx_{3}dx_{4} = \frac{193}{15}, 
&&c_{34} = \int_{P}x_{3}x_{4}dx_{1}dx_{2}dx_{3}dx_{4} = \frac{193}{30}, \\
&c_{44} = \int_{P}x_{4}^{2}dx_{1}dx_{2}dx_{3}dx_{4} = 13, 
\end{align*}
\begin{align*}
&\theta_P(x_{1}, x_{2}, x_{3}, x_{4}) = \frac{26763240{{x}_{1}} + 7896000{{x}_{3}}-13003449}{20356871}, \\
&M_{X} = \max_{P}\theta_P = \frac{29551791}{20356871} > 1. 
\end{align*}

\subsubsection{{$L_{4}$}}
\begin{align*}
&b_{0} = \mathrm{vol}(P) = 18, 
&&b_{1} = \int_{P}x_{1}dx_{1}dx_{2}dx_{3}dx_{4} = \frac{76}{15}, \\
&b_{2} = \int_{P}x_{2}dx_{1}dx_{2}dx_{3}dx_{4} = -\frac{169}{60}, 
&&b_{3} = \int_{P}x_{3}dx_{1}dx_{2}dx_{3}dx_{4} = \frac{241}{30}, \\
&b_{4} = \int_{P}x_{4}dx_{1}dx_{2}dx_{3}dx_{4} = \frac{241}{60}, 
&&c_{11} = \int_{P}x_{1}^{2}dx_{1}dx_{2}dx_{3}dx_{4} = \frac{286}{45}, \\
&c_{12} = \int_{P}x_{1}x_{2}dx_{1}dx_{2}dx_{3}dx_{4} = \frac{34}{45}, 
&&c_{13} = \int_{P}x_{1}x_{3}dx_{1}dx_{2}dx_{3}dx_{4} = \frac{22}{5}, \\
&c_{14} = \int_{P}x_{1}x_{4}dx_{1}dx_{2}dx_{3}dx_{4} = \frac{11}{5}, 
&&c_{22} = \int_{P}x_{2}^{2}dx_{1}dx_{2}dx_{3}dx_{4} = \frac{139}{18}, \\
&c_{23} = \int_{P}x_{2}x_{3}dx_{1}dx_{2}dx_{3}dx_{4} = -\frac{73}{15}, 
&&c_{24} = \int_{P}x_{2}x_{4}dx_{1}dx_{2}dx_{3}dx_{4} = -\frac{73}{30}, \\
&c_{33} = \int_{P}x_{3}^{2}dx_{1}dx_{2}dx_{3}dx_{4} = \frac{79}{5}, 
&&c_{34} = \int_{P}x_{3}x_{4}dx_{1}dx_{2}dx_{3}dx_{4} = \frac{79}{10}, \\
&c_{44} = \int_{P}x_{4}^{2}dx_{1}dx_{2}dx_{3}dx_{4} = \frac{251}{18}, 
\end{align*}
\begin{align*}
&\theta_P(x_{1}, x_{2}, x_{3}, x_{4}) = \frac{6334289640{{x}_{1}}-2670575400{{x}_{2}}+2239779600{{x}_{3}}}{6291985913} - \frac{3200486167}{6291985913},\\
&M_{X} = \max_{P}\theta_P = \frac{12523717673}{6291985913} > 1. 
\end{align*}

\subsubsection{{$L_{5}$}}
\begin{align*}
&b_{0} = \mathrm{vol}(P) = \frac{52}{3}, 
&&b_{1} = \int_{P}x_{1}dx_{1}dx_{2}dx_{3}dx_{4} = 0, \\
&b_{2} = \int_{P}x_{2}dx_{1}dx_{2}dx_{3}dx_{4} = -\frac{16}{3}, 
&&b_{3} = \int_{P}x_{3}dx_{1}dx_{2}dx_{3}dx_{4} = \frac{8}{3}, \\
&b_{4} = \int_{P}x_{4}dx_{1}dx_{2}dx_{3}dx_{4} = \frac{8}{3}, 
&&c_{11} = \int_{P}x_{1}^{2}dx_{1}dx_{2}dx_{3}dx_{4} = \frac{52}{9}, \\
&c_{12} = \int_{P}x_{1}x_{2}dx_{1}dx_{2}dx_{3}dx_{4} = 0, 
&&c_{13} = \int_{P}x_{1}x_{3}dx_{1}dx_{2}dx_{3}dx_{4} = 0, \\
&c_{14} = \int_{P}x_{1}x_{4}dx_{1}dx_{2}dx_{3}dx_{4} = 0, 
&&c_{22} = \int_{P}x_{2}^{2}dx_{1}dx_{2}dx_{3}dx_{4} = \frac{92}{15}, \\
&c_{23} = \int_{P}x_{2}x_{3}dx_{1}dx_{2}dx_{3}dx_{4} = -\frac{46}{15}, 
&&c_{24} = \int_{P}x_{2}x_{4}dx_{1}dx_{2}dx_{3}dx_{4} = -\frac{46}{15}, \\
&c_{33} = \int_{P}x_{3}^{2}dx_{1}dx_{2}dx_{3}dx_{4} = \frac{48}{5}, 
&&c_{34} = \int_{P}x_{3}x_{4}dx_{1}dx_{2}dx_{3}dx_{4} = \frac{23}{15}, \\
&c_{44} = \int_{P}x_{4}^{2}dx_{1}dx_{2}dx_{3}dx_{4} = \frac{48}{5}, 
\end{align*}
\begin{align*}
&\theta_P(x_{1}, x_{2}, x_{3}, x_{4}) = \frac{-260{{x}_{2}}-80}{219}, \\
&M_{X} = \max_{P}\theta_P = \frac{60}{73} < 1. 
\end{align*}

\subsubsection{{$L_{6}$}}
\begin{align*}
&b_{0} = \mathrm{vol}(P) = -\frac{99}{80}, 
&&b_{1} = \int_{P}x_{1}dx_{1}dx_{2}dx_{3}dx_{4} = 0, \\
&b_{2} = \int_{P}x_{2}dx_{1}dx_{2}dx_{3}dx_{4} = -4, 
&&b_{3} = \int_{P}x_{3}dx_{1}dx_{2}dx_{3}dx_{4} = \frac{16}{3}, \\
&b_{4} = \int_{P}x_{4}dx_{1}dx_{2}dx_{3}dx_{4} = \frac{8}{3}, 
&&c_{11} = \int_{P}x_{1}^{2}dx_{1}dx_{2}dx_{3}dx_{4} = \frac{50}{9}, \\
&c_{12} = \int_{P}x_{1}x_{2}dx_{1}dx_{2}dx_{3}dx_{4} = 0, 
&&c_{13} = \int_{P}x_{1}x_{3}dx_{1}dx_{2}dx_{3}dx_{4} = 0, \\
&c_{14} = \int_{P}x_{1}x_{4}dx_{1}dx_{2}dx_{3}dx_{4} = 0, 
&&c_{22} = \int_{P}x_{2}^{2}dx_{1}dx_{2}dx_{3}dx_{4} = \frac{86}{15}, \\
&c_{23} = \int_{P}x_{2}x_{3}dx_{1}dx_{2}dx_{3}dx_{4} = -\frac{64}{15}, 
&&c_{24} = \int_{P}x_{2}x_{4}dx_{1}dx_{2}dx_{3}dx_{4} = -\frac{32}{15}, \\
&c_{33} = \int_{P}x_{3}^{2}dx_{1}dx_{2}dx_{3}dx_{4} = \frac{51}{5}, 
&&c_{34} = \int_{P}x_{3}x_{4}dx_{1}dx_{2}dx_{3}dx_{4} = \frac{51}{10}, \\
&c_{44} = \int_{P}x_{4}^{2}dx_{1}dx_{2}dx_{3}dx_{4} = \frac{161}{15}, 
\end{align*}
\begin{align*}
&\theta_P(x_{1}, x_{2}, x_{3}, x_{4}) = \frac{-10150{{x}_{2}}+7600{{x}_{3}}-4868}{17787}, \\
&M_{X} = \max_{P}\theta_P = \frac{38}{33} > 1. 
\end{align*}

\subsubsection{{$L_{7}$}}
\begin{align*}
&b_{0} = \mathrm{vol}(P) = 16, 
&&b_{1} = \int_{P}x_{1}dx_{1}dx_{2}dx_{3}dx_{4} = \frac{8}{3}, \\
&b_{2} = \int_{P}x_{2}dx_{1}dx_{2}dx_{3}dx_{4} = \frac{4}{3}, 
&&b_{3} = \int_{P}x_{3}dx_{1}dx_{2}dx_{3}dx_{4} = \frac{8}{3}, \\
&b_{4} = \int_{P}x_{4}dx_{1}dx_{2}dx_{3}dx_{4} = \frac{4}{3}, 
&&c_{11} = \int_{P}x_{1}^{2}dx_{1}dx_{2}dx_{3}dx_{4} = \frac{16}{3}, \\
&c_{12} = \int_{P}x_{1}x_{2}dx_{1}dx_{2}dx_{3}dx_{4} = \frac{8}{3}, 
&&c_{13} = \int_{P}x_{1}x_{3}dx_{1}dx_{2}dx_{3}dx_{4} = \frac{4}{9}, \\
&c_{14} = \int_{P}x_{1}x_{4}dx_{1}dx_{2}dx_{3}dx_{4} = \frac{2}{9}, 
&&c_{22} = \int_{P}x_{2}^{2}dx_{1}dx_{2}dx_{3}dx_{4} = 8, \\
&c_{23} = \int_{P}x_{2}x_{3}dx_{1}dx_{2}dx_{3}dx_{4} = \frac{2}{9}, 
&&c_{24} = \int_{P}x_{2}x_{4}dx_{1}dx_{2}dx_{3}dx_{4} = \frac{1}{9}, \\
&c_{33} = \int_{P}x_{3}^{2}dx_{1}dx_{2}dx_{3}dx_{4} = \frac{16}{3}, 
&&c_{34} = \int_{P}x_{3}x_{4}dx_{1}dx_{2}dx_{3}dx_{4} = \frac{8}{3}, \\
&c_{44} = \int_{P}x_{4}^{2}dx_{1}dx_{2}dx_{3}dx_{4} = 8, 
\end{align*}
\begin{align*}
&\theta_P(x_{1}, x_{2}, x_{3}, x_{4}) = \frac{6{{x}_{1}}+6{{x}_{3}}-2}{11}, \\
&M_{X} = \max_{P}\theta_P = \frac{10}{11} < 1. 
\end{align*}

\subsubsection{{$L_{8}$}}
\begin{align*}
&b_{0} = \mathrm{vol}(P) = 16, 
&&b_{1} = \int_{P}x_{1}dx_{1}dx_{2}dx_{3}dx_{4} = 0, \\
&b_{2} = \int_{P}x_{2}dx_{1}dx_{2}dx_{3}dx_{4} = 0, 
&&b_{3} = \int_{P}x_{3}dx_{1}dx_{2}dx_{3}dx_{4} = 0, \\
&b_{4} = \int_{P}x_{4}dx_{1}dx_{2}dx_{3}dx_{4} = 0, 
&&c_{11} = \int_{P}x_{1}^{2}dx_{1}dx_{2}dx_{3}dx_{4} = \frac{16}{3}, \\
&c_{12} = \int_{P}x_{1}x_{2}dx_{1}dx_{2}dx_{3}dx_{4} = 0, 
&&c_{13} = \int_{P}x_{1}x_{3}dx_{1}dx_{2}dx_{3}dx_{4} = 0, \\
&c_{14} = \int_{P}x_{1}x_{4}dx_{1}dx_{2}dx_{3}dx_{4} = 0, 
&&c_{22} = \int_{P}x_{2}^{2}dx_{1}dx_{2}dx_{3}dx_{4} = \frac{16}{3}, \\
&c_{23} = \int_{P}x_{2}x_{3}dx_{1}dx_{2}dx_{3}dx_{4} = 0, 
&&c_{24} = \int_{P}x_{2}x_{4}dx_{1}dx_{2}dx_{3}dx_{4} = 0, \\
&c_{33} = \int_{P}x_{3}^{2}dx_{1}dx_{2}dx_{3}dx_{4} = \frac{16}{3}, 
&&c_{34} = \int_{P}x_{3}x_{4}dx_{1}dx_{2}dx_{3}dx_{4} = 0, \\
&c_{44} = \int_{P}x_{4}^{2}dx_{1}dx_{2}dx_{3}dx_{4} = \frac{16}{3}, 
\end{align*}
\begin{align*}
&\theta_P(x_{1}, x_{2}, x_{3}, x_{4}) = 0, \\
&M_{X} = \max_{P}\theta_P = 0 < 1. 
\end{align*}

\subsubsection{{$L_{9}$}}
\begin{align*}
&b_{0} = \mathrm{vol}(P) = 16, 
&&b_{1} = \int_{P}x_{1}dx_{1}dx_{2}dx_{3}dx_{4} = 0, \\
&b_{2} = \int_{P}x_{2}dx_{1}dx_{2}dx_{3}dx_{4} = 0, 
&&b_{3} = \int_{P}x_{3}dx_{1}dx_{2}dx_{3}dx_{4} = \frac{8}{3}, \\
&b_{4} = \int_{P}x_{4}dx_{1}dx_{2}dx_{3}dx_{4} = \frac{4}{3}, 
&&c_{11} = \int_{P}x_{1}^{2}dx_{1}dx_{2}dx_{3}dx_{4} = \frac{16}{3}, \\
&c_{12} = \int_{P}x_{1}x_{2}dx_{1}dx_{2}dx_{3}dx_{4} = 0, 
&&c_{13} = \int_{P}x_{1}x_{3}dx_{1}dx_{2}dx_{3}dx_{4} = 0, \\
&c_{14} = \int_{P}x_{1}x_{4}dx_{1}dx_{2}dx_{3}dx_{4} = 0, 
&&c_{22} = \int_{P}x_{2}^{2}dx_{1}dx_{2}dx_{3}dx_{4} = \frac{16}{3}, \\
&c_{23} = \int_{P}x_{2}x_{3}dx_{1}dx_{2}dx_{3}dx_{4} = 0, 
&&c_{24} = \int_{P}x_{2}x_{4}dx_{1}dx_{2}dx_{3}dx_{4} = 0, \\
&c_{33} = \int_{P}x_{3}^{2}dx_{1}dx_{2}dx_{3}dx_{4} = \frac{16}{3}, 
&&c_{34} = \int_{P}x_{3}x_{4}dx_{1}dx_{2}dx_{3}dx_{4} = \frac{8}{3}, \\
&c_{44} = \int_{P}x_{4}^{2}dx_{1}dx_{2}dx_{3}dx_{4} = 8, 
\end{align*}
\begin{align*}
&\theta_P(x_{1}, x_{2}, x_{3}, x_{4}) = \frac{6{{x}_{3}}-1}{11}, \\
&M_{X} = \max_{P}\theta_P = \frac{5}{11} < 1. 
\end{align*}

\subsubsection{{$L_{10}$}}
\begin{align*}
&b_{0} = \mathrm{vol}(P) = 16, 
&&b_{1} = \int_{P}x_{1}dx_{1}dx_{2}dx_{3}dx_{4} = \frac{71}{15}, \\
&b_{2} = \int_{P}x_{2}dx_{1}dx_{2}dx_{3}dx_{4} = \frac{71}{30}, 
&&b_{3} = \int_{P}x_{3}dx_{1}dx_{2}dx_{3}dx_{4} = \frac{71}{60}, \\
&b_{4} = \int_{P}x_{4}dx_{1}dx_{2}dx_{3}dx_{4} = \frac{71}{60}, 
&&c_{11} = \int_{P}x_{1}^{2}dx_{1}dx_{2}dx_{3}dx_{4} = \frac{256}{45}, \\
&c_{12} = \int_{P}x_{1}x_{2}dx_{1}dx_{2}dx_{3}dx_{4} = \frac{128}{45}, 
&&c_{13} = \int_{P}x_{1}x_{3}dx_{1}dx_{2}dx_{3}dx_{4} = \frac{64}{45}, \\
&c_{14} = \int_{P}x_{1}x_{4}dx_{1}dx_{2}dx_{3}dx_{4} = \frac{64}{45}, 
&&c_{22} = \int_{P}x_{2}^{2}dx_{1}dx_{2}dx_{3}dx_{4} = \frac{122}{15}, \\
&c_{23} = \int_{P}x_{2}x_{3}dx_{1}dx_{2}dx_{3}dx_{4} = -\frac{119}{45}, 
&&c_{24} = \int_{P}x_{2}x_{4}dx_{1}dx_{2}dx_{3}dx_{4} = \frac{61}{15}, \\
&c_{33} = \int_{P}x_{3}^{2}dx_{1}dx_{2}dx_{3}dx_{4} = \frac{53}{6}, 
&&c_{34} = \int_{P}x_{3}x_{4}dx_{1}dx_{2}dx_{3}dx_{4} = -\frac{119}{90}, \\
&c_{44} = \int_{P}x_{4}^{2}dx_{1}dx_{2}dx_{3}dx_{4} = \frac{53}{6}, 
\end{align*}
\begin{align*}
&\theta_P(x_{1}, x_{2}, x_{3}, x_{4}) = \frac{17040{{x}_{1}}-5041}{15439}, \\
&M_{X} = \max_{P}\theta_P = \frac{11999}{15439} < 1. 
\end{align*}

\subsubsection{{$L_{11}$}}
\begin{align*}
&b_{0} = \mathrm{vol}(P) = \frac{44}{3}, 
&&b_{1} = \int_{P}x_{1}dx_{1}dx_{2}dx_{3}dx_{4} = 0, \\
&b_{2} = \int_{P}x_{2}dx_{1}dx_{2}dx_{3}dx_{4} = 0, 
&&b_{3} = \int_{P}x_{3}dx_{1}dx_{2}dx_{3}dx_{4} = 0, \\
&b_{4} = \int_{P}x_{4}dx_{1}dx_{2}dx_{3}dx_{4} = 0, 
&&c_{11} = \int_{P}x_{1}^{2}dx_{1}dx_{2}dx_{3}dx_{4} = \frac{44}{9}, \\
&c_{12} = \int_{P}x_{1}x_{2}dx_{1}dx_{2}dx_{3}dx_{4} = 0, 
&&c_{13} = \int_{P}x_{1}x_{3}dx_{1}dx_{2}dx_{3}dx_{4} = 0, \\
&c_{14} = \int_{P}x_{1}x_{4}dx_{1}dx_{2}dx_{3}dx_{4} = 0, 
&&c_{22} = \int_{P}x_{2}^{2}dx_{1}dx_{2}dx_{3}dx_{4} = \frac{68}{15}, \\
&c_{23} = \int_{P}x_{2}x_{3}dx_{1}dx_{2}dx_{3}dx_{4} = -\frac{34}{15}, 
&&c_{24} = \int_{P}x_{2}x_{4}dx_{1}dx_{2}dx_{3}dx_{4} = \frac{34}{15}, \\
&c_{33} = \int_{P}x_{3}^{2}dx_{1}dx_{2}dx_{3}dx_{4} = \frac{32}{5}, 
&&c_{34} = \int_{P}x_{3}x_{4}dx_{1}dx_{2}dx_{3}dx_{4} = -\frac{17}{15}, \\
&c_{44} = \int_{P}x_{4}^{2}dx_{1}dx_{2}dx_{3}dx_{4} = \frac{32}{5}, 
\end{align*}
\begin{align*}
&\theta_P(x_{1}, x_{2}, x_{3}, x_{4}) = 0, \\
&M_{X} = \max_{P}\theta_P = 0 < 1. 
\end{align*}

\subsubsection{{$L_{12}$}}
\begin{align*}
&b_{0} = \mathrm{vol}(P) = \frac{44}{3}, 
&&b_{1} = \int_{P}x_{1}dx_{1}dx_{2}dx_{3}dx_{4} = -\frac{17}{15}, \\
&b_{2} = \int_{P}x_{2}dx_{1}dx_{2}dx_{3}dx_{4} = -\frac{17}{30}, 
&&b_{3} = \int_{P}x_{3}dx_{1}dx_{2}dx_{3}dx_{4} = \frac{16}{5}, \\
&b_{4} = \int_{P}x_{4}dx_{1}dx_{2}dx_{3}dx_{4} = \frac{8}{5}, 
&&c_{11} = \int_{P}x_{1}^{2}dx_{1}dx_{2}dx_{3}dx_{4} = \frac{68}{15}, \\
&c_{12} = \int_{P}x_{1}x_{2}dx_{1}dx_{2}dx_{3}dx_{4} = \frac{34}{15}, 
&&c_{13} = \int_{P}x_{1}x_{3}dx_{1}dx_{2}dx_{3}dx_{4} = -\frac{136}{45}, \\
&c_{14} = \int_{P}x_{1}x_{4}dx_{1}dx_{2}dx_{3}dx_{4} = -\frac{68}{45}, 
&&c_{22} = \int_{P}x_{2}^{2}dx_{1}dx_{2}dx_{3}dx_{4} = \frac{271}{45}, \\
&c_{23} = \int_{P}x_{2}x_{3}dx_{1}dx_{2}dx_{3}dx_{4} = -\frac{68}{45}, 
&&c_{24} = \int_{P}x_{2}x_{4}dx_{1}dx_{2}dx_{3}dx_{4} = -\frac{34}{45}, \\
&c_{33} = \int_{P}x_{3}^{2}dx_{1}dx_{2}dx_{3}dx_{4} = \frac{113}{15}, 
&&c_{34} = \int_{P}x_{3}x_{4}dx_{1}dx_{2}dx_{3}dx_{4} = \frac{113}{30}, \\
&c_{44} = \int_{P}x_{4}^{2}dx_{1}dx_{2}dx_{3}dx_{4} = \frac{127}{15}, 
\end{align*}
\begin{align*}
&\theta_P(x_{1}, x_{2}, x_{3}, x_{4}) = \frac{5940{{x}_{1}}+58080{{x}_{3}}-12213}{118907}, \\
&M_{X} = \max_{P}\theta_P = \frac{98007}{118907} < 1. 
\end{align*}

\subsubsection{{$L_{13}$}}
\begin{align*}
&b_{0} = \mathrm{vol}(P) = \frac{44}{3}, 
&&b_{1} = \int_{P}x_{1}dx_{1}dx_{2}dx_{3}dx_{4} = \frac{34}{15}, \\
&b_{2} = \int_{P}x_{2}dx_{1}dx_{2}dx_{3}dx_{4} = \frac{17}{15}, 
&&b_{3} = \int_{P}x_{3}dx_{1}dx_{2}dx_{3}dx_{4} = \frac{17}{15}, \\
&b_{4} = \int_{P}x_{4}dx_{1}dx_{2}dx_{3}dx_{4} = -\frac{17}{15}, 
&&c_{11} = \int_{P}x_{1}^{2}dx_{1}dx_{2}dx_{3}dx_{4} = \frac{68}{15}, \\
&c_{12} = \int_{P}x_{1}x_{2}dx_{1}dx_{2}dx_{3}dx_{4} = \frac{34}{15}, 
&&c_{13} = \int_{P}x_{1}x_{3}dx_{1}dx_{2}dx_{3}dx_{4} = \frac{34}{15}, \\
&c_{14} = \int_{P}x_{1}x_{4}dx_{1}dx_{2}dx_{3}dx_{4} = -\frac{34}{15}, 
&&c_{22} = \int_{P}x_{2}^{2}dx_{1}dx_{2}dx_{3}dx_{4} = \frac{322}{45}, \\
&c_{23} = \int_{P}x_{2}x_{3}dx_{1}dx_{2}dx_{3}dx_{4} = \frac{17}{15}, 
&&c_{24} = \int_{P}x_{2}x_{4}dx_{1}dx_{2}dx_{3}dx_{4} = -\frac{17}{15}, \\
&c_{33} = \int_{P}x_{3}^{2}dx_{1}dx_{2}dx_{3}dx_{4} = \frac{322}{45}, 
&&c_{34} = \int_{P}x_{3}x_{4}dx_{1}dx_{2}dx_{3}dx_{4} = -\frac{17}{15}, \\
&c_{44} = \int_{P}x_{4}^{2}dx_{1}dx_{2}dx_{3}dx_{4} = \frac{254}{45}, 
\end{align*}
\begin{align*}
&\theta_P(x_{1}, x_{2}, x_{3}, x_{4}) = \frac{110{{x}_{1}}-17}{203}, \\
&M_{X} = \max_{P}\theta_P = \frac{93}{203} < 1. 
\end{align*}

\subsubsection{$I_{1}$}
\begin{align*}
&b_{0} = \mathrm{vol}(P) = \frac{62}{3}, 
&&b_{1} = \int_{P}x_{1}dx_{1}dx_{2}dx_{3}dx_{4} = \frac{173}{20}, \\
&b_{2} = \int_{P}x_{2}dx_{1}dx_{2}dx_{3}dx_{4} = \frac{521}{60}, 
&&b_{3} = \int_{P}x_{3}dx_{1}dx_{2}dx_{3}dx_{4} = \frac{521}{30}, \\
&b_{4} = \int_{P}x_{4}dx_{1}dx_{2}dx_{3}dx_{4} = -\frac{1}{30}, 
&&c_{11} = \int_{P}x_{1}^{2}dx_{1}dx_{2}dx_{3}dx_{4} = \frac{1393}{180}, \\
&c_{12} = \int_{P}x_{1}x_{2}dx_{1}dx_{2}dx_{3}dx_{4} = \frac{443}{72}, 
&&c_{13} = \int_{P}x_{1}x_{3}dx_{1}dx_{2}dx_{3}dx_{4} = \frac{443}{36}, \\
&c_{14} = \int_{P}x_{1}x_{4}dx_{1}dx_{2}dx_{3}dx_{4} = \frac{571}{360},	 
&&c_{22} = \int_{P}x_{2}^{2}dx_{1}dx_{2}dx_{3}dx_{4} = \frac{2197}{90}, \\
&c_{23} = \int_{P}x_{2}x_{3}dx_{1}dx_{2}dx_{3}dx_{4} = \frac{88}{5}, 
&&c_{24} = \int_{P}x_{2}x_{4}dx_{1}dx_{2}dx_{3}dx_{4} = -\frac{953}{360}, \\
&c_{33} = \int_{P}x_{3}^{2}dx_{1}dx_{2}dx_{3}dx_{4} = \frac{176}{5}, 
&&c_{34} = \int_{P}x_{3}x_{4}dx_{1}dx_{2}dx_{3}dx_{4} = -\frac{953}{180}, \\
&c_{44} = \int_{P}x_{4}^{2}dx_{1}dx_{2}dx_{3}dx_{4} = \frac{563}{45}, 
\end{align*}
\begin{align*}
&\theta_P(x_{1}, x_{2}, x_{3}, x_{4}) = \frac{455864610{{x}_{1}}+140393265{{x}_{3}}-308777028}{298795537}, \\
&M_{X} = \max_{P}\theta_P = \frac{708660642}{298795537} > 1. 
\end{align*}

\subsubsection{$I_{2}$}
\begin{align*}
&b_{0} = \mathrm{vol}(P) = \frac{463}{24}, 
&&b_{1} = \int_{P}x_{1}dx_{1}dx_{2}dx_{3}dx_{4} = 8, \\
&b_{2} = \int_{P}x_{2}dx_{1}dx_{2}dx_{3}dx_{4} = 4, 
&&b_{3} = \int_{P}x_{3}dx_{1}dx_{2}dx_{3}dx_{4} = \frac{57}{5}, \\
&b_{4} = \int_{P}x_{4}dx_{1}dx_{2}dx_{3}dx_{4} = \frac{23}{10}, 
&&c_{11} = \int_{P}x_{1}^{2}dx_{1}dx_{2}dx_{3}dx_{4} = \frac{141}{20}, \\
&c_{12} = \int_{P}x_{1}x_{2}dx_{1}dx_{2}dx_{3}dx_{4} = \frac{141}{40}, 
&&c_{13} = \int_{P}x_{1}x_{3}dx_{1}dx_{2}dx_{3}dx_{4} = \frac{1627}{180}, \\
&c_{14} = \int_{P}x_{1}x_{4}dx_{1}dx_{2}dx_{3}dx_{4} = \frac{911}{360}, 
&&c_{22} = \int_{P}x_{2}^{2}dx_{1}dx_{2}dx_{3}dx_{4} = \frac{4121}{360}, \\
&c_{23} = \int_{P}x_{2}x_{3}dx_{1}dx_{2}dx_{3}dx_{4} = \frac{1627}{360}, 
&&c_{24} = \int_{P}x_{2}x_{4}dx_{1}dx_{2}dx_{3}dx_{4} = \frac{911}{720}, \\
&c_{33} = \int_{P}x_{3}^{2}dx_{1}dx_{2}dx_{3}dx_{4} = \frac{4363}{180}, 
&&c_{34} = \int_{P}x_{3}x_{4}dx_{1}dx_{2}dx_{3}dx_{4} = -\frac{1109}{360}, \\
&c_{44} = \int_{P}x_{4}^{2}dx_{1}dx_{2}dx_{3}dx_{4} = \frac{1607}{120}, 
\end{align*}
\begin{align*}
&\theta_P(x_{1}, x_{2}, x_{3}, x_{4}) = \frac{3407814270{{x}_{1}}+302232510{{x}_{3}}-1591773552}{1752843533}, \\
&M_{X} = \max_{P}\theta_P = \frac{3024970758}{1752843533} > 1. 
\end{align*}

\subsubsection{$I_{3}$}
\begin{align*}
&b_{0} = \mathrm{vol}(P) = \frac{221}{12}, 
&&b_{1} = \int_{P}x_{1}dx_{1}dx_{2}dx_{3}dx_{4} = \frac{127}{20}, \\
&b_{2} = \int_{P}x_{2}dx_{1}dx_{2}dx_{3}dx_{4} = \frac{101}{12}, 
&&b_{3} = \int_{P}x_{3}dx_{1}dx_{2}dx_{3}dx_{4} = \frac{61}{6}, \\
&b_{4} = \int_{P}x_{4}dx_{1}dx_{2}dx_{3}dx_{4} = \frac{23}{10}, 
&&c_{11} = \int_{P}x_{1}^{2}dx_{1}dx_{2}dx_{3}dx_{4} = \frac{286}{45}, \\
&c_{12} = \int_{P}x_{1}x_{2}dx_{1}dx_{2}dx_{3}dx_{4} = \frac{679}{120}, 
&&c_{13} = \int_{P}x_{1}x_{3}dx_{1}dx_{2}dx_{3}dx_{4} = \frac{2771}{260}, \\
&c_{14} = \int_{P}x_{1}x_{4}dx_{1}dx_{2}dx_{3}dx_{4} = \frac{259}{120}, 
&&c_{22} = \int_{P}x_{2}^{2}dx_{1}dx_{2}dx_{3}dx_{4} = \frac{1583}{120}, \\
&c_{23} = \int_{P}x_{2}x_{3}dx_{1}dx_{2}dx_{3}dx_{4} = \frac{1193}{120}, 
&&c_{24} = \int_{P}x_{2}x_{4}dx_{1}dx_{2}dx_{3}dx_{4} = \frac{1069}{240}, \\
&c_{33} = \int_{P}x_{3}^{2}dx_{1}dx_{2}dx_{3}dx_{4} = \frac{1663}{72}, 
&&c_{34} = \int_{P}x_{3}x_{4}dx_{1}dx_{2}dx_{3}dx_{4} = -\frac{131}{48}, \\
&c_{44} = \int_{P}x_{4}^{2}dx_{1}dx_{2}dx_{3}dx_{4} = \frac{139}{10}, 
\end{align*}
\begin{align*}
&\theta_P(x_{1}, x_{2}, x_{3}, x_{4}) = \frac{1119451198840{{x}_{1}}+553671270165{{x}_{2}}+212169307415{{x}_{3}}}{1114834051022}-\frac{756143145443}{1114834051022}, \\
&M_{X} = \max_{P}\theta_P = \frac{1265748565401}{557417025511} > 1. 
\end{align*}

\subsubsection{$I_{4}$}
\begin{align*}
&b_{0} = \mathrm{vol}(P) = \frac{433}{24}, 
&&b_{1} = \int_{P}x_{1}dx_{1}dx_{2}dx_{3}dx_{4} = \frac{59}{10}, \\
&b_{2} = \int_{P}x_{2}dx_{1}dx_{2}dx_{3}dx_{4} = \frac{121}{30}, 
&&b_{3} = \int_{P}x_{3}dx_{1}dx_{2}dx_{3}dx_{4} = \frac{419}{30}, \\
&b_{4} = \int_{P}x_{4}dx_{1}dx_{2}dx_{3}dx_{4} = -\frac{13}{12}, 
&&c_{11} = \int_{P}x_{1}^{2}dx_{1}dx_{2}dx_{3}dx_{4} = \frac{557}{90}, \\
&c_{12} = \int_{P}x_{1}x_{2}dx_{1}dx_{2}dx_{3}dx_{4} = \frac{349}{180}, 
&&c_{13} = \int_{P}x_{1}x_{3}dx_{1}dx_{2}dx_{3}dx_{4} = \frac{151}{15}, \\
&c_{14} = \int_{P}x_{1}x_{4}dx_{1}dx_{2}dx_{3}dx_{4} = \frac{52}{45}, 
&&c_{22} = \int_{P}x_{2}^{2}dx_{1}dx_{2}dx_{3}dx_{4} = \frac{5023}{360}, \\
&c_{23} = \int_{P}x_{2}x_{3}dx_{1}dx_{2}dx_{3}dx_{4} = \frac{3533}{360}, 
&&c_{24} = \int_{P}x_{2}x_{4}dx_{1}dx_{2}dx_{3}dx_{4} = -\frac{2137}{720}, \\
&c_{33} = \int_{P}x_{3}^{2}dx_{1}dx_{2}dx_{3}dx_{4} = \frac{1069}{36}, 
&&c_{34} = \int_{P}x_{3}x_{4}dx_{1}dx_{2}dx_{3}dx_{4} = -\frac{1721}{360}, \\
&c_{44} = \int_{P}x_{4}^{2}dx_{1}dx_{2}dx_{3}dx_{4} = \frac{3607}{360}, 
\end{align*}
\begin{align*}
&\theta_P(x_{1}, x_{2}, x_{3}, x_{4}) = \frac{29597715{{x}_{1}}+23135190{{x}_{3}}-27588804}{42931741}, \\
&M_{X} = \max_{P}\theta_P = \frac{94549671}{42931741} > 1. 
\end{align*}

\subsubsection{$I_{5}$}
\begin{align*}
&b_{0} = \mathrm{vol}(P) = \frac{415}{24}, 
&&b_{1} = \int_{P}x_{1}dx_{1}dx_{2}dx_{3}dx_{4} = \frac{63}{10}, \\
&b_{2} = \int_{P}x_{2}dx_{1}dx_{2}dx_{3}dx_{4} = \frac{67}{20}, 
&&b_{3} = \int_{P}x_{3}dx_{1}dx_{2}dx_{3}dx_{4} = \frac{403}{69}, \\
&b_{4} = \int_{P}x_{4}dx_{1}dx_{2}dx_{3}dx_{4} = \frac{67}{10}, 
&&c_{11} = \int_{P}x_{1}^{2}dx_{1}dx_{2}dx_{3}dx_{4} = \frac{427}{72}, \\
&c_{12} = \int_{P}x_{1}x_{2}dx_{1}dx_{2}dx_{3}dx_{4} = \frac{551}{240}, 
&&c_{13} = \int_{P}x_{1}x_{3}dx_{1}dx_{2}dx_{3}dx_{4} = \frac{2263}{360}, \\
&c_{14} = \int_{P}x_{1}x_{4}dx_{1}dx_{2}dx_{3}dx_{4} = \frac{551}{120}, 
&&c_{22} = \int_{P}x_{2}^{2}dx_{1}dx_{2}dx_{3}dx_{4} = \frac{1577}{120}, \\
&c_{23} = \int_{P}x_{2}x_{3}dx_{1}dx_{2}dx_{3}dx_{4} = -\frac{43}{90}, 
&&c_{24} = \int_{P}x_{2}x_{4}dx_{1}dx_{2}dx_{3}dx_{4} = \frac{463}{60}, \\
&c_{33} = \int_{P}x_{3}^{2}dx_{1}dx_{2}dx_{3}dx_{4} = \frac{2881}{180}, 
&&c_{34} = \int_{P}x_{3}x_{4}dx_{1}dx_{2}dx_{3}dx_{4} = -\frac{43}{45}, \\
&c_{44} = \int_{P}x_{4}^{2}dx_{1}dx_{2}dx_{3}dx_{4} = \frac{463}{30}, 
\end{align*}
\begin{align*}
&\theta_P(x_{1}, x_{2}, x_{3}, x_{4}) = \frac{10897419402618300{{x}_{1}}+101247812541450{{x}_{3}}}{6984760752795427}  \\
&\quad \quad \quad \quad \quad \quad \quad \quad + \frac{1845227650200450{{x}_{4}}-4719505353973848}{6984760752795427}, \\
&M_{X} = \max_{P}\theta_P = \frac{11713596999245802}{6984760752795427} > 1. 
\end{align*}

\newpage
\subsubsection{$I_{6}$}
\begin{align*}
&b_{0} = \mathrm{vol}(P) = \frac{137}{8}, 
&&b_{1} = \int_{P}x_{1}dx_{1}dx_{2}dx_{3}dx_{4} = \frac{53}{15},\quad \\
&b_{2} = \int_{P}x_{2}dx_{1}dx_{2}dx_{3}dx_{4} = \frac{43}{10}, 
&&b_{3} = \int_{P}x_{3}dx_{1}dx_{2}dx_{3}dx_{4} = \frac{43}{5}, \\
&b_{4} = \int_{P}x_{4}dx_{1}dx_{2}dx_{3}dx_{4} = -\frac{38}{15}, 
&&c_{11} = \int_{P}x_{1}^{2}dx_{1}dx_{2}dx_{3}dx_{4} = \frac{253}{45}, \\
&c_{12} = \int_{P}x_{1}x_{2}dx_{1}dx_{2}dx_{3}dx_{4} = \frac{239}{90}, 
&&c_{13} = \int_{P}x_{1}x_{3}dx_{1}dx_{2}dx_{3}dx_{4} = \frac{239}{45}, \\
&c_{14} = \int_{P}x_{1}x_{4}dx_{1}dx_{2}dx_{3}dx_{4} = \frac{7}{45}, 
&&c_{22} = \int_{P}x_{2}^{2}dx_{1}dx_{2}dx_{3}dx_{4} = \frac{4933}{360}, \\
&c_{23} = \int_{P}x_{2}x_{3}dx_{1}dx_{2}dx_{3}dx_{4} = \frac{923}{120}, 
&&c_{24} = \int_{P}x_{2}x_{4}dx_{1}dx_{2}dx_{3}dx_{4} = -\frac{1813}{720}, \\
&c_{33} = \int_{P}x_{3}^{2}dx_{1}dx_{2}dx_{3}dx_{4} = \frac{923}{60}, 
&&c_{34} = \int_{P}x_{3}x_{4}dx_{1}dx_{2}dx_{3}dx_{4} = -\frac{1813}{360}, \\
&c_{44} = \int_{P}x_{4}^{2}dx_{1}dx_{2}dx_{3}dx_{4} = \frac{2693}{360}, 
\end{align*}
\begin{align*}
&\theta_P(x_{1}, x_{2}, x_{3}, x_{4}) = \frac{48154815{{x}_{1}}+164153400{{x}_{3}}-92371752}{230998253}, \\
&M_{X} = \max_{P}\theta_P = \frac{448243263}{230998253} > 1. 
\end{align*}

\subsubsection{$I_{7}$}
\begin{align*}
&b_{0} = \mathrm{vol}(P) = -\frac{99}{80}, 
&&b_{1} = \int_{P}x_{1}dx_{1}dx_{2}dx_{3}dx_{4} = \frac{21}{4}, \\
&b_{2} = \int_{P}x_{2}dx_{1}dx_{2}dx_{3}dx_{4} = 0, 
&&b_{3} = \int_{P}x_{3}dx_{1}dx_{2}dx_{3}dx_{4} = 8, \\
&b_{4} = \int_{P}x_{4}dx_{1}dx_{2}dx_{3}dx_{4} = \frac{5}{4}, 
&&c_{11} = \int_{P}x_{1}^{2}dx_{1}dx_{2}dx_{3}dx_{4} = \frac{11}{2}, \\
&c_{12} = \int_{P}x_{1}x_{2}dx_{1}dx_{2}dx_{3}dx_{4} = 0, 
&&c_{13} = \int_{P}x_{1}x_{3}dx_{1}dx_{2}dx_{3}dx_{4} = \frac{34}{5}, \\
&c_{14} = \int_{P}x_{1}x_{4}dx_{1}dx_{2}dx_{3}dx_{4} = \frac{21}{10}, 
&&c_{22} = \int_{P}x_{2}^{2}dx_{1}dx_{2}dx_{3}dx_{4} = \frac{50}{9}, \\
&c_{23} = \int_{P}x_{2}x_{3}dx_{1}dx_{2}dx_{3}dx_{4} = 0, 
&&c_{24} = \int_{P}x_{2}x_{4}dx_{1}dx_{2}dx_{3}dx_{4} = 0, \\
&c_{33} = \int_{P}x_{3}^{2}dx_{1}dx_{2}dx_{3}dx_{4} = \frac{281}{15}, 
&&c_{34} = \int_{P}x_{3}x_{4}dx_{1}dx_{2}dx_{3}dx_{4} = -\frac{77}{30}, \\
&c_{44} = \int_{P}x_{4}^{2}dx_{1}dx_{2}dx_{3}dx_{4} = \frac{109}{10}, 
\end{align*}
\begin{align*}
&\theta_P(x_{1}, x_{2}, x_{3}, x_{4}) = \frac{527400{{x}_{1}}+99600{{x}_{3}}-213939}{467581}, \\
&M_{X} = \max_{P}\theta_P = \frac{711861}{467581} > 1. 
\end{align*}

\subsubsection{$I_{8}$}
\begin{align*}
&b_{0} = \mathrm{vol}(P) = 16, 
&&b_{1} = \int_{P}x_{1}dx_{1}dx_{2}dx_{3}dx_{4} = -\frac{53}{20}, \\
&b_{2} = \int_{P}x_{2}dx_{1}dx_{2}dx_{3}dx_{4} = \frac{88}{15}, 
&&b_{3} = \int_{P}x_{3}dx_{1}dx_{2}dx_{3}dx_{4} = \frac{163}{60}, \\
&b_{4} = \int_{P}x_{4}dx_{1}dx_{2}dx_{3}dx_{4} = \frac{1}{4}, 
&&c_{11} = \int_{P}x_{1}^{2}dx_{1}dx_{2}dx_{3}dx_{4} = \frac{443}{90}, \\
&c_{12} = \int_{P}x_{1}x_{2}dx_{1}dx_{2}dx_{3}dx_{4} = -\frac{1547}{360}, 
&&c_{13} = \int_{P}x_{1}x_{3}dx_{1}dx_{2}dx_{3}dx_{4} = \frac{431}{360}, \\
&c_{14} = \int_{P}x_{1}x_{4}dx_{1}dx_{2}dx_{3}dx_{4} = -\frac{103}{360}, 
&&c_{22} = \int_{P}x_{2}^{2}dx_{1}dx_{2}dx_{3}dx_{4} = \frac{3611}{360}, \\
&c_{23} = \int_{P}x_{2}x_{3}dx_{1}dx_{2}dx_{3}dx_{4} = \frac{581}{360}, 
&&c_{24} = \int_{P}x_{2}x_{4}dx_{1}dx_{2}dx_{3}dx_{4} = \frac{1483}{720}, \\
&c_{33} = \int_{P}x_{3}^{2}dx_{1}dx_{2}dx_{3}dx_{4} = \frac{3743}{360}, 
&&c_{34} = \int_{P}x_{3}x_{4}dx_{1}dx_{2}dx_{3}dx_{4} = -\frac{2731}{720}, \\
&c_{44} = \int_{P}x_{4}^{2}dx_{1}dx_{2}dx_{3}dx_{4} = \frac{829}{90}, 
\end{align*}
\begin{align*}
&\theta_P(x_{1}, x_{2}, x_{3}, x_{4}) = \frac{-354416417280{{x}_{1}}+991006584000{{x}_{2}}+428321791680{{x}_{3}}}{1577014856729} - \frac{494794770791}{1577014856729},  \\
&M_{X} = \max_{P}\theta_P = \frac{3126600189529}{1577014856729} > 1. 
\end{align*}

\subsubsection{$I_{9}$}
\begin{align*}
&b_{0} = \mathrm{vol}(P) = \frac{65}{4}, 
&&b_{1} = \int_{P}x_{1}dx_{1}dx_{2}dx_{3}dx_{4} = \frac{71}{60},\quad \\
&b_{2} = \int_{P}x_{2}dx_{1}dx_{2}dx_{3}dx_{4} = \frac{169}{60}, 
&&b_{3} = \int_{P}x_{3}dx_{1}dx_{2}dx_{3}dx_{4} = \frac{409}{60}, \\
&b_{4} = \int_{P}x_{4}dx_{1}dx_{2}dx_{3}dx_{4} = -\frac{169}{60}, 
&&c_{11} = \int_{P}x_{1}^{2}dx_{1}dx_{2}dx_{3}dx_{4} = \frac{919}{180}, \\
&c_{12} = \int_{P}x_{1}x_{2}dx_{1}dx_{2}dx_{3}dx_{4} = -\frac{43}{120}, 
&&c_{13} = \int_{P}x_{1}x_{3}dx_{1}dx_{2}dx_{3}dx_{4} = \frac{79}{18}, \\
&c_{14} = \int_{P}x_{1}x_{4}dx_{1}dx_{2}dx_{3}dx_{4} = \frac{43}{120}, 
&&c_{22} = \int_{P}x_{2}^{2}dx_{1}dx_{2}dx_{3}dx_{4} = \frac{637}{60}, \\
&c_{23} = \int_{P}x_{2}x_{3}dx_{1}dx_{2}dx_{3}dx_{4} = \frac{37}{8}, 
&&c_{24} = \int_{P}x_{2}x_{4}dx_{1}dx_{2}dx_{3}dx_{4} = -\frac{299}{120}, \\
&c_{33} = \int_{P}x_{3}^{2}dx_{1}dx_{2}dx_{3}dx_{4} = \frac{491}{36}, 
&&c_{34} = \int_{P}x_{3}x_{4}dx_{1}dx_{2}dx_{3}dx_{4} = -\frac{37}{8}, \\
&c_{44} = \int_{P}x_{4}^{2}dx_{1}dx_{2}dx_{3}dx_{4} = \frac{247}{36}, 
\end{align*}
\begin{align*}
&\theta_P(x_{1}, x_{2}, x_{3}, x_{4}) = \frac{-11160375{{x}_{1}} + 23983575{{x}_{3}}-9248098}{31553027}, \\
&M_{X} = \max_{P}\theta_P = \frac{51542252}{31553027} > 1. 
\end{align*}

\subsubsection{$I_{10}$}
\begin{align*}
&b_{0} = \mathrm{vol}(P) = \frac{389}{24}, 
&&b_{1} = \int_{P}x_{1}dx_{1}dx_{2}dx_{3}dx_{4} = \frac{41}{10}, \\
&b_{2} = \int_{P}x_{2}dx_{1}dx_{2}dx_{3}dx_{4} = \frac{41}{20}, 
&&b_{3} = \int_{P}x_{3}dx_{1}dx_{2}dx_{3}dx_{4} = \frac{161}{30}, \\
&b_{4} = \int_{P}x_{4}dx_{1}dx_{2}dx_{3}dx_{4} = -\frac{19}{30}, 
&&c_{11} = \int_{P}x_{1}^{2}dx_{1}dx_{2}dx_{3}dx_{4} = \frac{313}{60}, \\
&c_{12} = \int_{P}x_{1}x_{2}dx_{1}dx_{2}dx_{3}dx_{4} = \frac{313}{120}, 
&&c_{13} = \int_{P}x_{1}x_{3}dx_{1}dx_{2}dx_{3}dx_{4} = \frac{202}{45}, \\
&c_{14} = \int_{P}x_{1}x_{4}dx_{1}dx_{2}dx_{3}dx_{4} = \frac{131}{360}, 
&&c_{22} = \int_{P}x_{2}^{2}dx_{1}dx_{2}dx_{3}dx_{4} = \frac{1021}{120}, \\
&c_{23} = \int_{P}x_{2}x_{3}dx_{1}dx_{2}dx_{3}dx_{4} = \frac{101}{45}, 
&&c_{24} = \int_{P}x_{2}x_{4}dx_{1}dx_{2}dx_{3}dx_{4} = \frac{131}{720}, \\
&c_{33} = \int_{P}x_{3}^{2}dx_{1}dx_{2}dx_{3}dx_{4} = \frac{106}{9}, 
&&c_{34} = \int_{P}x_{3}x_{4}dx_{1}dx_{2}dx_{3}dx_{4} = -\frac{164}{45}, \\
&c_{44} = \int_{P}x_{4}^{2}dx_{1}dx_{2}dx_{3}dx_{4} = \frac{551}{72}, 
\end{align*}
\begin{align*}
&\theta_P(x_{1}, x_{2}, x_{3}, x_{4}) = \frac{762471120{{x}_{1}}+302223825{{x}_{3}}-292939812}{1008086248}, \\
&M_{X} = \max_{P}\theta_P = \frac{1376202783}{1008086248} > 1. 
\end{align*}

\subsubsection{$I_{11}$}
\begin{align*}
&b_{0} = \mathrm{vol}(P) = 16, 
&&b_{1} = \int_{P}x_{1}dx_{1}dx_{2}dx_{3}dx_{4} = \frac{7}{4}, \\
&b_{2} = \int_{P}x_{2}dx_{1}dx_{2}dx_{3}dx_{4} = \frac{11}{3}, 
&&b_{3} = \int_{P}x_{3}dx_{1}dx_{2}dx_{3}dx_{4} = \frac{59}{12}, \\
&b_{4} = \int_{P}x_{4}dx_{1}dx_{2}dx_{3}dx_{4} = \frac{1}{4}, 
&&c_{11} = \int_{P}x_{1}^{2}dx_{1}dx_{2}dx_{3}dx_{4} = \frac{443}{90}, \\
&c_{12} = \int_{P}x_{1}x_{2}dx_{1}dx_{2}dx_{3}dx_{4} = \frac{2}{9}, 
&&c_{13} = \int_{P}x_{1}x_{3}dx_{1}dx_{2}dx_{3}dx_{4} = \frac{341}{90}, \\
&c_{14} = \int_{P}x_{1}x_{4}dx_{1}dx_{2}dx_{3}dx_{4} = \frac{61}{90}, 
&&c_{22} = \int_{P}x_{2}^{2}dx_{1}dx_{2}dx_{3}dx_{4} = \frac{248}{45}, \\
&c_{23} = \int_{P}x_{2}x_{3}dx_{1}dx_{2}dx_{3}dx_{4} = \frac{116}{45}, 
&&c_{24} = \int_{P}x_{2}x_{4}dx_{1}dx_{2}dx_{3}dx_{4} = \frac{71}{45}, \\
&c_{33} = \int_{P}x_{3}^{2}dx_{1}dx_{2}dx_{3}dx_{4} = \frac{1169}{90}, 
&&c_{34} = \int_{P}x_{3}x_{4}dx_{1}dx_{2}dx_{3}dx_{4} = -\frac{149}{45}, \\
&c_{44} = \int_{P}x_{4}^{2}dx_{1}dx_{2}dx_{3}dx_{4} = \frac{829}{90}, 
\end{align*}
\begin{align*}
&\theta_P(x_{1}, x_{2}, x_{3}, x_{4}) = \frac{1527360{{x}_{1}}+5306880{{x}_{2}}+2119680{{x}_{3}}-2034575}{7524593}, \\
&M_{X} = \max_{P}\theta_P = \frac{13278385}{7524593} > 1. 
\end{align*}

\subsubsection{$I_{12}$}
\begin{align*}
&b_{0} = \mathrm{vol}(P) = \frac{347}{24}, 
&&b_{1} = \int_{P}x_{1}dx_{1}dx_{2}dx_{3}dx_{4} = -\frac{3}{5},\quad \\
&b_{2} = \int_{P}x_{2}dx_{1}dx_{2}dx_{3}dx_{4} = \frac{3}{10},  
&&b_{3} = \int_{P}x_{3}dx_{1}dx_{2}dx_{3}dx_{4} = \frac{3}{5}, \\
&b_{4} = \int_{P}x_{4}dx_{1}dx_{2}dx_{3}dx_{4} = -\frac{6}{5}, 
&&c_{11} = \int_{P}x_{1}^{2}dx_{1}dx_{2}dx_{3}dx_{4} = \frac{251}{60}, \\
&c_{12} = \int_{P}x_{1}x_{2}dx_{1}dx_{2}dx_{3}dx_{4} = -\frac{251}{120}, 
&&c_{13} = \int_{P}x_{1}x_{3}dx_{1}dx_{2}dx_{3}dx_{4} = \frac{119}{45}, \\
&c_{14} = \int_{P}x_{1}x_{4}dx_{1}dx_{2}dx_{3}dx_{4} = \frac{277}{360}, 
&&c_{22} = \int_{P}x_{2}^{2}dx_{1}dx_{2}dx_{3}dx_{4} = \frac{2309}{360}, \\
&c_{23} = \int_{P}x_{2}x_{3}dx_{1}dx_{2}dx_{3}dx_{4} = -\frac{119}{90}, 
&&c_{24} = \int_{P}x_{2}x_{4}dx_{1}dx_{2}dx_{3}dx_{4} = -\frac{277}{720}, \\
&c_{33} = \int_{P}x_{3}^{2}dx_{1}dx_{2}dx_{3}dx_{4} = \frac{373}{45}, 
&&c_{34} = \int_{P}x_{3}x_{4}dx_{1}dx_{2}dx_{3}dx_{4} = -\frac{127}{45}, \\
&c_{44} = \int_{P}x_{4}^{2}dx_{1}dx_{2}dx_{3}dx_{4} = \frac{2309}{360}, 
\end{align*}
\begin{align*}
&\theta_P(x_{1}, x_{2}, x_{3}, x_{4}) = -\frac{1780035048{{x}_{1}}-859184145{{x}_{3}}}{10127095471}-\frac{1344488976{{x}_{4}}+221112504}{10127095471}, \\
&M_{X} = \max_{P}\theta_P = \frac{3762595665}{10127095471} < 1. 
\end{align*}

\subsubsection{$I_{13}$}
\begin{align*}
&b_{0} = \mathrm{vol}(P) = \frac{46}{3}, 
&&b_{1} = \int_{P}x_{1}dx_{1}dx_{2}dx_{3}dx_{4} = \frac{7}{4},\quad \\
&b_{2} = \int_{P}x_{2}dx_{1}dx_{2}dx_{3}dx_{4} = 0, 
&&b_{3} = \int_{P}x_{3}dx_{1}dx_{2}dx_{3}dx_{4} = \frac{43}{12}, \\
&b_{4} = \int_{P}x_{4}dx_{1}dx_{2}dx_{3}dx_{4} = -\frac{11}{12}, 
&&c_{11} = \int_{P}x_{1}^{2}dx_{1}dx_{2}dx_{3}dx_{4} = \frac{47}{10}, \\
&c_{12} = \int_{P}x_{1}x_{2}dx_{1}dx_{2}dx_{3}dx_{4} = 0, 
&&c_{13} = \int_{P}x_{1}x_{3}dx_{1}dx_{2}dx_{3}dx_{4} = \frac{107}{30}, \\
&c_{14} = \int_{P}x_{1}x_{4}dx_{1}dx_{2}dx_{3}dx_{4} = \frac{17}{30}, 
&&c_{22} = \int_{P}x_{2}^{2}dx_{1}dx_{2}dx_{3}dx_{4} = \frac{46}{9}, \\
&c_{23} = \int_{P}x_{2}x_{3}dx_{1}dx_{2}dx_{3}dx_{4} = 0, 
&&c_{24} = \int_{P}x_{2}x_{4}dx_{1}dx_{2}dx_{3}dx_{4} = 0, \\
&c_{33} = \int_{P}x_{3}^{2}dx_{1}dx_{2}dx_{3}dx_{4} = \frac{301}{30}, 
&&c_{34} = \int_{P}x_{3}x_{4}dx_{1}dx_{2}dx_{3}dx_{4} = -\frac{97}{30}, \\
&c_{44} = \int_{P}x_{4}^{2}dx_{1}dx_{2}dx_{3}dx_{4} = \frac{211}{30}, 
\end{align*}
\begin{align*}
&\theta_P(x_{1}, x_{2}, x_{3}, x_{4}) = \frac{98900{{x}_{1}}+219420{{x}_{3}}-62565}{650251}, \\
&M_{X} = \max_{P}\theta_P = \frac{694595}{650251} > 1. 
\end{align*}

\subsubsection{$I_{14}$}
\begin{align*}
&b_{0} = \mathrm{vol}(P) = \frac{119}{8}, 
&&b_{1} = \int_{P}x_{1}dx_{1}dx_{2}dx_{3}dx_{4} = \frac{61}{30},\quad \\
&b_{2} = \int_{P}x_{2}dx_{1}dx_{2}dx_{3}dx_{4} = \frac{13}{10}, 
&&b_{3} = \int_{P}x_{3}dx_{1}dx_{2}dx_{3}dx_{4} = \frac{59}{30}, \\
&b_{4} = \int_{P}x_{4}dx_{1}dx_{2}dx_{3}dx_{4} = \frac{13}{5}, 
&&c_{11} = \int_{P}x_{1}^{2}dx_{1}dx_{2}dx_{3}dx_{4} = \frac{1619}{360}, \\
&c_{12} = \int_{P}x_{1}x_{2}dx_{1}dx_{2}dx_{3}dx_{4} = \frac{427}{720}, 
&&c_{13} = \int_{P}x_{1}x_{3}dx_{1}dx_{2}dx_{3}dx_{4} = \frac{142}{45}, \\
&c_{14} = \int_{P}x_{1}x_{4}dx_{1}dx_{2}dx_{3}dx_{4} = \frac{427}{360}, 
&&c_{22} = \int_{P}x_{2}^{2}dx_{1}dx_{2}dx_{3}dx_{4} = \frac{3043}{360}, \\
&c_{23} = \int_{P}x_{2}x_{3}dx_{1}dx_{2}dx_{3}dx_{4} = -\frac{1003}{720}, 
&&c_{24} = \int_{P}x_{2}x_{4}dx_{1}dx_{2}dx_{3}dx_{4} = \frac{473}{120}, \\
&c_{33} = \int_{P}x_{3}^{2}dx_{1}dx_{2}dx_{3}dx_{4} = \frac{2963}{360}, 
&&c_{34} = \int_{P}x_{3}x_{4}dx_{1}dx_{2}dx_{3}dx_{4} = -\frac{1003}{360}, \\
&c_{44} = \int_{P}x_{4}^{2}dx_{1}dx_{2}dx_{3}dx_{4} = \frac{473}{60}, 
\end{align*}
\begin{align*}
&\theta_P(x_{1}, x_{2}, x_{3}, x_{4}) = \frac{21660926805{{x}_{1}}+90160555275{{x}_{3}}}{218756440813}+ \frac{112125015975{{x}_{4}}-34479642672}{218756440813}, \\
&M_{X} = \max_{P}\theta_P = \frac{211431316083}{218756440813} < 1. 
\end{align*}

\subsubsection{{$I_{15}$}}
\begin{align*}
&b_{0} = \mathrm{vol}(P) = \frac{337}{24}, 
&&b_{1} = \int_{P}x_{1}dx_{1}dx_{2}dx_{3}dx_{4} = \frac{5}{2}, \\
&b_{2} = \int_{P}x_{2}dx_{1}dx_{2}dx_{3}dx_{4} = -\frac{5}{4}, 
&&b_{3} = \int_{P}x_{3}dx_{1}dx_{2}dx_{3}dx_{4} = \frac{23}{5}, \\
&b_{4} = \int_{P}x_{4}dx_{1}dx_{2}dx_{3}dx_{4} = \frac{1}{5}, 
&&c_{11} = \int_{P}x_{1}^{2}dx_{1}dx_{2}dx_{3}dx_{4} = \frac{79}{20}, \\
&c_{12} = \int_{P}x_{1}x_{2}dx_{1}dx_{2}dx_{3}dx_{4} = -\frac{79}{40}, 
&&c_{13} = \int_{P}x_{1}x_{3}dx_{1}dx_{2}dx_{3}dx_{4} = \frac{821}{180}, \\
&c_{14} = \int_{P}x_{1}x_{4}dx_{1}dx_{2}dx_{3}dx_{4} = \frac{601}{360}, 
&&c_{22} = \int_{P}x_{2}^{2}dx_{1}dx_{2}dx_{3}dx_{4} = \frac{1859}{360}, \\
&c_{23} = \int_{P}x_{2}x_{3}dx_{1}dx_{2}dx_{3}dx_{4} = -\frac{821}{360}, 
&&c_{24} = \int_{P}x_{2}x_{4}dx_{1}dx_{2}dx_{3}dx_{4} = -\frac{601}{720}, \\
&c_{33} = \int_{P}x_{3}^{2}dx_{1}dx_{2}dx_{3}dx_{4} = \frac{2381}{180}, 
&&c_{34} = \int_{P}x_{3}x_{4}dx_{1}dx_{2}dx_{3}dx_{4} = -\frac{739}{360}, \\
&c_{44} = \int_{P}x_{4}^{2}dx_{1}dx_{2}dx_{3}dx_{4} = \frac{1009}{120}, 
\end{align*}
\begin{align*}
&\theta_P(x_{1}, x_{2}, x_{3}, x_{4}) = \frac{329975235{{x}_{1}}+184724865{{x}_{3}}-119264508}{739116617}, \\
&M_{X} = \max_{P}\theta_P = \frac{949610187}{739116617} > 1. 
\end{align*}

\subsubsection{{$M_{1}$}}
\begin{align*}
&b_{0} = \mathrm{vol}(P) = \frac{385}{24}, 
&&b_{1} = \int_{P}x_{1}dx_{1}dx_{2}dx_{3}dx_{4} = -\frac{79}{30}, \\
&b_{2} = \int_{P}x_{2}dx_{1}dx_{2}dx_{3}dx_{4} = \frac{79}{20}, 
&&b_{3} = \int_{P}x_{3}dx_{1}dx_{2}dx_{3}dx_{4} = \frac{79}{30}, \\
&b_{4} = \int_{P}x_{4}dx_{1}dx_{2}dx_{3}dx_{4} = \frac{79}{30}, 
&&c_{11} = \int_{P}x_{1}^{2}dx_{1}dx_{2}dx_{3}dx_{4} = \frac{193}{36}, \\
&c_{12} = \int_{P}x_{1}x_{2}dx_{1}dx_{2}dx_{3}dx_{4} = -\frac{1949}{720}, 
&&c_{13} = \int_{P}x_{1}x_{3}dx_{1}dx_{2}dx_{3}dx_{4} = -\frac{19}{720}, \\
&c_{14} = \int_{P}x_{1}x_{4}dx_{1}dx_{2}dx_{3}dx_{4} = -\frac{19}{720}, 
&&c_{22} = \int_{P}x_{2}^{2}dx_{1}dx_{2}dx_{3}dx_{4} = \frac{2411}{180}, \\
&c_{23} = \int_{P}x_{2}x_{3}dx_{1}dx_{2}dx_{3}dx_{4} = \frac{1949}{720}, 
&&c_{24} = \int_{P}x_{2}x_{4}dx_{1}dx_{2}dx_{3}dx_{4} = \frac{1949}{720}, \\
&c_{33} = \int_{P}x_{3}^{2}dx_{1}dx_{2}dx_{3}dx_{4} = \frac{193}{36}, 
&&c_{34} = \int_{P}x_{3}x_{4}dx_{1}dx_{2}dx_{3}dx_{4} = \frac{19}{720}, \\
&c_{44} = \int_{P}x_{4}^{2}dx_{1}dx_{2}dx_{3}dx_{4} = \frac{193}{36}, 
\end{align*}
\begin{align*}
&\theta_P(x_{1}, x_{2}, x_{3}, x_{4}) = \frac{-1824900{{x}_{1}}+1824900{{x}_{3}}+1824900{{x}_{4}}-898704}{2853121}, \\
&M_{X} = \max_{P}\theta_P = \frac{4575996}{2853121} > 1. 
\end{align*}

\subsubsection{{$M_{2}$}}
\begin{align*}
&b_{0} = \mathrm{vol}(P) = \frac{139}{8}, 
&&b_{1} = \int_{P}x_{1}dx_{1}dx_{2}dx_{3}dx_{4} = \frac{161}{30}, \\
&b_{2} = \int_{P}x_{2}dx_{1}dx_{2}dx_{3}dx_{4} = \frac{239}{60}, 
&&b_{3} = \int_{P}x_{3}dx_{1}dx_{2}dx_{3}dx_{4} = \frac{20}{3}, \\
&b_{4} = \int_{P}x_{4}dx_{1}dx_{2}dx_{3}dx_{4} = \frac{20}{3}, 
&&c_{11} = \int_{P}x_{1}^{2}dx_{1}dx_{2}dx_{3}dx_{4} = \frac{1109}{180}, \\
&c_{12} = \int_{P}x_{1}x_{2}dx_{1}dx_{2}dx_{3}dx_{4} = \frac{481}{240}, 
&&c_{13} = \int_{P}x_{1}x_{3}dx_{1}dx_{2}dx_{3}dx_{4} = \frac{481}{240}, \\
&c_{14} = \int_{P}x_{1}x_{4}dx_{1}dx_{2}dx_{3}dx_{4} = \frac{481}{240}, 
&&c_{22} = \int_{P}x_{2}^{2}dx_{1}dx_{2}dx_{3}dx_{4} = \frac{2503}{180}, \\
&c_{23} = \int_{P}x_{2}x_{3}dx_{1}dx_{2}dx_{3}dx_{4} = \frac{1223}{240}, 
&&c_{24} = \int_{P}x_{2}x_{4}dx_{1}dx_{2}dx_{3}dx_{4} = \frac{1223}{240}, \\
&c_{33} = \int_{P}x_{3}^{2}dx_{1}dx_{2}dx_{3}dx_{4} = \frac{2101}{180}, 
&&c_{34} = \int_{P}x_{3}x_{4}dx_{1}dx_{2}dx_{3}dx_{4} = \frac{173}{48}, \\
&c_{44} = \int_{P}x_{4}^{2}dx_{1}dx_{2}dx_{3}dx_{4} = \frac{2101}{180}, 
\end{align*}
\begin{align*}
&\theta_P(x_{1}, x_{2}, x_{3}, x_{4}) = \frac{2555701260{{x}_{1}}+2483476860{{x}_{3}}}{4944865931}+ \frac{2483476860{{x}_{4}}-2695172464}{4944865931}, \\
&M_{X} = \max_{P}\theta_P = \frac{9794436236}{4944865931} > 1. 
\end{align*}

\subsubsection{{$M_{3}$}}
\begin{align*}
&b_{0} = \mathrm{vol}(P) = \frac{123}{8}, 
&&b_{1} = \int_{P}x_{1}dx_{1}dx_{2}dx_{3}dx_{4} = \frac{89}{30}, \\
&b_{2} = \int_{P}x_{2}dx_{1}dx_{2}dx_{3}dx_{4} = \frac{13}{10}, 
&&b_{3} = \int_{P}x_{3}dx_{1}dx_{2}dx_{3}dx_{4} = \frac{7}{60}, \\
&b_{4} = \int_{P}x_{4}dx_{1}dx_{2}dx_{3}dx_{4} = \frac{109}{20}, 
&&c_{11} = \int_{P}x_{1}^{2}dx_{1}dx_{2}dx_{3}dx_{4} = \frac{185}{36}, \\
&c_{12} = \int_{P}x_{1}x_{2}dx_{1}dx_{2}dx_{3}dx_{4} = \frac{19}{720}, 
&&c_{13} = \int_{P}x_{1}x_{3}dx_{1}dx_{2}dx_{3}dx_{4} = \frac{1549}{720}, \\
&c_{14} = \int_{P}x_{1}x_{4}dx_{1}dx_{2}dx_{3}dx_{4} = \frac{2189}{720}, 
&&c_{22} = \int_{P}x_{2}^{2}dx_{1}dx_{2}dx_{3}dx_{4} = \frac{1433}{180}, \\
&c_{23} = \int_{P}x_{2}x_{3}dx_{1}dx_{2}dx_{3}dx_{4} = \frac{413}{720}, 
&&c_{24} = \int_{P}x_{2}x_{4}dx_{1}dx_{2}dx_{3}dx_{4} = \frac{1733}{720}, \\
&c_{33} = \int_{P}x_{3}^{2}dx_{1}dx_{2}dx_{3}dx_{4} = \frac{1279}{180}, 
&&c_{34} = \int_{P}x_{3}x_{4}dx_{1}dx_{2}dx_{3}dx_{4} = \frac{-2741}{720}, \\
&c_{44} = \int_{P}x_{4}^{2}dx_{1}dx_{2}dx_{3}dx_{4} = \frac{2099}{180}, 
\end{align*}
\begin{align*}
&\theta_P(x_{1}, x_{2}, x_{3}, x_{4}) = \frac{4151273350320{{x}_{1}}+4524155739540{{x}_{3}}}{16440690503203} + \frac{10150485305940{{x}_{4}}-4433392348592}{16440690503203},\\
&M_{X} = \max_{P}\theta_P = \frac{25645181180008}{16440690503203} > 1. 
\end{align*}

\subsubsection{{$M_{4}$}}
\begin{align*}
&b_{0} = \mathrm{vol}(P) = \frac{123}{8}, 
&&b_{1} = \int_{P}x_{1}dx_{1}dx_{2}dx_{3}dx_{4} = \frac{41}{30}, \\
&b_{2} = \int_{P}x_{2}dx_{1}dx_{2}dx_{3}dx_{4} = \frac{157}{60}, 
&&b_{3} = \int_{P}x_{3}dx_{1}dx_{2}dx_{3}dx_{4} = \frac{119}{30}, \\
&b_{4} = \int_{P}x_{4}dx_{1}dx_{2}dx_{3}dx_{4} = \frac{79}{30}, 
&&c_{11} = \int_{P}x_{1}^{2}dx_{1}dx_{2}dx_{3}dx_{4} = \frac{893}{180}, \\
&c_{12} = \int_{P}x_{1}x_{2}dx_{1}dx_{2}dx_{3}dx_{4} = -\frac{253}{720}, 
&&c_{13} = \int_{P}x_{1}x_{3}dx_{1}dx_{2}dx_{3}dx_{4} = \frac{553}{144}, \\
&c_{14} = \int_{P}x_{1}x_{4}dx_{1}dx_{2}dx_{3}dx_{4} = \frac{301}{720}, 
&&c_{22} = \int_{P}x_{2}^{2}dx_{1}dx_{2}dx_{3}dx_{4} = \frac{1879}{180}, \\
&c_{23} = \int_{P}x_{2}x_{3}dx_{1}dx_{2}dx_{3}dx_{4} = \frac{1813}{720}, 
&&c_{24} = \int_{P}x_{2}x_{4}dx_{1}dx_{2}dx_{3}dx_{4} = \frac{1789}{720}, \\
&c_{33} = \int_{P}x_{3}^{2}dx_{1}dx_{2}dx_{3}dx_{4} = \frac{1553}{180}, 
&&c_{34} = \int_{P}x_{3}x_{4}dx_{1}dx_{2}dx_{3}dx_{4} = \frac{179}{720}, \\
&c_{44} = \int_{P}x_{4}^{2}dx_{1}dx_{2}dx_{3}dx_{4} = \frac{185}{36}, 
\end{align*}
\begin{align*}
&\theta_P(x_{1}, x_{2}, x_{3}, x_{4}) = \frac{-2793968096940{{x}_{1}}+8560548364140{{x}_{3}}}{13060549547443}+ \frac{8232843774060{{x}_{4}}-3370292079152}{13060549547443}, \\
&M_{X} = \max_{P}\theta_P = \frac{19189680326248}{13060549547443} > 1. 
\end{align*}

\subsubsection{{$M_{5}$}}
\begin{align*}
&b_{0} = \mathrm{vol}(P) = \frac{91}{6}, 
&&b_{1} = \int_{P}x_{1}dx_{1}dx_{2}dx_{3}dx_{4} = \frac{347}{60}, \\
&b_{2} = \int_{P}x_{2}dx_{1}dx_{2}dx_{3}dx_{4} = -\frac{3}{2}, 
&&b_{3} = \int_{P}x_{3}dx_{1}dx_{2}dx_{3}dx_{4} = \frac{169}{30}, \\
&b_{4} = \int_{P}x_{4}dx_{1}dx_{2}dx_{3}dx_{4} = 3, 
&&c_{11} = \int_{P}x_{1}^{2}dx_{1}dx_{2}dx_{3}dx_{4} = \frac{577}{45}, \\
&c_{12} = \int_{P}x_{1}x_{2}dx_{1}dx_{2}dx_{3}dx_{4} = -\frac{2287}{720}, 
&&c_{13} = \int_{P}x_{1}x_{3}dx_{1}dx_{2}dx_{3}dx_{4} = \frac{141}{16}, \\
&c_{14} = \int_{P}x_{1}x_{4}dx_{1}dx_{2}dx_{3}dx_{4} = \frac{2287}{360}, 
&&c_{22} = \int_{P}x_{2}^{2}dx_{1}dx_{2}dx_{3}dx_{4} = \frac{2399}{360}, \\
&c_{23} = \int_{P}x_{2}x_{3}dx_{1}dx_{2}dx_{3}dx_{4} = -\frac{2423}{720}, 
&&c_{24} = \int_{P}x_{2}x_{4}dx_{1}dx_{2}dx_{3}dx_{4} = -\frac{313}{80}, \\
&c_{33} = \int_{P}x_{3}^{2}dx_{1}dx_{2}dx_{3}dx_{4} = \frac{1039}{90}, 
&&c_{34} = \int_{P}x_{3}x_{4}dx_{1}dx_{2}dx_{3}dx_{4} = \frac{2423}{360}, \\
&c_{44} = \int_{P}x_{4}^{2}dx_{1}dx_{2}dx_{3}dx_{4} = \frac{313}{40}, 
\end{align*}
\begin{align*}
&\theta_P(x_{1}, x_{2}, x_{3}, x_{4}) = \frac{2003379647360{{x}_{1}}+2700009652340{{x}_{3}}}{5748810067571}+ \frac{-1155042818020{{x}_{4}}-1538316805364}{5748810067571}, \\
&M_{X} = \max_{P}\theta_P = \frac{12265145105056}{5748810067571} > 1. 
\end{align*}

\subsubsection{{$J_{1}$}}
\begin{align*}
&b_{0} = \mathrm{vol}(P) = \frac{46}{3}, 
&&b_{1} = \int_{P}x_{1}dx_{1}dx_{2}dx_{3}dx_{4} = -\frac{139}{30}, \\
&b_{2} = \int_{P}x_{2}dx_{1}dx_{2}dx_{3}dx_{4} = \frac{71}{20}, 
&&b_{3} = \int_{P}x_{3}dx_{1}dx_{2}dx_{3}dx_{4} = \frac{37}{30}, \\
&b_{4} = \int_{P}x_{4}dx_{1}dx_{2}dx_{3}dx_{4} = \frac{37}{30}, 
&&c_{11} = \int_{P}x_{1}^{2}dx_{1}dx_{2}dx_{3}dx_{4} = \frac{871}{180}, \\
&c_{12} = \int_{P}x_{1}x_{2}dx_{1}dx_{2}dx_{3}dx_{4} = -\frac{2173}{720}, 
&&c_{13} = \int_{P}x_{1}x_{3}dx_{1}dx_{2}dx_{3}dx_{4} = -\frac{431}{720}, \\
&c_{14} = \int_{P}x_{1}x_{4}dx_{1}dx_{2}dx_{3}dx_{4} = -\frac{431}{720}, 
&&c_{22} = \int_{P}x_{2}^{2}dx_{1}dx_{2}dx_{3}dx_{4} = \frac{1417}{120}, \\
&c_{23} = \int_{P}x_{2}x_{3}dx_{1}dx_{2}dx_{3}dx_{4} = \frac{313}{180}, 
&&c_{24} = \int_{P}x_{2}x_{4}dx_{1}dx_{2}dx_{3}dx_{4} = \frac{313}{180}, \\
&c_{33} = \int_{P}x_{3}^{2}dx_{1}dx_{2}dx_{3}dx_{4} = \frac{2827}{360}, 
&&c_{34} = \int_{P}x_{3}x_{4}dx_{1}dx_{2}dx_{3}dx_{4} = -\frac{3581}{720}, \\
&c_{44} = \int_{P}x_{4}^{2}dx_{1}dx_{2}dx_{3}dx_{4} = \frac{2827}{360}, 
\end{align*}
\begin{align*}
&\theta_P(x_{1}, x_{2}, x_{3}, x_{4}) = \frac{-353629140{{x}_{1}}+95218620{{x}_{3}}+95218620{{x}_{4}}-122175279}{271746296}, \\
&M_{X} = \max_{P}\theta_P = \frac{326672481}{271746296} > 1. 
\end{align*}

\subsubsection{$J_{2}$}
\begin{align*}
&b_{0} = \mathrm{vol}(P) = \frac{163}{12}, 
&&b_{1} = \int_{P}x_{1}dx_{1}dx_{2}dx_{3}dx_{4} = 0,\quad \\
&b_{2} = \int_{P}x_{2}dx_{1}dx_{2}dx_{3}dx_{4} = \frac{2}{3}, 
&&b_{3} = \int_{P}x_{3}dx_{1}dx_{2}dx_{3}dx_{4} = \frac{2}{3}, \\
&b_{4} = \int_{P}x_{4}dx_{1}dx_{2}dx_{3}dx_{4} = \frac{2}{3}, 
&&c_{11} = \int_{P}x_{1}^{2}dx_{1}dx_{2}dx_{3}dx_{4} = \frac{137}{36}, \\
&c_{12} = \int_{P}x_{1}x_{2}dx_{1}dx_{2}dx_{3}dx_{4} = -\frac{137}{144}, 
&&c_{13} = \int_{P}x_{1}x_{3}dx_{1}dx_{2}dx_{3}dx_{4} = \frac{137}{144}, \\
&c_{14} = \int_{P}x_{1}x_{4}dx_{1}dx_{2}dx_{3}dx_{4} = \frac{137}{144}, 
&&c_{22} = \int_{P}x_{2}^{2}dx_{1}dx_{2}dx_{3}dx_{4} = \frac{2381}{360}, \\
&c_{23} = \int_{P}x_{2}x_{3}dx_{1}dx_{2}dx_{3}dx_{4} = \frac{34}{45}, 
&&c_{24} = \int_{P}x_{2}x_{4}dx_{1}dx_{2}dx_{3}dx_{4} = \frac{34}{45}, \\
&c_{33} = \int_{P}x_{3}^{2}dx_{1}dx_{2}dx_{3}dx_{4} = \frac{2381}{360}, 
&&c_{34} = \int_{P}x_{3}x_{4}dx_{1}dx_{2}dx_{3}dx_{4} = -\frac{2989}{720}, \\
&c_{44} = \int_{P}x_{4}^{2}dx_{1}dx_{2}dx_{3}dx_{4} = \frac{2381}{360}, 
\end{align*}
\begin{align*}
&\theta_P(x_{1}, x_{2}, x_{3}, x_{4}) = \frac{-78240{{x}_{1}}+156480{{x}_{3}}+156480{{x}_{4}}-15360}{450983}, \\
&M_{X} = \max_{P}\theta_P = \frac{141120}{450983} < 1. 
\end{align*}

\subsubsection{{$Q_{1}$}}
\begin{align*}
&b_{0} = \mathrm{vol}(P) = \frac{221}{12}, 
&&b_{1} = \int_{P}x_{1}dx_{1}dx_{2}dx_{3}dx_{4} = \frac{329}{60}, \\
&b_{2} = \int_{P}x_{2}dx_{1}dx_{2}dx_{3}dx_{4} = \frac{329}{30}, 
&&b_{3} = \int_{P}x_{3}dx_{1}dx_{2}dx_{3}dx_{4} = \frac{329}{60}, \\
&b_{4} = \int_{P}x_{4}dx_{1}dx_{2}dx_{3}dx_{4} = \frac{329}{60}, 
&&c_{11} = \int_{P}x_{1}^{2}dx_{1}dx_{2}dx_{3}dx_{4} = \frac{233}{36}, \\
&c_{12} = \int_{P}x_{1}x_{2}dx_{1}dx_{2}dx_{3}dx_{4} = \frac{437}{60}, 
&&c_{13} = \int_{P}x_{1}x_{3}dx_{1}dx_{2}dx_{3}dx_{4} = \frac{437}{120}, \\
&c_{14} = \int_{P}x_{1}x_{4}dx_{1}dx_{2}dx_{3}dx_{4} = \frac{437}{120}, 
&&c_{22} = \int_{P}x_{2}^{2}dx_{1}dx_{2}dx_{3}dx_{4} = \frac{437}{30}, \\
&c_{23} = \int_{P}x_{2}x_{3}dx_{1}dx_{2}dx_{3}dx_{4} = \frac{437}{60}, 
&&c_{24} = \int_{P}x_{2}x_{4}dx_{1}dx_{2}dx_{3}dx_{4} = \frac{437}{60}, \\
&c_{33} = \int_{P}x_{3}^{2}dx_{1}dx_{2}dx_{3}dx_{4} = \frac{293}{20}, 
&&c_{34} = \int_{P}x_{3}x_{4}dx_{1}dx_{2}dx_{3}dx_{4} = \frac{437}{120}, \\
&c_{44} = \int_{P}x_{4}^{2}dx_{1}dx_{2}dx_{3}dx_{4} = \frac{293}{20}, 
\end{align*}
\begin{align*}
&\theta_P(x_{1}, x_{2}, x_{3}, x_{4}) = \frac{363545{{x}_{2}}-216482}{266403}, \\
&M_{X} = \max_{P}\theta_P = \frac{510608}{266403} > 1. 
\end{align*}

\subsubsection{{$Q_{2}$}}
\begin{align*}
&b_{0} = \mathrm{vol}(P) = \frac{135}{8}, 
&&b_{1} = \int_{P}x_{1}dx_{1}dx_{2}dx_{3}dx_{4} = \frac{257}{60}, \\
&b_{2} = \int_{P}x_{2}dx_{1}dx_{2}dx_{3}dx_{4} = \frac{257}{30}, 
&&b_{3} = \int_{P}x_{3}dx_{1}dx_{2}dx_{3}dx_{4} = \frac{124}{15}, \\
&b_{4} = \int_{P}x_{4}dx_{1}dx_{2}dx_{3}dx_{4} = \frac{62}{15}, 
&&c_{11} = \int_{P}x_{1}^{2}dx_{1}dx_{2}dx_{3}dx_{4} = \frac{2071}{360}, \\
&c_{12} = \int_{P}x_{1}x_{2}dx_{1}dx_{2}dx_{3}dx_{4} = \frac{145}{24}, 
&&c_{13} = \int_{P}x_{1}x_{3}dx_{1}dx_{2}dx_{3}dx_{4} = \frac{109}{24}, \\
&c_{14} = \int_{P}x_{1}x_{4}dx_{1}dx_{2}dx_{3}dx_{4} = \frac{109}{48}, 
&&c_{22} = \int_{P}x_{2}^{2}dx_{1}dx_{2}dx_{3}dx_{4} = \frac{145}{12}, \\
&c_{23} = \int_{P}x_{2}x_{3}dx_{1}dx_{2}dx_{3}dx_{4} = \frac{109}{12}, 
&&c_{24} = \int_{P}x_{2}x_{4}dx_{1}dx_{2}dx_{3}dx_{4} = \frac{109}{24}, \\
&c_{33} = \int_{P}x_{3}^{2}dx_{1}dx_{2}dx_{3}dx_{4} = \frac{311}{20}, 
&&c_{34} = \int_{P}x_{3}x_{4}dx_{1}dx_{2}dx_{3}dx_{4} = \frac{311}{40}, \\
&c_{44} = \int_{P}x_{4}^{2}dx_{1}dx_{2}dx_{3}dx_{4} = \frac{4883}{360}, 
\end{align*}
\begin{align*}
&\theta_P(x_{1}, x_{2}, x_{3}, x_{4}) = \frac{211857525{{x}_{2}}+80463375{{x}_{3}}-146967508}{237174992}, \\
&M_{X} = \max_{P}\theta_P = \frac{518137667}{237174992} > 1. 
\end{align*}

\subsubsection{{$Q_{3}$}}
\begin{align*}
&b_{0} = \mathrm{vol}(P) = \frac{197}{12}, 
&&b_{1} = \int_{P}x_{1}dx_{1}dx_{2}dx_{3}dx_{4} = \frac{88}{15}, \\
&b_{2} = \int_{P}x_{2}dx_{1}dx_{2}dx_{3}dx_{4} = \frac{391}{60}, 
&&b_{3} = \int_{P}x_{3}dx_{1}dx_{2}dx_{3}dx_{4} = \frac{44}{15}, \\
&b_{4} = \int_{P}x_{4}dx_{1}dx_{2}dx_{3}dx_{4} = \frac{44}{15}, 
&&c_{11} = \int_{P}x_{1}^{2}dx_{1}dx_{2}dx_{3}dx_{4} = \frac{173}{30}, \\
&c_{12} = \int_{P}x_{1}x_{2}dx_{1}dx_{2}dx_{3}dx_{4} = \frac{65}{12}, 
&&c_{13} = \int_{P}x_{1}x_{3}dx_{1}dx_{2}dx_{3}dx_{4} = \frac{173}{60}, \\
&c_{14} = \int_{P}x_{1}x_{4}dx_{1}dx_{2}dx_{3}dx_{4} = \frac{173}{60}, 
&&c_{22} = \int_{P}x_{2}^{2}dx_{1}dx_{2}dx_{3}dx_{4} = \frac{1861}{180}, \\
&c_{23} = \int_{P}x_{2}x_{3}dx_{1}dx_{2}dx_{3}dx_{4} = \frac{65}{24}, 
&&c_{24} = \int_{P}x_{2}x_{4}dx_{1}dx_{2}dx_{3}dx_{4} = \frac{65}{24}, \\
&c_{33} = \int_{P}x_{3}^{2}dx_{1}dx_{2}dx_{3}dx_{4} = \frac{187}{20}, 
&&c_{34} = \int_{P}x_{3}x_{4}dx_{1}dx_{2}dx_{3}dx_{4} = \frac{173}{120}, \\
&c_{44} = \int_{P}x_{4}^{2}dx_{1}dx_{2}dx_{3}dx_{4} = \frac{187}{20}, 
\end{align*}
\begin{align*}
&\theta_P(x_{1}, x_{2}, x_{3}, x_{4}) = \frac{269739295{{x}_{1}}+61718130{{x}_{2}}-120893422}{201232113}, \\
&M_{X} = \max_{P}\theta_P = \frac{90760711}{67077371} > 1. 
\end{align*}

\subsubsection{{$Q_{4}$}}
\begin{align*}
&b_{0} = \mathrm{vol}(P) = \frac{135}{8}, 
&&b_{1} = \int_{P}x_{1}dx_{1}dx_{2}dx_{3}dx_{4} = \frac{28}{5}, \\
&b_{2} = \int_{P}x_{2}dx_{1}dx_{2}dx_{3}dx_{4} = \frac{257}{30}, 
&&b_{3} = \int_{P}x_{3}dx_{1}dx_{2}dx_{3}dx_{4} = \frac{257}{60}, \\
&b_{4} = \int_{P}x_{4}dx_{1}dx_{2}dx_{3}dx_{4} = \frac{14}{5}, 
&&c_{11} = \int_{P}x_{1}^{2}dx_{1}dx_{2}dx_{3}dx_{4} = \frac{119}{20}, \\
&c_{12} = \int_{P}x_{1}x_{2}dx_{1}dx_{2}dx_{3}dx_{4} = \frac{1123}{180}, 
&&c_{13} = \int_{P}x_{1}x_{3}dx_{1}dx_{2}dx_{3}dx_{4} = \frac{1123}{360}, \\
&c_{14} = \int_{P}x_{1}x_{4}dx_{1}dx_{2}dx_{3}dx_{4} = \frac{119}{40}, 
&&c_{22} = \int_{P}x_{2}^{2}dx_{1}dx_{2}dx_{3}dx_{4} = \frac{145}{12}, \\
&c_{23} = \int_{P}x_{2}x_{3}dx_{1}dx_{2}dx_{3}dx_{4} = \frac{145}{24}, 
&&c_{24} = \int_{P}x_{2}x_{4}dx_{1}dx_{2}dx_{3}dx_{4} = \frac{1123}{360}, \\
&c_{33} = \int_{P}x_{3}^{2}dx_{1}dx_{2}dx_{3}dx_{4} = \frac{1501}{120}, 
&&c_{34} = \int_{P}x_{3}x_{4}dx_{1}dx_{2}dx_{3}dx_{4} = \frac{1123}{720}, \\
&c_{44} = \int_{P}x_{4}^{2}dx_{1}dx_{2}dx_{3}dx_{4} = \frac{379}{40}, 
\end{align*}
\begin{align*}
&\theta_P(x_{1}, x_{2}, x_{3}, x_{4}) = \frac{51832575{{x}_{1}}+58443525{{x}_{2}}-46869844}{73313456}, \\
&M_{X} = \max_{P}\theta_P = \frac{121849781}{73313456} > 1. 
\end{align*}

\subsubsection{{$Q_{5}$}}
\begin{align*}
&b_{0} = \mathrm{vol}(P) = \frac{373}{24}, 
&&b_{1} = \int_{P}x_{1}dx_{1}dx_{2}dx_{3}dx_{4} = \frac{281}{60}, \\
&b_{2} = \int_{P}x_{2}dx_{1}dx_{2}dx_{3}dx_{4} = \frac{82}{15}, 
&&b_{3} = \int_{P}x_{3}dx_{1}dx_{2}dx_{3}dx_{4} = \frac{167}{30}, \\
&b_{4} = \int_{P}x_{4}dx_{1}dx_{2}dx_{3}dx_{4} = \frac{167}{60}, 
&&c_{11} = \int_{P}x_{1}^{2}dx_{1}dx_{2}dx_{3}dx_{4} = \frac{21}{4}, \\
&c_{12} = \int_{P}x_{1}x_{2}dx_{1}dx_{2}dx_{3}dx_{4} = \frac{577}{120}, 
&&c_{13} = \int_{P}x_{1}x_{3}dx_{1}dx_{2}dx_{3}dx_{4} = \frac{49}{12}, \\
&c_{14} = \int_{P}x_{1}x_{4}dx_{1}dx_{2}dx_{3}dx_{4} = \frac{49}{24}, 
&&c_{22} = \int_{P}x_{2}^{2}dx_{1}dx_{2}dx_{3}dx_{4} = \frac{671}{72}, \\
&c_{23} = \int_{P}x_{2}x_{3}dx_{1}dx_{2}dx_{3}dx_{4} = \frac{473}{120}, 
&&c_{24} = \int_{P}x_{2}x_{4}dx_{1}dx_{2}dx_{3}dx_{4} = \frac{473}{240}, \\
&c_{33} = \int_{P}x_{3}^{2}dx_{1}dx_{2}dx_{3}dx_{4} = \frac{199}{20}, 
&&c_{34} = \int_{P}x_{3}x_{4}dx_{1}dx_{2}dx_{3}dx_{4} = \frac{199}{40}, \\
&c_{44} = \int_{P}x_{4}^{2}dx_{1}dx_{2}dx_{3}dx_{4} = \frac{3727}{360}, 
\end{align*}
\begin{align*}
&\theta_P(x_{1}, x_{2}, x_{3}, x_{4}) = \frac{394596985345{{x}_{1}}+192548075640{{x}_{2}}+234003144575{{x}_{3}}}{573603226348}- \frac{270449943142}{573603226348},  \\
&M_{X} = \max_{P}\theta_P = \frac{977249482633}{573603226348} > 1. 
\end{align*}

\subsubsection{{$Q_{6}$}}
\begin{align*}
&b_{0} = \mathrm{vol}(P) = \frac{46}{3}, 
&&b_{1} = \int_{P}x_{1}dx_{1}dx_{2}dx_{3}dx_{4} = \frac{37}{12}, \\
&b_{2} = \int_{P}x_{2}dx_{1}dx_{2}dx_{3}dx_{4} = \frac{37}{6}, 
&&b_{3} = \int_{P}x_{3}dx_{1}dx_{2}dx_{3}dx_{4} = 0, \\
&b_{4} = \int_{P}x_{4}dx_{1}dx_{2}dx_{3}dx_{4} = \frac{37}{12}, 
&&c_{11} = \int_{P}x_{1}^{2}dx_{1}dx_{2}dx_{3}dx_{4} = \frac{151}{30}, \\
&c_{12} = \int_{P}x_{1}x_{2}dx_{1}dx_{2}dx_{3}dx_{4} = \frac{24}{5}, 
&&c_{13} = \int_{P}x_{1}x_{3}dx_{1}dx_{2}dx_{3}dx_{4} = 0, \\
&c_{14} = \int_{P}x_{1}x_{4}dx_{1}dx_{2}dx_{3}dx_{4} = \frac{12}{5}, 
&&c_{22} = \int_{P}x_{2}^{2}dx_{1}dx_{2}dx_{3}dx_{4} = \frac{48}{5}, \\
&c_{23} = \int_{P}x_{2}x_{3}dx_{1}dx_{2}dx_{3}dx_{4} = 0, 
&&c_{24} = \int_{P}x_{2}x_{4}dx_{1}dx_{2}dx_{3}dx_{4} = \frac{24}{5}, \\
&c_{33} = \int_{P}x_{3}^{2}dx_{1}dx_{2}dx_{3}dx_{4} = \frac{46}{9}, 
&&c_{34} = \int_{P}x_{3}x_{4}dx_{1}dx_{2}dx_{3}dx_{4} = 0, \\
&c_{44} = \int_{P}x_{4}^{2}dx_{1}dx_{2}dx_{3}dx_{4} = \frac{311}{30}, 
\end{align*}
\begin{align*}
&\theta_P(x_{1}, x_{2}, x_{3}, x_{4}) = \frac{17020{{x}_{2}}-6845}{19651}, \\
&M_{X} = \max_{P}\theta_P = \frac{27195}{19651} > 1. 
\end{align*}

\subsubsection{{$Q_{7}$}}
\begin{align*}
&b_{0} = \mathrm{vol}(P) = \frac{121}{8}, 
&&b_{1} = \int_{P}x_{1}dx_{1}dx_{2}dx_{3}dx_{4} = \frac{97}{30}, \\
&b_{2} = \int_{P}x_{2}dx_{1}dx_{2}dx_{3}dx_{4} = \frac{97}{15}, 
&&b_{3} = \int_{P}x_{3}dx_{1}dx_{2}dx_{3}dx_{4} = \frac{97}{60}, \\
&b_{4} = \int_{P}x_{4}dx_{1}dx_{2}dx_{3}dx_{4} = \frac{97}{60}, 
&&c_{11} = \int_{P}x_{1}^{2}dx_{1}dx_{2}dx_{3}dx_{4} = \frac{59}{12}, \\
&c_{12} = \int_{P}x_{1}x_{2}dx_{1}dx_{2}dx_{3}dx_{4} = \frac{226}{45}, 
&&c_{13} = \int_{P}x_{1}x_{3}dx_{1}dx_{2}dx_{3}dx_{4} = \frac{59}{24}, \\
&c_{14} = \int_{P}x_{1}x_{4}dx_{1}dx_{2}dx_{3}dx_{4} = \frac{19}{360}, 
&&c_{22} = \int_{P}x_{2}^{2}dx_{1}dx_{2}dx_{3}dx_{4} = \frac{452}{45}, \\
&c_{23} = \int_{P}x_{2}x_{3}dx_{1}dx_{2}dx_{3}dx_{4} = \frac{113}{45}, 
&&c_{24} = \int_{P}x_{2}x_{4}dx_{1}dx_{2}dx_{3}dx_{4} = \frac{113}{45}, \\
&c_{33} = \int_{P}x_{3}^{2}dx_{1}dx_{2}dx_{3}dx_{4} = \frac{931}{120}, 
&&c_{34} = \int_{P}x_{3}x_{4}dx_{1}dx_{2}dx_{3}dx_{4} = \frac{19}{720}, \\
&c_{44} = \int_{P}x_{4}^{2}dx_{1}dx_{2}dx_{3}dx_{4} = \frac{931}{120}, 
\end{align*}
\begin{align*}
&\theta_P(x_{1}, x_{2}, x_{3}, x_{4}) = \frac{176055{{x}_{2}}-75272}{198188}, \\
&M_{X} = \max_{P}\theta_P = \frac{138419}{99094} > 1. 
\end{align*}

\subsubsection{{$Q_{8}$}}
\begin{align*}
&b_{0} = \mathrm{vol}(P) = \frac{44}{3}, 
&&b_{1} = \int_{P}x_{1}dx_{1}dx_{2}dx_{3}dx_{4} = \frac{7}{2}, \\
&b_{2} = \int_{P}x_{2}dx_{1}dx_{2}dx_{3}dx_{4} = \frac{53}{12}, 
&&b_{3} = \int_{P}x_{3}dx_{1}dx_{2}dx_{3}dx_{4} = 0, \\
&b_{4} = \int_{P}x_{4}dx_{1}dx_{2}dx_{3}dx_{4} = \frac{7}{4}, 
&&c_{11} = \int_{P}x_{1}^{2}dx_{1}dx_{2}dx_{3}dx_{4} = \frac{71}{15}, \\
&c_{12} = \int_{P}x_{1}x_{2}dx_{1}dx_{2}dx_{3}dx_{4} = \frac{21}{5}, 
&&c_{13} = \int_{P}x_{1}x_{3}dx_{1}dx_{2}dx_{3}dx_{4} = 0, \\
&c_{14} = \int_{P}x_{1}x_{4}dx_{1}dx_{2}dx_{3}dx_{4} = \frac{71}{30}, 
&&c_{22} = \int_{P}x_{2}^{2}dx_{1}dx_{2}dx_{3}dx_{4} = \frac{83}{10}, \\
&c_{23} = \int_{P}x_{2}x_{3}dx_{1}dx_{2}dx_{3}dx_{4} = 0, 
&&c_{24} = \int_{P}x_{2}x_{4}dx_{1}dx_{2}dx_{3}dx_{4} = \frac{21}{10}, \\
&c_{33} = \int_{P}x_{3}^{2}dx_{1}dx_{2}dx_{3}dx_{4} = \frac{44}{9}, 
&&c_{34} = \int_{P}x_{3}x_{4}dx_{1}dx_{2}dx_{3}dx_{4} = 0, \\
&c_{44} = \int_{P}x_{4}^{2}dx_{1}dx_{2}dx_{3}dx_{4} = \frac{229}{30}, 
\end{align*}
\begin{align*}
&\theta_P(x_{1}, x_{2}, x_{3}, x_{4}) = \frac{1663200{{x}_{1}}+982960{{x}_{2}}-692905}{2735927}, \\
&M_{X} = \max_{P}\theta_P = \frac{2936215}{2735927} > 1. 
\end{align*}

\subsubsection{{$Q_{9}$}}
\begin{align*}
&b_{0} = \mathrm{vol}(P) = \frac{341}{24}, 
&&b_{1} = \int_{P}x_{1}dx_{1}dx_{2}dx_{3}dx_{4} = \frac{113}{30}, \\
&b_{2} = \int_{P}x_{2}dx_{1}dx_{2}dx_{3}dx_{4} = \frac{71}{30}, 
&&b_{3} = \int_{P}x_{3}dx_{1}dx_{2}dx_{3}dx_{4} = \frac{113}{60}, \\
&b_{4} = \int_{P}x_{4}dx_{1}dx_{2}dx_{3}dx_{4} = \frac{7}{10}, 
&&c_{11} = \int_{P}x_{1}^{2}dx_{1}dx_{2}dx_{3}dx_{4} = \frac{91}{20}, \\
&c_{12} = \int_{P}x_{1}x_{2}dx_{1}dx_{2}dx_{3}dx_{4} = \frac{152}{45}, 
&&c_{13} = \int_{P}x_{1}x_{3}dx_{1}dx_{2}dx_{3}dx_{4} = \frac{911}{40}, \\
&c_{14} = \int_{P}x_{1}x_{4}dx_{1}dx_{2}dx_{3}dx_{4} = \frac{211}{360}, 
&&c_{22} = \int_{P}x_{2}^{2}dx_{1}dx_{2}dx_{3}dx_{4} = \frac{59}{9}, \\
&c_{23} = \int_{P}x_{2}x_{3}dx_{1}dx_{2}dx_{3}dx_{4} = \frac{76}{45}, 
&&c_{24} = \int_{P}x_{2}x_{4}dx_{1}dx_{2}dx_{3}dx_{4} = -\frac{143}{90}, \\
&c_{33} = \int_{P}x_{3}^{2}dx_{1}dx_{2}dx_{3}dx_{4} = \frac{901}{120}, 
&&c_{34} = \int_{P}x_{3}x_{4}dx_{1}dx_{2}dx_{3}dx_{4} = \frac{211}{720}, \\
&c_{44} = \int_{P}x_{4}^{2}dx_{1}dx_{2}dx_{3}dx_{4} = \frac{479}{72}, 
\end{align*}
\begin{align*}
&\theta_P(x_{1}, x_{2}, x_{3}, x_{4}) = \frac{461229780{{x}_{1}}-53988825{{x}_{2}}-113280372}{395454118}, \\
&M_{X} = \max_{P}\theta_P = \frac{173974704}{197727059} < 1. 
\end{align*}

\subsubsection{{$Q_{10}$}}
\begin{align*}
&b_{0} = \mathrm{vol}(P) = 14, 
&&b_{1} = \int_{P}x_{1}dx_{1}dx_{2}dx_{3}dx_{4} = \frac{4}{3}, \\
&b_{2} = \int_{P}x_{2}dx_{1}dx_{2}dx_{3}dx_{4} = \frac{8}{3}, 
&&b_{3} = \int_{P}x_{3}dx_{1}dx_{2}dx_{3}dx_{4} = \frac{7}{3}, \\
&b_{4} = \int_{P}x_{4}dx_{1}dx_{2}dx_{3}dx_{4} = \frac{7}{6}, 
&&c_{11} = \int_{P}x_{1}^{2}dx_{1}dx_{2}dx_{3}dx_{4} = \frac{13}{3}, \\
&c_{12} = \int_{P}x_{1}x_{2}dx_{1}dx_{2}dx_{3}dx_{4} = \frac{7}{2}, 
&&c_{13} = \int_{P}x_{1}x_{3}dx_{1}dx_{2}dx_{3}dx_{4} = \frac{2}{9}, \\
&c_{14} = \int_{P}x_{1}x_{4}dx_{1}dx_{2}dx_{3}dx_{4} = \frac{1}{9}, 
&&c_{22} = \int_{P}x_{2}^{2}dx_{1}dx_{2}dx_{3}dx_{4} = 7, \\
&c_{23} = \int_{P}x_{2}x_{3}dx_{1}dx_{2}dx_{3}dx_{4} = \frac{4}{9}, 
&&c_{24} = \int_{P}x_{2}x_{4}dx_{1}dx_{2}dx_{3}dx_{4} = \frac{2}{9}, \\
&c_{33} = \int_{P}x_{3}^{2}dx_{1}dx_{2}dx_{3}dx_{4} = \frac{14}{3}, 
&&c_{34} = \int_{P}x_{3}x_{4}dx_{1}dx_{2}dx_{3}dx_{4} = \frac{7}{3}, \\
&c_{44} = \int_{P}x_{4}^{2}dx_{1}dx_{2}dx_{3}dx_{4} = 7, 
\end{align*}
\begin{align*}
&\theta_P(x_{1}, x_{2}, x_{3}, x_{4}) = \frac{1848{{x}_{2}}+2454{{x}_{3}}-761}{4499}, \\
&M_{X} = \max_{P}\theta_P = \frac{5389}{4499} > 1. 
\end{align*}

\subsubsection{{$Q_{11}$}}
\begin{align*}
&b_{0} = \mathrm{vol}(P) = 14, 
&&b_{1} = \int_{P}x_{1}dx_{1}dx_{2}dx_{3}dx_{4} = \frac{4}{3}, \\
&b_{2} = \int_{P}x_{2}dx_{1}dx_{2}dx_{3}dx_{4} = \frac{8}{3}, 
&&b_{3} = \int_{P}x_{3}dx_{1}dx_{2}dx_{3}dx_{4} = 0, \\
&b_{4} = \int_{P}x_{4}dx_{1}dx_{2}dx_{3}dx_{4} = 0, 
&&c_{11} = \int_{P}x_{1}^{2}dx_{1}dx_{2}dx_{3}dx_{4} = \frac{13}{3}, \\
&c_{12} = \int_{P}x_{1}x_{2}dx_{1}dx_{2}dx_{3}dx_{4} = \frac{7}{2}, 
&&c_{13} = \int_{P}x_{1}x_{3}dx_{1}dx_{2}dx_{3}dx_{4} = 0, \\
&c_{14} = \int_{P}x_{1}x_{4}dx_{1}dx_{2}dx_{3}dx_{4} = 0, 
&&c_{22} = \int_{P}x_{2}^{2}dx_{1}dx_{2}dx_{3}dx_{4} = 7, \\
&c_{23} = \int_{P}x_{2}x_{3}dx_{1}dx_{2}dx_{3}dx_{4} = 0, 
&&c_{24} = \int_{P}x_{2}x_{4}dx_{1}dx_{2}dx_{3}dx_{4} = 0, \\
&c_{33} = \int_{P}x_{3}^{2}dx_{1}dx_{2}dx_{3}dx_{4} = \frac{14}{3}, 
&&c_{34} = \int_{P}x_{3}x_{4}dx_{1}dx_{2}dx_{3}dx_{4} = 0, \\
&c_{44} = \int_{P}x_{4}^{2}dx_{1}dx_{2}dx_{3}dx_{4} = \frac{14}{3}, 
\end{align*}
\begin{align*}
&\theta_P(x_{1}, x_{2}, x_{3}, x_{4}) = \frac{168{{x}_{2}}-32}{409}, \\
&M_{X} = \max_{P}\theta_P = \frac{304}{409} < 1. 
\end{align*}

\subsubsection{{$Q_{12}$}}
\begin{align*}
&b_{0} = \mathrm{vol}(P) = \frac{331}{24}, 
&&b_{1} = \int_{P}x_{1}dx_{1}dx_{2}dx_{3}dx_{4} = \frac{16}{5}, \\
&b_{2} = \int_{P}x_{2}dx_{1}dx_{2}dx_{3}dx_{4} = \frac{113}{30}, 
&&b_{3} = \int_{P}x_{3}dx_{1}dx_{2}dx_{3}dx_{4} = \frac{113}{60}, \\
&b_{4} = \int_{P}x_{4}dx_{1}dx_{2}dx_{3}dx_{4} = -\frac{17}{60}, 
&&c_{11} = \int_{P}x_{1}^{2}dx_{1}dx_{2}dx_{3}dx_{4} = \frac{203}{45}, \\
&c_{12} = \int_{P}x_{1}x_{2}dx_{1}dx_{2}dx_{3}dx_{4} = \frac{169}{45}, 
&&c_{13} = \int_{P}x_{1}x_{3}dx_{1}dx_{2}dx_{3}dx_{4} = \frac{169}{90}, \\
&c_{14} = \int_{P}x_{1}x_{4}dx_{1}dx_{2}dx_{3}dx_{4} = \frac{17}{45}, 
&&c_{22} = \int_{P}x_{2}^{2}dx_{1}dx_{2}dx_{3}dx_{4} = \frac{427}{60}, \\
&c_{23} = \int_{P}x_{2}x_{3}dx_{1}dx_{2}dx_{3}dx_{4} = \frac{427}{120}, 
&&c_{24} = \int_{P}x_{2}x_{4}dx_{1}dx_{2}dx_{3}dx_{4} = -\frac{121}{72}, \\
&c_{33} = \int_{P}x_{3}^{2}dx_{1}dx_{2}dx_{3}dx_{4} = \frac{329}{40}, 
&&c_{34} = \int_{P}x_{3}x_{4}dx_{1}dx_{2}dx_{3}dx_{4} = -\frac{121}{144}, \\
&c_{44} = \int_{P}x_{4}^{2}dx_{1}dx_{2}dx_{3}dx_{4} = \frac{2081}{360}, 
\end{align*}
\begin{align*}
&\theta_P(x_{1}, x_{2}, x_{3}, x_{4}) = \frac{115654710{{x}_{1}}+66679950{{x}_{2}}-45045768}{196251577}, \\
&M_{X} = \max_{P}\theta_P = \frac{203968842}{196251577} > 1. 
\end{align*}

\subsubsection{{$Q_{13}$}}
\begin{align*}
&b_{0} = \mathrm{vol}(P) = \frac{55}{4}, 
&&b_{1} = \int_{P}x_{1}dx_{1}dx_{2}dx_{3}dx_{4} = \frac{139}{60}, \\
&b_{2} = \int_{P}x_{2}dx_{1}dx_{2}dx_{3}dx_{4} = -\frac{29}{60}, 
&&b_{3} = \int_{P}x_{3}dx_{1}dx_{2}dx_{3}dx_{4} = \frac{7}{5}, \\
&b_{4} = \int_{P}x_{4}dx_{1}dx_{2}dx_{3}dx_{4} = \frac{7}{5}, 
&&c_{11} = \int_{P}x_{1}^{2}dx_{1}dx_{2}dx_{3}dx_{4} = \frac{733}{180}, \\
&c_{12} = \int_{P}x_{1}x_{2}dx_{1}dx_{2}dx_{3}dx_{4} = \frac{227}{90}, 
&&c_{13} = \int_{P}x_{1}x_{3}dx_{1}dx_{2}dx_{3}dx_{4} = \frac{31}{40}, \\
&c_{14} = \int_{P}x_{1}x_{4}dx_{1}dx_{2}dx_{3}dx_{4} = \frac{31}{40}, 
&&c_{22} = \int_{P}x_{2}^{2}dx_{1}dx_{2}dx_{3}dx_{4} = \frac{185}{36}, \\
&c_{23} = \int_{P}x_{2}x_{3}dx_{1}dx_{2}dx_{3}dx_{4} = -\frac{157}{120}, 
&&c_{24} = \int_{P}x_{2}x_{4}dx_{1}dx_{2}dx_{3}dx_{4} = -\frac{157}{120}, \\
&c_{33} = \int_{P}x_{3}^{2}dx_{1}dx_{2}dx_{3}dx_{4} = \frac{1243}{180}, 
&&c_{34} = \int_{P}x_{3}x_{4}dx_{1}dx_{2}dx_{3}dx_{4} = \frac{25}{24}, \\
&c_{44} = \int_{P}x_{4}^{2}dx_{1}dx_{2}dx_{3}dx_{4} = \frac{1243}{180}, 
\end{align*}
\begin{align*}
&\theta_P(x_{1}, x_{2}, x_{3}, x_{4}) = \frac{38978775{{x}_{1}}-23199825{{x}_{2}}-7382842}{35875483}, \\
&M_{X} = \max_{P}\theta_P = \frac{31595933}{35875483} < 1. 
\end{align*}

\subsubsection{{$Q_{14}$}}
\begin{align*}
&b_{0} = \mathrm{vol}(P) = \frac{325}{24}, 
&&b_{1} = \int_{P}x_{1}dx_{1}dx_{2}dx_{3}dx_{4} = \frac{61}{30}, \\
&b_{2} = \int_{P}x_{2}dx_{1}dx_{2}dx_{3}dx_{4} = \frac{13}{60}, 
&&b_{3} = \int_{P}x_{3}dx_{1}dx_{2}dx_{3}dx_{4} = \frac{103}{30}, \\
&b_{4} = \int_{P}x_{4}dx_{1}dx_{2}dx_{3}dx_{4} = \frac{103}{60}, 
&&c_{11} = \int_{P}x_{1}^{2}dx_{1}dx_{2}dx_{3}dx_{4} = \frac{1459}{360}, \\
&c_{12} = \int_{P}x_{1}x_{2}dx_{1}dx_{2}dx_{3}dx_{4} = \frac{122}{45}, 
&&c_{13} = \int_{P}x_{1}x_{3}dx_{1}dx_{2}dx_{3}dx_{4} = \frac{427}{360}, \\
&c_{14} = \int_{P}x_{1}x_{4}dx_{1}dx_{2}dx_{3}dx_{4} = \frac{427}{720}, 
&&c_{22} = \int_{P}x_{2}^{2}dx_{1}dx_{2}dx_{3}dx_{4} = \frac{1951}{360}, \\
&c_{23} = \int_{P}x_{2}x_{3}dx_{1}dx_{2}dx_{3}dx_{4} = -\frac{119}{72}, 
&&c_{24} = \int_{P}x_{2}x_{4}dx_{1}dx_{2}dx_{3}dx_{4} = -\frac{119}{144}, \\
&c_{33} = \int_{P}x_{3}^{2}dx_{1}dx_{2}dx_{3}dx_{4} = \frac{437}{60}, 
&&c_{34} = \int_{P}x_{3}x_{4}dx_{1}dx_{2}dx_{3}dx_{4} = \frac{437}{120}, \\
&c_{44} = \int_{P}x_{4}^{2}dx_{1}dx_{2}dx_{3}dx_{4} = \frac{2911}{360}, 
\end{align*}
\begin{align*}
&\theta_P(x_{1}, x_{2}, x_{3}, x_{4}) = \frac{1233812146125{{x}_{1}}-151909753125{{x}_{2}}+1142554261875{{x}_{3}}}{2450627729348}- \frac{472512532902}{2450627729348}, \ \\
&M_{X} = \max_{P}\theta_P = \frac{3046408136973}{2450627729348} > 1. 
\end{align*}

\subsubsection{{$Q_{15}$}}
\begin{align*}
&b_{0} = \mathrm{vol}(P) = \frac{40}{3}, 
&&b_{1} = \int_{P}x_{1}dx_{1}dx_{2}dx_{3}dx_{4} = \frac{7}{4}, \\
&b_{2} = \int_{P}x_{2}dx_{1}dx_{2}dx_{3}dx_{4} = \frac{11}{12}, 
&&b_{3} = \int_{P}x_{3}dx_{1}dx_{2}dx_{3}dx_{4} = 0, \\
&b_{4} = \int_{P}x_{4}dx_{1}dx_{2}dx_{3}dx_{4} = \frac{5}{12}, 
&&c_{11} = \int_{P}x_{1}^{2}dx_{1}dx_{2}dx_{3}dx_{4} = \frac{121}{30}, \\
&c_{12} = \int_{P}x_{1}x_{2}dx_{1}dx_{2}dx_{3}dx_{4} = \frac{29}{10}, 
&&c_{13} = \int_{P}x_{1}x_{3}dx_{1}dx_{2}dx_{3}dx_{4} = 0, \\
&c_{14} = \int_{P}x_{1}x_{4}dx_{1}dx_{2}dx_{3}dx_{4} = \frac{17}{30}, 
&&c_{22} = \int_{P}x_{2}^{2}dx_{1}dx_{2}dx_{3}dx_{4} = \frac{57}{10}, \\
&c_{23} = \int_{P}x_{2}x_{3}dx_{1}dx_{2}dx_{3}dx_{4} = 0, 
&&c_{24} = \int_{P}x_{2}x_{4}dx_{1}dx_{2}dx_{3}dx_{4} = -\frac{7}{5}, \\
&c_{33} = \int_{P}x_{3}^{2}dx_{1}dx_{2}dx_{3}dx_{4} = \frac{40}{9}, 
&&c_{34} = \int_{P}x_{3}x_{4}dx_{1}dx_{2}dx_{3}dx_{4} = 0, \\
&c_{44} = \int_{P}x_{4}^{2}dx_{1}dx_{2}dx_{3}dx_{4} = \frac{181}{30}, 
\end{align*}
\begin{align*}
&\theta_P(x_{1}, x_{2}, x_{3}, x_{4}) = \frac{210720{{x}_{1}}-39680{{x}_{2}}-24929}{394975}, \\
&M_{X} = \max_{P}\theta_P = \frac{185791}{394975} < 1. 
\end{align*}

\subsubsection{{$Q_{16}$}}
\begin{align*}
&b_{0} = \mathrm{vol}(P) = \frac{155}{12}, 
&&b_{1} = \int_{P}x_{1}dx_{1}dx_{2}dx_{3}dx_{4} = \frac{17}{15}, \\
&b_{2} = \int_{P}x_{2}dx_{1}dx_{2}dx_{3}dx_{4} = \frac{139}{60}, 
&&b_{3} = \int_{P}x_{3}dx_{1}dx_{2}dx_{3}dx_{4} = \frac{17}{30}, \\
&b_{4} = \int_{P}x_{4}dx_{1}dx_{2}dx_{3}dx_{4} = -\frac{17}{30}, 
&&c_{11} = \int_{P}x_{1}^{2}dx_{1}dx_{2}dx_{3}dx_{4} = \frac{37}{10}, \\
&c_{12} = \int_{P}x_{1}x_{2}dx_{1}dx_{2}dx_{3}dx_{4} = \frac{179}{80}, 
&&c_{13} = \int_{P}x_{1}x_{3}dx_{1}dx_{2}dx_{3}dx_{4} = \frac{37}{20}, \\
&c_{14} = \int_{P}x_{1}x_{4}dx_{1}dx_{2}dx_{3}dx_{4} = -\frac{37}{20}, 
&&c_{22} = \int_{P}x_{2}^{2}dx_{1}dx_{2}dx_{3}dx_{4} = \frac{1127}{180}, \\
&c_{23} = \int_{P}x_{2}x_{3}dx_{1}dx_{2}dx_{3}dx_{4} = \frac{179}{120}, 
&&c_{24} = \int_{P}x_{2}x_{4}dx_{1}dx_{2}dx_{3}dx_{4} = -\frac{179}{120}, \\
&c_{33} = \int_{P}x_{3}^{2}dx_{1}dx_{2}dx_{3}dx_{4} = \frac{71}{12}, 
&&c_{34} = \int_{P}x_{3}x_{4}dx_{1}dx_{2}dx_{3}dx_{4} = -\frac{37}{40}, \\
&c_{44} = \int_{P}x_{4}^{2}dx_{1}dx_{2}dx_{3}dx_{4} = \frac{929}{180}, 
\end{align*}
\begin{align*}
&\theta_P(x_{1}, x_{2}, x_{3}, x_{4}) = \frac{1544575{{x}_{1}}+43444950{{x}_{2}}-7927586}{111477439}, \\
&M_{X} = \max_{P}\theta_P = \frac{80506889}{111477439} < 1. 
\end{align*}

\subsubsection{{$Q_{17}$}}
\begin{align*}
&b_{0} = \mathrm{vol}(P) = \frac{299}{24}, 
&&b_{1} = \int_{P}x_{1}dx_{1}dx_{2}dx_{3}dx_{4} = -\frac{4}{15}, \\
&b_{2} = \int_{P}x_{2}dx_{1}dx_{2}dx_{3}dx_{4} = -\frac{8}{15}, 
&&b_{3} = \int_{P}x_{3}dx_{1}dx_{2}dx_{3}dx_{4} = \frac{2}{15}, \\
&b_{4} = \int_{P}x_{4}dx_{1}dx_{2}dx_{3}dx_{4} = \frac{2}{15}, 
&&c_{11} = \int_{P}x_{1}^{2}dx_{1}dx_{2}dx_{3}dx_{4} = \frac{211}{60}, \\
&c_{12} = \int_{P}x_{1}x_{2}dx_{1}dx_{2}dx_{3}dx_{4} = \frac{109}{45}, 
&&c_{13} = \int_{P}x_{1}x_{3}dx_{1}dx_{2}dx_{3}dx_{4} = \frac{197}{360}, \\
&c_{14} = \int_{P}x_{1}x_{4}dx_{1}dx_{2}dx_{3}dx_{4} = -\frac{211}{120}, 
&&c_{22} = \int_{P}x_{2}^{2}dx_{1}dx_{2}dx_{3}dx_{4} = \frac{218}{45}, \\
&c_{23} = \int_{P}x_{2}x_{3}dx_{1}dx_{2}dx_{3}dx_{4} = -\frac{109}{90}, 
&&c_{24} = \int_{P}x_{2}x_{4}dx_{1}dx_{2}dx_{3}dx_{4} = -\frac{109}{90}, \\
&c_{33} = \int_{P}x_{3}^{2}dx_{1}dx_{2}dx_{3}dx_{4} = \frac{1949}{360}, 
&&c_{34} = \int_{P}x_{3}x_{4}dx_{1}dx_{2}dx_{3}dx_{4} = -\frac{197}{720}, \\
&c_{44} = \int_{P}x_{4}^{2}dx_{1}dx_{2}dx_{3}dx_{4} = \frac{1949}{360}, 
\end{align*}
\begin{align*}
&\theta_P(x_{1}, x_{2}, x_{3}, x_{4}) = \frac{-17940{{x}_{2}}-768}{162187}, \\
&M_{X} = \max_{P}\theta_P = \frac{17172}{162187} < 1. 
\end{align*}

\subsubsection{{$K_{1}$}}
\begin{align*}
&b_{0} = \mathrm{vol}(P) = \frac{91}{6}, 
&&b_{1} = \int_{P}x_{1}dx_{1}dx_{2}dx_{3}dx_{4} = 5, \\
&b_{2} = \int_{P}x_{2}dx_{1}dx_{2}dx_{3}dx_{4} = \frac{5}{2}, 
&&b_{3} = \int_{P}x_{3}dx_{1}dx_{2}dx_{3}dx_{4} = \frac{10}{3}, \\
&b_{4} = \int_{P}x_{4}dx_{1}dx_{2}dx_{3}dx_{4} = \frac{10}{3}, 
&&c_{11} = \int_{P}x_{1}^{2}dx_{1}dx_{2}dx_{3}dx_{4} = \frac{281}{60}, \\
&c_{12} = \int_{P}x_{1}x_{2}dx_{1}dx_{2}dx_{3}dx_{4} = \frac{281}{120}, 
&&c_{13} = \int_{P}x_{1}x_{3}dx_{1}dx_{2}dx_{3}dx_{4} = \frac{281}{90}, \\
&c_{14} = \int_{P}x_{1}x_{4}dx_{1}dx_{2}dx_{3}dx_{4} = \frac{281}{90}, 
&&c_{22} = \int_{P}x_{2}^{2}dx_{1}dx_{2}dx_{3}dx_{4} = \frac{253}{60}, \\
&c_{23} = \int_{P}x_{2}x_{3}dx_{1}dx_{2}dx_{3}dx_{4} = \frac{281}{180}, 
&&c_{24} = \int_{P}x_{2}x_{4}dx_{1}dx_{2}dx_{3}dx_{4} = \frac{281}{180}, \\
&c_{33} = \int_{P}x_{3}^{2}dx_{1}dx_{2}dx_{3}dx_{4} = \frac{2527}{180}, 
&&c_{34} = \int_{P}x_{3}x_{4}dx_{1}dx_{2}dx_{3}dx_{4} = -\frac{1403}{360}, \\
&c_{44} = \int_{P}x_{4}^{2}dx_{1}dx_{2}dx_{3}dx_{4} = \frac{2527}{180}, 
\end{align*}
\begin{align*}
&\theta_P(x_{1}, x_{2}, x_{3}, x_{4}) = \frac{27300{{x}_{1}}-9000}{16571}, \\
&M_{X} = \max_{P}\theta_P = \frac{18300}{16571} > 1. 
\end{align*}

\subsubsection{{$K_{2}$}}
\begin{align*}
&b_{0} = \mathrm{vol}(P) = \frac{59}{4}, 
&&b_{1} = \int_{P}x_{1}dx_{1}dx_{2}dx_{3}dx_{4} = \frac{15}{4}, \\
&b_{2} = \int_{P}x_{2}dx_{1}dx_{2}dx_{3}dx_{4} = \frac{15}{4}, 
&&b_{3} = \int_{P}x_{3}dx_{1}dx_{2}dx_{3}dx_{4} = \frac{5}{2}, \\
&b_{4} = \int_{P}x_{4}dx_{1}dx_{2}dx_{3}dx_{4} = \frac{5}{2}, 
&&c_{11} = \int_{P}x_{1}^{2}dx_{1}dx_{2}dx_{3}dx_{4} = \frac{13}{3}, \\
&c_{12} = \int_{P}x_{1}x_{2}dx_{1}dx_{2}dx_{3}dx_{4} = \frac{281}{120}, 
&&c_{13} = \int_{P}x_{1}x_{3}dx_{1}dx_{2}dx_{3}dx_{4} = \frac{89}{40}, \\
&c_{14} = \int_{P}x_{1}x_{4}dx_{1}dx_{2}dx_{3}dx_{4} = \frac{89}{40}, 
&&c_{22} = \int_{P}x_{2}^{2}dx_{1}dx_{2}dx_{3}dx_{4} = \frac{13}{3}, \\
&c_{23} = \int_{P}x_{2}x_{3}dx_{1}dx_{2}dx_{3}dx_{4} = \frac{89}{40}, 
&&c_{24} = \int_{P}x_{2}x_{4}dx_{1}dx_{2}dx_{3}dx_{4} = \frac{89}{40}, \\
&c_{33} = \int_{P}x_{3}^{2}dx_{1}dx_{2}dx_{3}dx_{4} = \frac{121}{10}, 
&&c_{34} = \int_{P}x_{3}x_{4}dx_{1}dx_{2}dx_{3}dx_{4} = -\frac{153}{40}, \\
&c_{44} = \int_{P}x_{4}^{2}dx_{1}dx_{2}dx_{3}dx_{4} = \frac{121}{10}, 
\end{align*}
\begin{align*}
&\theta_P(x_{1}, x_{2}, x_{3}, x_{4}) = \frac{2950{{x}_{1}}+2950{{x}_{2}}-1500}{3751}, \\
&M_{X} = \max_{P}\theta_P = \frac{400}{341} > 1. 
\end{align*}

\subsubsection{{$K_{3}$}}
\begin{align*}
&b_{0} = \mathrm{vol}(P) = \frac{167}{12}, 
&&b_{1} = \int_{P}x_{1}dx_{1}dx_{2}dx_{3}dx_{4} = \frac{5}{2}, \\
&b_{2} = \int_{P}x_{2}dx_{1}dx_{2}dx_{3}dx_{4} = \frac{5}{4}, 
&&b_{3} = \int_{P}x_{3}dx_{1}dx_{2}dx_{3}dx_{4} = \frac{5}{6}, \\
&b_{4} = \int_{P}x_{4}dx_{1}dx_{2}dx_{3}dx_{4} = \frac{5}{6}, 
&&c_{11} = \int_{P}x_{1}^{2}dx_{1}dx_{2}dx_{3}dx_{4} = \frac{239}{60}, \\
&c_{12} = \int_{P}x_{1}x_{2}dx_{1}dx_{2}dx_{3}dx_{4} = \frac{239}{120}, 
&&c_{13} = \int_{P}x_{1}x_{3}dx_{1}dx_{2}dx_{3}dx_{4} = \frac{239}{180}, \\
&c_{14} = \int_{P}x_{1}x_{4}dx_{1}dx_{2}dx_{3}dx_{4} = \frac{239}{180}, 
&&c_{22} = \int_{P}x_{2}^{2}dx_{1}dx_{2}dx_{3}dx_{4} = \frac{58}{15}, \\
&c_{23} = \int_{P}x_{2}x_{3}dx_{1}dx_{2}dx_{3}dx_{4} = \frac{239}{180}, 
&&c_{24} = \int_{P}x_{2}x_{4}dx_{1}dx_{2}dx_{3}dx_{4} = \frac{239}{180}, \\
&c_{33} = \int_{P}x_{3}^{2}dx_{1}dx_{2}dx_{3}dx_{4} = \frac{761}{90}, 
&&c_{34} = \int_{P}x_{3}x_{4}dx_{1}dx_{2}dx_{3}dx_{4} = -\frac{1283}{360}, \\
&c_{44} = \int_{P}x_{4}^{2}dx_{1}dx_{2}dx_{3}dx_{4} = \frac{761}{90}, 
\end{align*}
\begin{align*}
&\theta_P(x_{1}, x_{2}, x_{3}, x_{4}) = \frac{25050{{x}_{1}}-4500}{35413}, \\
&M_{X} = \max_{P}\theta_P = \frac{20550}{35413} < 1. 
\end{align*}

\subsubsection{{$K_{4}$}}
\begin{align*}
&b_{0} = \mathrm{vol}(P) = \frac{27}{2}, 
&&b_{1} = \int_{P}x_{1}dx_{1}dx_{2}dx_{3}dx_{4} = 0, \\
&b_{2} = \int_{P}x_{2}dx_{1}dx_{2}dx_{3}dx_{4} = 0, 
&&b_{3} = \int_{P}x_{3}dx_{1}dx_{2}dx_{3}dx_{4} = 0, \\
&b_{4} = \int_{P}x_{4}dx_{1}dx_{2}dx_{3}dx_{4} = 0, 
&&c_{11} = \int_{P}x_{1}^{2}dx_{1}dx_{2}dx_{3}dx_{4} = \frac{15}{4}, \\
&c_{12} = \int_{P}x_{1}x_{2}dx_{1}dx_{2}dx_{3}dx_{4} = \frac{15}{8}, 
&&c_{13} = \int_{P}x_{1}x_{3}dx_{1}dx_{2}dx_{3}dx_{4} = 0, \\
&c_{14} = \int_{P}x_{1}x_{4}dx_{1}dx_{2}dx_{3}dx_{4} = 0, 
&&c_{22} = \int_{P}x_{2}^{2}dx_{1}dx_{2}dx_{3}dx_{4} = \frac{15}{4}, \\
&c_{23} = \int_{P}x_{2}x_{3}dx_{1}dx_{2}dx_{3}dx_{4} = 0, 
&&c_{24} = \int_{P}x_{2}x_{4}dx_{1}dx_{2}dx_{3}dx_{4} = 0, \\
&c_{33} = \int_{P}x_{3}^{2}dx_{1}dx_{2}dx_{3}dx_{4} = \frac{27}{4}, 
&&c_{34} = \int_{P}x_{3}x_{4}dx_{1}dx_{2}dx_{3}dx_{4} = -\frac{27}{8}, \\
&c_{44} = \int_{P}x_{4}^{2}dx_{1}dx_{2}dx_{3}dx_{4} = \frac{27}{4}, 
\end{align*}
\begin{align*}
&\theta_P(x_{1}, x_{2}, x_{3}, x_{4}) = 0, \\
&M_{X} = \max_{P}\theta_P = 0 < 1. 
\end{align*}

\subsubsection{{$R_{1}$}}
\begin{align*}
&b_{0} = \mathrm{vol}(P) = \frac{83}{6}, 
&&b_{1} = \int_{P}x_{1}dx_{1}dx_{2}dx_{3}dx_{4} = 1, \\
&b_{2} = \int_{P}x_{2}dx_{1}dx_{2}dx_{3}dx_{4} = 1, 
&&b_{3} = \int_{P}x_{3}dx_{1}dx_{2}dx_{3}dx_{4} = \frac{35}{8}, \\
&b_{4} = \int_{P}x_{4}dx_{1}dx_{2}dx_{3}dx_{4} = -\frac{23}{12}, 
&&c_{11} = \int_{P}x_{1}^{2}dx_{1}dx_{2}dx_{3}dx_{4} = \frac{2551}{360}, \\
&c_{12} = \int_{P}x_{1}x_{2}dx_{1}dx_{2}dx_{3}dx_{4} = -\frac{3149}{720}, 
&&c_{13} = \int_{P}x_{1}x_{3}dx_{1}dx_{2}dx_{3}dx_{4} = \frac{1709}{720}, \\
&c_{14} = \int_{P}x_{1}x_{4}dx_{1}dx_{2}dx_{3}dx_{4} = \frac{61}{360}, 
&&c_{22} = \int_{P}x_{2}^{2}dx_{1}dx_{2}dx_{3}dx_{4} = \frac{2551}{360}, \\
&c_{23} = \int_{P}x_{2}x_{3}dx_{1}dx_{2}dx_{3}dx_{4} = \frac{1709}{720}, 
&&c_{24} = \int_{P}x_{2}x_{4}dx_{1}dx_{2}dx_{3}dx_{4} = \frac{61}{360}, \\
&c_{33} = \int_{P}x_{3}^{2}dx_{1}dx_{2}dx_{3}dx_{4} = \frac{2033}{180}, 
&&c_{34} = \int_{P}x_{3}x_{4}dx_{1}dx_{2}dx_{3}dx_{4} = -\frac{2357}{720}, \\
&c_{44} = \int_{P}x_{4}^{2}dx_{1}dx_{2}dx_{3}dx_{4} = \frac{521}{120}, 
\end{align*}
\begin{align*}
&\theta_P(x_{1}, x_{2}, x_{3}, x_{4}) = \frac{-54894540{{x}_{1}}-54894540{{x}_{2}}+238278060{{x}_{3}}-92542140}{324135871}, \\
&M_{X} = \max_{P}\theta_P = \frac{567397500}{324135871} > 1. 
\end{align*}

\newpage
\subsubsection{{$R_{2}$}}
\begin{align*}
&b_{0} = \mathrm{vol}(P) = \frac{107}{8}, 
&&b_{1} = \int_{P}x_{1}dx_{1}dx_{2}dx_{3}dx_{4} = \frac{31}{20}, \\
&b_{2} = \int_{P}x_{2}dx_{1}dx_{2}dx_{3}dx_{4} = \frac{31}{20}, 
&&b_{3} = \int_{P}x_{3}dx_{1}dx_{2}dx_{3}dx_{4} = \frac{269}{60}, \\
&b_{4} = \int_{P}x_{4}dx_{1}dx_{2}dx_{3}dx_{4} = -\frac{157}{60}, 
&&c_{11} = \int_{P}x_{1}^{2}dx_{1}dx_{2}dx_{3}dx_{4} = \frac{1471}{180}, \\
&c_{12} = \int_{P}x_{1}x_{2}dx_{1}dx_{2}dx_{3}dx_{4} = -\frac{2989}{720}, 
&&c_{13} = \int_{P}x_{1}x_{3}dx_{1}dx_{2}dx_{3}dx_{4} = \frac{79}{30}, \\
&c_{14} = \int_{P}x_{1}x_{4}dx_{1}dx_{2}dx_{3}dx_{4} = -\frac{5}{16}, 
&&c_{22} = \int_{P}x_{2}^{2}dx_{1}dx_{2}dx_{3}dx_{4} = \frac{1471}{180}, \\
&c_{23} = \int_{P}x_{2}x_{3}dx_{1}dx_{2}dx_{3}dx_{4} = \frac{79}{30}, 
&&c_{24} = \int_{P}x_{2}x_{4}dx_{1}dx_{2}dx_{3}dx_{4} = -\frac{5}{16}, \\
&c_{33} = \int_{P}x_{3}^{2}dx_{1}dx_{2}dx_{3}dx_{4} = \frac{2929}{360}, 
&&c_{34} = \int_{P}x_{3}x_{4}dx_{1}dx_{2}dx_{3}dx_{4} = -\frac{2483}{720}, \\
&c_{44} = \int_{P}x_{4}^{2}dx_{1}dx_{2}dx_{3}dx_{4} = \frac{731}{180}, 
\end{align*}
\begin{align*}
&\theta_P(x_{1}, x_{2}, x_{3}, x_{4}) = \frac{451036894560{{x}_{1}}+451036894560{{x}_{2}}+701882402340{{x}_{3}}}{2028544442569} \\
&\quad \quad \quad \quad \quad \quad \quad \quad + \frac{-984670344060{{x}_{4}}-532451679296}{2028544442569}, \\
&M_{X} = \max_{P}\theta_P = \frac{2758057258564}{2028544442569} > 1. 
\end{align*}

\newpage
\subsubsection{{$R_{3}$}}
\begin{align*}
&b_{0} = \mathrm{vol}(P) = \frac{305}{24}, 
&&b_{1} = \int_{P}x_{1}dx_{1}dx_{2}dx_{3}dx_{4} = \frac{23}{15}, \\
&b_{2} = \int_{P}x_{2}dx_{1}dx_{2}dx_{3}dx_{4} = \frac{23}{15}, 
&&b_{3} = \int_{P}x_{3}dx_{1}dx_{2}dx_{3}dx_{4} = \frac{23}{10}, \\
&b_{4} = \int_{P}x_{4}dx_{1}dx_{2}dx_{3}dx_{4} = \frac{23}{60}, 
&&c_{11} = \int_{P}x_{1}^{2}dx_{1}dx_{2}dx_{3}dx_{4} = \frac{95}{12}, \\
&c_{12} = \int_{P}x_{1}x_{2}dx_{1}dx_{2}dx_{3}dx_{4} = -\frac{943}{240}, 
&&c_{13} = \int_{P}x_{1}x_{3}dx_{1}dx_{2}dx_{3}dx_{4} = \frac{463}{240}, \\
&c_{14} = \int_{P}x_{1}x_{4}dx_{1}dx_{2}dx_{3}dx_{4} = \frac{247}{240}, 
&&c_{22} = \int_{P}x_{2}^{2}dx_{1}dx_{2}dx_{3}dx_{4} = \frac{95}{12}, \\
&c_{23} = \int_{P}x_{2}x_{3}dx_{1}dx_{2}dx_{3}dx_{4} = \frac{463}{240}, 
&&c_{24} = \int_{P}x_{2}x_{4}dx_{1}dx_{2}dx_{3}dx_{4} = \frac{247}{240}, \\
&c_{33} = \int_{P}x_{3}^{2}dx_{1}dx_{2}dx_{3}dx_{4} = \frac{17}{4}, 
&&c_{34} = \int_{P}x_{3}x_{4}dx_{1}dx_{2}dx_{3}dx_{4} = -\frac{47}{240}, \\
&c_{44} = \int_{P}x_{4}^{2}dx_{1}dx_{2}dx_{3}dx_{4} = \frac{659}{180}, 
\end{align*}
\begin{align*}
&\theta_P(x_{1}, x_{2}, x_{3}, x_{4}) = \frac{91335300{{x}_{1}}+91335300{{x}_{2}}+142965700{{x}_{3}}-47914704}{369479321}, \\
&M_{X} = \max_{P}\theta_P = \frac{277721596}{369479321} < 1. 
\end{align*}

\subsubsection{}
\begin{align*}
&b_{0} = \mathrm{vol}(P) = \frac{155}{12}, 
&&b_{1} = \int_{P}x_{1}dx_{1}dx_{2}dx_{3}dx_{4} = \frac{19}{12}, \\
&b_{2} = \int_{P}x_{2}dx_{1}dx_{2}dx_{3}dx_{4} = \frac{7}{3}, 
&&b_{3} = \int_{P}x_{3}dx_{1}dx_{2}dx_{3}dx_{4} = -\frac{3}{2}, \\
&b_{4} = \int_{P}x_{4}dx_{1}dx_{2}dx_{3}dx_{4} = -\frac{19}{12}, 
&&c_{11} = \int_{P}x_{1}^{2}dx_{1}dx_{2}dx_{3}dx_{4} = \frac{1417}{360}, \\
&c_{12} = \int_{P}x_{1}x_{2}dx_{1}dx_{2}dx_{3}dx_{4} = \frac{2093}{720}, 
&&c_{13} = \int_{P}x_{1}x_{3}dx_{1}dx_{2}dx_{3}dx_{4} = \frac{253}{720}, \\
&c_{14} = \int_{P}x_{1}x_{4}dx_{1}dx_{2}dx_{3}dx_{4} = -\frac{1}{9}, 
&&c_{22} = \int_{P}x_{2}^{2}dx_{1}dx_{2}dx_{3}dx_{4} = \frac{2749}{360}, \\
&c_{23} = \int_{P}x_{2}x_{3}dx_{1}dx_{2}dx_{3}dx_{4} = -\frac{1741}{720}, 
&&c_{24} = \int_{P}x_{2}x_{4}dx_{1}dx_{2}dx_{3}dx_{4} = -\frac{71}{90}, \\
&c_{33} = \int_{P}x_{3}^{2}dx_{1}dx_{2}dx_{3}dx_{4} = \frac{997}{180}, 
&&c_{34} = \int_{P}x_{3}x_{4}dx_{1}dx_{2}dx_{3}dx_{4} = -\frac{253}{720}, \\
&c_{44} = \int_{P}x_{4}^{2}dx_{1}dx_{2}dx_{3}dx_{4} = \frac{1417}{360}, 
\end{align*}
\begin{align*}
&\theta_P(x_{1}, x_{2}, x_{3}, x_{4}) = \frac{74703180{{x}_{1}}-57721380{{x}_{3}}-74703180{{x}_{4}}-25017456}{153124261}, \\
&M_{X} = \max_{P}\theta_P = \frac{182110284}{153124261} > 1. 
\end{align*}

\subsubsection{{$U_{1}$}}
\begin{align*}
&b_{0} = \mathrm{vol}(P) = \frac{77}{6}, 
&&b_{1} = \int_{P}x_{1}dx_{1}dx_{2}dx_{3}dx_{4} = \frac{10}{3}, \\
&b_{2} = \int_{P}x_{2}dx_{1}dx_{2}dx_{3}dx_{4} = \frac{5}{3}, 
&&b_{3} = \int_{P}x_{3}dx_{1}dx_{2}dx_{3}dx_{4} = \frac{5}{3}, \\
&b_{4} = \int_{P}x_{4}dx_{1}dx_{2}dx_{3}dx_{4} = \frac{5}{3}, 
&&c_{11} = \int_{P}x_{1}^{2}dx_{1}dx_{2}dx_{3}dx_{4} = \frac{19}{5}, \\
&c_{12} = \int_{P}x_{1}x_{2}dx_{1}dx_{2}dx_{3}dx_{4} = \frac{19}{10}, 
&&c_{13} = \int_{P}x_{1}x_{3}dx_{1}dx_{2}dx_{3}dx_{4} = \frac{19}{10}, \\
&c_{14} = \int_{P}x_{1}x_{4}dx_{1}dx_{2}dx_{3}dx_{4} = \frac{19}{10}, 
&&c_{22} = \int_{P}x_{2}^{2}dx_{1}dx_{2}dx_{3}dx_{4} = \frac{107}{30}, \\
&c_{23} = \int_{P}x_{2}x_{3}dx_{1}dx_{2}dx_{3}dx_{4} = \frac{19}{20}, 
&&c_{24} = \int_{P}x_{2}x_{4}dx_{1}dx_{2}dx_{3}dx_{4} = \frac{19}{20}, \\
&c_{33} = \int_{P}x_{3}^{2}dx_{1}dx_{2}dx_{3}dx_{4} = \frac{599}{90}, 
&&c_{34} = \int_{P}x_{3}x_{4}dx_{1}dx_{2}dx_{3}dx_{4} = \frac{19}{20}, \\
&c_{44} = \int_{P}x_{4}^{2}dx_{1}dx_{2}dx_{3}dx_{4} = \frac{599}{90}, 
\end{align*}
\begin{align*}
&\theta_P(x_{1}, x_{2}, x_{3}, x_{4}) = \frac{3850{{x}_{1}}-1000}{3389}, \\
&M_{X} = \max_{P}\theta_P = \frac{2850}{3389} < 1. 
\end{align*}

\subsubsection{{$U_{2}$}}
\begin{align*}
&b_{0} = \mathrm{vol}(P) = \frac{149}{12}, 
&&b_{1} = \int_{P}x_{1}dx_{1}dx_{2}dx_{3}dx_{4} = \frac{5}{2}, \\
&b_{2} = \int_{P}x_{2}dx_{1}dx_{2}dx_{3}dx_{4} = \frac{5}{4}, 
&&b_{3} = \int_{P}x_{3}dx_{1}dx_{2}dx_{3}dx_{4} = -\frac{7}{6}, \\
&b_{4} = \int_{P}x_{4}dx_{1}dx_{2}dx_{3}dx_{4} = \frac{11}{6}, 
&&c_{11} = \int_{P}x_{1}^{2}dx_{1}dx_{2}dx_{3}dx_{4} = \frac{107}{30}, \\
&c_{12} = \int_{P}x_{1}x_{2}dx_{1}dx_{2}dx_{3}dx_{4} = \frac{107}{60}, 
&&c_{13} = \int_{P}x_{1}x_{3}dx_{1}dx_{2}dx_{3}dx_{4} = \frac{41}{45}, \\
&c_{14} = \int_{P}x_{1}x_{4}dx_{1}dx_{2}dx_{3}dx_{4} = \frac{239}{180}, 
&&c_{22} = \int_{P}x_{2}^{2}dx_{1}dx_{2}dx_{3}dx_{4} = \frac{69}{20}, \\
&c_{23} = \int_{P}x_{2}x_{3}dx_{1}dx_{2}dx_{3}dx_{4} = \frac{41}{90}, 
&&c_{24} = \int_{P}x_{2}x_{4}dx_{1}dx_{2}dx_{3}dx_{4} = \frac{239}{180}, \\
&c_{33} = \int_{P}x_{3}^{2}dx_{1}dx_{2}dx_{3}dx_{4} = \frac{238}{45}, 
&&c_{34} = \int_{P}x_{3}x_{4}dx_{1}dx_{2}dx_{3}dx_{4} = -\frac{197}{90}, \\
&c_{44} = \int_{P}x_{4}^{2}dx_{1}dx_{2}dx_{3}dx_{4} = \frac{1387}{180}, 
\end{align*}
\begin{align*}
&\theta_P(x_{1}, x_{2}, x_{3}, x_{4}) = \frac{907410{{x}_{1}}-409005{{x}_{3}}-221130}{924382}, \\
&M_{X} = \max_{P}\theta_P = \frac{1095285}{924382} > 1. 
\end{align*}

\subsubsection{{$U_{3}$}}
\begin{align*}
&b_{0} = \mathrm{vol}(P) = \frac{149}{12}, 
&&b_{1} = \int_{P}x_{1}dx_{1}dx_{2}dx_{3}dx_{4} = \frac{5}{2}, \\
&b_{2} = \int_{P}x_{2}dx_{1}dx_{2}dx_{3}dx_{4} = \frac{5}{2}, 
&&b_{3} = \int_{P}x_{3}dx_{1}dx_{2}dx_{3}dx_{4} = \frac{5}{4}, \\
&b_{4} = \int_{P}x_{4}dx_{1}dx_{2}dx_{3}dx_{4} = \frac{5}{4}, 
&&c_{11} = \int_{P}x_{1}^{2}dx_{1}dx_{2}dx_{3}dx_{4} = \frac{107}{30}, \\
&c_{12} = \int_{P}x_{1}x_{2}dx_{1}dx_{2}dx_{3}dx_{4} = \frac{19}{10}, 
&&c_{13} = \int_{P}x_{1}x_{3}dx_{1}dx_{2}dx_{3}dx_{4} = \frac{107}{60}, \\
&c_{14} = \int_{P}x_{1}x_{4}dx_{1}dx_{2}dx_{3}dx_{4} = \frac{19}{20}, 
&&c_{22} = \int_{P}x_{2}^{2}dx_{1}dx_{2}dx_{3}dx_{4} = \frac{107}{30}, \\
&c_{23} = \int_{P}x_{2}x_{3}dx_{1}dx_{2}dx_{3}dx_{4} = \frac{19}{20}, 
&&c_{24} = \int_{P}x_{2}x_{4}dx_{1}dx_{2}dx_{3}dx_{4} = \frac{107}{60}, \\
&c_{33} = \int_{P}x_{3}^{2}dx_{1}dx_{2}dx_{3}dx_{4} = \frac{1109}{180}, 
&&c_{34} = \int_{P}x_{3}x_{4}dx_{1}dx_{2}dx_{3}dx_{4} = \frac{19}{40}, \\
&c_{44} = \int_{P}x_{4}^{2}dx_{1}dx_{2}dx_{3}dx_{4} = \frac{1109}{180}, 
\end{align*}
\begin{align*}
&\theta_P(x_{1}, x_{2}, x_{3}, x_{4}) = \frac{11175{{x}_{1}}+11175{{x}_{2}}-4500}{19936}, \\
&M_{X} = \max_{P}\theta_P = \frac{1275}{1424} < 1. 
\end{align*}

\subsubsection{{$U_{4}$}}
\begin{align*}
&b_{0} = \mathrm{vol}(P) = 12, 
&&b_{1} = \int_{P}x_{1}dx_{1}dx_{2}dx_{3}dx_{4} = 0, \\
&b_{2} = \int_{P}x_{2}dx_{1}dx_{2}dx_{3}dx_{4} = 0, 
&&b_{3} = \int_{P}x_{3}dx_{1}dx_{2}dx_{3}dx_{4} = -2, \\
&b_{4} = \int_{P}x_{4}dx_{1}dx_{2}dx_{3}dx_{4} = 1, 
&&c_{11} = \int_{P}x_{1}^{2}dx_{1}dx_{2}dx_{3}dx_{4} = \frac{10}{3}, \\
&c_{12} = \int_{P}x_{1}x_{2}dx_{1}dx_{2}dx_{3}dx_{4} = \frac{5}{3}, 
&&c_{13} = \int_{P}x_{1}x_{3}dx_{1}dx_{2}dx_{3}dx_{4} = 0, \\
&c_{14} = \int_{P}x_{1}x_{4}dx_{1}dx_{2}dx_{3}dx_{4} = 0, 
&&c_{22} = \int_{P}x_{2}^{2}dx_{1}dx_{2}dx_{3}dx_{4} = \frac{10}{3}, \\
&c_{23} = \int_{P}x_{2}x_{3}dx_{1}dx_{2}dx_{3}dx_{4} = 0, 
&&c_{24} = \int_{P}x_{2}x_{4}dx_{1}dx_{2}dx_{3}dx_{4} = 0, \\
&c_{33} = \int_{P}x_{3}^{2}dx_{1}dx_{2}dx_{3}dx_{4} = 4, 
&&c_{34} = \int_{P}x_{3}x_{4}dx_{1}dx_{2}dx_{3}dx_{4} = -2, \\
&c_{44} = \int_{P}x_{4}^{2}dx_{1}dx_{2}dx_{3}dx_{4} = 6, 
\end{align*}
\begin{align*}
&\theta_P(x_{1}, x_{2}, x_{3}, x_{4}) = \frac{-6{{x}_{3}}-1}{11}, \\
&M_{X} = \max_{P}\theta_P = \frac{5}{11} < 1. 
\end{align*}

\subsubsection{{$U_{5}$}}
\begin{align*}
&b_{0} = \mathrm{vol}(P) = 12, 
&&b_{1} = \int_{P}x_{1}dx_{1}dx_{2}dx_{3}dx_{4} = 0, \\
&b_{2} = \int_{P}x_{2}dx_{1}dx_{2}dx_{3}dx_{4} = 0, 
&&b_{3} = \int_{P}x_{3}dx_{1}dx_{2}dx_{3}dx_{4} = 0, \\
&b_{4} = \int_{P}x_{4}dx_{1}dx_{2}dx_{3}dx_{4} = 0, 
&&c_{11} = \int_{P}x_{1}^{2}dx_{1}dx_{2}dx_{3}dx_{4} = \frac{10}{3}, \\
&c_{12} = \int_{P}x_{1}x_{2}dx_{1}dx_{2}dx_{3}dx_{4} = \frac{5}{3}, 
&&c_{13} = \int_{P}x_{1}x_{3}dx_{1}dx_{2}dx_{3}dx_{4} = 0, \\
&c_{14} = \int_{P}x_{1}x_{4}dx_{1}dx_{2}dx_{3}dx_{4} = 0, 
&&c_{22} = \int_{P}x_{2}^{2}dx_{1}dx_{2}dx_{3}dx_{4} = \frac{10}{3}, \\
&c_{23} = \int_{P}x_{2}x_{3}dx_{1}dx_{2}dx_{3}dx_{4} = 0, 
&&c_{24} = \int_{P}x_{2}x_{4}dx_{1}dx_{2}dx_{3}dx_{4} = 0, \\
&c_{33} = \int_{P}x_{3}^{2}dx_{1}dx_{2}dx_{3}dx_{4} = 4, 
&&c_{34} = \int_{P}x_{3}x_{4}dx_{1}dx_{2}dx_{3}dx_{4} = 0, \\
&c_{44} = \int_{P}x_{4}^{2}dx_{1}dx_{2}dx_{3}dx_{4} = 4, 
\end{align*}
\begin{align*}
&\theta_P(x_{1}, x_{2}, x_{3}, x_{4}) = 0, \\
&M_{X} = \max_{P}\theta_P = 0 < 1. 
\end{align*}

\subsubsection{{$U_{6}$}}
\begin{align*}
&b_{0} = \mathrm{vol}(P) = 12, 
&&b_{1} = \int_{P}x_{1}dx_{1}dx_{2}dx_{3}dx_{4} = \frac{5}{3}, \\
&b_{2} = \int_{P}x_{2}dx_{1}dx_{2}dx_{3}dx_{4} = \frac{5}{6}, 
&&b_{3} = \int_{P}x_{3}dx_{1}dx_{2}dx_{3}dx_{4} = \frac{5}{6}, \\
&b_{4} = \int_{P}x_{4}dx_{1}dx_{2}dx_{3}dx_{4} = 0, 
&&c_{11} = \int_{P}x_{1}^{2}dx_{1}dx_{2}dx_{3}dx_{4} = \frac{10}{3}, \\
&c_{12} = \int_{P}x_{1}x_{2}dx_{1}dx_{2}dx_{3}dx_{4} = \frac{5}{3}, 
&&c_{13} = \int_{P}x_{1}x_{3}dx_{1}dx_{2}dx_{3}dx_{4} = \frac{5}{3}, \\
&c_{14} = \int_{P}x_{1}x_{4}dx_{1}dx_{2}dx_{3}dx_{4} = 0, 
&&c_{22} = \int_{P}x_{2}^{2}dx_{1}dx_{2}dx_{3}dx_{4} = \frac{10}{3}, \\
&c_{23} = \int_{P}x_{2}x_{3}dx_{1}dx_{2}dx_{3}dx_{4} = \frac{5}{6}, 
&&c_{24} = \int_{P}x_{2}x_{4}dx_{1}dx_{2}dx_{3}dx_{4} = 0, \\
&c_{33} = \int_{P}x_{3}^{2}dx_{1}dx_{2}dx_{3}dx_{4} = \frac{17}{3}, 
&&c_{34} = \int_{P}x_{3}x_{4}dx_{1}dx_{2}dx_{3}dx_{4} = 0, \\
&c_{44} = \int_{P}x_{4}^{2}dx_{1}dx_{2}dx_{3}dx_{4} = 4, 
\end{align*}
\begin{align*}
&\theta_P(x_{1}, x_{2}, x_{3}, x_{4}) = \frac{36{{x}_{1}}-5}{67}, \\
&M_{X} = \max_{P}\theta_P = \frac{31}{67} < 1. 
\end{align*}

\subsubsection{{$U_{7}$}}
\begin{align*}
&b_{0} = \mathrm{vol}(P) = \frac{139}{12}, 
&&b_{1} = \int_{P}x_{1}dx_{1}dx_{2}dx_{3}dx_{4} = \frac{5}{6}, \\
&b_{2} = \int_{P}x_{2}dx_{1}dx_{2}dx_{3}dx_{4} = \frac{5}{3}, 
&&b_{3} = \int_{P}x_{3}dx_{1}dx_{2}dx_{3}dx_{4} = \frac{5}{12}, \\
&b_{4} = \int_{P}x_{4}dx_{1}dx_{2}dx_{3}dx_{4} = \frac{5}{12}, 
&&c_{11} = \int_{P}x_{1}^{2}dx_{1}dx_{2}dx_{3}dx_{4} = \frac{31}{10}, \\
&c_{12} = \int_{P}x_{1}x_{2}dx_{1}dx_{2}dx_{3}dx_{4} = \frac{5}{3}, 
&&c_{13} = \int_{P}x_{1}x_{3}dx_{1}dx_{2}dx_{3}dx_{4} = \frac{31}{20}, \\
&c_{14} = \int_{P}x_{1}x_{4}dx_{1}dx_{2}dx_{3}dx_{4} = -\frac{43}{60}, 
&&c_{22} = \int_{P}x_{2}^{2}dx_{1}dx_{2}dx_{3}dx_{4} = \frac{10}{3}, \\
&c_{23} = \int_{P}x_{2}x_{3}dx_{1}dx_{2}dx_{3}dx_{4} = \frac{5}{6}, 
&&c_{24} = \int_{P}x_{2}x_{4}dx_{1}dx_{2}dx_{3}dx_{4} = \frac{5}{6}, \\
&c_{33} = \int_{P}x_{3}^{2}dx_{1}dx_{2}dx_{3}dx_{4} = \frac{931}{180}, 
&&c_{34} = \int_{P}x_{3}x_{4}dx_{1}dx_{2}dx_{3}dx_{4} = -\frac{43}{120}, \\
&c_{44} = \int_{P}x_{4}^{2}dx_{1}dx_{2}dx_{3}dx_{4} = \frac{931}{180}, 
\end{align*}
\begin{align*}
&\theta_P(x_{1}, x_{2}, x_{3}, x_{4}) = \frac{139{{{{x}_{}}}_{2}}-20}{258}, \\
&M_{X} = \max_{P}\theta_P = \frac{119}{258} < 1. 
\end{align*}

\subsubsection{{$U_{8}$}}
\begin{align*}
&b_{0} = \mathrm{vol}(P) = \frac{67}{6}, 
&&b_{1} = \int_{P}x_{1}dx_{1}dx_{2}dx_{3}dx_{4} = 0, \\
&b_{2} = \int_{P}x_{2}dx_{1}dx_{2}dx_{3}dx_{4} = 0, 
&&b_{3} = \int_{P}x_{3}dx_{1}dx_{2}dx_{3}dx_{4} = 0, \\
&b_{4} = \int_{P}x_{4}dx_{1}dx_{2}dx_{3}dx_{4} = 0, 
&&c_{11} = \int_{P}x_{1}^{2}dx_{1}dx_{2}dx_{3}dx_{4} = \frac{43}{15}, \\
&c_{12} = \int_{P}x_{1}x_{2}dx_{1}dx_{2}dx_{3}dx_{4} = \frac{43}{30}, 
&&c_{13} = \int_{P}x_{1}x_{3}dx_{1}dx_{2}dx_{3}dx_{4} = \frac{43}{30}, \\
&c_{14} = \int_{P}x_{1}x_{4}dx_{1}dx_{2}dx_{3}dx_{4} = -\frac{43}{30}, 
&&c_{22} = \int_{P}x_{2}^{2}dx_{1}dx_{2}dx_{3}dx_{4} = \frac{31}{10}, \\
&c_{23} = \int_{P}x_{2}x_{3}dx_{1}dx_{2}dx_{3}dx_{4} = \frac{43}{60}, 
&&c_{24} = \int_{P}x_{2}x_{4}dx_{1}dx_{2}dx_{3}dx_{4} = -\frac{43}{60}, \\
&c_{33} = \int_{P}x_{3}^{2}dx_{1}dx_{2}dx_{3}dx_{4} = \frac{421}{90}, 
&&c_{34} = \int_{P}x_{3}x_{4}dx_{1}dx_{2}dx_{3}dx_{4} = -\frac{43}{60}, \\
&c_{44} = \int_{P}x_{4}^{2}dx_{1}dx_{2}dx_{3}dx_{4} = \frac{421}{90}, 
\end{align*}
\begin{align*}
&\theta_P(x_{1}, x_{2}, x_{3}, x_{4}) = 0, \\
&M_{X} = \max_{P}\theta_P = 0 < 1. 
\end{align*}

\subsubsection{{$\widetilde{V}^{4}$}}
\begin{align*}
&b_{0} = \mathrm{vol}(P) = \frac{307}{24}, 
&&b_{1} = \int_{P}x_{1}dx_{1}dx_{2}dx_{3}dx_{4} = \frac{13}{10}, \\
&b_{2} = \int_{P}x_{2}dx_{1}dx_{2}dx_{3}dx_{4} = \frac{13}{10}, 
&&b_{3} = \int_{P}x_{3}dx_{1}dx_{2}dx_{3}dx_{4} = \frac{13}{10}, \\
&b_{4} = \int_{P}x_{4}dx_{1}dx_{2}dx_{3}dx_{4} = \frac{13}{10}, 
&&c_{11} = \int_{P}x_{1}^{2}dx_{1}dx_{2}dx_{3}dx_{4} = \frac{743}{180}, \\
&c_{12} = \int_{P}x_{1}x_{2}dx_{1}dx_{2}dx_{3}dx_{4} = -\frac{263}{720}, 
&&c_{13} = \int_{P}x_{1}x_{3}dx_{1}dx_{2}dx_{3}dx_{4} = -\frac{263}{720}, \\
&c_{14} = \int_{P}x_{1}x_{4}dx_{1}dx_{2}dx_{3}dx_{4} = -\frac{263}{720}, 
&&c_{22} = \int_{P}x_{2}^{2}dx_{1}dx_{2}dx_{3}dx_{4} = \frac{743}{180}, \\
&c_{23} = \int_{P}x_{2}x_{3}dx_{1}dx_{2}dx_{3}dx_{4} = -\frac{263}{720}, 
&&c_{24} = \int_{P}x_{2}x_{4}dx_{1}dx_{2}dx_{3}dx_{4} = -\frac{263}{720}, \\
&c_{33} = \int_{P}x_{3}^{2}dx_{1}dx_{2}dx_{3}dx_{4} = \frac{743}{180}, 
&&c_{34} = \int_{P}x_{3}x_{4}dx_{1}dx_{2}dx_{3}dx_{4} = -\frac{263}{720}, \\
&c_{44} = \int_{P}x_{4}^{2}dx_{1}dx_{2}dx_{3}dx_{4} = \frac{743}{180}, 
\end{align*}
\begin{align*}
&\theta_P(x_{1}, x_{2}, x_{3}, x_{4}) = \frac{1436760{{x}_{1}}+1436760{{x}_{2}}+1436760{{x}_{3}}}{2766841} + \frac{1436760{{x}_{4}}-584064}{2766841}, \\
&M_{X} = \max_{P}\theta_P = \frac{737568}{395263} > 1. 
\end{align*}

\subsubsection{{$V^{4}$}}
\begin{align*}
&b_{0} = \mathrm{vol}(P) = \frac{115}{12}, 
&&b_{1} = \int_{P}x_{1}dx_{1}dx_{2}dx_{3}dx_{4} = 0, \\
&b_{2} = \int_{P}x_{2}dx_{1}dx_{2}dx_{3}dx_{4} = 0, 
&&b_{3} = \int_{P}x_{3}dx_{1}dx_{2}dx_{3}dx_{4} = 0, \\
&b_{4} = \int_{P}x_{4}dx_{1}dx_{2}dx_{3}dx_{4} = 0, 
&&c_{11} = \int_{P}x_{1}^{2}dx_{1}dx_{2}dx_{3}dx_{4} = \frac{263}{90}, \\
&c_{12} = \int_{P}x_{1}x_{2}dx_{1}dx_{2}dx_{3}dx_{4} = -\frac{263}{360}, 
&&c_{13} = \int_{P}x_{1}x_{3}dx_{1}dx_{2}dx_{3}dx_{4} = -\frac{263}{360}, \\
&c_{14} = \int_{P}x_{1}x_{4}dx_{1}dx_{2}dx_{3}dx_{4} = -\frac{263}{360}, 
&&c_{22} = \int_{P}x_{2}^{2}dx_{1}dx_{2}dx_{3}dx_{4} = \frac{263}{90}, \\
&c_{23} = \int_{P}x_{2}x_{3}dx_{1}dx_{2}dx_{3}dx_{4} = -\frac{263}{360}, 
&&c_{24} = \int_{P}x_{2}x_{4}dx_{1}dx_{2}dx_{3}dx_{4} = -\frac{263}{360}, \\
&c_{33} = \int_{P}x_{3}^{2}dx_{1}dx_{2}dx_{3}dx_{4} = \frac{263}{90}, 
&&c_{34} = \int_{P}x_{3}x_{4}dx_{1}dx_{2}dx_{3}dx_{4} = -\frac{263}{360}, \\
&c_{44} = \int_{P}x_{4}^{2}dx_{1}dx_{2}dx_{3}dx_{4} = \frac{263}{90}, 
\end{align*}
\begin{align*}
&\theta_P(x_{1}, x_{2}, x_{3}, x_{4}) = 0, \\
&M_{X} = \max_{P}\theta_P = 0 < 1. 
\end{align*}

\subsubsection{{$dP_{7} \times dP_{7}$}}
\begin{align*}
&b_{0} = \mathrm{vol}(P) = \frac{49}{4}, 
&&b_{1} = \int_{P}x_{1}dx_{1}dx_{2}dx_{3}dx_{4} = -\frac{7}{6}, \\
&b_{2} = \int_{P}x_{2}dx_{1}dx_{2}dx_{3}dx_{4} = -\frac{7}{6}, 
&&b_{3} = \int_{P}x_{3}dx_{1}dx_{2}dx_{3}dx_{4} = -\frac{7}{6}, \\
&b_{4} = \int_{P}x_{4}dx_{1}dx_{2}dx_{3}dx_{4} = -\frac{7}{6}, 
&&c_{11} = \int_{P}x_{1}^{2}dx_{1}dx_{2}dx_{3}dx_{4} = \frac{91}{24}, \\
&c_{12} = \int_{P}x_{1}x_{2}dx_{1}dx_{2}dx_{3}dx_{4} = -\frac{35}{48}, 
&&c_{13} = \int_{P}x_{1}x_{3}dx_{1}dx_{2}dx_{3}dx_{4} = \frac{1}{9}, \\
&c_{14} = \int_{P}x_{1}x_{4}dx_{1}dx_{2}dx_{3}dx_{4} = \frac{1}{9}, 
&&c_{22} = \int_{P}x_{2}^{2}dx_{1}dx_{2}dx_{3}dx_{4} = \frac{91}{24}, \\
&c_{23} = \int_{P}x_{2}x_{3}dx_{1}dx_{2}dx_{3}dx_{4} = \frac{1}{9}, 
&&c_{24} = \int_{P}x_{2}x_{4}dx_{1}dx_{2}dx_{3}dx_{4} = \frac{1}{9}, \\
&c_{33} = \int_{P}x_{3}^{2}dx_{1}dx_{2}dx_{3}dx_{4} = \frac{91}{24}, 
&&c_{34} = \int_{P}x_{3}x_{4}dx_{1}dx_{2}dx_{3}dx_{4} = -\frac{35}{48}, \\
&c_{44} = \int_{P}x_{4}^{2}dx_{1}dx_{2}dx_{3}dx_{4} = \frac{91}{24}, 
\end{align*}
\begin{align*}
&\theta_P(x_{1}, x_{2}, x_{3}, x_{4}) = -\frac{168{{x}_{1}}+168{{x}_{2}}+168{{x}_{3}}+168{{x}_{4}}+64}{409}, \\
&M_{X} = \max_{P}\theta_P = \frac{608}{409} > 1. 
\end{align*}

\subsubsection{{$dP_{7} \times dP_{6}$}}
\begin{align*}
&b_{0} = \mathrm{vol}(P) = \frac{21}{2}, 
&&b_{1} = \int_{P}x_{1}dx_{1}dx_{2}dx_{3}dx_{4} = -1, \\
&b_{2} = \int_{P}x_{2}dx_{1}dx_{2}dx_{3}dx_{4} = -1, 
&&b_{3} = \int_{P}x_{3}dx_{1}dx_{2}dx_{3}dx_{4} = -1, \\
&b_{4} = \int_{P}x_{4}dx_{1}dx_{2}dx_{3}dx_{4} = 0, 
&&c_{11} = \int_{P}x_{1}^{2}dx_{1}dx_{2}dx_{3}dx_{4} = \frac{13}{4}, \\
&c_{12} = \int_{P}x_{1}x_{2}dx_{1}dx_{2}dx_{3}dx_{4} = -\frac{5}{8}, 
&&c_{13} = \int_{P}x_{1}x_{3}dx_{1}dx_{2}dx_{3}dx_{4} = 0, \\
&c_{14} = \int_{P}x_{1}x_{4}dx_{1}dx_{2}dx_{3}dx_{4} = 0, 
&&c_{22} = \int_{P}x_{2}^{2}dx_{1}dx_{2}dx_{3}dx_{4} = \frac{13}{4}, \\
&c_{23} = \int_{P}x_{2}x_{3}dx_{1}dx_{2}dx_{3}dx_{4} = 0, 
&&c_{24} = \int_{P}x_{2}x_{4}dx_{1}dx_{2}dx_{3}dx_{4} = 0, \\
&c_{33} = \int_{P}x_{3}^{2}dx_{1}dx_{2}dx_{3}dx_{4} = \frac{35}{12}, 
&&c_{34} = \int_{P}x_{3}x_{4}dx_{1}dx_{2}dx_{3}dx_{4} = -\frac{35}{24}, \\
&c_{44} = \int_{P}x_{4}^{2}dx_{1}dx_{2}dx_{3}dx_{4} = \frac{35}{12}, 
\end{align*}
\begin{align*}
&\theta_P(x_{1}, x_{2}, x_{3}, x_{4}) = -\frac{168{{x}_{1}}+168{{x}_{2}}+32}{409}, \\
&M_{X} = \max_{P}\theta_P = \frac{304}{409} < 1. 
\end{align*}

\subsubsection{{$dP_{6} \times dP_{6}$}}
\begin{align*}
&b_{0} = \mathrm{vol}(P) = 9, 
&&b_{1} = \int_{P}x_{1}dx_{1}dx_{2}dx_{3}dx_{4} = 0, \\
&b_{2} = \int_{P}x_{2}dx_{1}dx_{2}dx_{3}dx_{4} = 0, 
&&b_{3} = \int_{P}x_{3}dx_{1}dx_{2}dx_{3}dx_{4} = 0, \\
&b_{4} = \int_{P}x_{4}dx_{1}dx_{2}dx_{3}dx_{4} = 0, 
&&c_{11} = \int_{P}x_{1}^{2}dx_{1}dx_{2}dx_{3}dx_{4} = \frac{5}{2}, \\
&c_{12} = \int_{P}x_{1}x_{2}dx_{1}dx_{2}dx_{3}dx_{4} = -\frac{5}{4}, 
&&c_{13} = \int_{P}x_{1}x_{3}dx_{1}dx_{2}dx_{3}dx_{4} = 0, \\
&c_{14} = \int_{P}x_{1}x_{4}dx_{1}dx_{2}dx_{3}dx_{4} = 0, 
&&c_{22} = \int_{P}x_{2}^{2}dx_{1}dx_{2}dx_{3}dx_{4} = \frac{5}{2}, \\
&c_{23} = \int_{P}x_{2}x_{3}dx_{1}dx_{2}dx_{3}dx_{4} = 0, 
&&c_{24} = \int_{P}x_{2}x_{4}dx_{1}dx_{2}dx_{3}dx_{4} = 0, \\
&c_{33} = \int_{P}x_{3}^{2}dx_{1}dx_{2}dx_{3}dx_{4} = \frac{5}{2}, 
&&c_{34} = \int_{P}x_{3}x_{4}dx_{1}dx_{2}dx_{3}dx_{4} = -\frac{5}{4}, \\
&c_{44} = \int_{P}x_{4}^{2}dx_{1}dx_{2}dx_{3}dx_{4} = \frac{5}{2}, 
\end{align*}
\begin{align*}
&\theta_P(x_{1}, x_{2}, x_{3}, x_{4}) = 0, \\
&M_{X} = \max_{P}\theta_P = 0 < 1. 
\end{align*}

\subsubsection{{$Z_{1}$}}
\begin{align*}
&b_{0} = \mathrm{vol}(P) = \frac{161}{12}, 
&&b_{1} = \int_{P}x_{1}dx_{1}dx_{2}dx_{3}dx_{4} = \frac{31}{30}, \\
&b_{2} = \int_{P}x_{2}dx_{1}dx_{2}dx_{3}dx_{4} = -\frac{31}{15}, 
&&b_{3} = \int_{P}x_{3}dx_{1}dx_{2}dx_{3}dx_{4} = \frac{31}{15}, \\
&b_{4} = \int_{P}x_{4}dx_{1}dx_{2}dx_{3}dx_{4} = \frac{31}{15}, 
&&c_{11} = \int_{P}x_{1}^{2}dx_{1}dx_{2}dx_{3}dx_{4} = \frac{421}{72}, \\
&c_{12} = \int_{P}x_{1}x_{2}dx_{1}dx_{2}dx_{3}dx_{4} = -\frac{217}{80}, 
&&c_{13} = \int_{P}x_{1}x_{3}dx_{1}dx_{2}dx_{3}dx_{4} = \frac{145}{36}, \\
&c_{14} = \int_{P}x_{1}x_{4}dx_{1}dx_{2}dx_{3}dx_{4} = -\frac{1069}{720}, 
&&c_{22} = \int_{P}x_{2}^{2}dx_{1}dx_{2}dx_{3}dx_{4} = \frac{217}{40}, \\
&c_{23} = \int_{P}x_{2}x_{3}dx_{1}dx_{2}dx_{3}dx_{4} = -\frac{1831}{720}, 
&&c_{24} = \int_{P}x_{2}x_{4}dx_{1}dx_{2}dx_{3}dx_{4} = -\frac{1831}{720}, \\
&c_{33} = \int_{P}x_{3}^{2}dx_{1}dx_{2}dx_{3}dx_{4} = \frac{145}{18}, 
&&c_{34} = \int_{P}x_{3}x_{4}dx_{1}dx_{2}dx_{3}dx_{4} = -\frac{1069}{360}, \\
&c_{44} = \int_{P}x_{4}^{2}dx_{1}dx_{2}dx_{3}dx_{4} = \frac{145}{18}, 
\end{align*}
\begin{align*}
&\theta_P(x_{1}, x_{2}, x_{3}, x_{4}) = \frac{598920{{x}_{3}}+598920{{x}_{4}}-184512}{1289443}, \\
&M_{X} = \max_{P}\theta_P = \frac{1013328}{1289443} < 1. 
\end{align*}

\subsubsection{{$Z_{2}$}}
\begin{align*}
&b_{0} = \mathrm{vol}(P) = \frac{109}{8}, 
&&b_{1} = \int_{P}x_{1}dx_{1}dx_{2}dx_{3}dx_{4} = -\frac{19}{12}, \\
&b_{2} = \int_{P}x_{2}dx_{1}dx_{2}dx_{3}dx_{4} = \frac{19}{6}, 
&&b_{3} = \int_{P}x_{3}dx_{1}dx_{2}dx_{3}dx_{4} = \frac{9}{10}, \\
&b_{4} = \int_{P}x_{4}dx_{1}dx_{2}dx_{3}dx_{4} = \frac{9}{10}, 
&&c_{11} = \int_{P}x_{1}^{2}dx_{1}dx_{2}dx_{3}dx_{4} = \frac{46}{9}, \\
&c_{12} = \int_{P}x_{1}x_{2}dx_{1}dx_{2}dx_{3}dx_{4} = -\frac{303}{80}, 
&&c_{13} = \int_{P}x_{1}x_{3}dx_{1}dx_{2}dx_{3}dx_{4} = \frac{1351}{720}, \\
&c_{14} = \int_{P}x_{1}x_{4}dx_{1}dx_{2}dx_{3}dx_{4} = -\frac{61}{40}, 
&&c_{22} = \int_{P}x_{2}^{2}dx_{1}dx_{2}dx_{3}dx_{4} = \frac{303}{40}, \\
&c_{23} = \int_{P}x_{2}x_{3}dx_{1}dx_{2}dx_{3}dx_{4} = -\frac{253}{720}, 
&&c_{24} = \int_{P}x_{2}x_{4}dx_{1}dx_{2}dx_{3}dx_{4} = -\frac{253}{720}, \\
&c_{33} = \int_{P}x_{3}^{2}dx_{1}dx_{2}dx_{3}dx_{4} = \frac{2327}{360}, 
&&c_{34} = \int_{P}x_{3}x_{4}dx_{1}dx_{2}dx_{3}dx_{4} = -\frac{257}{72}, \\
&c_{44} = \int_{P}x_{4}^{2}dx_{1}dx_{2}dx_{3}dx_{4} = \frac{2327}{360}, 
\end{align*}
\begin{align*}
&\theta_P(x_{1}, x_{2}, x_{3}, x_{4}) = \frac{1384138680{{x}_{2}}+1120256220{{x}_{4}}+1120256220{{x}_{3}}}{2592694163} - \frac{469692992}{2592694163},\\
&M_{X} = \max_{P}\theta_P = \frac{3418840588}{2592694163} > 1. 
\end{align*}

\subsubsection{$W$}
\begin{align*}
&b_{0} = \mathrm{vol}(P) = \frac{83}{8}, 
&&b_{1} = \int_{P}x_{1}dx_{1}dx_{2}dx_{3}dx_{4} = 0, \\
&b_{2} = \int_{P}x_{2}dx_{1}dx_{2}dx_{3}dx_{4} = 0, 
&&b_{3} = \int_{P}x_{3}dx_{1}dx_{2}dx_{3}dx_{4} = 0, \\
&b_{4} = \int_{P}x_{4}dx_{1}dx_{2}dx_{3}dx_{4} = 0, 
&&c_{11} = \int_{P}x_{1}^{2}dx_{1}dx_{2}dx_{3}dx_{4} = \frac{463}{120}, \\
&c_{12} = \int_{P}x_{1}x_{2}dx_{1}dx_{2}dx_{3}dx_{4} = -\frac{463}{240}, 
&&c_{13} = \int_{P}x_{1}x_{3}dx_{1}dx_{2}dx_{3}dx_{4} = -\frac{601}{360}, \\
&c_{14} = \int_{P}x_{1}x_{4}dx_{1}dx_{2}dx_{3}dx_{4} = \frac{601}{720}, 
&&c_{22} = \int_{P}x_{2}^{2}dx_{1}dx_{2}dx_{3}dx_{4} = \frac{463}{120}, \\
&c_{23} = \int_{P}x_{2}x_{3}dx_{1}dx_{2}dx_{3}dx_{4} = \frac{601}{720}, 
&&c_{24} = \int_{P}x_{2}x_{4}dx_{1}dx_{2}dx_{3}dx_{4} = -\frac{601}{360}, \\
&c_{33} = \int_{P}x_{3}^{2}dx_{1}dx_{2}dx_{3}dx_{4} = \frac{463}{120}, 
&&c_{34} = \int_{P}x_{3}x_{4}dx_{1}dx_{2}dx_{3}dx_{4} = -\frac{463}{240}, \\
&c_{44} = \int_{P}x_{4}^{2}dx_{1}dx_{2}dx_{3}dx_{4} = \frac{463}{120}, 
\end{align*}
\begin{align*}
&\theta_P(x_{1}, x_{2}, x_{3}, x_{4}) = 0, \\
&M_{X} = \max_{P}\theta_P = 0 < 1. 
\end{align*}


\begin{thebibliography}{A}

\bibitem{ACGT-F08}
V.~Apostolov, D.M.J.~Calderbank, P.~Gauduchon, and C.W.~T\o nnesen-Friedman, 
Hamiltonian 2-forms in K\"ahler geometry. III. Extremal metrics and stability. 
\emph{Invent. Math.} {\bf{173}} (2008), 547--601.


\bibitem{BCT-F19}
C.P.~Boyer, D.M.J.~Calderbank, and C.W. T\o nnesen-Friedman, 
The K\"ahler geometry of Bott manifolds. 
\emph{Adv. Math.} {\bf{350}} (2019), 1--62.


\bibitem{A82}
M.~F.~Atiyah, 
Convexity and commuting hamiltonians. 
\emph{Bull. London Math. Soc.} \textbf{14} (1982), 1--15. 

\bibitem{B81}
V.~V.~Batyrev, 
Troidal Fano 3-folds.  
\emph{Izv. Akad. Nauk SSSR.} \textbf{45} (1981), 704--717. 

\bibitem{B98}
V.~V.~Batyrev, 
On the classification of toric Fano 4-folds. 
\emph{J. Math. Sci.} \textbf{94} (1999), 1021--1050. 

\bibitem{BS99}
V.~V.~Batyrev and E.~Selivanova, 
Einstein-K\"ahler metrics on symmetric toric Fano manifolds. 
\emph{J. Reine Angew. Math.} \textbf{512} (1999), 225--236. 


\bibitem{BBJ21}
R.~J.~Berman, S.~Boucksom, and M.~Jonsson, 
A variational approach to the Yau-Tian-Donaldson conjecture. 
\emph{ J. Amer. Math.}
\textbf{34} (2021) 605--652.


\bibitem{B16}
R.~J.~Berman, 
$K$-polystability of $\mathbf{Q}$-Fano varieties admitting K\"ahler-Einstein metrics. 
\emph{Invent. Math.} \textbf{203} (2016), 973--1025. 

\bibitem{BHJ17}
S.~Boucksom, T.~Hisamoto, and M.~Jonsson, 
Uniform $K$-stability, Duistermaat-Heckman measures and singularity of pairs. 
\emph{Ann. Inst. Fourier (Grenoble)} \textbf{67} (2017), 743--841. 


\bibitem{Normaliz}
W.~Bruns, B.~Ichim, T.~R\"omer and C.~S\"oger,
Normaliz. Algorithms for rational cones and affine monoids. Available from http://www.math.uos.de/normaliz.



\bibitem{CC18}
X.~Chen and J.~Cheng,
On the constant scalar curvature K\"ahler metrics (II) -- Existence results.
arXiv:1801.00656



\bibitem{CLMP20}
Y.~Cho, E.~Lee, M.~Masuda, and S.~Park,
Unique toric structure on a Fano Bott manifold.
arXiv:2005.02740



\bibitem{CLMP21}
Y.~Cho, E.~Lee, M.~Masuda, and S.~Park,
On the enumeration of Fano Bott manifolds.
arXiv:2106.12788



\bibitem{CLS11}
D.~A.~Cox, J.~B.~Little, and H.~K.~Schenck, 
\emph{Toric varieties}. 
Graduate Studies in Mathematics \textbf{124}. \emph{American Mathematical Society, Providence, RI}, (2011). xxiv+841 pp. 


\bibitem{D20}
T.~Delcroix,
Uniform $K$-stability of polarized spherical varieties.
arXiv:2009.06463v2.


\bibitem{D02}
S.~K.~Donaldson, 
Scalar curvature and stability of toric varieties.
\emph{J. Differential Geom.} \textbf{62} (2002), 289--349. 

\bibitem{D17}
S.~K.~Donaldson, 
The Ding Functional, Berndtsson Convexity and Moment Maps. 
in \emph{Geometry, Analysis and Probability, Progress in Math.} \textbf{310} (2017), 57--67. 

\bibitem{F18}
K.~Fujita, 
Optimal bounds for the volumes of K\"ahler-Einstein Fano manifolds. 
\emph{Amer. J. Math.} \textbf{140} (2018), 391--414. 

\bibitem{FM95} A.~Futaki and T.~Mabuchi, 
Bilinear forms and extremal K\"ahler vector fields associated with K\"ahler classes. 
\emph{Math. Ann.} \textbf{301} (1995), 199--210. 


\bibitem{G95}
D.~Guan, 
Existence of extremal metrics on compact almost homogeneous K\"ahler manifolds with two ends. 
\emph{Trans. Amer. Math. Soc.} \textbf{347}  (1995), 2255--2262.


\bibitem{GS82}
V.~Guillemin and S.~Sternberg, 
Convexity properties of the moment mapping I and II. 
\emph{Invent. Math.} \textbf{67} (1982), 491--513; \emph{Invent. Math.} \textbf{77} (1984), 533--546. 


\bibitem{HL22}
J.~Han and C.~Li,
On the Yau-Tian-Donaldson conjecture for generalized K\"ahler-Ricci soliton equations.
To appear in \emph{Commun. Pur. Appl. Math.}
arXiv:2006.00903



\bibitem{HKM20}
A.~Higashitani, K.~Kurimoto, and M.~Masuda,
Cohomological rigidity for toric Fano manifolds of small dimensions or large Picard numbers.
arXiv:2005.13795


\bibitem{H16}
T.~Hisamoto, 
Stability and coercivity for toric polarizations. 
arXiv:1610.07998v2. 
 


\bibitem{H94}
A.D.~Hwang, 
On existence of K\"ahler metrics with constant scalar curvature. 
Osaka J. Math. \textbf{31} (1994), 561--595.



\bibitem{LL22}
Y.~Li and Z.~Li,
Equivariant $\mathbb R$-test configurations of polarized spherical varieties.
arXiv:2206.04880v3.


\bibitem{M01}
T.~Mabuchi, 
K\"ahler-Einstein metrics for manifolds with nonvanishing Futaki character. 
\emph{Tohoku Math. J.} \textbf{53} (2001), 171--182. 

\bibitem{N98}
Y.~Nakagawa, 
Combinatorial formulae for Futaki characters and generalized killing forms of toric Fano orbifolds. 
\emph{The Third Pacific Rim Geometry Conference (Seoul, 1996)}, 223--260, Monogr. Geom. Topology \textbf{25}, \emph{Int. Press, Cambridge, MA}, 1998. 


\bibitem{NS19}
Y.~Nitta and S.~Saito, 
A uniform version of the Yau-Tian-Donaldson correspondence for extremal K\"ahler metrics on polarized toric manifolds. 
arXiv:2110.10386. 


\bibitem{S00}
H.~Sato, 
Toward the classification of higher-dimensional toric Fano varieties. 
\emph{Tohoku Math. J.} \textbf{52} (2000) 383--413.

\bibitem{S07}
G.~Sz\'ekelyhidi, 
Extremal metrics and K-stability. 
\emph{Bull. Lond. Math. Soc.} \textbf{39} (2007), 76--84. 

\bibitem{WW82}
K.~Watanabe and M.~Watanabe, 
The classification of Fano 3-folds with torus embeddings. 
\emph{Tokyo. J. Math.} \textbf{5} (1982), 37--48. 

\bibitem{Y17}
Y.~Yao, 
Mabuchi solitons and relative Ding stability of toric Fano varieties. 
\emph{Int. Math. Res. Not.}, published electronically 22 September 2021,
\newblock  ~ {\tt{https://doi.org/10.1093/imrn/rnab226}}
 

\bibitem{Y19}
Y.~Yao, 
Relative Ding stability and an obstruction to the existence of Mabuchi solitons. 
J. Geom. Anal. \textbf{32} (2022), Article No. 105, 51 pp.
 


\bibitem{YZ19}
N.~Yotsutani and B.~Zhou,
Relative algebro-geometric stabilities of toric manifolds.
Tohoku Math. J. \textbf{71} (2019), 495--524. 



\bibitem{ZZ08}
B.~Zhou and X.~Zhu,
$K$-stability on toric manifolds.
\emph{Proc. Amer. Math. Soc.} \textbf{136} (2008), 3301--3307.


\end{thebibliography}
\end{document}